\documentclass[a4paper,11pt]{article}
\input{preamble}

\title{
  Strong local nondeterminism for stochastic time-fractional\\ slow and fast diffusion equations
}
\author{
 Le Chen\footnote{Department of Mathematics and Statistics, Auburn University, Auburn, Al., USA. Email: \href{mailto:le.chen@auburn.edu}{le.chen@auburn.edu}.}
 \and
 Cheuk Yin Lee\footnote{School of Science and Engineering, The Chinese University of Hong Kong (Shenzhen), Longgang, Shenzhen, Guangdong, 518172, China. Email: \href{mailto:leecheukyin@cuhk.edu.cn}{leecheukyin@cuhk.edu.cn}.}
 \and
 Panqiu Xia\footnote{School of Mathematics, Cardiff University, Cardiff, Wales, UK, CF24 4AG. Email: \href{mailto:xiap@cardiff.ac.uk}{xiap@cardiff.ac.uk}.}
}
\date{\small\today}

\begin{document}
% \tdplotsetmaincoords{70}{120}
\maketitle

\begin{abstract}
	We study a class of stochastic time-fractional equations on $\R^d$ driven by a
	centered Gaussian noise, involving a Caputo time derivative of order
	$\beta>0$, a fractional (power) Laplacian of order $\alpha>0$, and a
	Riemann--Liouville time integral of order $\gamma\ge0$ acting on the noise.
	The noise is fractional in time (index $H$) and Riesz-type in space (index
	$\ell$). We derive sharp Dalang-type necessary and sufficient conditions for
	the existence of a random field solution across almost full parameter range
	$(\alpha,\beta,\gamma;H,\ell)$. Under the Dalang-type conditions, we prove sharp variance bounds
	for temporal and spatial increments, as well as strong local nondeterminism in
	time in several regimes (two-sided version for $\beta=1$ and for parts of the case
	$\beta=2$; one-sided version for $0<\beta<2$) and strong local nondeterminism
	in space for the whole range of parameters. As applications, we derive exact
	uniform and local moduli of continuity, Chung-type laws of the iterated
	logarithm, and quantitative bounds on small ball probabilities. Along the way,
	we obtain sharp asymptotics for the fundamental solution kernels at $0$ and
	$\infty$, which may be of independent interest.

	\bigskip \noindent{\textit{Keywords.}}
	Stochastic time-fractional diffusion equations\index{stochastic time-fractional diffusion equation};
	fractional Laplacian\index{fractional Laplacian};
	Dalang-type condition\index{Dalang condition};
	fundamental solution asymptotics\index{fundamental solution!asymptotics};
	centered Gaussian noise\index{Gaussian noise!centered};
	strong local nondeterminism\index{strong local nondeterminism (SLND)};
	modulus of continuity\index{modulus of continuity};
	laws of the iterated logarithm\index{law of the iterated logarithm (LIL)};
	small ball probabilities\index{small ball probability};
	Fox H-function\index{Fox H-function}.

	\medskip\noindent\textit{2020 Mathematics Subject Classifications.~} Primary
	\MFC{60G17}, \MFC{60H15}; Secondary \MFC{60G15}, \MFC{60G22}, \MFC{33G12}.

	\medskip\noindent\textit{Acknowledgements.~}
	L.~C. and P.~X. were partially supported by NSF grant DMS-2246850. L.~C. was
	partially supported by NSF grant DMS-2443823 and a Collaboration Grant for
	Mathematicians (\#959981) from the Simons Foundation. C.-Y.~L. was supported
	in part by the Shenzhen Peacock grant 2025TC0013.

\end{abstract}

% Change link color to black for TOC, LOF, LOT
\hypersetup{linkcolor=black}
\tableofcontents

\listoffigures
\addcontentsline{toc}{section}{List of Figures}

\listoftables
\addcontentsline{toc}{section}{List of Tables}
% Reset link color to default (red) after lists
\hypersetup{linkcolor=red}

\section{Introduction}

In this paper, we study the following \textit{stochastic partial differential
  equation}\index{stochastic partial differential equation (SPDE)} (SPDE) on $\R^d$ with fractional differential and integral operators,
perturbed by a centered Gaussian noise\index{Gaussian noise}\index{Gaussian noise!centered}:
\begin{equation}\label{E:fde}
  \left(\partial_t^{\beta}+\dfrac{\nu}{2}\left(-\Delta\right)^{\alpha / 2}\right) u(t, x)
  = I_{t}^{\gamma}\left[\dot{W}(t, x)\right], \quad  t>0,\, x \in \mathbb{R}^{d},\, \nu>0.
\end{equation}

In~\eqref{E:fde}, the operator $\Delta$ is the Laplace operator and the exponent
$\alpha>0$ represents its power, i.e., the fractional (power) Laplacian of order $\alpha>0$\index{fractional Laplacian}. The symbol $\partial_t^\beta$ denotes the
\emph{Caputo fractional differential}\index{Caputo fractional derivative} operator of order $\beta> 0$:
\begin{align*}
  \partial_t^{\beta} f(t)\coloneqq
  \begin{dcases} \dfrac{1}{\Gamma(n-\beta)}\int_{0}^{t}\frac{f^{(n)}(\tau)}{(t-\tau)^{\beta+1-n}}\ud\tau, & \text{if $\beta \neq n$}, \\
               \dfrac{\ud^n}{\ud t^n}f(t),                                                              & \text{if $\beta = n$},
  \end{dcases}
\end{align*}
where $n = \Ceil{\beta} \coloneqq \min \{n \in \mathbb{Z} \colon n \geq \beta\}$
(i.e., $\Ceil{\cdot}$ is the ceiling function), and $\Gamma(\cdot)$ is the
\textit{gamma function}. We use $I_t^\gamma$ with $\gamma \ge 0$ to refer to the
\textit{Riemann-Liouville integral}\index{Riemann-Liouville integral} in the time variable to the right of zero:
\begin{align*}
  I_t^\gamma f(t) = \frac{1}{\Gamma(\gamma)} \int_0^t (t-\tau)^{\gamma-1} f(\tau) \ud \tau.
\end{align*}
We impose zero initial conditions, namely,
\begin{align*}
  u(0,\cdot) = 0\quad \text{if $\beta\in (0,1]$} \quad \text{and} \quad
  u(0,\cdot) = \frac{\partial}{\partial t} u(0, \cdot) = 0 \quad \text{if $\beta \in(1,2]$.}
\end{align*}

The parameters $(\alpha,\beta,\gamma)$ specify the deterministic evolution and
its memory, while the noise $\dot W$ (defined precisely below) is parameterized
by two indices $(H,\ell)$ that encode temporal and spatial roughness.
Equation~\eqref{E:fde} generates a rich family of Gaussian random fields whose
correlation structure can be highly nonlocal and typically does not exhibit
independent increments in time or space. A useful tool for studying fine sample
path properties of such fields is \emph{strong local
  nondeterminism}\index{strong local nondeterminism (SLND)} (SLND). The notion of
\textit{local nondeterminism}\index{local nondeterminism (LND)} (LND) was first
introduced by Berman~\cite{berman:73:local} in the 1970s to study local times of
Gaussian processes and extended to Gaussian random fields by
Pitt~\cite{pitt:78:local}. In the 1980s, SLND was further developed by Cuzick
and DuPreez~\cite{cuzick.dupreez:82:joint} and Monrad and
Pitt~\cite{monrad.pitt:87:local}. Since then, SLND has been effectively used to
analyze sample path properties of Gaussian random fields; see, e.g., Xiao's
articles~\cites{xiao:06:properties, xiao:08:strong, xiao:09:sample} and the
references therein.

Let us recall that a real-valued Gaussian random field $X=\{X(z): z \in I\}$
satisfies LND on $I \subset \R^N$ if for every $n \in \mathbb{N}$, there exist
constants $C>0$ and $\delta>0$ such that
\begin{align}\label{def:LND}
  \Var(X(z)|X(z_1),\dots,X(z_n)) \ge C \min_{1 \le i \le n} \Var(X(z)-X(z_i))
\end{align}
uniformly for all $z, z_1,\dots, z_n \in I$ with $\max_{1 \le i \le n}|z-z_i|
  \le \delta$. Moreover, $X$ is said to satisfy SLND if the constants $C>0$ and
$\delta>0$ do not depend on $n$ and \eqref{def:LND} holds uniformly for all $n
  \in \mathbb{N}$ and for all $z, z_1,\dots, z_n \in I$ with $\max_{1 \le i \le
    n}|z-z_i| \le \delta$.

The aim of this paper is to establish sharp temporal and spatial increment
estimates\index{increment estimates} for the solution to~\eqref{E:fde}, and to derive from them strong local
nondeterminism properties and optimal fine sample path results.
As special cases, SLND for solutions to stochastic heat and wave equations in
various parameter regimes has been proved and applied to study moduli of
continuity, local times, and other path properties; see, for instance,
~\cites{herrell.song.ea:20:sharp, lee.xiao:23:chung-type, lee.xiao:23:local,
  khoshnevisan.lee:25:on, lee.xiao:19:local, lee:22:local, hu.lee:25:on,
  khoshnevisan.lee.ea:25:uniform} and the references therein.

\subsection{Motivation}

There are multiple reasons to study equation~\eqref{E:fde}. From the
probabilistic point of view,~\eqref{E:fde} interpolates between several
classical parabolic and hyperbolic SPDEs and generates Gaussian random fields
with long-range temporal dependence and nonlocal spatial interactions. While
existence and rough regularity are well understood in a number of classical
cases (e.g., heat or wave equations driven by white/fractional noise), and our
results in particular recover and interpolate many known results for the
stochastic heat ($\beta=1$) and wave ($\beta=2$) equations. Sharp two-sided
increment estimates and strong local nondeterminism across the full parameter
range $(\alpha,\beta,\gamma;H,\ell)$ are much less developed. The goal of this
paper is to provide such sharp estimates, which then yield optimal H\"older-type
regularity and related fine sample path properties.
Across the three motivations below, our analysis follows a common probabilistic
scheme: first, we identify sharp necessary and sufficient conditions for
well-posedness (existence of a random field solution\index{random field solution}) in the form of a
Dalang-type integrability condition; next, we prove sharp two-sided variance
bounds\index{variance bounds} for temporal and spatial increments\index{variance bounds!temporal}\index{variance bounds!spatial}; then, we establish strong local
nondeterminism from these increment bounds. These inputs yield exact uniform and
local moduli of continuity, Chung-type laws of the iterated
logarithm, as well as quantitative small ball probability bounds for the
temporal and spatial processes (see Corollaries~\ref{C:MC}, \ref{C:LIL},
\ref{C:Chung}, and~\ref{C:smallball}).

\paragraph{Diffusion and anomalous diffusion.}

To explain the emergence of fractional diffusion equations\index{fractional diffusion equation}\index{time-fractional diffusion equation}, let us consider a stochastic process $\left\{X_t, t\ge 0\right\}$ and denote its mean
square displacement (MSD) by
\begin{align*}
  \InPrd{X^2_t} \coloneqq \E\left(\left[X_t-\E(X_t)\right]^2\right).
\end{align*}
If $X_t$ is a normal diffusion (e.g., Brownian motion), then
$\InPrd{X^2_t} = C t$ for all $t\ge 0$. When this linear growth law is violated,
the diffusion is called \textit{anomalous}\index{anomalous diffusion}. A commonly used scaling ansatz is
$\InPrd{X^2_t} \asymp C t^\theta$ as $t\to\infty$, where the exponent $\theta$
distinguishes subdiffusion\index{subdiffusion} ($\theta<1$) from superdiffusion\index{superdiffusion} ($\theta>1$); see
Table~\ref{Tb:Diffusion} and, e.g., \cite{luchko:12:anomalous}.

\begin{table}[htpb]
  \centering
  \caption{Classes of anomalous diffusions.}
  \label{Tb:Diffusion}

  \renewcommand{\arraystretch}{1.2}
  \begin{tabular}{|c|c|c|c|c|}
    \hline \rowcolor{gray!30}
    $\theta \in (0,1)$ & $\theta = 1$     & $\theta \in (1,2)$ & $\theta = 2$     & $\theta >2$           \\ \hline
    Subdiffusion       & Normal diffusion & Superdiffusion     & Ballistic motion & Hyperballistic motion \\
    Slow diffusion     &                  & Fast diffusion     &                  &                       \\ \hline
  \end{tabular}
\end{table}

Normal diffusion is ubiquitous in the physical sciences, in part because the
\textit{central limit theorem} yields universal Gaussian scaling when increment
distributions have finite variance. In complex media, heavy-tailed waiting times
or jump lengths can lead to anomalous scaling; see, for instance,
\cites{sagi.brook.ea:12:observation, saporta-katz.efrati:19:self-driven,
metzler.klafter:00:random} and the references therein.

Fractional diffusion equations are recognized as standard and powerful tools to
characterize and model anomalous diffusion phenomena; see the survey paper by
Metzler and Klafter (2000)~\cite{metzler.klafter:00:random}, or the
books~\cites{evangelista.lenzi:18:fractional,
  meerschaert.sikorskii:19:stochastic}. Specifically, as formulated
in~\cites{evangelista.lenzi:18:fractional, meerschaert.sikorskii:19:stochastic},
the microscopic prescription of the diffusions can be given by the
\textit{continuous time random walk}\index{continuous time random walk (CTRW)} (CTRW). CTRW consists of two distributions
(which may be jointly distributed): the waiting time distribution and the jump
length distribution. In~\eqref{E:fde}, the case $\beta\in (0,1)$ is associated
with a heavy-tailed waiting time distribution; the smaller the value of $\beta$,
the longer the waiting time and the more profound the memory effect. On the
other hand, the case $\alpha \in (0,2)$ in~\eqref{E:fde} aligns with a
heavy-tailed jump length distribution (L\'evy flight); a smaller value of
$\alpha$ leads to a faster spread. The parameter $\gamma$ in~\eqref{E:fde}
provides a flexible mechanism to adjust the memory effects of the external
forcing on the system. Specifically, a smaller $\gamma$ value corresponds to
shorter memory effects, while a larger $\gamma$ leads to a more pronounced
effect of the historical external stimuli.

Note that for anomalous diffusions, one usually requires $\beta\le 1$ and
$\alpha\in (0,2]$. A probabilistic representation for the case $\beta\in (1,2)$
is less clear and trickier; see Meerschaert et al.
\cite{meerschaert.schilling.ea:15:stochastic}. In what follows, we will provide
different motivations for the study of the parameter range $\beta\in (1,2]$.
In particular, our results cover this regime and quantify, in terms of
$(\alpha,\beta,\gamma;H,\ell)$, sharp two-sided variance bounds and SLND
exponents (whenever a random field solution exists).

\paragraph{Viscoelasticity and rheology.}

In materials science and continuum mechanics, \textit{viscoelasticity}\index{viscoelasticity} is the
property of materials that display both viscous and elastic characteristics
during deformation; see, e.g., \cites{dimitrienko:11:nonlinear,
  lemaitre.jean-louis:90:mechanics, perkins.lach:11:viscoelasticity}. Viscous
liquids (e.g., water) resist shear flow and, under constant stress, their
deformation grows with time. In contrast, elastic solids deform under stress and
(ideally) recover their original shape once the stress is removed. Viscoelastic
materials (e.g., honey, polymers, and biological tissues) exhibit both effects.
The study of the deformation and flow of such materials is called
\textit{rheology}\index{rheology}; see, e.g., \cites{macosko:96:rheology, tanner:00:engineering,
  barnes.hutton.ea:89:introduction}. Fractional calculus has long played a central
role in linear viscoelasticity because it provides parsimonious constitutive
laws with power-law memory/relaxation; see
Mainardi~\cite{mainardi:10:fractional} and the references therein. Applications
include chemistry (e.g., \cites{doi.edwards:86:theory, ferry:80:viscoelastic}),
seismology (e.g., \cites{aki.richards:09:quantitative, carcione:08:theory,
  zhu.harris:14:modeling}), and bioengineering (e.g.,
\cites{treeby.cox:10:modeling, caputo.carcione.ea:11:wave}).

In the context of~\eqref{E:fde}, $\beta = 2$ corresponds to ballistic motion
(solid-like behavior), while $\beta = 1$ corresponds to diffusion (liquid-like
materials). The case $\beta\in (1,2)$ represents an interpolation of these two
states. Together with the choice of $\alpha$ in~\eqref{E:fde}, one has a model
for describing the wave propagation in viscoelastic materials.  Including noise
in~\eqref{E:fde} accounts for the wave propagation in viscoelastic medium
subject to random perturbations, with the parameter $\gamma$ controlling the
memory effect of the medium.

Seismic wave propagation provides a representative example:
at relevant scales, Earth is often modeled as a heterogeneous, attenuating
viscoelastic medium, and constant-$Q$ models naturally lead to fractional
operators (either in time or space); see, e.g.,
\cites{zhu.harris:14:modeling, carcione:08:theory}. In this spirit,
equation~\eqref{E:fde} with $\beta\in(1,2]$ provides a deterministic model for
wave propagation with memory, while the Riemann--Liouville factor $I_t^\gamma$
allows additional history dependence in the forcing. Adding Gaussian noise then
accounts for unresolved fluctuations, and motivates our sharp increment
estimates and SLND analysis for the resulting random field.

\paragraph{Fractional wave equations.}

When $\beta\in(1,2)$, equation~\eqref{E:fde} interpolates between diffusion
($\beta=1$) and wave propagation ($\beta=2$), and is often viewed as a
wave--diffusion (or diffusion--wave) model. Related fractional wave equations,
integrodifferential equations with memory, and their fundamental solutions have
been studied extensively in the deterministic literature; see, e.g.,
\cites{schneider.wyss:89:fractional, fujita:90:integrodifferential,
  orsingher.beghin:09:fractional, eidelman.kochubei:04:cauchy,
  luchko:13:fractional, luchko:14:multi-dimensional, hanyga:02:multidimensional,
  luchko:15:wave-diffusion, kochubeui:89:cauchy}. On the analytic side, in the
deterministic setting (ignoring the noise and $I_t^\gamma$), Laplace--Fourier
analysis yields the symbol $s^\beta+\frac{\nu}{2}|\xi|^\alpha$ and solution
operators governed by Mittag--Leffler\index{Mittag-Leffler function}/Wright\index{Wright function}/Fox-type\index{Fox H-function} special functions. In the
stochastic setting, these kernels control the Dalang-type integrability
condition and the sharp temporal and spatial increment exponents (and hence the
SLND property). Our main results make this connection quantitative for
equation~\eqref{E:fde} across the full parameter range, including the additional
memory parameter $\gamma$. In particular, the sharp estimates on the fundamental
solution kernels developed in Section~\ref{S:Asymptotics} are of independent
interest and may be useful for related fractional models beyond the present
SPDE.

\subsection{The centered Gaussian noise}

The noise $\dot{W}$ in~\eqref{E:fde} is generated by an isonormal Gaussian
process\index{isonormal Gaussian process} $W =\left\{W(\phi) \colon \phi \in \mathcal{H}\right\}$ (see, e.g.,
\cite{nualart:06:malliavin}*{Definition 1.1.1}) on a Hilbert space $\mathcal{H}$
of functions on $\R_+\times \R^d$. Namely, $W$ is a centered Gaussian family
with covariance given by the following inner product:
\begin{align}\label{E:Inner_1}
  \langle \phi, \psi \rangle_{\mathcal{H}}
  = \E \big[W(\phi) W(\psi) \big]
  \coloneqq \int_\R \ud \tau \int_{\R^d} \ud \xi \; \Lambda_H (\tau) K_{\ell} (\xi)  \mathcal{F}(\phi)(\tau,\xi) \overline{\mathcal{F}(\psi)(\tau,\xi)}\;,
\end{align}
for all $\phi$ and $\psi \in C_c^{\infty} (\R_+ \times \R^d)$ (smooth functions
with compact support), where
\begin{align*}
  \Lambda_H (\tau) \coloneqq |\tau|^{1-2H} \quad \text{and} \quad K_{\ell} (\xi) \coloneqq |\xi|^{\ell - d},\qquad  (\tau,\xi) \in \R\times\R^d,
\end{align*}
and $\mathcal{F}(\phi)$ denotes the space-time Fourier transformation of $\phi$,
namely,
\begin{align*}
  \mathcal{F}(\phi)(\tau,\xi) \coloneqq \int_\R \ud t \int_{\R^d} \ud x \; \phi(t,x) e^{-i(\tau t + \xi \cdot x)}.
\end{align*}
We will also use the notation $\widehat{\phi}$ to denote the Fourier transform
in only one variable (either time or space). The noise is parameterized by two
parameters $(H,\ell)$, with respect to which we have the following cases:
\begin{enumerate}[(i)]

  \item If $H = 1/2$ and $\ell = d$, then $W$ is a Brownian sheet on
        $\R_+\times \R^d$ and $\dot{W}$ is the \textit{space-time white noise}\index{white noise!space-time}.

  \item If $H \in (\sfrac{1}{2},1)$ and $\ell < d$, the noise is
        \textit{regular} in both time and space. In this case, the (inverse) Fourier
        transforms $\widehat{\Lambda_H}$ and $\widehat{K_\ell}$ are function-valued,
        and thus $\mathcal{H}$ can be understood as a usual Hilbert space of
        functions. The inner product in~\eqref{E:Inner_1} can be equivalently
        written as
        \begin{align*} % \label{E:Inner_2}
          \langle \phi, \psi \rangle_{\mathcal{H}}
          =  (2\pi)^{d+1} \iint_{\R_+^2} \ud s \,\ud t \iint_{\R^{2d}} \ud x\, \ud y \; \widehat{\Lambda_H}(t-s) \widehat{K_\ell}(x-y) \phi(t,x) \psi(s,y).
        \end{align*}

  \item If $H\in (\sfrac{1}{2},1)$ and $\ell \ge d$, the noise is regular in
        time and white ($\ell=d$) or rough ($\ell>d$) in space. The inner product
        in~\eqref{E:Inner_1} can be equivalently written as
        \begin{align}\label{E:Inner_3}
          \langle \phi, \psi \rangle_{\mathcal{H}}
          =  2\pi \iint_{\R_+^2} \ud s\, \ud t \int_{\R^d} \ud \xi \; \widehat{\Lambda_H}(t-s) K_\ell(\xi) \mathcal{F}[\phi(t,\cdot)](\xi) \overline{\mathcal{F}[\phi(s,\cdot)](\xi)}.
        \end{align}
        When $d < \ell < 2 d$, the Hilbert space $\mathcal{H}$ should be understood
        as a fractional Sobolev space; see, e.g.,
        \cite{di-nezza.palatucci.ea:12:hitchhikers}. Analogously, if $\ell \geq 2
          d$, the Hilbert space $\mathcal{H}$ should be associated with a higher order
        (fractional) Sobolev space. We also refer the reader to
        \cite{lodhia.sheffield.ea:16:fractional} for an alternative treatment of
        such random fields.

  \item If $H\in (0,\sfrac{1}{2})$, the noise is rough in time.

\end{enumerate}

In this paper, we will focus on Cases (1)--(3), i.e., $H \in [1/2,1)$. When
$\ell \in (0,d)$, the noise can be understood as
\begin{gather*}
  \E\left(\dot{W}(t,x)\dot{W}(s,y)\right) = \widehat{\Lambda_H}(t-s)
  |x-y|^{-\ell} \quad \text{with} \quad \widehat{\Lambda_H}(t-s) =
  \begin{cases}
    a_H |t-s|^{2H-2}, & \text{$H \in (1/2,1)$}, \\
    \delta_0(t-s),    & \text{$H = 1/2$},
  \end{cases}
\end{gather*}
where $\delta_0$ denotes the Dirac delta at $0$ and
\begin{align*} %\label{E:aH}
  a_H\coloneqq H(2H-1).
\end{align*}

\subsection{Main results}

Let $G(t,x)$ be the fundamental solution\index{fundamental solution} to~\eqref{E:fde}. The solution
to~\eqref{E:fde} with zero initial condition(s) is understood in the following
\textit{mild form}\index{mild solution}
\begin{align}\label{E:Mild}
  u(t,x) = \int_0^t\int_{\R^d} G(t-s,x-y) W(\ud s, \ud y), \quad t>0,\: x\in\R^d.
\end{align}
The solution $\left\{u(t,x):\, t>0,x\in\R^d\right\}$ is a centered Gaussian
process with correlation
\begin{align}\label{E:u_corr}
  \E [u(t,x) u(s,y)] = \left\langle G(t - \cdot, x - *) \one_{[0, t]}(\cdot), G(s - \cdot, y - *)\one_{[0, s]}(\cdot)\right\rangle_{\mathcal{H}},
\end{align}
if $G(t - \cdot, x - *) \one_{[0,t]}(\cdot) \in \mathcal{H}$ for all $(t,x) \in
  \R_+\times \R^d$. The integral in~\eqref{E:Mild} is understood as the
\textit{It\^o--Walsh integral}\index{Ito-Walsh integral@It\^o--Walsh integral}~\cites{walsh:86:introduction,
  dalang:99:extending, dalang.sanz-sole:26:stochastic} in case of $H = 1/2$ and the \textit{Skorohod
  integral}\index{Skorohod integral}~\cites{nualart:06:malliavin, skorohod:75:on} in case $H > 1/2$.
\smallskip

Throughout this paper, the following hypothesis on the parameters are assumed:
\begin{hypothesis}\label{H:para}
  Assume that
  \begin{align*} % \label{E:Main_Assump}
    \alpha >   0,       \quad
    \beta  \in (0,2],   \quad
    H      \in [1/2,1), \quad \ell \in (0,2d),  \quad \text{and}\quad
    \gamma \geq 0.
  \end{align*}
\end{hypothesis}

Our first main result is a Dalang-type criterion\index{Dalang condition} for the existence of random
field solutions\footnote{The uniqueness of solutions is straightforward, as the
  solution is given explicitly by the formula in~\eqref{E:Mild}.}.
Theorem~\ref{T:Main} below summarizes the conclusion, while the full technical
statement is given in Section~\ref{S:Dalang}; see Theorem~\ref{T:Dalang}. A key
input is the sharp behavior of the fundamental solution near the origin and at
infinity (proved in Section~\ref{S:Asymptotics}; see Theorems~\ref{T:Zero}
and~\ref{T:Infinity}). Moreover, the existence statement in
Theorem~\ref{T:Main}(i) (equivalently, Theorem~\ref{T:Dalang}(i)--(ii) below) is
sharp (necessary and sufficient). While fractional SPDEs have been studied
extensively, existing works typically focus on special parameter regimes (either
for the fractional PDE or for the noise). To the best of our knowledge, this is
the first comprehensive study of~\eqref{E:fde} that allows simultaneously the
full fractional parameter range $(\alpha,\beta, \gamma)$ and the full noise
parameter range $(H,\ell)$; see Remark~\ref{R:Dalang-known-cases} for related
works. Let us first introduce some notation:
\begin{align}\label{E:Dalang}
  \rho = \rho \left(\alpha,\beta,\gamma;H,\ell\right) \coloneqq
  \begin{dcases}
    \beta + \gamma + H -\frac{\ell \beta}{2\alpha} - 1, & \beta < 2, \text{ or } \beta = 2, \gamma > 1, \\
    \gamma + H - \ell/\alpha + 1/2,                     & \beta = 2, \gamma \in [0,1],
  \end{dcases}
\end{align}
and
\begin{align*}
  \widetilde{\rho} \coloneqq \min \left\{\frac{\alpha\rho}{\beta},\; \alpha - \frac{\ell}{2} \right\}.
\end{align*}

\begin{theorem}[Sharp well-posedness]\label{T:Main}
  Assume Hypothesis~\ref{H:para}, and let $\rho$ be defined as
  in~\eqref{E:Dalang}.
  \begin{enumerate}[\rm (i)]
    \item If either $\beta \in (0,2)$; or $\beta = 2$ and $\gamma > 1$; or
          $\beta = 2$ and $\gamma = 0$, then,~\eqref{E:fde} has a random field
          solution, namely, $u(t,x)$ defined as in~\eqref{E:Mild} has a finite
          second moment, if and only if
          \begin{align}\label{E:main}
            \rho > 0 \quad \text{and}\quad \ell < 2\alpha.
          \end{align}

    \item If $\beta = 2$ and $\gamma \in (0,1]$, then,~\eqref{E:fde} has a
          random field solution if both conditions in~\eqref{E:main} hold.

  \end{enumerate}
\end{theorem}

\begin{proof}
  This is a direct consequence of Theorem~\ref{T:Dalang}; see
  Section~\ref{S:Dalang}.
\end{proof}

% See Remark~\ref{R:Dalang-jump-beta2} for the property jump at $\beta = 2$, and
% Remark~\ref{R:Dalang-known-cases} for comparisons with known special cases. See
% also Lemma~\ref{L:K} and Subsection~\ref{SS:K} for explicit computations.
% \bigskip

We would like to point out that Theorem~\ref{T:Main} reveals a qualitative change at
the endpoint $\beta = 2$: the well-posedness threshold exhibits a \emph{property
  jump}, depending on the forcing memory parameter $\gamma$ (and interacting with
the noise parameters $(H,\ell)$). We discuss this phenomenon in
Remark~\ref{R:Dalang-jump-beta2}. We compare with previously studied special
regimes in Remark~\ref{R:Dalang-known-cases}. For some parameter choices,
Lemma~\ref{L:K} provide explicit computations. \bigskip

We next establish sharp two-sided variance bounds\index{variance bounds!two-sided} for temporal and spatial
increments of $u$, which will be combined our strong local nondeterminism
results to yield the ensuing fine sample path corollaries (moduli of continuity,
Chung-type laws of the iterated logarithm, and small ball probabilities). Set
\begin{align}\label{E:m1-m2}
  m_1 (r) \coloneqq \begin{dcases}
                      r^{2 (\rho \wedge 1)},        & \rho \neq 1, \\
                      r^2 \big(1 + |\log(r)| \big), & \rho = 1;
                    \end{dcases}
  \quad\text{and} \quad
  m_2 (r) \coloneqq \begin{dcases}
                      r^{2 (\widetilde{\rho} \wedge 1)}, & \widetilde{\rho} \neq 1, \\
                      r^2 \big(1 + |\log(r)| \big),      & \widetilde{\rho} = 1.
                    \end{dcases}
\end{align}

\begin{theorem}[Two-sided variance bounds]\label{T:main-holder}
  % In the scenarios of Case (i) of Theorem~\ref{T:Main}, with an additional
  % assumption $\ell < \alpha (\gamma + 1/2)$ if $\beta = 2$ and $\gamma \in
  %   (0,2]$, 
  Assume Hypothesis \eqref{H:para} and condition \eqref{E:main}.
  For part (i) below, assume in addition that $\ell < \alpha(\gamma+1/2)$ if $\beta=2$ and $\gamma \in (0,2]$.
  Then, for all
  % positive\footnote{``Positive'' means ``strictly
  % positive'' in the paper.} 
  $0<S<T$ and $M > 0$ with $S<T$, there exist some constants $0<C\le C'$ depending on $S,T,M$ such that the following bounds hold:
  \begin{enumerate}[\rm (i)]
      \item For all $(t,s,x)\in [S,T]^2\times \R^d$,
      \begin{align}\label{E:main-holder-t}
    C m_1 (|t - s|) & \le \E\left[\big|(u(t, x) - u(s, x)\big|^2\right]
    \le C'm_1 (|t - s|),
  \end{align}
  where the lower bound is valid when $\beta \in (0,2)$, $\gamma \ge 0$ or $\beta = 2$, $\gamma \in [0,1]$.
    \item For all $(t,s,x,y)\in [S,T]^2\times B(0,M)^2$,
    \begin{align}\label{E:main-holder-s}
    C m_2 (|x - y|) & \le \E\left[\big|u(t, x) - u(t, y)\big|^2\right]
    \le C'm_2 (|x - y|).
  \end{align}
  \end{enumerate}
  In the above,
  $B(0,M)$ denotes the ball centered at zero and of radius $M$; and $m_1, m_2$
  are nonnegative functions on $\R_+$ given by~\eqref{E:m1-m2}.
\end{theorem}

\begin{proof}
  % The existence criterion~\eqref{E:main} is established in Section~\ref{S:Dalang}. 
  The upper bounds in~\eqref{E:main-holder-t}
  and~\eqref{E:main-holder-s} are derived in Section~\ref{S:Upper-t}, while the
  corresponding lower bounds are proved in Sections~\ref{S:SLN_t}
  and~\ref{S:SLN_x}.
\end{proof}

\begin{remark}\label{R:main-holder-assumptions}
  \begin{enumerate}[(i)]

    \item In Theorem~\ref{T:main-holder}, for $\beta = 2$ and $\gamma \in
            [0,1/2]$, with $\rho > 0$, we observe that $\ell < \alpha(\gamma + H +
            1/2) < 2\alpha$ as $H \in (0,1)$. This implies that the second condition
          in~\eqref{E:main}, i.e., $\ell < 2\alpha$, is guaranteed by the first
          condition.

    \item The additional assumption $\ell < \alpha (\gamma + 1/2)$ in
          Theorem~\ref{T:main-holder} is required when $\beta = 2$ and $\gamma \in
            (1,2]$. This condition is specifically necessary for the two-sided bounds
          for the temporal increments~\eqref{E:main-holder-t}, but not for the
          spatial increments~\eqref{E:main-holder-s}. It arises as a prerequisite
          for employing a particular technique, borrowed
          from~\cite{guo.song.ea:24:stochastic}, in establishing the upper bound for
          the temporal increment. A comparable condition was thus also observed
          in~\cite{guo.song.ea:24:stochastic}*{Formula (4.1)}. We believe that this
          condition is non-essential and hope that it can be eliminated in future
          research via alternative methodologies.

    \item The assumption $\ell < \alpha (\gamma + 1/2)$ is rather strong.
          Indeed, in this case
          \[
            \rho = 2 + \gamma + H - \ell/\alpha - 1 > 2 + \gamma
            + H - (\gamma + 1/2) - 1 = H + 1/2.
          \]
          It indicates that the temporal H\"{o}lder continuity is almost of order
          $1$ on account of~\eqref{E:main-holder-t}, rendering it essentially
          trivial.

          Furthermore, the two-sided inequality~\eqref{E:main-holder-t} can be also
          extended to the case $\beta = 2$ and $\gamma \in (0,1]$ under the
          assumption $\ell < \alpha \min \{2,\gamma + H - 1/2\}$, again yield a
          trivial bound, namely the exponent equal to $2$. However, in this setting
          we are unable to establish the corresponding upper bounds for the moments
          of the spatial increments. For these reasons, we choose to retain the
          current formulation of Theorem~\ref{T:main-holder}.

  \end{enumerate}
\end{remark}

Our third and final main result concerns strong local nondeterminism (SLND) of
the solution in the time variable or in the space variable, with the other
variable fixed. In the classical parabolic and hyperbolic cases $\beta\in\{1,2\}$
with $\gamma=0$, the SLND property has been established
in~\cites{lee.xiao:19:local, lee.xiao:23:chung-type, lee:22:local}. Here we
prove SLND in a substantially broader parameter range. Moreover, for the cases when the
two-sided\index{strong local nondeterminism (SLND)!two-sided} temporal\index{strong local nondeterminism (SLND)!temporal} SLND lower bound is out of reach, we still obtain a weaker
``one-sided''\index{strong local nondeterminism (SLND)!one-sided} version; see Theorem~\ref{T:SLND}(i)(b).

\begin{theorem}[Strong local nondeterminism (SLND)]\label{T:SLND}
  Assume Hypothesis \ref{H:para} and the Dalang condition~\eqref{E:main}. Let $0<S<T<\infty$.
  \begin{enumerate}[\rm (i)]

    \item \emph{(SLND in time)}\index{strong local nondeterminism (SLND)!temporal}
          \begin{enumerate}[\rm (a)]

            \item If $\beta = 1$, $\gamma \ge 0$; or $\beta = 2$, $\gamma \in [0,
                      1]$, then there exist constants $C > 0$ and $\delta>0$ such that for
                  all $x \in \R^d$, for all $n \ge 1$, and for all $t, t_1, \dots, t_n
                    \in [S, T]$ with $\max_{1\le j \le n}|t-t_j| \le \delta$,
                  \begin{align}\label{E:SLND:t}
                    \Var\big( u(t, x) \big| u(t_1, x), \dots, u(t_n, x) \big) \ge C \min_{1 \le j \le n} |t-t_j|^{2\rho}.
                  \end{align}

            \item If $\beta \in (0, 2)$, $\gamma \ge 0$, then there exists a constant $C>0$ such
                  that for all $x \in \R^d$, for all $n \ge 1$, and for all $t, t_1,
                    \dots, t_n \in [S, T]$ with $\max\{t_1, \dots, t_n \} \le t$,
                  \begin{align}\label{E:1SLND:t}
                    \Var\big( u(t, x) \big| u(t_1, x), \dots, u(t_n, x) \big) \ge C \min_{1 \le j \le n} (t-t_j)^{2\rho}.
                  \end{align}

          \end{enumerate}

    \item \emph{(SLND in space)}\index{strong local nondeterminism (SLND)!spatial} If $\beta \in (0, 2]$, $\gamma \ge 0$, then
          for any $M>0$, there exists a constant $C>0$ such that for all $t \in [S,
              T]$, for all $n \ge 1$, and for all $x, x_1, \dots, x_n \in B(0, M)$,
          \begin{align}\label{E:SLND:x}
            \Var\big( u(t, x) \big| u(t, x_1), \dots, u(t, x_n) \big) \ge C \min_{1 \le j \le n} |x-x_j|^{2\tilde\rho}.
          \end{align}

  \end{enumerate}
\end{theorem}

\begin{proof}
  Part (i)(a) follows from Theorem~\ref{T:LND_t}(ii) and (iii); part (i)(b)
  follows from Theorem~\ref{T:LND_t}(i); and part (ii) follows from
  Theorem~\ref{T:LND_x}.
\end{proof}

Theorem~\ref{T:SLND} is a key probabilistic input for the sharp sample path
regularity results stated below, including exact moduli of continuity, Chung-type
laws of the iterated logarithm, and small ball estimates. These applications
typically rely on sharp bounds on the conditional variances which are provided by the SLND results.
% on the two-sided SLND bound~\eqref{E:SLND:t} (and its spatial
% counterpart~\eqref{E:SLND:x}) in order to control conditional variances at fine
% scales.

Outside the regimes covered by Theorem~\ref{T:SLND}(i)(a), our methods still
yield the one-sided lower bound~\eqref{E:1SLND:t} for all $\beta\in(0,2)$, but
upgrading it to the two-sided form~\eqref{E:SLND:t} seems to require new ideas
when $\beta\in(0,1)\cup(1,2)$ and in the hyperbolic case $\beta=2$ with
$\gamma>1$.
% In the space variable, the obstruction is different: when $\beta=2$
% and $\gamma\in(0,1)$, the spatial behavior of the kernel suggests that SLND
% should still hold, but the available lower bounds are not strong enough to
% control the relevant conditional variances uniformly. 
% We therefore record the
% following as an open problem.

\begin{conjecture}\label{Conj:SLND}
  The solution $u(t, x)$ satisfies (two-sided) strong local nondeterminism with
  respect to $t$ when $\beta\in(0,1)\cup(1,2)$ and when $\beta=2$, $\gamma>1$.
  % and satisfies strong local nondeterminism with respect to $x$ when $\beta=2$
  % and $\gamma\in(0,1)$.
\end{conjecture}

\begin{remark}\label{R:SLND-spacetime}
  One may also study strong local nondeterminism with respect to the joint
  variable $(t, x)$ and ask if the following holds:
  \begin{align*}
    \Var\big( u(t, x) \big| u(t_1, x_1), \dots, u(t_n, x_n) \big) \ge C \min_{1 \le j \le n} \left(|t-t_j|^{\rho} + |x-x_j|^{\tilde\rho}\right)^2.
  \end{align*}
  When $\beta = 1$, $\gamma = 0$, this has been established
  in~\cite{lee.xiao:23:chung-type}. When $\beta = 2$, $\gamma = 0$, an integral
  form of strong local nondeterminism has been proved in~\cites{lee:22:local,
    lee.xiao:19:local}. The other cases are open and we leave them for future
  investigations.
\end{remark}

\subsection{Applications of SLND: Sample path properties}\label{SS:Applications}

Building on the sharp two-sided variance bounds in
\eqref{E:main-holder-t}--\eqref{E:main-holder-s} and the strong local
nondeterminism property in Theorem~\ref{T:SLND}, we derive a collection of sharp
sample path regularity results for the solution to~\eqref{E:fde}. These include
exact uniform moduli of continuity, local moduli of continuity (laws of the
iterated logarithm), Chung-type laws of the iterated logarithm (and related
moduli of non-differentiability), as well as small ball probability estimates.
At a high level, Theorem~\ref{T:main-holder} provides the sharp two-sided
variance bounds that identify the correct scales of the increments, while
Theorem~\ref{T:SLND} supplies the conditional-variance lower bounds needed to
upgrade these scales to exact almost sure statements. Moreover, the almost sure
limit constants appearing in the exact modulus and iterated-logarithm statements
are deterministic, thanks to the zero-one laws in Theorem~\ref{T:01law} proved
in Section~\ref{S:Pf:Cor}. To state these results, let us define the modulus
functions
\begin{align}\label{E:w1w2}
  w_1(r) = \begin{cases}
             r^\rho \sqrt{\log(1+r^{-1})}, & \rho <1,  \\
             r\log(1+r^{-1}),              & \rho = 1,
           \end{cases}
  \qquad
  w_2(r) = \begin{cases}
             r^{\tilde\rho} \sqrt{\log(1+r^{-1})}, & \tilde\rho <1,  \\
             r\log(1+r^{-1}),                      & \tilde\rho = 1.
           \end{cases}
\end{align}

\subsubsection{Uniform moduli of continuity}

We first obtain uniform modulus of continuity\index{modulus of continuity!uniform} results for the temporal process
$\{u(t, x_0): t\in I\}$ and the spatial process $\{u(t_0, x): x\in J\}$. In the
subcritical cases $\rho\in(0,1)$ and $\tilde\rho\in(0,1)$, these show that $w_1$
and $w_2$ are the optimal modulus functions, respectively. In particular, the
optimal temporal and spatial H\"older exponents\index{Holder continuity@H\"older continuity} are $\rho^-$ and $\tilde\rho^-$
and cannot, in general, be upgraded to $\rho$ and $\tilde\rho$ due to the
presence of the factor $\sqrt{\log(1+r^{-1})}$.

\begin{corollary}[Uniform moduli of continuity]\label{C:MC}
  Suppose condition \eqref{E:main} holds. Further assume $\ell<\alpha(\gamma+1/2)$ if $\beta=2$ and $\gamma \in (0,2]$ for part (i) below.
  Then the following results hold.
  \begin{enumerate}[\rm (i)]

    \item If $\rho \in (0, 1]$, then for any compact interval $I$ in $(0,
            \infty)$ and $x_0 \in \R^d$, there exists a random variable $Z_1$ with
          $\E[Z_1]<\infty$ such that
          \begin{align}\label{E:MC1}
            |u(t, x_0) - u(s, x_0)| \le Z_1 w_1(|t-s|) \quad \text{for all } t, s \in I.
          \end{align}
          Moreover, suppose $\beta \in (0, 2)$, $\gamma \ge 0$, or $\beta = 2$, $\gamma \in [0,1]$. If
          $\rho \in (0, 1)$, then we have the following exact uniform modulus of
          continuity: There exists a constant $C_1 \in (0, \infty)$ such that
          \begin{align}\label{E:exactMC1}
            \lim_{\varepsilon \to 0^+} \sup_{t, s \in I:\, 0<|t-s|\le \varepsilon} \frac{|u(t, x_0) - u(s, x_0)|}{w_1(|t-s|)} = C_1 \quad \text{a.s.}
          \end{align}

    \item If $\tilde\rho \in (0, 1]$, then for any $t_0 > 0$ and any compact
          rectangle $J$ in $\R^d$, there exists a random variable $Z_2$ with
          $\E[Z_2] < \infty$ such that
          \begin{align}\label{E:MC2}
            |u(t_0, x) - u(t_0, y)| \le Z_2 w_2(|x-y|) \quad \text{for all } x, y \in J.
          \end{align}
          If $\tilde\rho \in (0, 1)$, then we have the following exact uniform
          modulus of continuity: There exists a constant $C_2 \in (0, \infty)$ such
          that
          \begin{align}\label{E:exactMC2}
            \lim_{\varepsilon \to 0^+} \sup_{x, y \in J:\, 0<|x-y| \le \varepsilon} \frac{|u(t_0, x) - u(t_0, y)|}{w_2(|x-y|)} = C_2 \quad \text{a.s.}
          \end{align}

  \end{enumerate}
\end{corollary}

\begin{remark}\label{R:MC-critical}
  The exact uniform moduli of continuity in Corollary~\ref{C:MC} for the
  critical cases $\rho = 1$ and $\tilde{\rho} = 1$ remain unsolved. This is
  because we could not obtain the corresponding versions of strong local
  nondeterminism that match with the bounds
  in~\eqref{E:main-holder-t}--\eqref{E:main-holder-s} and capture the
  logarithmic factor. See also~\cite{herrell.song.ea:20:sharp,
    guo.song.ea:25:sample} where the critical cases were also unsolved. We believe
  that when $\rho=1$ and $\tilde\rho=1$, respectively, the modulus functions
  $w_1$ and $w_2$ defined by \eqref{E:w1w2} remain optimal
  in~\eqref{E:MC1}--\eqref{E:exactMC2}. Whether this is the case is an open
  question and is left for future investigation.
\end{remark}

\subsubsection{Local moduli of continuity (laws of the iterated logarithm)}

We next establish local moduli of continuity\index{modulus of continuity!local} (laws of the iterated logarithm\index{law of the iterated logarithm (LIL)})
for the temporal and spatial processes at a fixed base point $(t_0,x_0)$.

\begin{corollary}[Local moduli of continuity/laws of the iterated logarithm]\label{C:LIL}
  Assume condition \eqref{E:main} holds. For part (i) below, further assume $\ell < \alpha(\gamma+1/2)$ if $\beta = 2$ and
  $\gamma \in (0, 2]$.
  \begin{enumerate}[\rm (i)]
    \item If $\rho \in (0, 1)$, then for any $t_0 > 0$ and $x_0 \in \R^d$, there
          exists a constant $C_3 \in [0, \infty)$ such that
          \begin{align}\label{E:LIL1}
            \lim_{\varepsilon\to 0^+} \sup_{t:\, 0<|t-t_0|\le\varepsilon} \frac{|u(t, x_0) - u(t_0, x_0)|}{{|t-t_0|}^\rho \sqrt{\log\log(|t-t_0|^{-1})}} = C_3 \quad \text{a.s.}
          \end{align}
          Moreover, when $\gamma\in[0,1-\beta]$, we have $C_3>0$ and
          \begin{align}\label{E:LIL1b}
            \lim_{\varepsilon\to0^+} \sup_{t: 0<d_1(t,t_0)\le\varepsilon}\frac{|u(t,x_0)-u(t_0,x_0)|}{d_1(t,t_0)\sqrt{\log\log(1/d_1(t,t_0))}} = \sqrt{2} \quad \text{a.s.}
          \end{align}
          where $d_1(t,s)=\E[(u(t,x_0)-u(s,x_0))^2]^{1/2}$.

    \item If $\tilde\rho \in (0, 1)$, then for any $t_0 > 0$ and $x_0 \in \R^d$,
          there exists a constant $C_4 \in (0, \infty)$ such that
          \begin{align}\label{E:LIL2}
            \lim_{\varepsilon\to 0^+} \sup_{x:\, 0<|x-x_0| \le \varepsilon} \frac{|u(t_0, x) - u(t_0, x_0)|}{{|x-x_0|}^{\tilde\rho} \sqrt{\log\log(|x-x_0|^{-1})}} = C_4 \quad \text{a.s.}
          \end{align}
          and
          \begin{align}\label{E:LIL2b}
            \lim_{\varepsilon\to0^+}\sup_{x: 0<d_2(x,x_0)\le\varepsilon}\frac{|u(t_0,x)-u(t_0,x_0)|}{d_2(x,x_0)\sqrt{\log\log(1/d_2(x,x_0))}} = \sqrt{2} \quad \text{a.s.}
          \end{align}
          where $d_2(x,y) = \E[(u(t_0,x)-u(t_0,y))^2]^{1/2}$.

  \end{enumerate}
\end{corollary}

\begin{remark}
  The extra condition $\gamma \in [0, 1-\beta]$ in part (i) is imposed to ensure the technical Assumption 2.1 in \cite{lee.xiao:23:chung-type} holds (see Lemma \ref{L:LX23cond}).
  We believe that this condition can be removed or relaxed.
\end{remark}

\subsubsection{Chung-type laws of the iterated logarithm}

In addition, we establish $\liminf$ behavior of the temporal and spatial
increments and nowhere differentiability of the temporal and spatial processes
in the corresponding regimes.

\begin{corollary}[Chung-type laws of the iterated logarithm]\label{C:Chung}
  Assume condition~\eqref{E:main} holds. For part (i) below, further assume $\ell < \alpha(\gamma+1/2)$ if $\beta = 2$ and
  $\gamma \in (0, 2]$.
  \begin{enumerate}[\rm (i)]

    \item If $\rho \in (0, 1)$, then for any compact interval $I$ in $(0,
            \infty)$ and any $x_0 \in \R^d$, there exist constants $C_5, C_5' \in [0,
              \infty]$ such that for each $t_0 \in I$,
          \begin{align}\label{E:Chung1}
            \liminf_{\varepsilon \to 0^+} \sup_{t:\, |t-t_0|
              \le \varepsilon} \frac{|u(t, x_0) - u(t_0, x_0)|}{\varepsilon^\rho(\log\log(1/\varepsilon))^{-\rho}}
            = C_5 \quad \text{a.s.}
          \end{align}
          and
          \begin{align}\label{E:mod-non1}
            \liminf_{\varepsilon \to 0^+} \inf_{t_0 \in I} \sup_{t: |t-t_0|
              \le \varepsilon} \frac{|u(t, x_0) - u(t_0, x_0)|}{\varepsilon^\rho(\log(1/\varepsilon))^{-\rho}}
            = C_5' \quad \text{a.s.}
          \end{align}
          Moreover, when $\beta = 1$, $\gamma \ge 0$; or $\beta = 2$, $\gamma=[0,1]$, $\ell<\alpha(\gamma+1/2)$, we
          have $C_5>0$, $C_5'>0$, and $t \mapsto u(t, x_0)$ is a.s.\ nowhere
          differentiable on $(0, \infty)$.\\
          When $\gamma \in [0,1-\beta]$, we have $0 < C_5 < \infty$.

    \item If $\tilde\rho \in (0, 1)$, then for any $t_0 > 0$ and any compact
          rectangle $J$ in $\R^d$, there exist constants $C_6 \in (0,\infty)$ and
          $C_6' \in (0, \infty]$ such that for each $x_0 \in J$,
          \begin{align}\label{E:Chung2}
            \liminf_{\varepsilon \to 0^+} \sup_{x:\, |x-x_0|
              \le \varepsilon} \frac{|u(t_0, x) - u(t_0, x_0)|}{\varepsilon^{\tilde\rho}(\log\log(1/\varepsilon))^{-\tilde\rho/d}}
            = C_6 \quad \text{a.s.}
          \end{align}
          and
          \begin{align}\label{E:mod-non2}
            \liminf_{\varepsilon \to 0^+} \inf_{x_0 \in J} \sup_{x:\, |x-x_0|
              \le \varepsilon} \frac{|u(t_0, x) - u(t_0, x_0)|}{\varepsilon^{\tilde\rho}(\log(1/\varepsilon))^{-\tilde\rho/d}}
            = C_6' \quad \text{a.s.}
          \end{align}
          In particular, $x \mapsto u(t_0, x)$ is a.s.\ nowhere differentiable on
          $\R^d$.
  \end{enumerate}
\end{corollary}

\begin{remark}\label{R:Chung-zeroone}
  By Theorem~\ref{T:01law} below, a zero-one law holds for each of the limits in
  Corollaries~\ref{C:MC},~\ref{C:LIL} and~\ref{C:Chung} (see also
  \cite{marcus.rosen:06:markov}*{Theorem 6.3.3}
  and~\cite{lee.xiao:23:chung-type}*{Lemma 3.1}). The results \eqref{E:Chung1}
  and~\eqref{E:Chung2} are known as \emph{Chung's law of the iterated logarithm}\index{Chung's law of the iterated logarithm}
  \cite{chung:48:on} (see also \cites{lee.xiao:23:chung-type,lee:22:local,lee:22:hausdorff,chen:23:chungs,
    khoshnevisan.kim.ea:24:small-ball, guo.song.ea:25:sample, hu.lee:25:on} for
  related results in the context of Gaussian random fields and SPDEs). If the constants
  in~\eqref{E:mod-non1} and~\eqref{E:mod-non2} are strictly positive and finite,
  the rate functions would be optimal and referred to as the \emph{moduli of
    non-differentiability} \cites{csorgo.revesz:78:how, csorgo.revesz:81:on,
    wang.xiao:19:csorgo-revesz, wang.su.ea:20:moduli}. We propose the following
  questions for future studies:
  \begin{enumerate}

    \item[(i)] Are the constants in~\eqref{E:mod-non1} and~\eqref{E:mod-non2}
          strictly positive and finite?

    \item[(ii)] What are the optimal rate functions
          in~\eqref{E:Chung1}--\eqref{E:mod-non2} when $\rho = 1$ and $\tilde{\rho}
            = 1$?

  \end{enumerate}
\end{remark}

\subsubsection{Small ball probabilities.}

Finally, we obtain the following result regarding small ball probabilities\index{small ball probability} for
the temporal and spatial processes. We refer to
\cites{athreya.joseph.ea:21:small, chen:24:small, foondun.joseph.ea:23:small,
  guo.song.ea:25:sample, hu.lee:25:on, khoshnevisan.kim.ea:24:small-ball} for
related studies in the context of SPDEs, and to \cite{li.shao:01:gaussian} for a
survey on small ball probabilities of Gaussian processes.
The temporal estimate in Corollary~\ref{C:smallball} relies on a general small
ball bound under one-sided strong local nondeterminism (Lemma~\ref{L:X:smallball}
in Section~\ref{S:Pf:Cor}), which is tailored to the regime where we only have
one-sided time-SLND.

\begin{corollary}[Small ball probabilities]\label{C:smallball}
  Suppose condition \eqref{E:main} holds.
  \begin{enumerate}[\rm (i)]

    \item Suppose $\beta\in(0,2)$; or $\beta=2$, $\gamma\in[0,1]$, $\ell<\alpha(\gamma+1/2)$. If $\rho \in (0,1)$,
          then for any compact interval $I\subset (0,\infty)$, there exist two
          constants $0<C_7\le C_7'<\infty$ such that for all $x_0\in \R^d$ and all
          $\varepsilon \in (0, 1)$,
          \begin{align}\label{E:smallball1}
            \exp\left( - \frac{C_7'}{\varepsilon^{1/\rho}} \right)
            \le \mathbb{P}\left\{ \sup_{t\in I} |u(t,x_0)| \le \varepsilon \right\}
            \le \exp\left( - \frac{C_7}{\varepsilon^{1/\rho}} \right).
          \end{align}

    \item If $\tilde\rho\in (0,1)$, then for any $t_0>0$ and compact rectangle
          $J$ in $\R^d$, there exist two constants $0<C_8\le C_8'<\infty$ such that
          for all $\varepsilon\in(0,1)$,
          \begin{align}\label{E:smallball2}
            \exp\left( - \frac{C_8'}{\varepsilon^{d/\tilde\rho}} \right)
            \le \mathbb{P}\left\{ \sup_{x\in J} |u(t_0,x)| \le \varepsilon \right\}
            \le \exp\left( - \frac{C_8}{\varepsilon^{d/\tilde\rho}} \right).
          \end{align}

  \end{enumerate}
\end{corollary}

\begin{remark}\label{R:smallball-critical}
  The small ball probability bounds for $\rho = 1$ and $\tilde\rho = 1$ remain
  open.
\end{remark}

\subsubsection{Related work and further directions}

Further results on local times and fractal properties of the temporal and
spatial processes can be obtained by applying the results
of~\cites{lee.xiao:23:local, xiao:08:strong, xiao:09:sample}, thanks to our
sharp variance bounds and strong local nondeterminism results.

\begin{remark}\label{R:Guo-sample-path}
  In a recent work by Guo et al. \cite{guo.song.ea:25:sample}, the
  solution $u(t,x)$ to the same SPDE~\eqref{E:fde} driven by a Gaussian
  noise but with covariance
  \[
    \E[\dot{W}(t,x)\dot{W}(s,y)]
    = \widehat{\Lambda_H}(t-s) \prod_{i=1}^d \widehat{\Lambda_{H_i}}(x_i-y_i)
  \]
  is investigated, and the space-time exact uniform and local moduli of
  continuity, Chung's law of the iterated logarithm, and small ball
  probabilities for $u(t,x)$ are established through an extension of the
  Mueller-Tribe ``pinned string'' decomposition (see
  \cite{mueller.tribe:02:hitting}): $u(t, x)=U(t,x)-V(t,x)$, where
  $U(t,x)$ has stationary increments and satisfies strong local
  nondeterminism, whereas $V(t,x)$ is much smoother than $U(t,x)$. In
  \cite{guo.song.ea:25:sample}, it is not shown whether or not the
  solution $u(t,x)$ itself satisfies strong local nondeterminism. Whether
  this is the case is a nontrivial question because $U$ and $V$ are not
  independent. In the present paper, we show directly that the solution to
  \eqref{E:fde} itself satisfies strong local nondeterminism in $t$ and in
  $x$, and our results do not rely on the ``pinned string'' decomposition.

  Another related work \cite{khoshnevisan.sanz-sole:23:optimal} studies
  equation \eqref{E:fde} with parabolic ($\beta = 1$) and hyperbolic
  ($\beta = 2$) integro-differential operators. The authors establish
  necessary and sufficient (optimal) conditions for the H\"{o}lder
  regularity of the random-field solution, expressed in terms of the
  characteristic exponent of the underlying L\'evy operator and the
  spectral measure of the noise.
\end{remark}

\paragraph{Scope and limitations.} We prove two-sided strong local
nondeterminism in time in the regimes covered by Theorem~\ref{T:SLND} and a
one-sided version for $\beta\in(0,1)\cup(1,2)$ in Theorem~\ref{T:SLND}(i)(b).
Two-sided strong local nondeterminism in time outside these regimes and joint
strong local nondeterminism in $(t,x)$ remain open in general (see the
conjecture and remarks above). Further fine spatio-temporal properties,
sharpness of the constants and rate functions, and the cases $\rho = 1$ and
$\tilde\rho = 1$ are also left for future work.

\subsection*{Organization of the paper}
\addcontentsline{toc}{subsection}{Organization of the paper}

This paper is organized as follows: In Section~\ref{S:G}, we recall some basic
facts about the fundamental solution $G(t,x)$ to be used in this paper. In
Section~\ref{S:Asymptotics}, we study the asymptotic behavior of the fundamental
solution\index{fundamental solution!asymptotics} at $0$\index{asymptotics!origin} and at infinity\index{asymptotics!infinity} for the spatial variable; these sharp kernel
estimates are of independent interest and may be useful in other deterministic
and stochastic problems involving fractional operators. In
Section~\ref{S:Dalang}, we study the well-posedness of~\eqref{E:fde}. In Section
\ref{S:Upper-t}, we derive moment upper bounds for temporal and spatial
increments of the solution to \eqref{E:fde}. In Sections~\ref{S:SLN_t} and
\ref{S:SLN_x}, we establish strong local nondeterminism and moment lower bounds
for increments in $t$ and $x$, respectively. The zero-one laws and Corollaries
\ref{C:MC},~\ref{C:LIL},~\ref{C:Chung} and~\ref{C:smallball} are proved in
Section~\ref{S:Pf:Cor}. Finally, the appendices contain auxiliary estimates used
in the proofs: Appendix~\ref{S:Trigonometry} establishes trigonometric integral
estimates needed for the wave case ($\beta=2$ and $\gamma=0$);
Appendix~\ref{A:Tech} collects technical bounds (e.g., Mittag--Leffler
estimates); Appendix~\ref{A:Balan} proves a generalized Balan-type lemma used in
the strong local nondeterminism analysis; and Appendix~\ref{A:2F1} summarizes
the hypergeometric identities used throughout the paper.

\section{Some preliminaries---Fundamental solution}\label{S:G}

As stated in the Introduction, we assume that
\begin{align*}
  \alpha > 0, \quad \beta \in (0, 2], \quad \text{and} \quad \gamma \ge 0.
\end{align*}
When $\beta \in (0, 2)$, the fundamental solution\index{fundamental solution} $G(t,x)$ can be expressed using
the \textit{Fox H-function}\index{Fox H-function} (see, e.g., \cite{kilbas.saigo:04:h-transforms} for
a comprehensive study of this special function) as follows:
\begin{align*} % \label{E:G-Fox}
  G(t,x) = \pi^{-d/2} |x|^{-d} t^{\beta + \gamma - 1}
  \FoxH{2,1}{2,3}{\frac{|x|^\alpha}{2^{\alpha-1} \nu t^\beta}}
  {(1,1), \:(\beta + \gamma, \beta)}{(d/2, \alpha/2), \:(1,1), \:(1, \alpha/2)}.
\end{align*}
Unfortunately, the above expression does not hold for the case $\beta = 2$.
However, in all cases $\beta \in (0, 2]$, the following Fourier transform always
holds:
\begin{align}\label{E:FG}
  \mathcal{F}G(t, \cdot)(\xi) = t^{\beta + \gamma - 1} E_{\beta, \beta + \gamma}\left(-2^{-1} \nu t^\beta |\xi|^\alpha\right),
\end{align}
where $E_{\beta, \beta + \gamma}(\cdot)$ denotes the \textit{two-parameter
  Mittag-Leffler function}\index{Mittag-Leffler function} defined by
\begin{align}\label{E:ML}
  E_{a, b}(z) \coloneqq \sum_{k = 0}^{\infty} \frac{z^k}{\Gamma(a k + b)} \quad \text{for all } a \in \R_+, b \in \R, \text{ and } z \in \mathbb{C}.
\end{align}

In this section, we present the fundamental solutions along with their primary
properties in Section~\ref{SS:Fundamental}. Prior to this, we introduce related
special functions, including the Mittag-Leffler function in
Section~\ref{SS:MittagLeffler} and the Fox H-function in Section~\ref{SS:Fox}.

\subsection{Mittag-Leffler function} \label{SS:MittagLeffler}

The two-parameter Mittag-Leffler function $E_{a,b}(\cdot)$ plays an important
role in this work. Below, we list some properties of this special function that
will be used in this paper. For a more detailed account, interested readers may
refer to Section 1.8 of~\cite{kilbas.srivastava.ea:06:theory}.

\begin{lemma}\label{L:ML}
  The Mittag--Leffler function $E_{a,b}(\cdot)$ defined in~\eqref{E:ML}
  satisfies the following properties:

  \begin{enumerate}[\rm (i)]

    \item For all $a>0$ and $b\in\R$, $E_{a,b}(z)$ is an analytic function on
          $\mathbb{C}$. Moreover, if $b>0$, then $E_{a,b} (0) =
            \Gamma(b)^{-1} > 0$.

    \item Suppose that $a \in (0,2]$ and $b \in \R$. Then, by restricting $x\in \R$, $E_{a, b} (- x)$
          admits the following asymptotic expansions as $x \to \infty$:
          \begin{subequations}\label{E:asym_MT1}
            \begin{empheq}[left={E_{a, b}(-x) = \empheqlbrace}]{align}
              & \frac{1}{\Gamma(b - a) x} + O\left(x^{-2}\right),                                                                     &  & a < 2, a \ne b,       \label{E:asym_MT1-1} \\
              & \frac{1}{\Gamma(-a) x^2} + O\left(x^{-3}\right),                                                                      &  & 1 \ne a = b \in (0,2),\label{E:asym_MT1-2} \\
              & \exp\left( - x \right),                                                                                               &  & a = b =  1,           \label{E:asym_MT1-3} \\
              & \frac{\cos \left(\sqrt{x} + (1 - b) \pi /2\right)}{x^{(b - 1)/2}} + \frac{1}{\Gamma(b - 2) x} + O\left(x^{-2}\right), &  & a = 2\ne b,           \label{E:asym_MT1-4} \\
              & \frac{\sin \left(\sqrt{x}\right)}{\sqrt{x}},                                                                          &  & a = b =  2.           \label{E:asym_MT1-5}
            \end{empheq}
          \end{subequations}

    \item For some finite constant $C>0$, the following uniform upper bounds
          hold for all $x\ge 0$:
          \begin{subequations} % \label{E:ML-UB}
            \begin{empheq}[left={\left|E_{a, b}(-x)\right| \le \empheqlbrace}]{align}
              & \dfrac{C}{1 + x^{(b-1)/2}} &  & a = 2, \; b \in [2,3),    \label{E:ML-UB-1} \\
              & \dfrac{C}{1+x},            &  & a \in (0,2), \: b \in \R. \label{E:ML-UB-2}
            \end{empheq}
          \end{subequations}

    \item For any $a>0$, $b$, $\lambda \in \R$, and $m=1,2,3,\cdots$, it holds
          that
          \begin{align}\label{E:ML-deriv}
            \frac{\ud^m}{\ud x^m} \left(x^{b - 1} E_{a,b} \left(\lambda x^a\right)\right)
            = x^{b-1-m} E_{a, b-m}(\lambda x^a).
          \end{align}

  \end{enumerate}
\end{lemma}
\begin{proof}
  Part (i) follows directly from the definition~\eqref{E:ML}. For part (ii), we
  denote the five cases as (a) through (e). Cases (a) and (b) are derived
  from~(1.8.28) in~\cite{kilbas.srivastava.ea:06:theory}. Case (c) is based
  on~(1.8.2) in~\cite{kilbas.srivastava.ea:06:theory}. Cases (d) and (e) can be
  found in~(1.8.31) of~\cite{kilbas.srivastava.ea:06:theory}.

  For part (iii), when $\alpha=2$, observe that $(b-1)/2< 1$ if and only if
  $b<3$. Therefore, by considering both the asymptotic behavior at infinity in
  part (ii) and the boundedness at zero in part (i), we can
  deduce~\eqref{E:ML-UB-1}. Similarly, parts (i) and (ii) can be used to deduce
  the upper bound $C/(1+x)$ for the case when $b\ge a\in(0,2)$. The more general
  statement, where $a\in (0,2)$ and $b\in\R$, follows from Theorem~1.6
  in~\cite{podlubny:99:fractional}. Part (iv) can be derived from a direct
  computation; see, for example, (1.83) in~\cite{podlubny:99:fractional}.
\end{proof}

Note that the cases in~\eqref{E:asym_MT1-4} and~\eqref{E:asym_MT1-5} can be simplified further. Specifically, as $x\to\infty$,
\begin{subequations}\label{E:asym_MT2}
  \begin{empheq}[left={E_{2, b}(-x) = \empheqlbrace}]{align}
    & \frac{\cos \left(\sqrt{x} + (1 - b) \pi /2\right)}{x^{(b - 1)/2}} + O\left(x^{-1}\right), &  & b < 3, \label{E:asym_MT2-1} \\
    & \frac{1-\cos\left(\sqrt{x}\right)}{x}  + O\left(x^{-2}\right),                            &  & b = 3, \label{E:asym_MT2-2} \\
    & \frac{1}{\Gamma(b - 2) x} + O\left(x^{-\min\{2,(b-1)/2\}}\right),                         &  & b > 3. \label{E:asym_MT2-3}
  \end{empheq}
\end{subequations}

\subsection{Fox $H$-function}\label{SS:Fox}

The \textit{Fox $H$-function}, named after Charles Fox~\cite{fox:61:g}, is a
further generalization of the \textit{Meijer $G$-function}, which can be found
in Chapter~16 of~\cite{olver.lozier.ea:10:nist}. The Fox $H$-function is
essential for expressing the fundamental solutions to our equations.
Comprehensive details about the Fox $H$-function are available in Chapters 1 and
2 of~\cite{kilbas.saigo:04:h-transforms}. Additional information can be found in
Section 1.12 of~\cite{kilbas.srivastava.ea:06:theory} and Section 8.2
of~\cite{prudnikov.brychkov.ea:90:integrals}.

The Fox $H$-function has a wide range of applications in science and engineering;
see, for example, \cites{eidelman.ivasyshen.ea:04:analytic,mathai.saxena.ea:10:h-function}. This subsection provides a concise
overview of this special function. Additionally, some computations can be
facilitated by the open-source package~\cite{chen.hu:23:some}.

Let $m,n,p,q$ be integers such that $0 \le m \le q$ and $0 \le n \le p.$ Let
$a_i,b_j\in \mathbb{C}$ and $\alpha_i, \beta_j \in\R_+$ for $i=1,\cdots, p$ and
$j=1,\cdots, q$. Denote
\begin{align}
  a^*    & \coloneqq \sum_{i=1}^n \alpha_i -\sum_{i=n+1}^p \alpha_i + \sum_{j=1}^m \beta_j - \sum_{j=m+1}^{q}\beta_j, \label{E:a^star} \\
  a_1^*  & \coloneqq \sum_{j=1}^m \beta_j-\sum_{i=n+1}^p \alpha_i, \label{E:a_1^star}                                                  \\
  \Delta & \coloneqq \sum_{j=1}^q\beta_j-\sum_{i=1}^p\alpha_i, \label{E:Delta}                                                         \\
  \delta & \coloneqq \prod_{i=1}^{p}\alpha_i^{-\alpha_i} \prod_{j=1}^{q}\beta_j^{\beta_i}, \label{E:delta}                             \\
  \mu    & \coloneqq \sum_{j=1}^{q}b_j - \sum_{i=1}^{p}a_i + \frac{p-q}{2}, \label{E:mu}                                               \\
  A_0    & \coloneqq (2\pi)^{(p-q+1)/2} \Delta^{-\mu} \prod_{i=1}^p \alpha^{-a_i + 1/2} \prod_{j=1}^q \beta^{b_j-1/2}. \label{E:A0}
\end{align}
Now we consider the following ratio of the Gamma functions:
\begin{align}\label{E:H}
  \mathcal{H}^{m,n}_{p,q}(s) \coloneqq
  \dfrac{ \displaystyle \prod_{i=1}^n\Gamma\left(1-a_i-\alpha_is\right) }{ \displaystyle \prod_{i=n+1}^p\Gamma\left(a_i+\alpha_is\right)    }
  \times \dfrac{ \displaystyle \prod_{j=1}^m\Gamma\left(b_j+\beta_js\right)    }{ \displaystyle \prod_{j=m+1}^q\Gamma\left(1- b_j-\beta_js\right) }\:.
\end{align}
Assume that the two sets of poles of $\mathcal{H}^{m,n}_{p,q}(s)$ in~\eqref{E:H}
do not overlap, i.e., $A\cap B = \emptyset$ with
\begin{align}\label{E:poles}
  \begin{aligned}
    A \coloneqq \bigcup_{i=1}^n A_i \quad \text{with}\quad A_i & \coloneqq \bigg\{a_{ik}   =\frac{1-a_i + k}{\alpha_i}: \;\; k    =0, 1, \cdots\bigg\} \quad \text{and} \\
    B \coloneqq \bigcup_{j=1}^m B_j \quad \text{with}\quad B_j & \coloneqq \bigg\{b_{j\ell}=\frac{-b_j-\ell}{\beta_j }: \;\; \ell =0, 1, \cdots\bigg\}.
  \end{aligned}
\end{align}
Let $\mathcal{L}$ be any contour that separates these two sets of poles. Then
the following \textit{Mellin-Barnes integral} of $\mathcal{H}^{m,n}_{p,q}$,
known as the \textit{Fox $H$-function}~\cite{fox:61:g},
\begin{align}\label{E:Fox-H}
  \frac{1}{2\pi i} \int_{\mathcal{L}} \mathcal{H}^{m,n}_{p,q}(s) z^{-s} \ud s
  \eqqcolon \FoxH{m,n}{p,q}{z}{(a_1,\alpha_1),\cdots, (a_p,\alpha_p)}{(b_1,\beta_1),\cdots, (b_q,\beta_q)},
\end{align}
is well-defined in many cases. To be more precise, it is analytic with respect
to $z$ in the sector
\begin{align*}
  | \text{arg} (z) | < a^* \times \frac{\pi}{2} \quad \text{provided} \quad a^*>0.
\end{align*}
In the case $a^*\le 0$, we do not have the sector analyticity, but if $\Delta
  =0$, it is still analytic for both $|z|>\delta$ and $0<|z|<\delta$. At $|z| =
  \delta$, the function is well-defined when $\text{Re}(\mu)<-1$. See Theorems~1.1
and 1.2 of~\cite{kilbas.saigo:04:h-transforms} for the precise statement and
\cite{chen.hu:23:some} for a diagram exposition.

\subsection{Fundamental solutions}\label{SS:Fundamental}

The fundamental solutions to~\eqref{E:fde} have been studied in various specific
cases. When $d = 1$, $\beta = 1$, $\gamma = 0$, and $\alpha \in 2 \mathbb{N}$,
they were investigated in~\cites{hochberg:78:signed, krylov:60:some}.
Debbi~\cite{debbi:06:explicit} explored the fundamental solutions for $d = 1$,
$\beta = 1$, $\gamma = 0$, and $\alpha \in (1, \infty) \setminus \mathbb{N}$. In
collaboration with Dozzi~\cite{debbi.dozzi:05:on}, Debbi also studied the
corresponding SPDEs with space-time white noise. For the case $\beta \ne 1$,
fundamental solutions were examined in earlier works,
including~\cites{mijena.nane:15:space-time, chen:17:nonlinear,
  chen.hu.ea:17:space-time, chen.hu.ea:19:nonlinear}, among others. The following
theorem, cited from~\cite{chen.guo.ea:24:moments}, was first proved for the case
$\alpha \in (0, 2]$, $\beta \in (0, 2)$, and $\gamma = 0$
in~\cite{chen.hu.ea:17:space-time}, and later generalized to $\gamma \ge 0$ in
Theorem 4.1 of~\cite{chen.hu.ea:19:nonlinear}. The current version
in~\cite{chen.guo.ea:24:moments} extends the result to $\alpha > 0$ and $\beta
  \in (0, 2]$.

\begin{theorem}[Theorem 2.8 of~\cite{chen.guo.ea:24:moments}]\label{T:PDE}
  For $\alpha\in (0,\infty)$, $\beta\in (0,2]$, and $\gamma\ge 0$, the solution
  to
  \begin{align*} % \label{E:PDE}
    \begin{cases}
      \left(\partial_t^\beta + \dfrac{\nu}{2} (-\Delta)^{\alpha/2}\right) u(t,x) = I_t^\gamma\left[f(t,x)\right], & \qquad t>0,\: x\in\R^d,                         \\[1em]
      \left.\dfrac{\partial^k}{\partial t^k} u(t,x)\right|_{t = 0} = u_k(x),                                      & \qquad 0\le k\le \Ceil{\beta}-1, \:\: x\in\R^d,
    \end{cases}
  \end{align*}
  with $f$ and $u_k$ being smooth functions with compact support (i.e., $f\in
    C_c^\infty(\R_+\times\R^d)$ and $u_k\in C_c^\infty(\R^d)$) is
  \begin{align*} % \label{E:Duhamel}
    u(t,x) = J_0(t,x) + \int_0^t \ud s \int_{\R^d} \ud y\: f(s,y) \: \DRL{0+}{\Ceil{\beta}-\beta-\gamma} Z(t-s,x-y),
  \end{align*}
  where $\DRL{0+}{\Ceil{\beta}-\beta-\gamma}$ denotes the
  \textit{Riemann-Liouville derivative}\index{Riemann-Liouville integral} $D_{0+}^{\Ceil{\beta}-\beta-\gamma}$
  acting on the time variable,
  \begin{align*} %\label{E:J0}
    J_0(t,x)\coloneqq \sum_{k = 0}^{\Ceil{\beta}-1}\int_{\R^d} u_{k}(y) \partial_t^{\Ceil{\beta}-1-k} Z(t,x-y) \ud y
  \end{align*}
  is the solution to the homogeneous equation and $Z(t,x) \coloneqq
    Z_{\alpha,\beta,d}(t,x)$ is the corresponding fundamental solution. Denote
  \begin{gather*}
    Y(t,x) \coloneqq Y_{\alpha,\beta,\gamma,d}(t,x) = \DRL{0+}{\Ceil{\beta}-\beta-\gamma}
    Z_{\alpha,\beta,d}(t,x), \\
    Z^* (t,x) \coloneqq Z^{*}_{\alpha,\beta,d}(t,x) = \frac{\partial}{\partial t} Z_{\alpha,\beta,d}(t,x),
    \quad \text{if $\beta\in(1,2]$.}
  \end{gather*}
  Then, we have the following Fourier transforms:
  \begin{align}\label{E:FZ}
    \mathcal{F} Z_{\alpha,\beta,d}(t,\cdot)(\xi)        & = t^{\Ceil{\beta}-1} E_{\beta,\Ceil{\beta}}(-\tfrac12 \nu t^\beta |\xi|^\alpha),              \\
    \mathcal{F} Y_{\alpha,\beta,\gamma,d}(t,\cdot)(\xi) & = t^{\beta+\gamma-1} E_{\beta,\beta+\gamma}(-\tfrac12 \nu t^\beta |\xi|^\alpha), \label{E:FY} \\
    \mathcal{F} Z^*_{\alpha,\beta,d}(t,\cdot)(\xi)      & = t^{k} E_{\beta,k+1}(-\tfrac12\nu t^\beta |\xi|^\alpha), \quad \text{if $\beta\in (1,2]$}.
    \label{E:FZ*}
  \end{align}
  Moreover, when $\beta\in (0,2)$, we have the following explicit expressions:
  \begin{align}\label{E:Zab}
    Z(t,x) = \pi^{-\frac d2} t^{\Ceil{\beta}-1} |x|^{-d}
    \FoxH{2,1}{2,3}{\frac{|x|^\alpha}{2^{\alpha-1}\nu t^\beta}}{(1,1),\:(\Ceil{\beta},\beta)}
    {(d/2,\alpha/2),\:(1,1),\:(1,\alpha/2)},
  \end{align}
  \begin{align}\label{E:Yab}
    \begin{aligned}
      Y(t,x) = \pi^{-\frac d2} |x|^{-d}t^{\beta+\gamma-1}
      \FoxH{2,1}{2,3}{\frac{|x|^\alpha}{2^{\alpha-1}\nu t^\beta}}
      {(1,1),\:(\beta+\gamma,\beta)}{(d/2,\alpha/2),\:(1,1),\:(1,\alpha/2)},
    \end{aligned}
  \end{align}
  and, if $\beta\in (1,2)$,
  \begin{align}\label{E:Z*ab}
    Z^{*}(t,x)=
    \pi^{-\frac d2} |x|^{-d}
    \FoxH{2,1}{2,3}{\frac{|x|^\alpha}{2^{\alpha-1}\nu t^\beta}}{(1,1),\:(1,\beta)}
    {(d/2,\alpha/2),\:(1,1),\:(1,\alpha/2)},
  \end{align}
  where $H_{2,3}^{2,1}$ %$\FoxH{2,1}{2,3}{\cdots}{\cdots}{\cdots}$
  refers to the Fox $H$-function defined as in~\eqref{E:Fox-H}.
\end{theorem}

\begin{remark}\label{R:Conditions}
  For future reference, we note that the corresponding parameters $(a^*, a_1^*,
    \Delta,\mu, \delta)$, as defined in~\eqref{E:a^star},~\eqref{E:a_1^star},~\eqref{E:Delta},~\eqref{E:mu}, and~\eqref{E:delta}, respectively, are
  identical across all three Fox $H$-functions appearing in $Z(1,x)$, $Z^*(1,x)$ and
  $Y(1,x)$:
  % \begin{multicols}{2}
  \begin{align}\label{E:H-Parameters}
    \left\{ \quad
    \begin{aligned}
      a^*    & = 2-\beta,                                                              \\
      a_1^*  & = \alpha/2+1-\beta,                                                     \\
      \Delta & = \alpha - \beta,                                                       \\
      \mu    & = \frac{d+1}{2} - (\beta+\gamma),                                       \\
      \delta & = \delta(\alpha,\beta) = \left(\alpha/2\right)^{\alpha} \beta^{-\beta}.
    \end{aligned}
    \right.
  \end{align}
  % \end{multicols}
  In particular, $\delta(\alpha,\alpha) = 2^{-\alpha} +(2\alpha)^{-\alpha/2}$,
  $\delta(2,\beta) = 2 \beta^{-\beta}$, and $\delta(2,2) = 1/2$. However, the
  values of parameter $\mu$ defined in~\eqref{E:mu} are different, which are
  equal to
  \begin{align*}
    \frac{1}{2} (-2 \lceil \beta \rceil +d+1), \quad
    \frac{d-1}{2}, \quad \text{and} \quad
    \frac{1}{2} (-2 \beta -2 \gamma +d+1),
  \end{align*}
  for $Z(1,x)$, $Z^*(1,x)$ and $Y(1,x)$, respectively. We refer the interested
  reader to the symbolic computational tools and the accompanying documentation
  in~\cite{chen.hu:23:some} for handling these tedious computations associated
  with the Fox H-function.
\end{remark}

Throughout the paper, we use the convention that $G(t,x) =
  Y_{\alpha,\beta,\gamma,d}(t,x)$. From the expression~\eqref{E:Yab}, we obtain
the following scaling property
\begin{gather}\label{E:Gscaling}
  G(rs, r^{\beta/\alpha}y) = r^{\beta+\gamma-1-d\beta/\alpha} G(s, y).
\end{gather}

\section{Asymptotic properties of the fundamental solutions}\label{S:Asymptotics}

In this section, we establish sharp asymptotic properties of the fundamental
solution\index{fundamental solution}\index{fundamental solution!asymptotics} at both zero\index{asymptotics!origin} and infinity\index{asymptotics!infinity} for the spatial variable. While these
estimates are used as inputs in the solvability analysis
(Section~\ref{S:Dalang}) and the subsequent increment/SLND arguments, they are
also of independent interest in their own right. Thanks to the scaling
property~\eqref{E:Gscaling}, we can focus on the behavior of $G(1,x)$. Moreover,
it suffices to establish the asymptotic properties of $G(1,x)$ for the case
$\beta\in (0,2)$ via studying
\begin{align}\label{E:f(x)}
  f(z) \coloneqq \FoxH{2,1}{2,3}{z}{(1,1),\;(\beta+\gamma,\beta)}{(d/2,\alpha/2),\:(1,1),\:(1,\alpha/2)}, \quad |\arg z| < (1-\beta /2) \pi,
\end{align}
since for $x\in\R^d$,
\begin{align}\label{E:G-f}
  G(1,x) = C_1 |x|^{-d} f\left(C_2|x|^\alpha\right), \quad \text{with}\quad
  C_1 = \pi^{-d/2}2^{d (1-1/\alpha)}\nu^{d/\alpha} \quad \text{and} \quad
  C_2 = 2^{1-\alpha}\nu^{-1}.
\end{align}

In the following, we will use the notation: $\mathbb{N}_0\coloneqq
  \{0,1,2,\ldots\}$, $\mathbb{N} \coloneqq \{1,2,\ldots\}$, $2\mathbb{N} \coloneqq
  \{2,4,\ldots\}$, $2\mathbb{N}-1 \coloneqq \{1,3,5,\ldots\}$, $2\mathbb{N}+2
  \coloneqq \{4,6,\ldots\}$, and so forth.

\subsection*{Main results}
\addcontentsline{toc}{subsection}{Main results}

\paragraph{Asymptotics at the origin.}\index{asymptotics!origin}
The behavior of $G(1,x)$ near $x=0$ governs the local singularity of the kernel,
and hence it plays a decisive role in the integrability estimates that enter the
Dalang-type solvability condition (Section~\ref{S:Dalang}) and the increment/SLND
analysis. The next theorem provides a sharp algebraic expansion at the origin,
including the logarithmic corrections that arise at the ``arithmetic'' values of
$\alpha$ through pole collisions in the Fox $H$-function representation. We
prove it in Subsection~\ref{SS:Zero}.

\begin{theorem}[Asymptotics at the origin]\label{T:Zero}
  Suppose $\alpha>0$, $\beta\in(0,2)$, $\gamma \ge 0$, and $d \in \mathbb{N}$.
  The function $f(z)$, defined in~\eqref{E:f(x)}, is an analytic function within
  the cone $|\arg z| < (1-\beta/2)\pi$.  Moreover, we have the following
  algebraic asymptotic expansion for $f(z)$ at zero:

  \begin{enumerate}[\rm (i)]

    \item If $\alpha>0$ satisfies the following condition
          \begin{align}\label{E:Rationals}
            \alpha \not\in \left\{\frac{2\ell_1+d}{\ell_2+1} :\: \ell_1, \ell_2 = 0, 1,2,\cdots\right\},
          \end{align}
          then
          \begin{align}\label{E:ZeroExpand}
            f(z) \sim \sum_{\ell=0}^{\infty} h_{1,\ell}^* z^{\frac{d+2\ell}{\alpha}} + \sum_{\ell=0}^{\infty} h_{2,\ell}^*  z^{\ell+1} \quad \text{as $z\to 0$ with $|\arg z |< (1-\beta/2)\pi$},
          \end{align}
          where
          \begin{align}\label{E:hstar}
            \begin{aligned}
              h_{1,\ell}^* & = \frac{2(-1)^\ell \pi \csc\left((d+2\ell)\pi/\alpha\right)}{\alpha \ell! \Gamma(\ell +d/2) \Gamma\left(\beta+\gamma - (d+2\ell) \beta/\alpha\right)} \quad \text{and} \\
              h_{2,\ell}^* & = \frac{(-1)^\ell \Gamma \left((d-(\ell+1) \alpha)/2\right)}{\Gamma \left((\ell+1) \alpha/2\right) \Gamma (\gamma - \ell \beta)}.
            \end{aligned}
          \end{align}

    \item Otherwise, suppose that
          \begin{align}\label{E:alpha-Q}
            \alpha = \frac{2\ell_1 +d}{\ell_2+1} \quad \text{for some $(\ell_1,\ell_2) \in \mathbb{N}_0^2$}.
          \end{align}
          Let $(\ell_1^*,\ell_2^*)$ be the smallest such integers, i.e.,
          \begin{align*} % \label{E:ell-star}
            (\ell_1^*,\ell_2^*) \coloneqq \argmin \left\{(\ell_1,\ell_2)\in \mathbb{N}_0^2:\: \alpha = \frac{2\ell_1 +d}{\ell_2+1}\right\}.
            % \alpha = \frac{2\ell_1^*+d}{\ell_2^* + 1}.
          \end{align*}
          Denote $\mathcal{L}_1^c \coloneqq \mathbb{N}_0 \setminus \mathcal{L}_1$ and $\mathcal{L}_2^c \coloneqq \mathbb{N}_0 \setminus \mathcal{L}_2$ with
          \begin{align}\label{E:calL_1-2}
            \begin{aligned}
              \mathcal{L}_1 & \coloneqq \left\{\frac{(2\ell_1^*+d) r -d}{2}:\: \text{$r \in \mathbb{N}$ when $d$ is even or $r\in 2\mathbb{N}-1$ when $d$ is odd} \right\} \quad \text{and} \\
              \mathcal{L}_2 & \coloneqq \bigg\{(\ell_2^*+1) r - 1:\: \text{$r \in \mathbb{N}$ when $d$ is even or $r \in 2\mathbb{N}-1$ when $d$ is odd} \bigg\}.
            \end{aligned}
          \end{align}
          Then for $|\arg z|< (1-\beta/2) \pi$ and as $z\to 0$,
          \begin{align}\label{E:PowerLogSeries}
            \begin{aligned}
              f(z) = & \sum_{\ell\in \mathcal{L}_1^c} h_{1,\ell}^* z^{\frac{d+2\ell}{\alpha}}
              + \sum_{\ell\in \mathcal{L}_2^c} h_{2,\ell}^* z^{\ell+1}                                                                                                                             \\
                     & + \frac{(-1)^{\ell_1^*+\ell_2^*+1} \left(\ell_2^*+1\right)}{\ell_1^*!\: \Gamma \left(d/2+\ell_1^*+1\right) \Gamma \left(\gamma -\beta \ell_2^*\right)} z^{\ell_2^*+1}\log z
              + o\left(z^{\ell_2^*+1}\log z\right),
            \end{aligned}
          \end{align}
          where the coefficients $h_{j,\ell}^*$, $j=1,2$, are defined in~\eqref{E:hstar}.

  \end{enumerate}
\end{theorem}

Together with~\eqref{E:G-f}, Theorem~\ref{T:Zero} gives the complete description
of the behavior of $G(1,x)$ near $x=0$. Part of its statement extends Lemma 4.3
from~\cite{chen.hu.ea:19:nonlinear}, generalizing the result from
$\alpha\in(0,2]$ to $\alpha > 0$.

\begin{remark}\label{rmk:G:+ve}
  The nonnegativity of $G(t,x)$ over its entire domain is highly nontrivial and
  holds only in certain special cases. For a detailed discussion, see Remark 1.2
  in~\cite{chen.eisenberg:23:interpolating}. However, in the present paper, we
  only need $G(t,x)$ to be either nonnegative or nonpositive in a neighborhood
  of $(0_+,0)$; see~\eqref{E:Gpositive} below. In other words,
  Theorem~\ref{T:Zero} is sufficient for this purpose.
\end{remark}

Thanks to the relation~\eqref{E:G-f}, we have the following corollary:

\begin{corollary}\label{C:GZero}
  If $\alpha>0$, $\beta\in(0,2)$, $\gamma \ge 0$, and $d \in \mathbb{N}$, then
  $G(1,x)$ (and similarly $Z(1,x)$ and $Z^*(1,x)$) is a smooth function for
  $x\in\R^d \setminus \{0\}$. At $x=0$, it may have a singularity of the form
  either $|x|^{-\theta}$ or $|x|^{-\theta}\log|x|$ for some $\theta \in [0, d)$.
\end{corollary}

\begin{remark}\label{R:Asymptotics-Rationals-finiteness}
  Condition~\eqref{E:Rationals} is a necessary and sufficient condition for all
  coefficients $h_{j,\ell}^*$, $j=1,2$ and $\ell \ge 0$, as defined
  in~\eqref{E:hstar}, to be finite. Conversely, these coefficients
  $h_{1,\ell_1}^*$ and $h_{2,\ell_2}^*$ blow up precisely when
  condition~\eqref{E:alpha-Q} is satisfied and for those indices $\ell_j$,
  $j=1,2$, such that $\ell_j\in \mathcal{L}_j$ (see~\eqref{E:ell_1-2}). In this
  case, we have that
  \begin{align*}
    \frac{(d+2\ell_1)\pi}{\alpha} = (\ell_2+1)\pi \quad \text{and} \quad
    \frac{d-(\ell_2+1)\alpha}{2} = -\ell_1,
  \end{align*}
  which correspond to simple poles of $\csc(\cdot)$ in the expression for
  $h_{1,\ell_1}^*$ and $\Gamma(\cdot)$ in the expression for $h_{2,\ell_2}^*$,
  respectively.
\end{remark}

\begin{remark}\label{R:Asymptotics-Rationals-arithmetic}
  Condition~\eqref{E:Rationals} is equivalent to stating that if $d$ is even, this
  condition excludes every positive rational number, so $\alpha$ must be an
  irrational number. If $d$ is odd, then only those rational $\alpha=M/N$ whose
  numerator $M$ is odd are excluded. Thus, for odd $d$, rational numbers with an
  even numerator and all irrational numbers remain permissible for $\alpha$.
\end{remark}

\begin{remark}\label{R:Asymptotics-Rationals-special-cases}
  Here are some special cases when condition~\eqref{E:Rationals} is satisfied:

  \begin{enumerate}[(i)]

    \item If $\gamma = 0$ and $\beta \ne 1$, then only the first term of the
          second series expansion in~\eqref{E:ZeroExpand} vanishes. That is,
          $h_{2,0}^*\equiv 0$ and $h_{2,\ell}^*\ne 0$ for all $\ell \ge 1$.

    \item If $\gamma = 0$ and $\beta = 1$, then the entire second series
          expansion in~\eqref{E:ZeroExpand} vanishes because
          $1/\Gamma\left(\gamma-\ell \beta\right)\equiv 0$. More rigorously, this
          can be shown by using Property~2.2 of~\cite{kilbas.saigo:04:h-transforms}
          to reduce the function $f(x)$ to the following form
          \begin{align}\label{E:Reduced-f}
            f(x)=\FoxH{1,1}{1,2}{x}{(1,1)}{(d/2,\alpha/2),\:(1,\alpha/2)}
          \end{align}
          and then applying Theorem~1.3 (\textit{ibid.}) to
          obtain~\eqref{E:ZeroExpand}. Consequently, the second series expansion
          in~\eqref{E:ZeroExpand} does not appear.

    \item When the parameter $\Delta$ defined in~\eqref{E:Delta} is strictly
          positive (i.e., $\beta<\alpha$), Theorem~1.3 (\textit{ibid.}) implies that
          the asymptotic expansion in~\eqref{E:ZeroExpand} is actually a power
          series expansion valid for all $z \ne 0$ on the entire complex plane, and
          the cone condition in~\eqref{E:ZeroExpand} can be removed.

    \item Similarly, when $\beta=\alpha$, Theorem~1.3 (\textit{ibid.}) implies
          that the asymptotic expansion in~\eqref{E:ZeroExpand} is actually a power
          series expansion valid for all $z$ such that
          \begin{align*}
            0 < |z| < \delta(\alpha, \alpha) = 2^{-\alpha}
            \quad \text{(see~\eqref{E:H-Parameters} for the definition of $\delta(\alpha, \beta)$)},
          \end{align*}
          and the cone condition in~\eqref{E:ZeroExpand} can be removed.

  \end{enumerate}
\end{remark}

\begin{remark}\label{R:Asymptotics-alphaQ-special-cases}
  Here are some special cases under condition~\eqref{E:alpha-Q}:

  \begin{enumerate}[(i)]

    \item When $d$ is even and $\ell_2^*=0$, we have $\mathcal{L}_2 =
            \mathbb{N}_0$ and $\mathcal{L}_2^c=\emptyset$. Consequently, the second
          series expansion in~\eqref{E:PowerLogSeries} vanishes.

    \item If $d \in \{1,2\}$ and $\ell_1^*=0$, then $\mathcal{L}_1 =
            \mathbb{N}_0$ and $\mathcal{L}_1^c=\emptyset$. Consequently, the first
          series expansion in~\eqref{E:PowerLogSeries} vanishes.

    \item When $d$ is odd, $\mathcal{L}_2^c$ is never empty, and hence the
          second series expansion in~\eqref{E:PowerLogSeries} always persists.

    \item When $\gamma = 0$ and $\beta = 1$, the coefficient for
          $z^{\ell_2^*+1}\log z$ in~\eqref{E:PowerLogSeries} is zero. Indeed, in
          this case, there are no logarithm terms in the expansion. This is because
          $f(x)$ must first be reduced to the form~\eqref{E:Reduced-f}, and then it
          can be observed that the entire poles $B_2$ in~\eqref{E:AB-poles} vanish.
          There is only one list of poles $B_1$, and there won't be any
          possibilities of pole collisions.

  \end{enumerate}
\end{remark}

By the identity theorem (see, e.g.,
\cite{ablowitz.fokas:03:complex}*{Theorem~3.2.6}) and the fact that analytic
functions $Z(1,x)$, $Z^*(1,x)$ and $Y(1,x)$ are not identical to zero, we have
the following corollary:

\begin{corollary}\label{C:GSupport}
  Suppose $\alpha>0$, $\beta\in(0,2)$, $\gamma \ge 0$, and $d \in \mathbb{N}$.
  The support of $G(1,x)$ (and similarly $Z(1,x)$ and $Z^*(1,x)$) is the whole
  space $x\in\R^d$.
\end{corollary}

\paragraph{Asymptotics at infinity.}\index{asymptotics!infinity}
We next turn to the regime $|x|\to\infty$, which quantifies the spatial decay of
the fundamental solution and is equally essential for the global integrability
properties needed in the subsequent SPDE analysis. The following theorem gives
the leading-order behavior at infinity across the full parameter range. Its
proof and further discussion are presented in Subsection~\ref{SS:Infinity}.

\begin{theorem}[Asymptotics at infinity]\label{T:Infinity}
  For $\alpha>0$, $\beta\in (0, 2)$ and $\gamma\ge 0$, the fundamental solution
  $G(1,x)$ has the following asymptotic expansions:
  \begin{enumerate}[\rm (i)]

    \item If $\alpha\not\in 2 \mathbb{N}$, then, as $|x|\to+\infty$,
          \begin{align*}
            G(1,x) = \Theta_1\: |x|^{-(d+\alpha)} + o\left(|x|^{-(d+\alpha)}\right),
          \end{align*}
          with
          \begin{align*}
            \Theta_1 \coloneqq 2^{d (1-1/\alpha)+\alpha-1}\nu^{1+d/\alpha}\pi^{-(1+d/2)} \frac{\Gamma\left((d+\alpha)/2\right) \Gamma\left(1+\alpha/2\right)}{\Gamma(2\beta+\gamma)} \sin\left(\pi\alpha /2\right).
          \end{align*}

    \item If $\alpha = 2$, then, as $|x|\to+\infty$,
          \begin{align*}
            G(1,x)
             & = \Theta_{2,1}\: |x|^{\frac{\alpha(d/2+1-(\beta+\gamma))}{2-\beta}-d} \exp\left( - \Theta_{2,2} |x|^{\alpha/(2-\beta)} \right) \left(1+O\left(|x|^{-\alpha/(2-\beta)}\right)\right),
          \end{align*}
          with
          \begin{align*}
            \Theta_{2,1} & = 2^{d\left(1-\frac{1}{\alpha}\right) + \frac{(1-\alpha)\left(d/2+1-(\beta+\gamma)\right)}{2-\beta}} \pi^{-d/2} \nu^{\frac{d}{\alpha} -\frac{d/2+1-(\beta+\gamma)}{2-\beta}}
            (2-\beta)^{-1/2} \beta^{\frac{-4\gamma+\beta (d-3)+2}{4-2\beta}}>0,                                                                                                                         \\
            \Theta_{2,2} & = 2^{(1-\alpha)/(2-\beta)} \nu^{-1/(2-\beta)} (2-\beta) \beta^{\beta / (2-\beta)} > 0.
          \end{align*}

    \item If $\alpha \in 2\mathbb{N}+2$, then, as $|x|\to+\infty$, write
          $\ell\coloneqq \alpha/2\in\mathbb{N}$ (so $\ell\ge 2$) and
          \begin{align*}
            G(1,x) & = \Theta_{3,1}\: |x|^{\frac{\alpha(d/2+1-(\beta+\gamma))}{\alpha-\beta}-d}
            \cos\left(\Theta_{3,2} - \Theta_{3,3}\cos\left(\theta_{\alpha,\beta}\right)\: |x|^{\alpha/(\alpha-\beta)}\right)                                                                     \\
                   & \quad \times \exp\left(- \Theta_{3,3}\sin\left(\theta_{\alpha,\beta}\right)\: |x|^{\alpha/(\alpha-\beta)}\right) \left(1+O\left(|x|^{-\alpha/(\alpha-\beta)}\right)\right),
          \end{align*}
          where $\theta_{\alpha,\beta}\in (0, \pi/2)$ is defined in~\eqref{E:theta}
          and
          \begin{align*}
            \Theta_{3,1} & = 2^{\frac{\alpha ^3+\alpha ^2 (2 d-\beta )-2 \alpha  (2 (\beta +\gamma -1)+(\beta +1) d)+4 \beta  d}{4 \alpha  (\alpha -\beta)}}
            \pi^{(\alpha/2-d-1)/2}
            \nu^{\frac{d}{\alpha} -\frac{d/2+1-(\beta+\gamma)}{\alpha-\beta}}                                                                                \\
                         & \quad \times \alpha^{\frac{\beta(1-d)+2\ell(2\beta+2\gamma-3)}{4\ell-2\beta}}
            \left(\alpha-\beta\right)^{-1/2}
            \beta^{\frac{\beta(1+d)-2 \ell(2\beta +2\gamma-1)}{4\ell-2\beta}},                                                                               \\
            \Theta_{3,2} & = \pi\left(\frac{\alpha}{2}-1\right)\left(\frac{d}{2}+1-(\beta+\gamma)\right)(\alpha-\beta)^{-1},                                 \\
            \Theta_{3,3} & = 2^{1/(\alpha-\beta)} \nu^{-1/(\alpha-\beta)} \alpha^{\alpha/(\beta-\alpha)} (\alpha-\beta) \beta^{\beta/(\alpha-\beta)} >0.
          \end{align*}

  \end{enumerate}
\end{theorem}

\subsection{Proof of Theorem~\ref{T:Zero} (asymptotics at the origin)}\label{SS:Zero}

\begin{proof}[Proof of Theorem~\ref{T:Zero}]
  From~\eqref{E:H}, we denote
  \begin{align*} %\label{E_:H(s)}
    \calH(s) \coloneqq \calH_{2,3}^{2,1}(s)
    = \frac{\Gamma(d/2+\alpha s/2)\Gamma(1+s)\Gamma(-s)}{\Gamma(\gamma+\beta (s+1))\Gamma(-\alpha s/2)}.
  \end{align*}
  The two sets of poles of the above $\calH(s)$ do not overlap because
  from~\eqref{E:poles}, we see that $A \cap B = \emptyset$ where $A = A_1$, $B =
    B_1\cup B_2$,
  \begin{align}\label{E:AB-poles}
    A_1 = \{0,1,2,\cdots\}, \quad
    B_1 = \left\{-\frac{2l+d}{\alpha}: l=0,1,2,\cdots\right\}, \quad \text{and} \quad
    B_2 = \{-1,-2,-3,\cdots\}.
  \end{align}
  Notice that the parameter $a^*$ defined in~\eqref{E:a^star} is equal to
  \begin{align*}
    a^* = 1-\beta + \left(\frac{\alpha}{2}+1\right) -\frac{\alpha}{2} = 2 - \beta>0.
  \end{align*}
  Hence, by part (iii) of Theorem~1.2 in~\cite{kilbas.saigo:04:h-transforms}, we
  know that $f(z)$ is an analytic function of $z$ in the sector $|\arg z|<
    (1-\beta/2) \pi$. \medskip

  It remains to establish the algebraic asymptotic expansions of $f(z)$ at zero.
  This can be done following Section~1.8 of~\cite{kilbas.saigo:04:h-transforms}.
  Notice that, under condition $a^*=2-\beta>0$, both Theorems~1.11 and~1.12
  (\textit{ibid.}) are applicable.  In particular, we will need to consider
  several cases.

  \medskip\noindent\textit{Case~1.~} Assume that $B_1\cap B_2 = \emptyset$,
  where $B_j$ are defined in~\eqref{E:AB-poles}. This condition is equivalent to
  condition~\eqref{E:Rationals}.  In this case, one can obtain the following
  explicit asymptotic expansions at $z=0$ following (1.8.1) (\textit{ibid.}) as
  given in~\eqref{E:ZeroExpand}. The coefficients $h^*_{j,\ell}$ are equal to
  the residue of $\calH(s)$ at the pole of the $\ell$-th of $B_j$ for $j\in
    \{1,\cdots, m\}$ and $\ell\ge 0$. These residues can be obtained via the
  following formula
  \begin{align*}
    h_{j,\ell}^* = \frac{(-1)^\ell}{\ell!\beta_j} \times
    \frac{\displaystyle \prod_{i=1, i\ne j}^m \Gamma\left(b_i-[b_j+\ell] \beta_i/\beta_j\right)\prod_{i=1}^n \Gamma\left(1-a_i+[b_j+\ell] \alpha_i/\beta_j\right)}
    {\displaystyle \prod_{i=n+1}^p \Gamma\left(a_i-[b_j+\ell] \alpha_i/\beta_j\right)\prod_{i=m+1}^q \Gamma\left(1-b_i+[b_j+\ell] \beta_i/\beta_j\right)};
  \end{align*}
  See (1.3.6) (\textit{ibid.}) for full details of the computation.

  \medskip\noindent\textit{Case~2.~} Assume that $B_1 \cap B_2 \ne \emptyset$.
  Let $(\ell_1^*, \ell_2^*)$ be the smallest integers with $\ell_j^* \ge 0$ such
  that~\eqref{E:alpha-Q} holds. It follows from a straightforward exercise that
  there are infinitely many pairs of $(\ell_1, \ell_2)$ such that $\alpha =
    (2\ell_1 + d)/(\ell_2 + 1)$. These pairs are given by
  \begin{align}\label{E:ell_1-2}
    (\ell_1, \ell_2) = \left(\frac{(2\ell_1^* + d) r - d}{2}, (\ell_2^* + 1) r - 1\right)
  \end{align}
  for $r \ge 1$ for even $d$, and $r \ge 1$ odd for odd $d$. These values are
  listed in the sets~\eqref{E:calL_1-2}. In this case, we apply Theorem~1.12
  (\textit{ibid.}) to obtain~\eqref{E:PowerLogSeries}. This completes the proof
  of Theorem~\ref{T:Zero}.
\end{proof}

\subsection{Proof of Theorem~\ref{T:Infinity} (asymptotics at infinity)}\label{SS:Infinity}

\begin{proof}[Proof of Theorem~\ref{T:Infinity}]
  This is a direct consequence of Lemma~\ref{L:HAtInfty} below and the relation
  in~\eqref{E:G-f}.
\end{proof}

The next lemma provides the asymptotic expansion of $f(z)$, defined
in~\eqref{E:f(x)}, at infinity. Note that the relationship between $f(x)$ and
the fundamental solution $G(1,x)$ (see~\eqref{E:G-f}) is valid only for $\beta <
  2$. However, the following lemma also includes the case $\beta = 2$ for future
reference.

\begin{lemma}\label{L:HAtInfty}
  Suppose $\alpha>0$, $\beta\in(0,2]$, $\gamma \ge 0$, and $d \in \mathbb{N}$.
  Denote the following cases:
  \begin{multicols}{2}
    \begin{itemize}

      \item Case 1: $0 < \alpha \le \beta \le 2$;

      \item Case 2: $0 < \beta < 2 \wedge \alpha$ and $\alpha \not\in 2\mathbb{N}$;

      \item Case 3: $\alpha>\beta =2$ and $\alpha\not\in 2\mathbb{N}$;
            \begin{itemize}
              \item Subcase 3-1: $\gamma < \alpha - 3 + d/2$;
              \item Subcase 3-2: $\gamma > \alpha - 3 + d/2$;
              \item Subcase 3-3: $\gamma = \alpha - 3 + d/2$.
            \end{itemize}

      \item Case 4: $\alpha = 2$ and $\beta < 2$;

      \item Case 5: $\alpha \in 2 \mathbb{N}+2$ and $\beta \in (0, 2)$.

    \end{itemize}
  \end{multicols}
  \noindent Then the function $f(x)$, defined in~\eqref{E:f(x)}, has the
  following asymptotic expansion at infinity along the positive real line as
  $x\to+\infty$:
  \begin{align*}
    f(x) =
    \begin{dcases}
                                                & \text{\textit{Case(s)}}                                                                                     \\[1em]
      \frac{h_{1,1}}{x} + o\left(x^{-1}\right), & \text{1, 2, 3-1;}                                                                                           \\[1em]
      \frac{2^{\frac{4d-2-4\alpha\gamma-3\alpha}{2(\alpha-2)}} \alpha ^{\frac{\alpha  \gamma +\alpha -d}{\alpha-2}}}{\sqrt{\alpha -2}}
      \sin\left( C_3\: x^{1/(\alpha-2)} + \frac{\pi}{4} \left(d-2\gamma\right)\right) x^{-\frac{2\gamma-d+2}{2(\alpha-2)}}
      + o\left(x^{-\frac{2 \gamma -d+2}{2(\alpha-2)}}\right),
                                                & \text{3-2;}                                                                                                 \\[1em]
      \left( \frac{2^{-2\alpha-d+1/2} \alpha^{\alpha +d/2}}{\sqrt{\alpha-2}}\sin\left(C_3\: x^{1 /(\alpha-2)} + \pi(3-\alpha)/2 \right) + h_{1,1}\right)\frac{1}{x} + o\left(x^{-1}\right),
                                                & \text{3-3;}                                                                                                 \\[1em]
      \frac{\beta ^{\frac{-4\gamma+\beta (d-3)+2}{4-2\beta}}}{\sqrt{2-\beta}} x^{\frac{d/2+1-(\beta+\gamma)}{2-\beta}}                                        \\
      \quad \times \exp\left( -(2-\beta) \beta^{\beta / (2-\beta)} x^{1/(2-\beta)} \right) \left(1+O\left(x^{-1/(2-\beta)}\right)\right),
                                                & \text{4;}                                                                                                   \\[1em]
      C_{5,1}\: x^{\frac{d/2+1-(\beta+\gamma)}{\alpha-\beta}} \cos\left(C_{5,2} - C_{5,3}\cos\left(\theta_{\alpha,\beta}\right)\: x^{1/(\alpha-\beta)}\right) \\
      \quad \times \exp\left(- C_{5,3}\sin\left(\theta_{\alpha,\beta}\right)\: x^{1/(\alpha-\beta)}\right) \left(1+O\left(x^{-1/(\alpha-\beta)}\right)\right),
                                                & \text{5;}                                                                                                   \\[1em]
    \end{dcases}
  \end{align*}
  where the constants are given below
  \begin{gather}\label{E:h11}
    h_{1,1} = \frac{\Gamma\left((d+\alpha)/2\right) \Gamma\left(1+\alpha/2\right)}{\pi\: \Gamma(2\beta+\gamma)} \sin\left(\pi\alpha /2\right), \\
    C_3     = (\alpha -2)\: 2^{(\alpha+2)/(\alpha-2)} \alpha^{-\alpha/(\alpha -2)}, \nonumber
  \end{gather}
  and in Case 5,
  \begin{align}\label{E:theta}
    \theta_{\alpha, \beta} \coloneqq \frac{2-\beta}{\alpha-\beta} \times \frac{\pi}{2} \in (0, \pi/2),
  \end{align}
  and % by setting $\ell \coloneqq \alpha/2\in \mathbb{N}-1$,
  \begin{align*}
    C_{5,1} & =  (2\pi)^{(\alpha/2 -1)/2} \left(\alpha -\beta \right)^{-1/2}
    (\alpha/2)^{\frac{\beta(1-d)+ \alpha(2\beta+2\gamma-3)}{2\alpha-2\beta}}
    \beta^{\frac{\beta(1+d)- \alpha (2\beta +2\gamma-1)}{2\alpha-2\beta}},                         \\
    C_{5,2} & = \frac{\pi(\alpha/2-1)(d/2+1-(\beta+\gamma))}{\alpha-\beta}, \quad \text{and}       \\
    C_{5,3} & = (\alpha/2)^{\alpha/(\beta-\alpha)} (\alpha-\beta) \beta^{\beta/(\alpha-\beta)} >0.
  \end{align*}
  % \begin{gather*}
  %   C_{5,1} = (2 \pi)^{(\alpha-2)/4} \left(\frac{\alpha}{2}\right)^{\frac{d+\alpha-1}{2(1-\alpha)}}(\alpha-1)^{-1/2}, \quad
  %   C_{5,2} = -\frac{\pi d (\alpha-2)}{4(2\alpha-1)},                                                                                            \\
  %   C_{5,3} = (\alpha-1) \left( \frac{\alpha}{2} \right)^{\frac{\alpha}{1-\alpha}} \cos \left((\alpha-1)^{-1}\pi/2\right), \quad \text{and}\quad \\
  %   C_{5,4} = \left(\frac{\alpha}{2}\right)^{\frac{\alpha}{1-\alpha}} (\alpha-1) \sin \left((\alpha-1)^{-1}\pi/2\right) >0.
  % \end{gather*}
\end{lemma}

\begin{remark}\label{R:Asymptotics-theta-alpha2}
  Note that $\theta_{2,\beta} = \pi/2$. Hence, by setting $\alpha =2$ in Case 5,
  one obtains Case 4.
\end{remark}

\begin{proof}[Proof of Lemma~\ref{L:HAtInfty}]
  As demonstrated in the proof of Theorem~\ref{T:Zero}, the two sets of poles,
  $A$ and $B$ (see~\eqref{E:AB-poles}), do not overlap. Furthermore, since
  $n=1$, meaning there is only one list of poles in $A$, both conditions (1.1.6)
  and (1.3.2) of~\cite{kilbas.saigo:04:h-transforms} are satisfied.
  Consequently, we can apply Theorem 1.7
  from~\cite{kilbas.saigo:04:h-transforms}. The conditions either $\Delta \le 0$
  or $\Delta > 0$ and $a^* > 0$ translate, as seen in~\eqref{E:H-Parameters}, to
  either $\alpha \le \beta$ or $\beta < 2 \wedge \alpha$, which are satisfied
  under the condition that $\alpha > 0$ and $\beta \in (0,2)$.

  Now we consider the following cases:

  \medskip\noindent\textbf{Case~1:} $\alpha \le \beta \le 2$. This case
  corresponds to the case $\Delta \le 0$. Both this case and Case 2 below are
  covered in Section 1.5 of~\cite{kilbas.saigo:04:h-transforms}. By Eq.~(1.5.1)
  of~\cite{kilbas.saigo:04:h-transforms}, and noticing that $n=1$, we have that
  \begin{align}\label{E:h-1k}
    h_{1,k} = (-1)^{k+1} \frac{\Gamma\left((d+\alpha k)/2\right) \Gamma\left(1+k \alpha /2\right)}{\pi\: \Gamma((1+k)\beta+\gamma)} \sin\left(\pi k \alpha /2\right),
  \end{align}
  where $h_{i,k}$ are computed via Eq.~(1.3.9)
  of~\cite{kilbas.saigo:04:h-transforms}. In particular,
  \begin{align}\label{E:h-10}
    h_{1,0}\equiv 0 \quad \text{and} \quad
    h_{1,1} = \frac{\Gamma\left((d+\alpha)/2\right) \Gamma\left(1+\alpha/2\right)}{\pi\: \Gamma(2\beta+\gamma)} \sin\left(\pi\alpha /2\right).
  \end{align}
  Consequently, we have
  \begin{align}\label{E:f=h-11}
    f(z) = h_{1,1}\: z^{-1} + o\left(z^{-1}\right) \quad \text{as $z\to \infty$, $z\in \mathbb{C}$}.
  \end{align}

  \medskip\noindent\textbf{Case~2:} $\alpha > \beta$, $\beta<2$ and $\alpha
    \not\in 2 \mathbb{N}$. This case is a special case $\Delta > 0$ and $a^* >0$
  with some special values---even integers---removed. This case is treated the
  same as Case 1 (see Theorems 1.4 and 1.7
  of~\cite{kilbas.saigo:04:h-transforms}) and the only difference is that
  $z\to\infty$ in the cone $|\arg z|< (2-\beta) \pi/2$. Note that when $\alpha$
  is an even integer, then all $h_{1,k}$, $k\ge 0$, vanish; see~\eqref{E:h-1k}.

  \medskip\noindent\textbf{Case~3:} $\alpha>\beta = 2$ and $\alpha\not\in
    2\mathbb{N}$. This falls into the case $a^*=0$ and $\Delta>0$ that is studied
  in Section 1.6 of~\cite{kilbas.saigo:04:h-transforms}. According to
  Eq.~(1.6.6) of~\cite{kilbas.saigo:04:h-transforms}, let $\epsilon \in
    (0,\pi/2)$. Since $a_{i_0} = \alpha_{i_0} = 1$ (see Corollary~1.9.1
  of~\cite{kilbas.saigo:04:h-transforms}) and $\mu = (d-3)/2 -\gamma$
  (see~\eqref{E:H-Parameters}), according to Theorem~1.9
  of~\cite{kilbas.saigo:04:h-transforms}, and noticing that $h_{1,0} =0 $ and
  $h_{1,1}\ne 0$ (see~\eqref{E:h-10}), we need to consider three cases: $\mu +
    1/2<-\Delta$, $\mu + 1/2> -\Delta$, and $\mu + 1/2 = -\Delta$:

  \medskip\textbf{Case~3-1:} $\mu + 1/2 < -\Delta$, or equivalently, $\gamma <
    \alpha-3+d/2$. According to part (i) of Corollary~1.9.1
  of~\cite{kilbas.saigo:04:h-transforms}, the asymptotic given in~\eqref{E:f=h-11}
  holds here. Note that in order to assure that $h_{1,1}\ne 0$, we need to
  impose $\alpha \not\in 2 \mathbb{N}$. The case when $\alpha\in 2 \mathbb{N}$
  is left for future work.

  \medskip\textbf{Case~3-2:} $\mu + 1/2 > -\Delta$, or equivalently, $\gamma >
    \alpha-3+d/2$. According to Eq.~(1.6.5)
  of~\cite{kilbas.saigo:04:h-transforms}, the constant $A_0$ defined
  in~\eqref{E:A0} takes the following form
  \begin{align}\label{E_:A0}
    A_0 = \alpha^{d/2} 2^{-\gamma -(d+3)/2} (\alpha -2)^{\gamma -(d-3)/2}.
  \end{align}
  According to Eq.~(1.6.1) and Eq.~(1.6.2)
  of~\cite{kilbas.saigo:04:h-transforms}, let
  \begin{align*}
    c_0 \coloneqq 2 \pi e^{ i \pi \left(\gamma - (d-3)/2\right)} \quad \text{and} \quad
    d_0 \coloneqq 2 \pi e^{-i \pi \left(\gamma - (d-3)/2\right)}.
  \end{align*}
  According to Eq.~(1.6.9) of~\cite{kilbas.saigo:04:h-transforms}, let
  \begin{align}\label{E_:ABC}
    A \coloneqq \frac{A_0}{2\pi i \Delta} \left(\frac{\Delta^\Delta}{\delta}\right)^{(d/2 -1-\gamma)/\Delta}, \quad
    B \coloneqq \frac{1}{4} \pi \left(d - 2(1+\gamma)\right), \quad \text{and} \quad
    C \coloneqq \left(\frac{\Delta^\Delta}{\delta}\right)^{1/\Delta},
  \end{align}
  where $\Delta = \alpha -\beta>0$ and $\delta$ is given
  in~\eqref{E:H-Parameters}. Then according to part (ii) of Corollary~1.9.1
  of~\cite{kilbas.saigo:04:h-transforms}, we see that
  \begin{align*}
    f(z) =
     & A z^{(d/2 -1-\gamma)/\Delta} \left(
    c_0 \exp\left[ \left(B+C z^{1/\Delta}\right) i\right]
    - d_0 \exp\left[-\left(B+C z^{1/\Delta}\right) i\right]
    \right)                                                                                                  \\
     & + o\left(z^{(d/2 -1-\gamma)/\Delta}\right), \quad \text{as $z\to\infty$ with $|\arg z|\le \epsilon$.}
  \end{align*}
  Now, for the positive real argument $x>0$, noticing that $d_0 =
    \overline{c_0}$, we obtain that, as $x\to\infty$,
  \begin{align*}
    f(x) = & \frac{2 A_0}{\Delta} \left(\frac{\Delta^\Delta}{\delta}\right)^{(d/2 -1-\gamma)/\Delta}
    x^{(d/2 -1-\gamma)/\Delta} \sin\left( C_3\: x^{1/\Delta} + \frac{\pi}{4} \left(d-2\gamma\right)\right)
    + o\left(x^{(d/2 -1-\gamma)/\Delta}\right)                                                                                                \\
    =      & \frac{2^{\frac{4d-2-4\alpha\gamma-3\alpha}{2(\alpha-2)}} \alpha ^{\frac{\alpha  \gamma +\alpha -d}{\alpha-2}}}{\sqrt{\alpha -2}}
    x^{-\frac{2\gamma -d+2}{2(\alpha-2)}} \sin\left( C_3\: x^{1/(\alpha-2)} + \frac{\pi}{4} \left(d-2\gamma\right)\right)
    + o\left(x^{-\frac{2 \gamma -d+2}{2(\alpha-2)}}\right),
  \end{align*}
  where the constant $C_3$ is defined in~\eqref{E:h11}. Note that $\gamma >
    \alpha - 3 + d/2$ implies that the exponent of the leading power is smaller
  than $-1$: $-\frac{2\gamma -d+2}{2(\alpha-2)} < -1$.

  \medskip\textbf{Case~3-3:} $\mu + 1/2 = -\Delta$, or equivalently, $\gamma =
    \alpha-3+d/2$. By Eq.~(1.6.8) of~\cite{kilbas.saigo:04:h-transforms}, we have
  that
  \begin{align*}
    f(z) =
     & h_{1,1}z^{-1} + A z^{-1} \left(
    c_0 \exp\left[ \left(B+C z^{1/\Delta}\right) i\right]
    - d_0 \exp\left[-\left(B+C z^{1/\Delta}\right) i\right]
    \right)                                                                              \\
     & + o\left(z^{-1}\right), \quad \text{as $z\to\infty$ with $|\arg z|\le \epsilon$,}
  \end{align*}
  where the constants $A$, $B$, and $C$ are defined in~\eqref{E_:ABC}. With the
  substitution of $\gamma$ by $\alpha - 3 + d/2$, we obtain the following
  asymptotic expansion in the positive real line as Case 3-2:
  \begin{align*}
    f(x) =
     & \left( \frac{2^{-2\alpha-d+1/2} \alpha^{\alpha +d/2}}{\sqrt{\alpha-2}} \sin\left(C_3\: x^{1 /(\alpha-2)} + \frac{1}{2} \pi (3-\alpha)\right) + h_{1,1}\right)\frac{1}{x} + o\left(x^{-1}\right),
  \end{align*}
  as $x\to\infty$, where $h_{1,1}$ is given in~\eqref{E:h11}.

  \medskip\noindent\textbf{Case~4:} $\alpha = 2$ and $\beta < 2$. This
  corresponds to the case $n=0$ that is studied in Section 1.7
  of~\cite{kilbas.saigo:04:h-transforms}. In this case, by Property~2.2
  of~\cite{kilbas.saigo:04:h-transforms}, $f(z)$ reduces to
  \begin{align*}
    f(z)=\FoxH{2,0}{1,2}{z}{(\beta+\gamma,\beta)}{(d/2,1),\:(1,1)}.
  \end{align*}
  This case was previously studied in~\cite[Lemma 4.5]{chen.hu.ea:19:nonlinear},
  but the details were left to the readers. Here, we provide a complete and
  detailed explanation of the arguments. From~\eqref{E:H-Parameters}, we have
  that
  \begin{align*}
    a^* = a_1^*  = \Delta = 2-\beta > 0, \quad
    \mu = \frac{d+1}{2} -\beta-\gamma, \quad \text{and}\quad
    \delta = \beta^{-\beta}.
  \end{align*}
  Let $c_0 = 2 \pi e^{- i \mu \pi}$ (see Eq.~(1.7.1)
  of~\cite{kilbas.saigo:04:h-transforms}). Theorem~1.10
  of~\cite{kilbas.saigo:04:h-transforms} establishes the asymptotic behavior of
  $f(z)$ in the complex plane. Recall that the constant $A_0$ is defined
  in~\eqref{E_:A0}. According to Eq.~(1.7.2)
  of~\cite{kilbas.saigo:04:h-transforms}, with the above constants, we can
  compute
  \begin{align}\label{E_:C_1D_1}
    C_1 = \frac{c_0 A_0}{2\pi i} \Delta^{\mu - 1/2} \left(\frac{e^{i a_1^* \pi}}{\delta}\right)^{(\mu+1/2)/\Delta} \quad \text{and} \quad
    D_1 = \Delta \left(\frac{e^{i a_1^* \pi}}{\delta}\right)^{1/\Delta}.
  \end{align}
  under the current setup:
  \begin{align*}
    C_1 = \frac{\beta ^{\frac{-4 \gamma +\beta  (d-3)+2}{4-2 \beta }}}{\sqrt{2-\beta }} >0 \quad \text{and} \quad
    D_1 = -(2-\beta) \beta^{\beta / (2-\beta)} < 0.
  \end{align*}
  Note that $C_2 = C_1$ and $D_2 = D_1$ in our case; see Eq.~(1.7.3)
  of~\cite{kilbas.saigo:04:h-transforms}. Hence, according to part (i) of
  Theorem~1.10 in~\cite{kilbas.saigo:04:h-transforms}, as $x$ approaches
  infinity along the positive real line, we can use the asymptotic expression
  given either in Eq.~(1.7.7) or Eq.~(1.7.8)
  of~\cite{kilbas.saigo:04:h-transforms} to conclude that, as $x \to +\infty$,
  \begin{align*}
    f(x) & = C_1 e^{D_1 x^{1/\Delta}} x^{(\mu+1/2)/\Delta}\left(1+O\left(x^{-1/\Delta}\right)\right)                                                             \\
         & = \frac{\beta ^{\frac{-4 \gamma +\beta  (d-3)+2}{4-2 \beta }}}{\sqrt{2-\beta }} \exp\left(-(2-\beta) \beta^{\beta / (2-\beta)} x^{1/(2-\beta)}\right)
    x^{(d/2+1 - (\beta+\gamma))/(2-\beta)}\left(1+O\left(x^{-1/(2-\beta)}\right)\right).
  \end{align*}

  \medskip\noindent\textbf{Case~5:} $\alpha \in 2 \mathbb{N}+2$ and $\beta \in
    (0, 2)$. Set $\ell = \alpha/2$, which is greater or equal than $2$. In this
  case, one can apply Property~2.15 of~\cite{kilbas.saigo:04:h-transforms} to
  reduce the above function to the following form
  \begin{align*}
    f(z) \coloneqq \FoxH{2,0}{1,\ell+1}{z}{(\beta+\gamma,\beta)}{(d/2,\ell),\:(1, 1),\:(1/\ell,1),\:\cdots,\:((\ell-1)/\ell,1)}, \quad |\arg z| < \pi/2.
  \end{align*}
  For this Fox H-function, we have
  \begin{align*}
    \Delta = 2\ell - \beta > 0, \quad
    a^*    = 2-\beta > 0, \quad
    a_1^*  = \ell+1-\beta > 0, \quad\text{and}\quad
    n      = 0.
  \end{align*}
  Hence, we can apply Theorem~1.10 of~\cite{kilbas.saigo:04:h-transforms} to
  obtain the asymptotics of $f(x)$. First we compute the constants $C_1$ and
  $D_1$ defined in~\eqref{E_:C_1D_1} as follows:
  \begin{gather*}
    D_1 = \ell^{2\ell/(\beta -2 \ell)} (2 \ell -\beta) \beta ^{\beta/(2 \ell -\beta)}
    \exp\left( \frac{i \pi  (-\beta +\ell +1)}{2 \ell -\beta } \right)\quad \text{and}\\
    C_1 =  (2 \pi )^{\frac{\ell -1}{2}} \left(2 \ell -\beta\right)^{-\frac{1}{2}}
    \ell^{\frac{\beta +\beta  (-d)+\ell  (4 \beta +4 \gamma -6)}{4 \ell -2 \beta }}
    \beta^{\frac{\beta +\beta  d-4 \beta  \ell -4 \gamma  \ell +2 \ell }{4 \ell -2 \beta }}
    \exp \left(-\frac{i \pi  (\ell -1) (d-2 (\beta +\gamma -1))}{4 \ell -2 \beta }\right).
  \end{gather*}
  Therefore, by taking the real part of Eq.~(1.7.7)
  of~\cite{kilbas.saigo:04:h-transforms}, we obtain that
  \begin{align*}
    f(x)
    = & \Re \left[ C_1 e^{D_1 x^{1/\Delta}} x^{(\mu+1/2)/\Delta}\left(1+O\left(x^{-1/\Delta}\right)\right) \right]                                                                                                                    \\
    = & (2\pi)^{(\ell -1)/2} \left(2 \ell -\beta \right)^{-1/2}
    \ell^{\frac{\beta(1-d)+2\ell(2\beta+2\gamma-3)}{4\ell-2\beta}}
    \beta^{\frac{\beta(1+d)-2 \ell(2\beta +2\gamma-1)}{4\ell-2\beta}}
    x^{\frac{d-2(\beta+\gamma-1)}{4\ell-2\beta}}                                                                                                                                                                                      \\
      & \times \cos \left( \frac{\pi(\ell-1)(d-2(\beta+\gamma-1))}{4\ell-2\beta} -\ell^{2\ell/(\beta-2\ell)} (2\ell-\beta) \beta^{\beta/(2\ell-\beta)} \cos \left(\frac{\pi(\beta-2)}{4\ell-2\beta}\right) x^{1/(2\ell-\beta)}\right) \\
      & \times \exp \left( -\ell^{2\ell/(\beta-2\ell)} (2\ell-\beta) \beta^{\beta/(2\ell-\beta)} \sin \left(\frac{\pi(2-\beta)}{4\ell-2\beta}\right) x^{1/(2\ell-\beta)}\right)\left(1+O\left(x^{-1/(2\ell-\beta)}\right)\right).
  \end{align*}

  This completes the proof of Lemma~\ref{L:HAtInfty}.
\end{proof}

\begin{remark}\label{R:Asymptotics-oscillation-b2g0}
  For $\beta=2$ and $\gamma=0$, the oscillatory behavior of the fundamental
  solution can be related to signed-measure representations; see
  \cite{hochberg:78:signed} and also~\cite{eidelman.ivasisen:70:investigation}.
\end{remark}

% \section{Well-posedness---Dalang's Condition}\label{S:Dalang}
\section{Solvability---Dalang-type Conditions}\label{S:Dalang}

In this section, we establish the Dalang-type conditions\index{Dalang condition} for the solvability of~\eqref{E:fde}.
The main result is Theorem~\ref{T:Dalang} below, which provides the technical details underlying Theorem~\ref{T:Main} stated in the introduction.

\subsection*{Main results}
\addcontentsline{toc}{subsection}{Main results}

For ease of reference, we first state the main solvability criterion (Dalang's
condition) in Theorem~\ref{T:Dalang}. Its proof is split into
Subsections~\ref{SS:i},~\ref{SS:ii}, and~\ref{SS:iii}. In addition,
Subsection~\ref{SS:K} records explicit formulas for the variance constant $K$ in
some special cases.

A solution $u(t,x)$ to~\eqref{E:fde}, for fixed $(t,x)$, is a centered
Gaussian random variable, given its variance or second moment
is finite. If
$H=1/2$, from~\eqref{E:Inner_3}, the second moment  is equal to
\begin{align}\label{E:SecMom-W}
  \E\left[u(t,x)^2\right]
   & = 2\pi \int_{\R^d} \ud \xi\; |\xi|^{\ell - d} \int_0^t \ud s \: s^{2(\beta+\gamma-1)} E_{\beta, \beta+\gamma}^2 \left(-2^{-1}\nu |\xi|^\alpha s^\beta\right).
\end{align}
Similarly, if $H>1/2$, we can write
% the second moment of the solution as
\begin{align}\label{E:SecMom}
  \begin{aligned}
    \E\left[u(t,x)^2\right]
     & = 2\pi a_{H} \int_{\R^d} \ud \xi\; |\xi|^{\ell - d} \iint_{[0,t]^2} \ud s_1 \ud s_2 \: |s_1-s_2|^{2H-2} \\
     & \quad \times s_1^{\beta+\gamma-1} E_{\beta, \beta+\gamma}\left(-2^{-1}\nu |\xi|^\alpha s_1^\beta\right)
    \times s_2^{\beta+\gamma-1} E_{\beta, \beta+\gamma}\left(-2^{-1}\nu |\xi|^\alpha s_2^\beta\right),
  \end{aligned}
\end{align}
where $a_H$ is defined in~\eqref{E:Inner_3}. By a standard scaling argument, we
see that
\begin{subequations}\label{E:SecMom-K}
  \begin{gather}
    \E\left[u(t,x)^2\right] = K(\alpha,\beta,\gamma;H,\ell;\nu,d)\: t^{2\rho_0},  \label{E:utx-rho}
    \shortintertext{where}
    \rho_0 = \rho_0 \left(\alpha,\beta,\gamma;H,\ell\right) \coloneqq \beta + \gamma - 1 + H - \frac{\ell\beta}{2\alpha}, \label{E:rho0} \shortintertext{and}
    \begin{split}
      K(\alpha,\beta,\gamma;H,\ell;\nu,d)
       & = 2\pi H(2H-1) \int_{\R^d} \ud \xi\; |\xi|^{\ell - d} \iint_{[0,1]^2} \ud s_1 \ud s_2 \: |s_1-s_2|^{2H-2} \\
       & \quad \times s_1^{\beta+\gamma-1} E_{\beta, \beta+\gamma}\left(-2^{-1}\nu |\xi|^\alpha s_1^\beta\right)
      \times s_2^{\beta+\gamma-1} E_{\beta, \beta+\gamma}\left(-2^{-1}\nu |\xi|^\alpha s_2^\beta\right),
    \end{split} \shortintertext{if $H\in (1/2,1)$; and}
    K(\alpha,\beta,\gamma;1/2,\ell;\nu,d)
    = 2\pi \int_{\R^d} \ud \xi\; |\xi|^{\ell - d} \int_0^1 \ud s\:
    s^{2(\beta+\gamma-1)} E_{\beta, \beta+\gamma}^2\left(-2^{-1}\nu |\xi|^\alpha s^\beta\right),
  \end{gather}
  if $H=1/2$.
\end{subequations}
Therefore, the solvability of equation~\ref{E:fde} reduces to establishing the finiteness of the constant $K=K(\alpha,\beta,\gamma;H,\ell;\nu,d)$.

\begin{theorem}[Dalang's condition]\label{T:Dalang}
  Assume that
  \begin{align*}
    \alpha >0, \quad
    \beta  \in (0,2],   \quad
    H      \in [1/2,1), \quad \text{and} \quad
    \ell   \in (0,2d).
  \end{align*}
  We have the following three cases (see Figure~\ref{F:Dalang})
  \begin{enumerate}[\rm (i)]

    \item If either $\beta \in (0,2)$ and $\gamma \geq 0$ or $\beta = 2$ and
          $\gamma > 1$, then we have
          \begin{subequations}\label{E:Dalang2}
            \begin{align}\label{E:Dalang2-1}
              K(\alpha,\beta,\gamma;H,\ell;\nu,d) <\infty
               & \quad\Longleftrightarrow\quad \rho_0 = \beta + \gamma - 1 + H -\frac{\ell \beta}{2\alpha} >0 \quad \text{and}\quad \ell < 2\alpha \\
               & \quad\Longleftrightarrow\quad 0 <\ell < 2\alpha + \frac{2\alpha}{\beta}\min\left(0, \gamma + H -1\right). \label{E:Dalang2-2}
            \end{align}
          \end{subequations}

    \item If $\beta = 2$ and $\gamma = 0$, then
          \begin{align}\label{E:Dalang3}
            K(\alpha,2,0;H,\ell;\nu,d) <\infty \quad\Longleftrightarrow\quad
            0 < \ell < \left(1/2+H\right) \alpha.
          \end{align}

    \item If $\beta = 2$ and $\gamma \in (0, 1]$, then
          \begin{align}\label{E:Dalang4}
            K(\alpha,\beta,\gamma;H,\ell;\nu,d) <\infty \quad\Longleftarrow\quad
            0 < \ell <  \min\left(2, \gamma+H+1/2\right) \alpha .
          \end{align}

  \end{enumerate}
\end{theorem}

Figure~\ref{F:Dalang} illustrates the three cases in Theorem~\ref{T:Dalang}. We
will prove these three cases in Sections~\ref{SS:i}, \ref{SS:ii}
and~\ref{SS:iii}, respectively. \bigskip

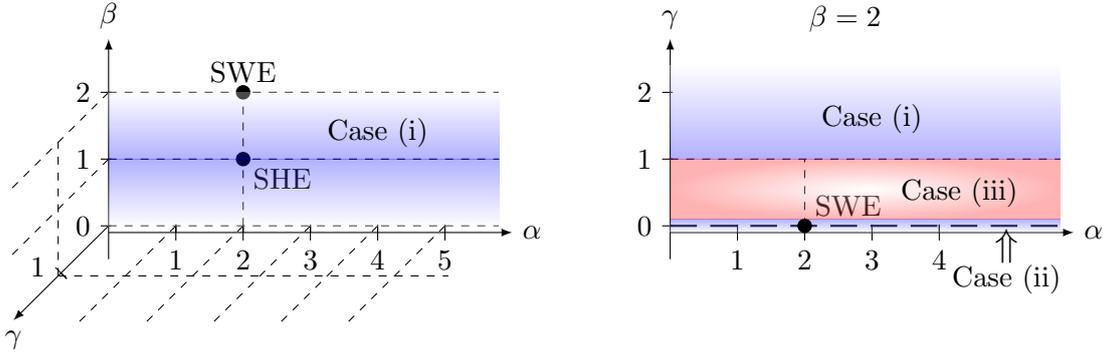
\begin{figure}[htpb]
  \begin{center}
    \begin{tikzpicture}[scale=1, transform shape, x = 2.3em, y = 2.3em]
      \tikzset{>=latex}
      \def\xshift{9.6em}
      \begin{scope}[xshift = -\xshift]
        \draw[->] (0,-0.1) -- (6,-0.1) node [right] {$\alpha$};
        \draw[->] (0,-0.5) -- (0,+2.8) node [above] {$\beta$};
        \foreach \x in {1,2,3,4,5}
        \draw (\x,-0.1) --++ (0,3pt) --++ (0,-6pt) node [below] {\x};
        \foreach \y in {0,1,2}
        \draw (1pt,\y) -- (-3pt,\y) node [left] {\y};

        \begin{scope}[dashed]
          \draw (0,2) -- (5.8,2);
          \draw (0,1) -- (5.8,1);
          \draw (0,0) -- (5.8,0);
          % \draw (1,0) -- (1,2);
          \draw (2,0) -- (2,2);
        \end{scope}

        \filldraw (2,1) circle (0.1) node [below right] {SHE};
        \filldraw (2,2) circle (0.1) node [above] {SWE};

        \begin{scope}[rotate = -45]
          \draw[->] (0,-0) --++ (0,-2.0) node [below, rotate =45] {$\gamma$};
          \draw (0,-1) --++ (+3pt,0) --++ (-6pt,0) node [left, rotate = 45] {$1$};
        \end{scope}

        \draw[dashed] (1,-0) --++ (-1.5,-1.5);
        \draw[dashed] (2,-0) --++ (-1.5,-1.5);
        \draw[dashed] (3,-0) --++ (-1.5,-1.5);
        \draw[dashed] (4,-0) --++ (-1.5,-1.5);
        \draw[dashed] (0,+1) --++ (-1.5,-1.5);
        \draw[dashed] (0,+2) --++ (-1.5,-1.5);
        \draw[dashed] (5,-0) --++ (-1.5,-1.5);
        \draw[dashed] (-0.75,-0.75) --++ (5.8,0);
        \draw[dashed] (-0.75,-0.75) --++ (0,2);

        \shade[bottom color=blue, top color=white, shading angle=180, opacity = 0.3] (0,0) rectangle (5.8,1);
        \shade[bottom color=blue, top color=white, shading angle=0, opacity = 0.3] (5.8,2) rectangle (0,1);
        \node at (4,1.4) {Case (i)};
      \end{scope}

      \begin{scope}[xshift = \xshift]
        \draw[->] (0,-0.1) -- (6,-0.1) node [right] {$\alpha$};
        \draw[->] (0,-0.5) -- (0,+2.8) node [above] {$\gamma$};
        \foreach \x in {1,2,3,4}
        \draw (\x,-0.1) --++ (0,3pt) --++ (0,-6pt) node [below] {\x};
        \foreach \y in {0,1,2}
        \draw (1pt,\y) -- (-3pt,\y) node [left] {\y};

        \begin{scope}[dashed]
          \draw (0,1) -- (5.8,1);
          % \draw (1,0) -- (1,1);
          \draw (2,0) -- (2,1);
        \end{scope}

        \draw[thick, dash pattern={on 10pt off 5pt}] (0,0) -- (5.8,0);
        \shade[top color=blue, bottom color=white, shading angle=0, opacity = 0.3] (0,-0.1) rectangle (5.8,0.1);
        \node (ii) at (5,-0.8) {Case (ii)};
        \node[] (Uparrow) at (5,-0.3) {\Large $\Uparrow$};

        \filldraw (2,0) circle (0.1) node [above right] {SWE};

        \shade[bottom color=blue, top color=white, shading angle=0, opacity = 0.3] (0,1) rectangle (5.8,2.4);
        \shade[inner color=white, outer color=red, shading angle=90, opacity = 0.3] (0,0.1) rectangle (5.8,1);
        \node at (3,1.6) {Case (i)};
        \node at (4.3,0.5) {Case (iii)};
        \node at (2.6,3.1) {$\beta=2$};

      \end{scope}
    \end{tikzpicture}
  \end{center}

  %   \begin{tikzpicture}[scale=0.7]
  %     % Draw grid lines
  %     \draw[step=1cm,gray!50,very thin] (0,0) grid (4,4);

  %     % Draw the main object
  %     \draw[fill=blue!50] (1,1) -- (3,1) -- (3,3) -- (1,3) -- cycle;

  %     % Draw the shadow
  %     \begin{scope}[opacity=0.5]
  %         \fill[gray!50] (1,0.7) -- (3,0.7) -- (3,1) -- (1,1) -- cycle;
  %         \fill[gray!50] (3,0.7) -- (3,2.3) -- (3.3,2.3) -- (3.3,1) -- cycle;
  %         \fill[gray!50] (1,3) -- (3,3) -- (3,3.3) -- (1,3.3) -- cycle;
  %     \end{scope}
  % \end{tikzpicture}

  % \begin{tikzpicture}[tdplot_main_coords]
  %     % Draw the grid
  %     \foreach \x in {0,...,4}
  %         \foreach \y in {0,...,4}
  %             \draw[gray!50] (\x,0,\y) -- (\x,4,\y) (\x,\y,0) -- (\x,\y,4) (0,\x,\y) -- (4,\x,\y);

  %     % Draw the box
  %     \draw[fill=blue!30] (1,1,1) -- (3,1,1) -- (3,3,1) -- (1,3,1) -- cycle;
  %     \draw[fill=blue!30] (1,1,1) -- (1,1,3) -- (1,3,3) -- (1,3,1) -- cycle;
  %     \draw[fill=blue!30] (1,1,1) -- (1,1,3) -- (3,1,3) -- (3,1,1) -- cycle;
  %     \draw[fill=blue!30] (1,1,3) -- (3,1,3) -- (3,3,3) -- (1,3,3) -- cycle;
  %     \draw[fill=blue!30] (3,1,1) -- (3,3,1) -- (3,3,3) -- (3,1,3) -- cycle;
  %     \draw[fill=blue!30] (1,3,1) -- (1,3,3) -- (3,3,3) -- (3,3,1) -- cycle;
  % \end{tikzpicture}

  \caption{Three cases and some special cases in Theorem~\ref{T:Dalang}.}

  \label{F:Dalang}
\end{figure}

\begin{remark}\label{R:Dalang-jump-beta2}
  One can easily check that when setting $\gamma=0$, the sufficient condition
  in~\eqref{E:Dalang4} reduces to the iff condition~\eqref{E:Dalang3}. One can
  also see that condition~\eqref{E:Dalang2-2} with $\gamma=1$ and $\beta=2$
  reduces to the sufficient condition~\eqref{E:Dalang4}. On the other hand, there is
  a property jump from $\beta<2$ to $\beta=2$, which can be seen from the
  condition in~\eqref{E:Dalang2-2} and the sufficient
  condition~\eqref{E:Dalang4} expressed equivalently as $\ell < 2\alpha + \alpha
    \min\left(0, \gamma+H-3/2\right)$.
\end{remark}

\begin{remark}\label{R:Dalang-known-cases}
  Theorem~\ref{T:Dalang} unifies and extends a number of solvability criteria in
  the literature. In particular, Theorem~\ref{T:Main}(i) (equivalently,
  Theorem~\ref{T:Dalang}(i) and (ii)) yields sharp (iff) conditions, while
  Theorem~\ref{T:Dalang}(iii) provides a clean sufficient condition in the
  remaining boundary regime. Although fractional SPDEs have been studied
  extensively, existing works usually focus on special parameter regimes. Below
  we indicate how some familiar models and recent results fit into our
  parameterization.
  \begin{enumerate}[(i)]

    \item Even though we require $H<1$, but if we formally set $H=1$, one
          obtains a time-independent noise. In this case,
          condition~\eqref{E:Dalang2} reduces to
          \begin{align*}
            0 <\ell < 2\alpha + \frac{2\alpha}{\beta}\min\left(0, \gamma \right),
          \end{align*}
          which should be compared with condition (1.11)
          of~\cite{chen.eisenberg:23:interpolating}, in the current
          parameterization,
          \begin{align*}
            0 <\ell < 2\alpha + \frac{2\alpha}{\beta}\min\left(0, \gamma - \sfrac{1}{2}\right).
          \end{align*}
          Property jump at $H=1$ has also been observed in,
          e.g.,~\cite{liu.hu.ea:24:in}.
          % This difference is due to the fact that the multiplicative noise was
          % studied (\textit{ibid.}). More examples that in case of $H>1/2$, the
          % condition for multiplicative noise case is more stringent than that for
          % the additive noise can be found in, e.g., \cite{liu.hu.ea:22:necessary}.

    \item For the space-time white noise case, i.e., $H=1/2$ and $\ell = d$,
          condition~\eqref{E:Dalang2} recovers that in Lemma~5.3
          of~\cite{chen.hu.ea:19:nonlinear}. The special case when $\beta\in (0,1)$
          and $\gamma = 1-\beta$ was earlier obtained
          in~\cite{mijena.nane:15:space-time}.

    \item Recently, Guo \textit{et al.}~\cite{guo.song.ea:24:stochastic} studied
          the problem with a multiplicative noise that is rough in space with $\ell
            \in (d,2d)$ for $d=1$ and regular in time, i.e., $H \in [1/2,1)$.
          Sufficient conditions for the existence of second moment of the solution
          are given; see Theorem~3.2 (\textit{ibid.}). In particular, in our
          parameterization, first part of condition~(3.12) (\textit{ibid.}) becomes
          \begin{align*}
            0 <\ell + \underbrace{(\ell-1)} = 2\ell -1 < 2\alpha +
            \frac{2\alpha}{\beta}\min\left(0, \gamma + H -1, \underbrace{\gamma -\frac{1}{2} + \frac{\beta}{2\alpha}(\ell-1)} \right),
          \end{align*}
          where the extra terms, compared to condition~\eqref{E:Dalang2}, have been
          indicated above using braces.

    \item For the stochastic wave equation ($\beta = 2$, $\gamma = 0$ and
          $\alpha = 2$), condition~\eqref{E:Dalang2} coincides
          with~\cite{liu.hu.ea:23:stochastic} for equation with an additive
          fractional noise and with~\cite{chen.deya.ea:25:solving} for equation with
          a multiplicative noise of the Riesz-type correlation in space: $0 < \ell <
            1+ 2H$, though the nonnegative assumption for the correlation function
          required in~\cite{chen.deya.ea:25:solving} fails when $\ell > d$. Assuming
          $\beta = 2$, $\alpha > 0$ and $\gamma = 0$, the equation was studied
          in~\cite{song.song.ea:20:fractional} with a multiplicative noise rough in
          space. In this paper, a sufficient condition for the existence to the
          equation was given, that is $H \in [1/2,1)$ and $\ell \in (1,3/2)$ in our
          notation. We only consider the SPDE with the additive noise, which avoids
          the complicated computation for moment estimates in higher order chaos,
          and thus do not require the nonnegativity and nonnegative definite of the
          correlation function as in~\cite{chen.deya.ea:25:solving}, and relax the
          condition. We conjecture that for multiplicative noise, our condition is
          sufficient to ensure the existence of the solution.

  \end{enumerate}
\end{remark}

\subsection{Some explicit formulas for the variance}\label{SS:K}

The next lemma specifies the constant $K$ in equation~\eqref{E:SecMom-K} in case
$\beta \in \{1, 2\}$ and $\gamma = 0$.

\begin{lemma}\label{L:K}
  When $\beta\in \{1,2\}$ and $\gamma = 0$, the constant
  $K(\alpha,\beta,\gamma;H,\ell;\nu,d)$, together with $\rho_0 = \beta + \gamma -
    1 + H - \ell\beta/(2\alpha)$, in~\eqref{E:SecMom-K} can be explicitly computed
  as follows:
  \begin{enumerate}[\rm (i)]
    \item For $\beta = 1$, we have that $\rho_0 = H-\ell/(2\alpha)$. Moreover, if $H\in(1/2,1)$,
          \begin{subequations}
            \begin{align}\label{E:K-0}
              K(\alpha,1,0;H,\ell;\nu,d)
               & = \frac{2^{3+\ell/\alpha} H \pi^{(2+d)/2} \Gamma\left(\ell/\alpha\right) \lMr{2}{F}{1}\left(1,\ell/\alpha;2H;-1\right) }{\nu^{\ell/\alpha} \Gamma\left(d/2\right) (2\alpha H-\ell)};
            \end{align}
            and if $H=1/2$
            \begin{align}\label{E:K+1}
              K(\alpha,1,0;1/2,\ell;\nu,d)
               & = \frac{2^{2}\pi^{(2+d)/2} \Gamma(\ell/\alpha) }{\nu^{\ell/\alpha}\Gamma (d / 2) (\alpha-\ell)}.
            \end{align}
          \end{subequations}
    \item
          For $\beta = 2$, we have the following cases:
          % \begin{subequations}
          \begin{enumerate}[\rm a)]

            \item if $H = 1/2$ and $\ell\in(0,\alpha)\setminus\{\alpha/2\}$, then $\rho_0
                    = 1+ H -\ell /\alpha$ and
                  \begin{align}\label{E:K-1}
                    K(\alpha,2,0;1/2,\ell;\nu,d)
                    = \frac{ 2^{3-\ell/\alpha} \nu^{1-\ell/\alpha} \pi^{(d+2)/2} \Gamma\left(2(\ell/\alpha-1)\right) \cos\left(\pi\ell/\alpha\right) }{\Gamma(d/2) \left(3\alpha-2\ell\right) };
                  \end{align}

            \item if $H = 1/2$ and $\ell = \alpha/2$, then $\rho_0 = 1$ and
                  \begin{align}\label{E:K-2}
                    K(\alpha,2,0;1/2,\alpha/2;\nu,d)
                    = \frac{ 2^{1/2} \nu^{1/2} \pi^{(d+4)/2} }{\Gamma(d/2) \alpha };
                  \end{align}

            \item if $H\in (1/2,1)$ and
                  $\ell\in(0,(1/2+H)\alpha)\setminus\{\alpha/2,\alpha\}$, then $\rho_0 = 1+ H
                    - \ell/\alpha$ and
                  \begin{align}\label{E:K-3}
                    \begin{aligned}
                       & K(\alpha,2,0;H,\ell;\nu,d)
                      =  \frac{ 2^{\ell/\alpha} \nu^{1-\ell/\alpha} H \pi^{(d+2)/2}  }{ \Gamma(d/2) } \times \Gamma(2(\ell/\alpha-1))\cos\left(\pi\ell/\alpha\right)                                                                                                   \\
                       & \qquad \times \Bigg[ \frac{ \alpha[1-2H] }{ (\alpha[1/2+H]-\ell) (\alpha[1+H]-\ell) }                                                               + 2 \frac{ \lMr{2}{F}{1}\left(1,2(\ell/\alpha-1);2H;-1\right) }{\alpha[1+H]-\ell} \Bigg];
                    \end{aligned}
                  \end{align}

            \item if $H\in (1/2,1)$ and $\ell=\alpha/2$, then $\rho_0 = 1/2 + H$ and
                  \begin{align}\label{E:K-4}
                    K(\alpha,2,0;H,\alpha/2;\nu,d)
                    = \frac{ 2^{3/2} \nu^{1/2} \pi^{(d+4)/2}}{ \alpha  (1+2H) \Gamma\left(d/2\right) };
                  \end{align}

            \item if $H\in (1/2,1)$ and $\ell=\alpha$, then $\rho_0 = H$ and
                  \begin{align}\label{E:K-5}
                    K(\alpha,2,0;H,\alpha;\nu,d)
                    = \frac{ 2^2 \pi^{(d+2)/2} }{ \alpha \Gamma (d/2) } \left(\frac{\lMr{2}{F}{1}(1,1;1+2H;-1)}{2H}+\frac{1}{2H-1}\right)
                  \end{align}
                  which does not depend on $\nu$. See Fig.\ref{F:K-5} for some plots related
                  to this constant.

          \end{enumerate}
  \end{enumerate}
  % \end{subequations}
\end{lemma}

The proof of Lemma~\ref{L:K} is given below.

\begin{remark}\label{R:K-limits}
  In order to compare these constants in Lemma~\ref{L:K}, we use
  $K_{\ref{E:K-0}}$ to denote the constant $K$ in~\eqref{E:K-0}, and similarly
  for others. We have the following limits:
  \begin{align*}
    \lim_{H \downarrow 1/2} K_{\ref{E:K-0}} = K_{\ref{E:K+1}}, \quad
    \lim_{H \downarrow 1/2} K_{\ref{E:K-3}} = K_{\ref{E:K-1}}, \quad\text{and}\quad
    \lim_{H \downarrow 1/2} K_{\ref{E:K-4}} = K_{\ref{E:K-2}},
  \end{align*}
  where the first two limits are due to~\eqref{E:2F1-1} and the last limit is
  trivial. Moreover,
  \begin{align*}
    \lim_{\ell \to \alpha/2} K_{\ref{E:K-3}} = K_{\ref{E:K-4}} \quad\text{and}\quad
    \lim_{\ell \to \alpha  } K_{\ref{E:K-3}} = K_{\ref{E:K-5}},
  \end{align*}
  where the first limit can be obtained via~\eqref{E:2F1-2}. As for the second
  limit, it is highly nontrivial. While we will not provide a proof of this
  limit, we have verified this limit using numerical computations.
\end{remark}

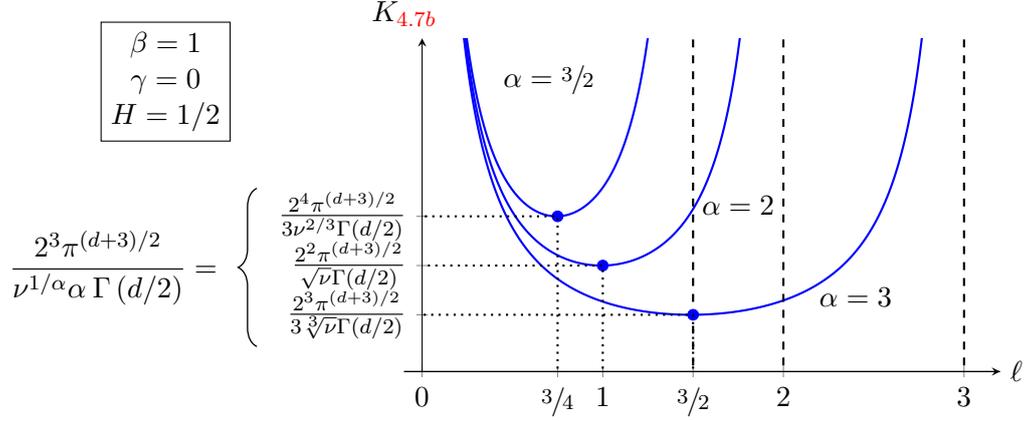
\begin{figure}[htpb]
  \centering
  \begin{tikzpicture}
    \begin{axis}[
        axis lines = center,
        domain=-0.2:3.2,
        ymin=0.5, ymax=4.5,
        xlabel={$\ell$},
        unit vector ratio*={6 1.67},
        ylabel={$K_{\ref{E:K+1}}$},
        ytick={0, 1.18164, 1.77245, 2.36327},
        yticklabels={1, $\frac{2^3 \pi^{(d+3)/2}}{3 \sqrt[3]{\nu} \Gamma\left(d/2\right)}$, $\frac{2^2\pi^{(d+3)/2}}{\sqrt{\nu} \Gamma\left(d/2\right)}$, $\frac{2^4 \pi^{(d+3)/2}}{3 \nu^{2/3} \Gamma \left(d/2\right)}$},
        every axis y label/.style={
            at={(axis description cs:0,1)},
            anchor=south,
            shift={(0,-15pt)}
          },
        xtick={-0.001,0.75,1.5,1,2,3},
        xticklabels={$0$, $\sfrac{3}{4}$, $\sfrac{3}{2}$, $1$, $2$, $3$},
        xlabel style={at={(ticklabel* cs:1.00)}, anchor=west},
        legend style={at={(1.2,0.8)}, anchor=north east,draw=none},
        clip mode = individual,
      ]
      \addplot[domain=-0.1:+3.2] {0.5};
      \addplot[domain=-0.0:1.5, blue, solid, thick] table {K_1-1.csv};
      \addplot[domain=-0.0:2.0, blue, solid, thick] table {K_1-2.csv};
      \addplot[domain=-0.0:3.0, blue, solid, thick] table {K_1-3.csv};
      \node at (axis cs:0.70,4.0) {$\alpha = \sfrac{3}{2}$};
      \node at (axis cs:1.75,2.5) {$\alpha = 2$};
      \node at (axis cs:2.40,1.4) {$\alpha = 3$};
      \addplot[dashed, thick] coordinates {(1.5,0) (1.5,7.0)};
      \addplot[dashed, thick] coordinates {(2.0,0) (2.0,7.0)};
      \addplot[dashed, thick] coordinates {(3.0,0) (3.0,7.0)};

      \def\myX{0.75}
      \def\myY{2.36327}
      \addplot[only marks, mark=*, color=blue, mark size = 2pt] coordinates {(\myX,\myY)};
      \addplot[dashed, thick, dotted] coordinates {(\myX,0) (\myX,\myY) (0,\myY)};

      \def\myX{1}
      \def\myY{1.77245}
      \addplot[only marks, mark=*, color=blue, mark size = 2pt] coordinates {(\myX,\myY)};
      \addplot[dashed, thick, dotted] coordinates {(\myX,0) (\myX,\myY) (0,\myY)};
      \draw[line width=0.5pt, decoration={brace, amplitude=7pt, raise=1pt}, decorate]
      (axis cs:-0.9,0.8) -- (axis cs:-0.9,2.7)
      node [midway, xshift=-5.0em] {$\dfrac{2^3\pi^{(d+3)/2}}{\nu^{1/\alpha}\alpha\:\Gamma\left(d/2\right)} = $};

      \def\myX{1.5}
      \def\myY{1.18164}
      \addplot[only marks, mark=*, color=blue, mark size = 2pt] coordinates {(\myX,\myY)};
      \addplot[dashed, thick, dotted] coordinates {(\myX,0) (\myX,\myY) (0,\myY)};

      \node[draw, align = center] at (axis description cs:-0.4,0.87) {$\beta=1$\\ $\gamma=0$\\ $H=1/2$};
    \end{axis}
  \end{tikzpicture}

  \caption{(SHE with white noise in time) Plots of the constant in
    $K_{\ref{E:K+1}}$ as a function of $\ell$ for $\alpha = 3/2$, $2$, and $3$,
    with special values when $\ell = \alpha/2$. Note that $\ell = \alpha/2$ are not
    the minimums of the respective curves.}

  % \label{F:K+1}
\end{figure}
% Do not remove the following commented codes.
% \begin{codes} %math
%
% (* Plot Theta *)
% \[Alpha] = 3/2;
% \[Nu] = 1;
% res1 = Table[{l, Gamma[l/\[Alpha]]/(\[Alpha] - l)}, {l, 0.18, \[Alpha] - 0.18, 0.009}];
% Export["K_1-1.csv", res1, "Table"]
% \[Alpha] = 2;
% res2 = Table[{l, Gamma[l/\[Alpha]]/(\[Alpha] - l)}, {l, 0.18, \[Alpha] - 0.18, 0.009}];
% Export["K_1-2.csv", res2, "Table"]
% \[Alpha] = 3;
% res3 = Table[{l, Gamma[l/\[Alpha]]/(\[Alpha] - l)}, {l, 0.18, \[Alpha] - 0.18, 0.009}];
% Export["K_1-3.csv", res3, "Table"]
% Clear[\[Alpha], \[Nu]]
%
% \end{codes}

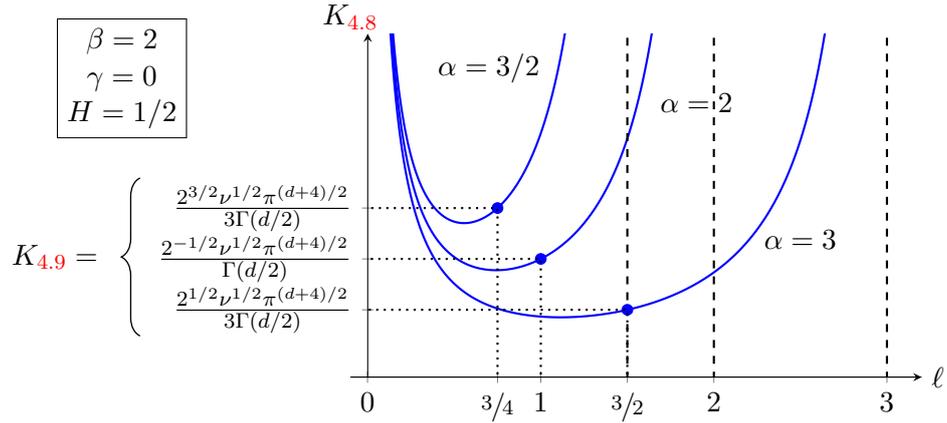
\begin{figure}[htpb]
  \centering
  \begin{tikzpicture}
    \begin{axis}[
        axis lines = center,
        domain=-0.2:3.2,
        ymin=0.5, ymax=5.5,
        xlabel={$\ell$},
        unit vector ratio*={10 2.0},
        ylabel={$K_{\ref{E:K-1}}$},
        ytick={0, 1.48096, 2.22144, 2.96192},
        yticklabels={1, $\frac{ 2^{1/2} \nu^{1/2} \pi^{(d+4)/2} }{3\Gamma(d/2)}$, $\frac{ 2^{-1/2} \nu^{1/2} \pi^{(d+4)/2} }{\Gamma(d/2)}$, $\frac{ 2^{3/2} \nu^{1/2} \pi^{(d+4)/2} }{3\Gamma(d/2)}$},
        every axis y label/.style={
            at={(axis description cs:0,1)},
            anchor=south,
            shift={(0,-15pt)}
          },
        xtick={-0.001,0.75,1.5,1,2,3},
        xticklabels={$0$, $\sfrac{3}{4}$, $\sfrac{3}{2}$, $1$, $2$, $3$},
        xlabel style={at={(ticklabel* cs:1.00)}, anchor=west},
        legend style={at={(1.2,0.8)}, anchor=north east,draw=none},
        clip mode = individual,
      ]
      \addplot[domain=-0.1:+3.2] {0.5};
      \addplot[domain=-0.0:1.5, blue, solid, thick] table {K1-1.csv};
      \addplot[domain=-0.0:2.0, blue, solid, thick] table {K1-2.csv};
      \addplot[domain=-0.0:3.0, blue, solid, thick] table {K1-3.csv};
      \node at (axis cs:0.70,5.0) {$\alpha = 3/2$};
      \node at (axis cs:1.90,4.5) {$\alpha = 2$};
      \node at (axis cs:2.50,2.5) {$\alpha = 3$};
      \addplot[dashed, thick] coordinates {(1.5,0) (1.5,7.0)};
      \addplot[dashed, thick] coordinates {(2.0,0) (2.0,7.0)};
      \addplot[dashed, thick] coordinates {(3.0,0) (3.0,7.0)};

      \def\myX{0.75}
      \def\myY{2.96192}
      \addplot[only marks, mark=*, color=blue, mark size = 2pt] coordinates {(\myX,\myY)};
      \addplot[dashed, thick, dotted] coordinates {(\myX,0) (\myX,\myY) (0,\myY)};

      \def\myX{1}
      \def\myY{2.22144}
      \addplot[only marks, mark=*, color=blue, mark size = 2pt] coordinates {(\myX,\myY)};
      \addplot[dashed, thick, dotted] coordinates {(\myX,0) (\myX,\myY) (0,\myY)};
      % \node at (axis cs:-2,2.22144) {$K_{\ref{E:K-2}} =   \resizebox{!}{1.5cm}{$\{$} $};
      % \draw[line width=0.5pt, decoration={brace, amplitude=10pt, mirror, raise=1pt}] (axis cs:-2,1.48096) -- (axis cs:-2,2.96192) node [midway, xshift=-5.0em] {$K_{\ref{E:K-2}} =$};
      \draw[line width=0.5pt, decoration={brace, amplitude=7pt, raise=1pt}, decorate]
      (axis cs:-1.3,1.1) -- (axis cs:-1.3,3.4)
      node [midway, xshift=-3.0em] {$K_{\ref{E:K-2}} =$};

      \def\myX{1.5}
      \def\myY{1.48096}
      \addplot[only marks, mark=*, color=blue, mark size = 2pt] coordinates {(\myX,\myY)};
      \addplot[dashed, thick, dotted] coordinates {(\myX,0) (\myX,\myY) (0,\myY)};

      \node[draw, align = center] at (axis description cs:-0.4,0.87)
      {$\beta=2$\\ $\gamma=0$\\ $H=1/2$};
    \end{axis}
  \end{tikzpicture}
  \caption{(SWE with white noise in time) Plots of the constant in
    $K_{\ref{E:K-1}}$ as a function of $\ell$ for $\alpha = 3/2$, $2$, and $3$.}

  \label{F:K-1}
\end{figure}
% Do not remove the following commented codes.
% \begin{codes} %math
%
% (* Plot Theta *)
% \[Alpha] = 3/2;
% \[Nu] = 1;
% res1 = Table[{l, (2^(3 - l/\[Alpha]) \[Nu]^(1 - l/\[Alpha]) Gamma[2 (l/\[Alpha] - 1)] Cos[\[Pi] l/\[Alpha]])/(3 \[Alpha] - 2 l)}, {l, 0.01, \[Alpha] - 0.01, 0.009}];
% Export["K1-1.csv", res1, "Table"]
% \[Alpha] = 2;
% res2 = Table[{l, (2^(3 - l/\[Alpha]) \[Nu]^(1 - l/\[Alpha])Gamma[2 (l/\[Alpha] - 1)] Cos[\[Pi] l/\[Alpha]])/(3 \[Alpha] - 2 l)}, {l, 0.01, \[Alpha] - 0.01, 0.009}];
% Export["K1-2.csv", res2, "Table"]
% \[Alpha] = 3;
% res3 = Table[{l, (2^(3 - l/\[Alpha]) \[Nu]^(1 - l/\[Alpha])Gamma[2 (l/\[Alpha] - 1)] Cos[\[Pi] l/\[Alpha]])/(3 \[Alpha] - 2 l)}, {l, 0.01, \[Alpha] - 0.01, 0.0091}];
% Export["K1-3.csv", res3, "Table"]
%
% \end{codes}

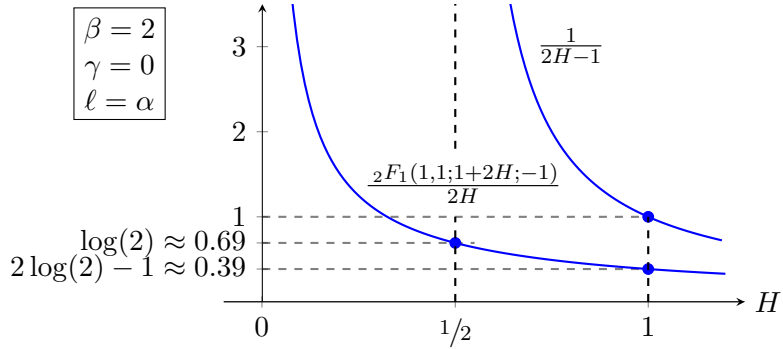
\begin{figure}[htpb]
  \centering
  \begin{tikzpicture}
    \begin{axis}[
        axis lines = center,
        domain=-0.2:1.2,
        ymin=0, ymax=3.5,
        xlabel={$H$},
        unit vector ratio*={18 2.0},
        ytick={0,0.386294,0.693147,1,2,3},
        yticklabels={1,$2\log(2)-1 \approx 0.39$,$\log(2) \approx 0.69$,1,2,3},
        % ytick align=outside,
        % yticklabel pos=right,
        every axis y label/.style={
            at={(ticklabel cs:0.75)},
            anchor=north,
            shift={(0,-0pt)}
          },
        xtick={-0.001,0.5,1},
        xticklabels={0, $\sfrac{1}{2}$, 1},
        xlabel style={at={(ticklabel* cs:1.00)}, anchor=west},
        legend style={at={(1.2,0.8)}, anchor=north east,draw=none},
        clip mode = individual,
      ]
      \addplot[domain=-0.0:+0.55, mark=none, thick, opacity=0.5, dashed] {0.693147};
      \addplot[domain=-0.0:+1.00, mark=none, thick, opacity=0.5, dashed] {0.386294};
      \addplot[domain=-0.0:+1.00, mark=none, thick, opacity=0.5, dashed] {1};
      \addplot[domain=-0.1:+1.25] {0};
      \addplot[domain=-0.0:1.2, blue, solid, thick] table {K5-1.csv}; % node [right, black, xshift = 0.2em] {$\frac{\Gamma(2H)\lMr{2}{F}{1}(1,1;1+2H;-1)}{\Gamma(1+2H)}$};
      \node at (axis cs:0.52,1.4) {$\frac{\lMr{2}{F}{1}(1,1;1+2H;-1)}{2H}$};
      \addplot[domain=+0.5:1.2, blue, solid, thick] table {K5-2.csv}; % node [right, black, xshift = 0.8em] {$\frac{1}{2H-1}$};
      \node at (axis cs:0.80,3.0) {$\frac{1}{2H-1}$};
      \addplot[only marks, mark=*, color=blue, mark size = 2pt] coordinates {(0.5,0.693147)};
      \addplot[only marks, mark=*, color=blue, mark size = 2pt] coordinates {(1,1)};
      \addplot[only marks, mark=*, color=blue, mark size = 2pt] coordinates {(1,0.386294)};
      \addplot[dashed, thick] coordinates {(0.5,0) (0.5,1.0)};
      \addplot[dashed, thick] coordinates {(0.5,1.8) (0.5,4)};
      \addplot[dashed, thick] coordinates {(1.0,0) (1.0,1)};

      \node[draw, align = center] at (axis description cs:-0.2,0.80)
      {$\beta=2$\\ $\gamma=0$\\ $\ell=\alpha$};
    \end{axis}
  \end{tikzpicture}
  \caption{(SWE with critical spatial noise) Plots of the constants in
    $K_{\ref{E:K-5}}$ as functions of $H$.} \label{F:K-5}
\end{figure}
% Do not remove the following commented codes.
% \begin{codes} %math
%
% (* Plot Theta *)
% res1 = Table[{H, Gamma[2 H]/Gamma[1 + 2 H] Hypergeometric2F1[1, 1, 1 + 2 H,
% -1]}, {H, 0.04, 1.2, 0.01}]
% res2 = Table[{H, 1/(2 H - 1)}, {H, 0.551, 1.2, 0.01}]
% Export["K5-1.csv", res1, "Table"]
% Export["K5-2.csv", res2, "Table"]
%
% \end{codes}

\begin{proof}[Proof of Lemma~\ref{L:K}] The proof for the case $\beta=2$ is
  contained in the proof of Theorem~\ref{T:Dalang}(ii) on Section~\ref{SS:ii}.
  It remains to prove the case $\beta=1$. \bigskip

  \noindent\textbf{Case~1 ($H=1/2$):} From~\eqref{E:Inner_3}
  or~\eqref{E:SecMom-W}, we see that
  \begin{align*}
    \E \left[u(t,x)^2\right]
     & = 2\pi \int_0^t \ud s \: \int_{\R^d} \ud \xi\; |\xi|^{\ell-d} \exp\left( -\nu s |\xi|^\alpha \right)                                                                                                                                                                                     \\
     & = 2\pi \times \frac{2\pi^{d/2}}{\Gamma(d/2)} \times \int_0^t \ud s \: \int_0^\infty \ud r\; r^{\ell-1} \exp\left( -\nu s r^\alpha \right)                                                                                                                                                \\
     & = 2\pi \times \frac{2\pi^{d/2}}{\Gamma(d/2)} \times \int_0^t \ud s \:\frac{ \Gamma\left(\ell/\alpha\right) (\nu s)^{-\ell/\alpha}}{\alpha}  =  \frac{4\pi^{(d + 2)/2}}{\Gamma(d/2)} \times \frac{ \Gamma\left(\ell/\alpha\right) }{ \nu^{\ell/\alpha} (\alpha-\ell) } t^{1-\ell/\alpha}.
  \end{align*}

  \noindent\textbf{Case~2 ($H>1/2$):} In this case, we deduce that
  % from~\eqref{E:bt=1}, we have
  \begin{align}\label{E:bt=1}
    \E \left[u(t,x)^2\right]
     & = 2\pi a_H \int_{\R^d} \ud \xi\; |\xi|^{\ell - d} \iint_{[0,t]^2} \ud s_1 \ud s_2 \: |s_1-s_2|^{2H-2}\exp\left(-2^{-1}\nu (s_1 + s_2) |\xi|^\alpha \right) \nonumber                                                                           \\
    % & = 2\pi H(2H-1) \times \frac{2\pi^{d/2}}{\Gamma(d/2)} \times \frac{2^{\ell/\alpha} \Gamma(\ell/\alpha)}{\nu^{\ell/\alpha} \alpha }\iint_{[0,t]^2} \ud s_1 \ud s_2 \: |s_1-s_2|^{2H-2} (s_1 + s_2)^{-\ell/\alpha} \\
     & = \frac{2^{2 + \ell/\alpha}H(2H-1)\pi^{(2+d)/2} \Gamma(\ell/\alpha) }{\nu^{\ell/\alpha} \Gamma(d/2) \alpha} \times 2 \int_0^t \ud s_1\int_0^{s_1} \ud s_2 \: |s_1-s_2|^{2H-2} (s_1 + s_2)^{-\ell/\alpha}                             \nonumber \\
     & = \frac{2^{3 + \ell/\alpha}H(2H-1)\pi^{(2+d)/2} \Gamma(\ell/\alpha) }{\nu^{\ell/\alpha} \Gamma(d/2) \alpha} \times \frac{ \lMr{2}{F}{1}\left(1,\ell/\alpha;2H;-1\right) }{2H-1} \int_0^t \ud s_1\; s_1^{2H-1-\ell/\alpha}     \nonumber        \\
    % & = \frac{2^{3 + \ell/\alpha}H(2H-1)\pi^{(2+d)/2} \Gamma(\ell/\alpha) }{\nu^{\ell/\alpha} \Gamma(d/2) \alpha} \times \frac{ \lMr{2}{F}{1}\left(1,\ell/\alpha;2H;-1\right) }{2H-1} \times \frac{t^{2 H-\ell/\alpha}}{2H-\ell/\alpha} \\
     & = \frac{2^{3 + \ell/\alpha}H\pi^{(2+d)/2}       \Gamma(\ell/\alpha) \lMr{2}{F}{1}\left(1,\ell/\alpha;2H;-1\right) }{\nu^{\ell/\alpha} \Gamma(d/2)  } \times \frac{t^{2 H-\ell/\alpha}}{2H\alpha-\ell},
  \end{align}
  where the second equality is due to~\eqref{E:HGF}. This completes the proof of
  Lemma~\ref{L:K}.
\end{proof}

\subsection{Proof of part (i) of Theorem~\ref{T:Dalang}}\label{SS:i}

We will prove the sufficiency\index{Dalang condition!sufficiency} and necessity\index{Dalang condition!necessity} of condition~\eqref{E:Dalang2}
separately below:

\begin{proof}[Proof of the sufficiency]
  The proof
  % of the sufficiency part of Theorem~\ref{T:Dalang} 
  relies on
  Lemma~\ref{L:ML-Est}. When $H>1/2$, applying the \textit{Hardy--Littlewood--Sobolev inequality} (see, e.g., \cite{memin.mishura.ea:01:inequalities}) to~\eqref{E:SecMom}, we see that
  \begin{align}\label{E:hls}
    \E\left(u(t,x)^2\right)
     & \le C_{H,d}' \int_{\R^d} \ud \xi\; |\xi|^{\ell - d}  \left(\int_0^t \ud s\, s^{(\beta+\gamma-1)/H}\left|E_{\beta, \beta+\gamma} \left(-2^{-1}\nu |\xi|^\alpha s^\beta\right)\right|^{1/H}\right)^{2H}.
  \end{align}
  It is clear that inequality~\eqref{E:hls} includes equality~\eqref{E:SecMom-W}
  as a special case. Consequently, in what follows we treat the cases $H>1/2$ and $H=1/2$ simultaneously.
  To this end, we distinguish two regimes: (i) $\beta\in(0,2)$ and $\gamma\ge 0$; and (ii) $\beta=2$ and $\gamma>1$.
  Notice
  that
  \begin{align*}
    %\text{condition~\eqref{E:Dalang2}} \quad \Longleftrightarrow \quad
    \text{Condition~\eqref{E:Dalang2}} \quad \implies \quad
    2\alpha + \frac{2\alpha}{\beta} \left(\gamma + H-1\right) > 0 \quad \Longleftrightarrow \quad
    \beta + \gamma + H > 1 .
  \end{align*}
  Hence, part~(a) of Lemma~\ref{L:ML-Est} applies to both cases (i) and (ii). As a result,
  \begin{align}\label{E:SecMomUB}
    \E\left(u(t,x)^2\right)
     & \le C_{H,d}' \int_{\R^d} \ud \xi\; |\xi|^{\ell - d}
    \frac{t^{2(\beta+\gamma+H-1)}}{1+ \left[ t (2/\nu)^{1 /\beta} |\xi|^{\alpha/\beta}\right]^{2\min\left(\beta, \beta+\gamma + H -1\right)}}.
  \end{align}
  Therefore, the above integral is finite if and only if
  \begin{align*}
    0 <\ell < \frac{2\alpha}{\beta}\min\left(\beta, \beta+\gamma + H -1\right).
  \end{align*}
  This proves the sufficiency of condition~\eqref{E:Dalang2}.
\end{proof}

\begin{proof}[Proof of the necessity]
  The necessity of condition~\eqref{E:Dalang2} is established by examining the asymptotic behavior of the Mittag--Leffler function $E_{\beta,\beta+\gamma}(\cdot)$ near $0$ and at infinity, as summarized in Lemma~\ref{L:ML}. We proceed by considering the following three cases. \bigskip

  \noindent\textbf{Case~1 ($\beta\in (0,2)$ and $\gamma > 0$):~} In this case, the equation
  $E_{\beta, \beta + \gamma}(-x) = 0$ has no positive root when $\beta\in (0,1]$
  and finitely many positive roots when $\beta\in (1,2)$; see Section~7.1
  of~\cite{diethelm:10:analysis} and references therein. Consequently, when $H>1/2$, one may fix $R_2>0$ sufficiently large such that $E_{\beta,\beta+\gamma}(-x)$ does not change sign on $[R_2,\infty)$. Assume that $\E[u(t,x)^2]<\infty$. For any $\epsilon\in(0,t)$, define
  \begin{align}\label{E:Repsilon}
    R(\epsilon)\coloneqq \left(\frac{2R_{2}}{\nu \epsilon^{\beta}}\right)^{1/\alpha} = \left(2 \nu^{-1} R_2 \right)^{1/\alpha} \epsilon^{-\beta / \alpha}.
  \end{align}
  Then, for all $s_1,s_2\ge\epsilon$, it holds that  \begin{gather}\begin{aligned}\label{E:abs_int}
      \int_{|\xi|\ge R(\epsilon)} & \ud \xi\; |\xi|^{\ell - d} \iint_{[\epsilon,t]^2} \ud s_1 \ud s_2 \: |s_1-s_2|^{2H-2} s_1^{\beta+\gamma-1} s_2^{\beta+\gamma-1} \\
                                  & \times \left|E_{\beta, \beta+\gamma}\left(-2^{-1}\nu |\xi|^\alpha s_1^\beta\right)\right|
      \times \left|E_{\beta, \beta+\gamma}\left(-2^{-1}\nu |\xi|^\alpha s_2^\beta\right)\right|
      < \infty .
    \end{aligned}\end{gather}
  % because the integrand in~\eqref{E:SecMom} is nonnegative on this region (by
  % the choice of $R_2$). 
  Since
  $\beta \in (0,2)$ and $\gamma > 0$, by Lemma~\ref{L:ML}, especially~\eqref{E:asym_MT1-1}, there exist positive numbers $R_1$, $C_*$, and (after
  possibly enlarging it) $R_2$ being chosen above, such that
  \begin{enumerate}[(i)]
    \item \label{I:pro_1} $E_{\beta,\beta + \gamma} (- x) \geq C_*$ for all $x \in [0, R_1)$;
    \item \label{I:pro_2} $\left| E_{\beta, \beta + \gamma} (- x) \right| \geq C_*/ x$ for all
          $x \in [R_2, \infty)$.
  \end{enumerate}
  It follows that
  \begin{align*}
    \int_{|\xi| \geq R(\epsilon)}
     & \ud \xi\; |\xi|^{\ell - d}
    \left|E_{\beta, \beta+\gamma}\left(-2^{-1}\nu |\xi|^\alpha s_1^\beta\right)\right| \times
    \left|E_{\beta, \beta+\gamma}\left(-2^{-1}\nu |\xi|^\alpha s_2^\beta\right)\right| \\
     & = C_d \int_{R(\epsilon)^{\alpha}}^{\infty}  \ud y\; y^{\ell/\alpha - 1}
    \left|E_{\beta, \beta+\gamma}\left(-2^{-1}\nu y s_1^\beta\right)\right| \times
    \left|E_{\beta, \beta+\gamma}\left(-2^{-1}\nu y s_2^\beta\right)\right|            \\
     & \gtrsim \int_{R(\epsilon)^{\alpha}}^{\infty}  \ud y\; y^{\ell/\alpha - 1}
    \left(y s_1^\beta\right)^{-1}
    \left(y s_2^\beta\right)^{-1}.
  \end{align*}
  Since this lower bound is valid for all $s_1, s_2 \in [\epsilon,t]$, we conclude that \eqref{E:abs_int} can hold only if $\ell < 2\alpha$.

  On the other hand, by
  examining the integral with respect to $\xi$ near the origin, we find that for any $0 < s_2/2 < s_1 < s_2 < t$,
  \begin{align}\label{E:int_xi_0}
    \int_{|\xi| < (R_1 /s_2^{\beta})^{1/\alpha}}
     & \ud \xi\; |\xi|^{\ell - d}
    E_{\beta, \beta+\gamma}\left(-2^{-1}\nu |\xi|^\alpha s_2^\beta\right)
    \times E_{\beta, \beta+\gamma}\left(-2^{-1}\nu |\xi|^\alpha s_1^\beta\right) \nonumber                             \\
     & \gtrsim \int_0^{R_1^{1/\alpha}  s_2^{-\beta/\alpha}}  \ud x \; x^{-(1 - \ell)} \asymp s_2^{- \ell\beta/\alpha}.
  \end{align}
  As a consequence, in order to ensure $\E[u(t,x)^2]<\infty$, it is necessary that
  \begin{align*}
    \int_0^t \ud s_2 \int_{s_2/2}^{s_2} \ud s_1  |s_1-s_2|^{2H-2} s_1^{\beta+\gamma-1} s_2^{\beta+\gamma-1} s_2^{- \ell\beta/\alpha} < \infty,
  \end{align*}
  which in turn implies $\ell < 2\alpha + \frac{2\alpha}{\beta} ( H + \gamma - 1)$.
  This proves the necessity for the case $H > 1/2$. \bigskip

  When $H = 1/2$, we may argue directly using formula~\eqref{E:SecMom-W}. Fix the same
  constants $R_1,R_2,C_*$ from Lemma~\ref{L:ML} as above. We first estimate the
  contribution from large values of $|\xi|$. Fix $\epsilon\in(0,t)$ and restrict to $s\in[\epsilon,t]$ and $|\xi|\ge R(\epsilon)$, where $R(\epsilon)$ is defined in~\eqref{E:Repsilon}. By property~\eqref{I:pro_2} above, for all such $(s,\xi)$,
  \begin{align*}
    \left|E_{\beta,\beta+\gamma}\left(-2^{-1}\nu|\xi|^\alpha s^\beta\right)\right|
    \gtrsim (|\xi|^\alpha s^\beta)^{-1}.
  \end{align*}
  Therefore, $\mathbb{E} [u(t, x)^2] < \infty$ forces $\ell<2\alpha$.

  We next estimate the contribution from small values of $|\xi|$. By property~\eqref{I:pro_1},
  $E_{\beta,\beta+\gamma}(-x)\ge C_*$ for all $x\in[0,R_1)$. Hence, for any
  $s\in(0,t)$, restricting to $|\xi|<(R_1/s^\beta)^{1/\alpha}$ yields
  \begin{align*}
    \E\left[u(t,x)^2\right]
     & \gtrsim \int_0^t \ud s\; s^{2(\beta+\gamma-1)}
    \int_{|\xi|<(R_1/s^\beta)^{1/\alpha}} \ud\xi\;|\xi|^{\ell-d}
    \asymp \int_0^t \ud s\; s^{2(\beta+\gamma-1)-\ell\beta/\alpha}.
  \end{align*}
  The last integral is finite only if $2(\beta+\gamma-1)-\ell\beta/\alpha>-1$,
  i.e.,
  $\ell<\frac{2\alpha}{\beta}(\beta+\gamma-\frac12)$, which is equivalent to
  $\rho_0=\beta+\gamma-\frac12-\frac{\ell\beta}{2\alpha}>0$ when $H=1/2$. This
  completes the necessity proof for the case $\beta \in (0,2)$ and $\gamma>0$.
  \bigskip

  \noindent\textbf{Case 2 ($\beta\in (0,2)$ and $\gamma = 0$):~} For this case,
  one needs to take care of asymptotics of $E_{\beta, \beta + \gamma} (- x)$ at
  infinity. In particular, when $\beta = 1$, $E_{1,1}(-x) = \exp(-x)$. Thus, if $H > 1/2$, it is not hard to check that equation \eqref{E:bt=1}
  % \begin{align}\label{E:bt=1}
  %   \E \left[u(t,x)^2\right]
  %   & = 2\pi a_H \int_{\R^d} \ud \xi\; |\xi|^{\ell - d} \iint_{[0,t]^2} \ud s_1 \ud s_2 \: |s_1-s_2|^{2H-2}\exp\left(-2^{-1}\nu (s_1 + s_2) |\xi|^\alpha \right)  \\
  %   & = 2\pi H(2H-1) \times \frac{2\pi^{d/2}}{\Gamma(d/2)} \times \frac{2^{\ell/\alpha} \Gamma(\ell/\alpha)}{\nu^{\ell/\alpha} \alpha }\iint_{[0,t]^2} \ud s_1 \ud s_2 \: |s_1-s_2|^{2H-2} (s_1 + s_2)^{-\ell/\alpha}, \nonumber
  % \end{align}
  % which 
  is finite if and only if $0 < \ell < 2 \alpha H$.
  This condition coincides with~\eqref{E:Dalang2}.

  On the other hand, if $\beta \neq 1$, $E_{\beta,\beta}(-x) \asymp x^{-2}$ as $x
    \uparrow \infty$ (see~\eqref{E:asym_MT1-2}), which relaxes the condition for
  the integration of $\xi$ at infinity in~\eqref{E:abs_int}. However, the
  integration of $\xi$ around zero contributes the same order of $s_2$;
  see~\eqref{E:int_xi_0}. As a result, in this case $\ell < 2\alpha +
    \frac{2\alpha}{\beta} ( H - 1) < 2\alpha$ is necessary for $\E [u(t,x)^2] <
    \infty$. This proves the necessity for the case $\beta \in (0,2)$ and $\gamma
    = 0$. \bigskip

  When $H=1/2$, we can argue similarly using~\eqref{E:SecMom-W}. Since
  $E_{\beta,\beta}(0)=\Gamma(\beta)^{-1}>0$ (Lemma~\ref{L:ML}(i)) and
  $E_{\beta,\beta}(\cdot)$ is continuous, there exist $R_0>0$ and $c_0>0$ such
  that $\left|E_{\beta,\beta}(-x)\right|\ge c_0$ for all $x\in[0,R_0]$. Hence,
  restricting to the region $|\xi|<(2R_0/(\nu s^\beta))^{1/\alpha}$ yields
  \begin{align*}
    \E\left[u(t,x)^2\right]
     & \gtrsim \int_0^t \ud s\; s^{2\beta-2}
    \int_{|\xi|<(2R_0/(\nu s^\beta))^{1/\alpha}} \ud\xi\;|\xi|^{\ell-d}
    \asymp \int_0^t \ud s\; s^{2\beta-2-\ell\beta/\alpha}.
  \end{align*}
  Therefore, $\E[u(t,x)^2]<\infty$ implies $2\beta-2-\ell\beta/\alpha>-1$, i.e.,
  $\ell<\frac{2\alpha}{\beta}(\beta-\frac12)$, which is equivalent to
  $\rho_0=\beta-\frac12-\frac{\ell\beta}{2\alpha}>0$ when $H=1/2$ and $\gamma=0$.
  \bigskip

  \noindent\textbf{Case~3 ($\beta = 2$ and $\gamma > 1$):~}
  In the asymptotics~\eqref{E:asym_MT1-4}, there are two terms competing for
  being the dominant one as $x\to\infty$. In particular, when $a=2$ and $b>3$
  in~\eqref{E:asym_MT1}, or equivalently $\beta=2$ and $\gamma > 1$, the
  non-oscillatory part is dominant, namely, $E_{2,2 + \gamma}(-x) \asymp x^{-1}$
  as $x\uparrow\infty$. Hence, the arguments in Case 1 can be applied here. This
  completes the proof of part (i) of Theorem~\ref{T:Dalang}.
\end{proof}

\subsection{Proof of part (ii) of Theorem~\ref{T:Dalang}}\label{SS:ii}

% \begin{proof}[Proof of Theorem~\ref{T:Dalang} (ii)]
By setting $\beta=2$ and $\gamma=0$ in~\eqref{E:FY}, or by direct computation,
the Fourier transform in space of the fundamental solution can be written as
follows
\begin{align}\label{E:FG_b2g0}
  \widehat{G}(t,\xi) = \frac{\sin\left( \sqrt{\nu/2}\: t\: |\xi|^{\alpha/2}\right)}{|\xi|^{\alpha/2}}.
\end{align}

\bigskip\noindent\textbf{Case~1 ($H = 1/2$):} In this case, we can write
\begin{align*}
  \E \left[ u(t,x)^2 \right]
  = & 2\pi \times \int_0^t \ud s \int_{\R^d} \ud \xi \; |\xi|^{\ell - d} \frac{\sin^2\left( \sqrt{\nu/2}\: s|\xi|^{\alpha/2}\right)}{|\xi|^{\alpha}}                                     \\
  = & 2\pi \times \frac{2 \pi^{d/2} }{\Gamma(d/2)} \int_0^t \ud s \int_0^{\infty} \ud x \; x^{\ell - \alpha - 1} \sin^2\left( \sqrt{\nu/2}\: s x^{\alpha/2}\right) \nonumber             \\
  = & 2\pi \times \frac{2 \pi^{d/2} }{\Gamma(d/2)} \times \frac{ \nu^{1-\ell/\alpha} 2^{\ell/\alpha} }{\alpha} \int_0^t \ud s \int_0^{\infty} \ud y \; y^{2\ell/\alpha - 3} \sin^2 (s y) \\
  = & \frac{2^{2+\ell/\alpha} \nu^{1-\ell/\alpha} \pi^{(d+2)/2}}{\Gamma(d/2) \alpha} \int_0^t \ud s \: F_{2\ell/\alpha-3}(s,s),
\end{align*}
where the function $F_{2\ell/\alpha-3}(s,s)$ comes from
Lemma~\ref{L:F_tht}(i), which is finite if and only if
\begin{align*}
  -3 < 2 \ell /\alpha -3 <-1 \quad \Longleftrightarrow \quad 0 < \ell < \alpha.
\end{align*}
If $\ell \ne \alpha/2$, by the first case in~\eqref{E:int_sin0},
\begin{align*}
  \E \left[ u(t,x)^2 \right]
  = & \frac{2^{2+\ell/\alpha} \nu^{1-\ell/\alpha} \pi^{(d+2)/2}}{\Gamma(d/2) \alpha}  \frac{\Gamma\left( 2(\ell/\alpha-1) \right) \cos\left(\pi\ell/\alpha\right) }{2^{2 \ell /\alpha -1}} \int_0^t \ud s \: s^{2(1-\ell/\alpha)}        \\
  = & \frac{2^{2+\ell/\alpha} \nu^{1-\ell/\alpha} \pi^{(d+2)/2}}{\Gamma(d/2) }  \frac{\Gamma\left( 2(\ell/\alpha-1) \right) \cos\left(\pi\ell/\alpha\right) }{2^{2 \ell /\alpha -1}}  \frac{1}{3\alpha-2\ell} \times t^{3-2\ell/\alpha}.
\end{align*}
Otherwise, if $\ell = \alpha /2$, by the second case in~\eqref{E:int_sin0},
\begin{align*}
  \E \left[ u(t,x)^2 \right]
  = \frac{2^{2+\ell/\alpha} \nu^{1-\ell/\alpha} \pi^{(d+2)/2}}{\Gamma(d/2) \alpha}  \frac{\pi}{2} \int_0^t s \ud s
  = \frac{2^{\ell/\alpha} \nu^{1-\ell/\alpha} \pi^{(d+4)/2}}{\Gamma(d/2) \alpha}  t^2.
\end{align*}
This proves part (ii) of Theorem~\ref{T:Dalang} under the assumption that $H =
  1/2$.

\bigskip\noindent\textbf{Case~2 ($(H > 1/2)$):} In this case, we can write
\begin{align*}
  \E \left[ u(t,x)^2 \right]
  = & 2 \pi a_H \times \frac{2 \pi^{d/2}}{\Gamma (d/2) } \times \frac{\nu^{1-\ell/\alpha} 2^{\ell/\alpha} }{\alpha} \times 2 \int_0^t \ud s_2 \int_0^{s_2} \ud s_1 |s_2 - s_1|^{2H - 2} \\
    & \times \int_0^{\infty} \ud y \; y^{2\ell/\alpha - 3} \sin\left( s_1 y\right) \sin\left( s_2 y\right)                                                                              \\
  = & H(2H-1) \frac{ 2^{3+\ell/\alpha} \nu^{1-\ell/\alpha} \pi^{(d+2)/2}}{\Gamma (d/2) \alpha} \int_0^t \ud s_2 \int_0^{s_2} \ud s_1 |s_2 - s_1|^{2H - 2} F_{2\ell/\alpha-3}(s_1,s_2).
\end{align*}
By Lemma~\ref{L:F_tht}(ii), $F_{2\ell/\alpha-3}(s_1,s_2)$ is finite
if and only if
\begin{align*}
  -3 < 2\ell/\alpha-3 < 0 \quad \Longleftrightarrow \quad 0 < \ell < 3\alpha/2.
\end{align*}
Now we have three sub-cases: \bigskip

\noindent\textbf{Sub-case 2-1 ($\ell \not\in \{\alpha /2, \alpha\}$):}
% If $\ell \not\in \{\alpha /2, \alpha\}$, then b
By the first case in~\eqref{E:int_sin}, we see that
\begin{align*}
  \E \left[ u(t,x)^2 \right]
  = & H(2H-1) \frac{ 2^{3+\ell/\alpha} \nu^{1-\ell/\alpha} \pi^{(d+2)/2}}{\Gamma (d/2) \alpha} \times \frac{1}{2} \Gamma(2(\ell/\alpha-1))\cos\left(\pi\ell/\alpha\right) \\
    & \times  \int_0^t \ud s_2\int_0^{s_2} \ud s_1 |s_2 - s_1|^{2H - 2}\left( (s_2 + s_1)^{2(1-\ell/\alpha)} -  (s_2 - s_1)^{2(1-\ell/\alpha)} \right).
\end{align*}
The above $\ud s_1$-integral is finite if and only if
\begin{align}\label{E_:ell<(1/2+H)*alpha}
  (2H-2) + 2 \left(1-\frac{\ell}{\alpha}\right) > -1 \quad \Longleftrightarrow \quad
  \ell < (1/2+H) \alpha.
\end{align}
Moreover, under \eqref{E_:ell<(1/2+H)*alpha}, this $\ud s_1$-integral can be
carried out using~\eqref{E:HGF} as follows
\begin{align*}
  \frac{\left(\alpha -2 \alpha  H+(\alpha[1+2H]-2 \ell) \, \lMr{2}{F}{1}\left(1,2(\ell/\alpha-1);2 H;-1\right)\right) }{(2 H-1) \times 2 (\alpha[1/2+H]-\ell)} s_2^{1+2H-2\ell/\alpha}.
\end{align*}
Finally, the $\ud s_2$-integral, which only requires that $\ell <
  (1+H)\alpha$, can be easily computed:
\begin{align*}
  \E \left[ u(t,x)^2 \right]
  = & H(2H-1) \frac{ 2^{3+\ell/\alpha} \nu^{1-\ell/\alpha} \pi^{(d+2)/2}}{\Gamma (d/2) } \times \frac{1}{2} \Gamma(2(\ell/\alpha-1))\cos\left(\pi\ell/\alpha\right)                                                            \\
    & \times \frac{\left(\alpha -2 \alpha  H+(\alpha[1+2H]-2 \ell) \, \lMr{2}{F}{1}\left(1,2(\ell/\alpha-1);2 H;-1\right)\right) }{(2 H-1) \times 2 (\alpha[1/2+H]-\ell) \times 2 (\alpha[1+H]-\ell) } t^{2(1+H-\ell/\alpha)}.
\end{align*}
\bigskip

\noindent\textbf{Sub-case 2-2 ($\ell=\alpha/2$):} % If $\ell=\alpha/2$, b
By the
second case in \eqref{E:int_sin}, we see that
\begin{align*}
  \E \left[ u(t,x)^2 \right]
  = & H(2H-1) \frac{ 2^{3+\ell/\alpha} \nu^{1-\ell/\alpha} \pi^{(d+2)/2}}{\Gamma (d/2) \alpha} \times \frac{\pi}{2} \int_0^t \ud s_2  \int_0^{s_2} \ud s_1\:\left(s_2-s_1\right)^{2H-2} s_1 \\
  = & H(2H-1) \frac{ 2^{3+\ell/\alpha} \nu^{1-\ell/\alpha} \pi^{(d+2)/2}}{\Gamma (d/2) \alpha} \times \frac{\pi}{2} \int_0^t \ud s_2\:\frac{\Gamma(2H-1)\Gamma(2)}{\Gamma(2H+1)} s_2^{2H}   \\
  = & \frac{2^{3/2} \nu^{1/2} \pi^{(d+4)/2}}{ \alpha (1+2H) \Gamma\left(d/2\right) } t^{1+2H},
\end{align*}
where the $\ud s_1$-integral is the \textit{Beta integral}.

\bigskip\noindent\textbf{Sub-case~2-3 ($\ell=\alpha$):} By the third case in
\eqref{E:int_sin}, we see that
\begin{align*}
  \E \left[ u(t,x)^2 \right]
  = & H(2H-1) \frac{ 2^{3+\ell/\alpha} \nu^{1-\ell/\alpha} \pi^{(d+2)/2}}{\Gamma (d/2) \alpha} \times \frac{1}{2} \int_0^t \ud s_2 \int_0^{s_2} \ud s_1\: \left(s_2-s_1\right)^{2H-2} \log\left(\frac{s_1+s_2}{s_2-s_1}\right) \\
  = & H(2H-1) \frac{ 2^{3+\ell/\alpha} \nu^{1-\ell/\alpha} \pi^{(d+2)/2}}{\Gamma (d/2) \alpha}                                                                                                                                 \\
    & \times \frac{1}{2} \left(\frac{\Gamma(2H-1) }{\Gamma(1+2H)}\lMr{2}{F}{1}(1,1;1+2H;-1)+\frac{1}{(2H-1)^2}\right) \int_0^t \ud s_2\: s_2^{2H-1}                                                                            \\
  = & H(2H-1) \frac{ 2^{3+\ell/\alpha} \nu^{1-\ell/\alpha} \pi^{(d+2)/2}}{\Gamma (d/2) \alpha}                                                                                                                                 \\
    & \times  \frac{1}{2} \left(\frac{\Gamma(2H-1) }{\Gamma(1+2H)}\lMr{2}{F}{1}(1,1;1+2H;-1)+\frac{1}{(2H-1)^2}\right) \times \frac{t^{2H}}{2H}                                                                                \\
  = & \frac{ 2^2 \pi^{(d+2)/2}}{\Gamma (d/2) \alpha} \times \left(\frac{1}{2H}\lMr{2}{F}{1}(1,1;1+2H;-1)+\frac{1}{2H-1}\right) \times t^{2H},
\end{align*}
where the $\ud s_1$-integral is computed via Lemma~\ref{L:2F1-Log} below. This
completes the proof of part (ii) of Theorem~\ref{T:Dalang}. \qed
% \end{proof}

\subsection{Proof of part (iii) of Theorem~\ref{T:Dalang}}\label{SS:iii}

% \textcolor{magenta}{It seems that the case $H=1/2$ has not been covered in the
% proof below. In this case, can one obtain iff condition?}

% \begin{proof}[Proof of Theorem~\ref{T:Dalang}(iii)]
First, consider $H=1/2$. In this case, by~\eqref{E:SecMom-W} with $\beta=2$,
\begin{align*}
  \E\left[u(t,x)^2\right]
  = 2\pi \int_{\R^d} \ud\xi\;|\xi|^{\ell-d}
  \int_0^t \ud s\; s^{2(1+\gamma)}
  E_{2,2+\gamma}^2\left(-2^{-1}\nu|\xi|^\alpha s^2\right).
\end{align*}
By Lemma~\ref{L:ML} (iii) (when $\gamma\in(0,1)$) and the explicit identity
$E_{2,3}(-x)=\frac{1-\cos(\sqrt{x})}{x}$ (when $\gamma=1$, which follows from
the series definition of the Mittag--Leffler function), for any fixed
$\gamma\in(0,1]$ there exists a constant $C_\gamma>0$ such that for all $x\ge0$,
\begin{align*}
  \left|E_{2,2+\gamma}(-x)\right| \le \frac{C_\gamma}{1+x^{(1+\gamma)/2}}.
\end{align*}
Consequently,
\begin{align*}
  \int_0^t \ud s\; s^{2(1+\gamma)}
  E_{2,2+\gamma}^2\left(-2^{-1}\nu|\xi|^\alpha s^2\right)
  \lesssim \int_0^t \ud s\; \min\left(s^{2(1+\gamma)},\,|\xi|^{-\alpha(1+\gamma)}\right)
  \lesssim 1 \wedge |\xi|^{-\alpha(1+\gamma)}.
\end{align*}
Since $\int_{|\xi|\le1}\ud\xi\,|\xi|^{\ell-d}<\infty$ for $\ell>0$, and
$\int_{|\xi|>1}\ud\xi\,|\xi|^{\ell-d-\alpha(1+\gamma)}<\infty$ is equivalent to
$\ell<(1+\gamma)\alpha$, we obtain $\E[u(t,x)^2]<\infty$ under
$\ell<(1+\gamma)\alpha$. This proves Theorem~\ref{T:Dalang} (iii) when $H=1/2$.
\bigskip

Now assume $H>1/2$. In case $\ell < (1 + \gamma) \alpha$, we can apply the
Hardy--Littlewood--Sobolev inequality and follow the same argument as in
Section~\ref{SS:i} to prove that $\E [u(t,x)^2] < \infty$ for all
$(t,x)\in \R_+\times \R^d$. In particular, one applies Lemma~\ref{L:ML-Est}
(b) when $\gamma\in(0,1)$ (i.e., $2+\gamma\in(2,3)$), and Lemma~\ref{L:ML-Est}
(a) when $\gamma=1$ (i.e., $2+\gamma=3$), to see that~\eqref{E:SecMomUB} takes
the following form:
\begin{align*}
  \E\left(u(t,x)^2\right)
   & \le C_{H,d}' \int_{\R^d} \ud \xi\; |\xi|^{\ell - d}
  \frac{t^{2(1+\gamma+H)}}{1+ \left[ t (2/\nu)^{1/2} |\xi|^{\alpha/2}\right]^{2(1+\gamma)}},
\end{align*}
which is finite if $\ell < (1 + \gamma) \alpha$. Hence, in the following, it
suffices to consider the case when
\begin{align}\label{E_:iii-cond}
  (1+\gamma)\alpha \le \ell < \alpha \times \min(2,\: 1/2+\gamma+H).
\end{align}

Performing a change of variable $x^{\alpha} = y$, we can write
\begin{align*}
  \E\left[u(t,x)^2\right]
  \asymp & \int_0^t \ud s_2 \int_0^{s_2} \ud s_1\; |s_1 - s_2|^{2H - 2} s_1^{1 + \gamma} s_2^{1 + \gamma} \int_0^{\infty} \ud x \;  x^{\ell - 1} E_{2, 2+\gamma}\left(- x^\alpha s_1^2\right)  E_{2, 2+\gamma}\left(- x^\alpha s_2^2\right) \\
  \asymp & \int_0^t \ud s_2 \int_0^{s_2} \ud s_1\; |s_1 - s_2|^{2H - 2} s_1^{1 + \gamma} s_2^{1 + \gamma} \int_0^{\infty} \ud y \;  y^{\ell/\alpha - 1} E_{2, 2+\gamma}\left(- s_1^2 y\right)  E_{2, 2+\gamma}\left(- s_2^2 y\right).
\end{align*}
Denote by $b = \gamma + 2$ and $\theta = \ell/\alpha - 1$. Since $\gamma \in
  (0,1]$, $\alpha > 0$, $H \in (1/2,1)$, and~\eqref{E_:iii-cond} holds, we find
that $b \in (2,3]$, $-1 < \theta < \min(1,\: \gamma + H - 1/2)$, and $b -
  \theta > \gamma + 2 - (\gamma + H - 1/2) = 5/2 - H > 3/2$. Then, all
conditions in
part (ii) of Lemma~\ref{L:intml} are satisfied, and thus
\begin{align*}
  \E\left[u(t,x)^2\right] \lesssim
   & \int_0^t \ud s_2 \int_0^{s_2} \ud s_1\; |s_1 - s_2|^{2H - 2} s_1^{1 + \gamma} s_2^{1 + \gamma}                                                                                                 \\
   & \times \begin{dcases}
              s_1^{\gamma  - 2\ell/\alpha + 1} s_2^{-\gamma - 1} + s_1^{-\gamma - 1} s_2^{-\gamma - 1} (s_2 - s_1)^{2\gamma - 2\ell/\alpha +2}, & \ell > \alpha (1 + \gamma),                         \\
              s_1^{- \gamma - 1} s_2^{-\gamma - 1} \left(1 + \left|\log \left((s_2 - s_1)/s_2\right) \right| \right),                           & \ell =  \alpha (1 + \gamma)  ,                      \\
              s_1^{\gamma - 2\ell/\alpha + 1} s_2^{-\gamma - 1},                                                                                & \alpha (1 + \gamma)/2 < \ell < \alpha (1 + \gamma), \\
              s_2^{- 2\ell/\alpha} \log \big(s_2/s_1 \big),                                                                                     & \ell = \alpha (1 + \gamma)/2 ,                      \\
              s_2^{- 2\ell/\alpha},                                                                                                             & \ell < \alpha (1 + \gamma)/2.
            \end{dcases}
\end{align*}
Then, it is not difficult to see that under the condition of part (iii) of
Theorem~\ref{T:Dalang}, the above expression is finite. This completes the
proof of part (iii) of Theorem~\ref{T:Dalang}. \qed

\begin{remark}\label{R:Dalang-oscillatory-proof}
  Due to the presence of the oscillatory term $\cos(\cdot)$
  in~\eqref{E:asym_MT1-4}, instead of controlling a single leading term as given
  in~\eqref{E:asym_MT2}, the proof requires controlling both leading terms as
  given in~\eqref{E:asym_MT1-4}. This makes the proof more complex than usual.
  We leave this for future work.
\end{remark}

\section{Moment upper bounds for increments}\label{S:Upper-t}\index{increment bounds}\index{moment bounds}

The aim of this section is to prove the following propositions on moment upper
bounds\index{Holder continuity@H\"older continuity} for temporal and spatial increments in
Theorem~\ref{T:main-holder}.

\subsection*{Main results}
\addcontentsline{toc}{subsection}{Main results}

We state the two main upper bound estimates---for time increments
(Proposition~\ref{T:Upper-t}) and for space increments
(Proposition~\ref{T:Upper-s})---before proving them in the subsequent
subsections.

Denote by
\begin{align}\label{E:rho1-rho2}
  \rho_1 \coloneqq \beta + \gamma + H -\frac{\ell \beta}{2\alpha} - 1, \quad \rho_2 \coloneqq \gamma + H - \ell/\alpha + 1/2,
\end{align}
\begin{align}\label{E:trho1-trho2}
  \widetilde{\rho}_1 \coloneqq \min
  \{\alpha\rho_1/\beta, \alpha - \ell/2\}, \quad \text{and} \quad \widetilde{\rho}_2 \coloneqq \min
  \{\alpha\rho_2/2, \alpha - \ell/2\},
\end{align}
\begin{proposition}\label{T:Upper-t}
  Under the setting of Theorem~\ref{T:main-holder}, for all positive constants
  $S$ and $T$ with $S<T$, there exists a positive constant $C = C_{S,T}$, such
  that the following statements hold with $\rho_1$ and $\rho_2$ defined as
  in~\eqref{E:rho1-rho2}.

  \begin{enumerate}[\rm (i)]

    \item If either a) $\beta \in (0,2)$; or b) $\beta = 2$ and $\gamma > 2$; or
          c) $\beta = 2$, $\gamma \in (0,2]$ and $\ell < \alpha (\gamma + 1/2)$\footnote{This condition is required to apply part (ii) of Lemma~\ref{L:intml} on $E_{2,2 + \gamma - 1}$.}, then for all $(t,s,x)\in [S,T]^2 \times \R^d$ it
          holds that
          \begin{align*} % \label{E:upt}
            \E\left[\left(u(t, x) - u(s, x)\right)^2\right]
            \le C
            \begin{dcases}
              |t-s|^{2 (\rho_1 \wedge 1)},                    & \ell \neq \alpha(\gamma + 1) - 2, \\
              |t-s|^2 \big(1 + \big| \log(|t-s|) \big| \big), & \ell = \alpha(\gamma + 1) - 2.
            \end{dcases}
          \end{align*}

    \item If $\beta = 2$ and $\gamma = 0$, then for all $(t,s,x)\in [S,T]^2
            \times \R^d$ it holds that
          \begin{align*} % \label{E:upt1}
            \E\left[\left(u(t, x) - u(s, x)\right)^2\right]
            \le C
            \begin{dcases}
              |t-s|^{2(\rho_2 \wedge 1)},                     & \rho_2 \neq 1, \\
              |t-s|^2 \big(1 + \big| \log(|t-s|) \big| \big), & \rho_2 = 1.
            \end{dcases}
          \end{align*}

  \end{enumerate}
\end{proposition}

\begin{proposition}\label{T:Upper-s}
  Under the setting of Theorem~\ref{T:main-holder}, for all positive constants
  $S$ and $T$ with $S<T$, there exists a positive constant $C = C_{S,T}$, such
  that the following statements hold with $\widetilde{\rho}_1$
  and $\widetilde{\rho}_2$ defined as in~\eqref{E:trho1-trho2}.

  \begin{enumerate}[\rm (i)]

    \item If either $\beta \in (0,2)$ or $\beta = 2$ and $\gamma \ge 1$, then
          for all $(t, x, y)\in [S,T]\times \R^{2d}$, it holds that
          \begin{align} \label{E:ups}
            \E\left[\left(u(t, x) - u(t, y)\right)^2\right]
            \le C
            \begin{dcases}
              |x-y|^{2 (\widetilde{\rho}_1 \wedge 1)},        & \widetilde{\rho}_1 \neq 1, \\
              |x-y|^2 \big(1 + \big| \log(|x-y|) \big| \big), & \widetilde{\rho}_1 = 1.
            \end{dcases}
          \end{align}

          %   \item If $\beta = 2$ and $\gamma \in (0, 1)$, then for all $(t, x, y)\in [S,T]\times
          % \R^{2d}$, it holds that

          % \begin{align} \label{E:ups'}
          %     \E\left[\left(u(t, x) - u(t, y)\right)^2\right]
          %     \le C
          %     \begin{dcases}
          %       |x-y|^{2 (\alpha (\gamma - \ell/\alpha + 1)/2 \wedge 1)},        & \widetilde{\rho}_1 \neq 1, \\
          %       |x-y|^2 \big(1 + \big| \log(|x-y|) \big| \big), & \widetilde{\rho}_1 = 1.
          %     \end{dcases}
          %   \end{align}

    \item If $\beta = 2$ and $\gamma \in [0, 1)$, then for all $(t, x, y)\in [S,T]\times
            \R^{2d}$, it holds that
          \begin{align}\label{E:ups1}
            \E\left[\left(u(t, x) - u(t, y)\right)^2\right]
            \le C
            \begin{dcases}
              |x-y|^{2 (\widetilde{\rho}_2 \wedge 1)},        & \widetilde{\rho}_2 \neq 1, \\
              |x-y|^2 \big(1 + \big| \log(|x-y|) \big| \big), & \widetilde{\rho}_2 = 1.
            \end{dcases}.
          \end{align}

  \end{enumerate}
\end{proposition}

% \begin{remark}
%   In part (ii) of Proposition~\ref{T:Upper-t}, we have that $\rho(\alpha,2,0,H,\ell)
%   = H-\ell/\alpha+1$. Hence, \eqref{E:upt1} can be equivalently written as
%   follows using $\rho$:
%   \begin{align}\label{E:upt1-2}
%     \E\left[\left(u(t, x) - u(s, x)\right)^2\right]
%     \le C_{S,T}
%     \begin{dcases}
%       |t-s|^{2\rho -1},                & 2\rho < 3, \\
%       |t-s|^2 \big| \log(|t-s|) \big|, & 2\rho = 3, \\
%       |t-s|^2,                         & 2\rho > 3.
%     \end{dcases}
%   \end{align}
%   Comparing~\eqref{E:upt} and~\eqref{E:upt1-2}, we see there is a property jump
%   from $\beta<2$ to $\beta = 2$.
% \end{remark}

\subsection{Proof of Proposition~\ref{T:Upper-t}}
Without loss of generality, we assume that $S \leq s < t \leq T$ with $t-s \ll S
  < s$.  Otherwise, as a result of~\eqref{E:utx-rho}, the proposition follows from
the following trivial arguments:
\begin{align*}
  \E \left[ \left(u(t,x) - u(s,x)\right)^2 \right]
  \lesssim \E \left[ u(t,x)^2 \right] + \E \left[ u(s,x)^2 \right] \lesssim 1.
\end{align*}

From the mild form~\eqref{E:Mild}, we can write
\begin{gather*}
  \E\left[\left(u(t, x) - u(s, x)\right)^2\right] \le 2 (I_1 +  I_2),
\end{gather*}
where
\begin{align}\label{E:upp-t-I1}
  I_1 \coloneqq & \E \left[ \left( \int_s^t \int_{\R^d} G(t-r,x-y) W(\ud r,\ud y) \right)^2\right]  \quad \text{and}                             \\
  I_2 \coloneqq & \E \left[ \left( \int_0^s \int_{\R^d} \left(G(t-r,x-y) - G(s-r,x-y)\right) W(\ud r,\ud y) \right)^2\right]. \label{E:upp-t-I2}
\end{align}
The proof of this proposition is split into two parts depending on the following
cases of parameters.

\subsubsection{Proof of Proposition~\ref{T:Upper-t}: part (i)}\label{S:T4.1-i}

The proof of part (i) of Proposition~\ref{T:Upper-t} proceeds by employing a
similar argument as in~\cite[Section 4]{guo.song.ea:24:stochastic}.
Additionally, we need the next two lemmas corresponding to cases where $H = 1/2$
and $H > 1/2$, respectively.

\begin{lemma}\label{L:lmm_tu1}
  Let $0<s<t$, and $a,b \in \R$ satisfy the property that $\min \{a , b\} \geq -2$ and $a
    + b > -3$. Furthermore, suppose that $t - s \ll s$. Then,
  \begin{align}\label{E:lmm_tu1}
    \int_{s}^{t} \ud \theta_2 \int_{s}^{\theta_2} \ud \theta_1  \int_0^s \ud r\;  (\theta_1 - r)^{a}  (\theta_2 - r)^{b}
    \lesssim
    \begin{dcases}
      (t-s)^{(a + b + 3)\wedge 2}, & a + b \neq -1, \\
      (t-s)^2 (1 + |\log (t-s)|),  & a + b = -1.
    \end{dcases}
  \end{align}
\end{lemma}

\begin{lemma}\label{L:lmm_tu}
  Let $0<s<t$, and $a,b,c \in \R$ satisfy the property that $-1 <a <0$, $b > c
    \geq -2$ and $a + b + c > -4$. Furthermore, suppose that $t - s \ll s$. Denote
  by
  \begin{align*}
    I  \coloneqq & \int_{s}^{t} \ud \theta_2 \int_{s}^{\theta_2} \ud \theta_1  \int_0^s \ud r_2 \int_{0 \vee (r_2 - \theta_2 + \theta_1)}^{r_2} \ud r_1 \: (r_2-r_1)^{a}  (\theta_1 - r_1)^{b}  (\theta_2 - r_2)^{c} \shortintertext{and}
    I' \coloneqq & \int_{s}^{t} \ud \theta_2 \int_{s}^{t} \ud \theta_1  \int_0^s \ud r_2 \int_{0}^{(r_2 - \theta_2 + \theta_1) \wedge r_2} \ud r_1 \: (r_2-r_1)^{a}  (\theta_1 - r_1)^{c} (\theta_2 - r_2)^{b}.
  \end{align*}
  Then, the following inequalities hold.
  \begin{align*}
    I \lesssim (t-s)^{(a + b + c + 4)\wedge 2} \quad  \text{and} \quad
    I' \lesssim \begin{dcases}
                  (t-s)^{(a + b + c + 4)\wedge 2}, & a + b + c \neq -2, \\
                  (t-s)^2 (1 + |\log (t-s)|),      & a + b + c = -2.
                \end{dcases}
  \end{align*}
\end{lemma}

\begin{proof}[Proof of part (i) of Proposition~\ref{T:Upper-t}]
  We only need to estimate $I_1$ and $I_2$ as defined in~\eqref{E:upp-t-I1}
  and~\eqref{E:upp-t-I2}, respectively.

  \bigskip\noindent\textbf{Estimate for $I_1$.~} Here we assume that $H > 1/2$.
  In this case,
  \begin{align*}
    I_1 =
            & 2\pi a_H \int_{\R^d} \ud \xi\; |\xi|^{\ell - d} \int_{0}^{t-s} \ud r_1 \int_0^{ t - s} \ud r_2 \: (r_2 - r_1)^{2H-2} \\
            & \times r_1^{\beta + \gamma -1} E_{\beta, \beta + \gamma}\left(-2^{-1}\nu |\xi|^\alpha r_1^\beta\right)
    \times r_2^{\beta + \gamma -1} E_{\beta, \beta + \gamma}\left(-2^{-1}\nu |\xi|^\alpha r_2^\beta\right)                         \\
    \approx & \int_{0}^{t-s} \ud r_1 \int_0^{ t - s} \ud r_2 \: (r_2 - r_1)^{2H-2} r_1^{\beta + \gamma -1} r_2^{\beta + \gamma -1} \\
            & \times  \int_{0}^{\infty} \ud z\; z^{\ell/\alpha - 1}  E_{\beta, \beta + \gamma}\left(- r_1^\beta z\right)
    \times  E_{\beta, \beta + \gamma}\left(-  r_2^\beta z\right).
  \end{align*}
  It follows from part (i) of Lemma~\ref{L:intml} that in case $\beta \in (0,2)$; or $\beta
    = 2$ and $\gamma \geq 1$,
  \begin{align*}
    I_1 \lesssim \begin{dcases}
                   \int_0^{t - s} \ud r_2 \int_0^{r_2} \ud r_1 \: (r_2 - r_1)^{2H - 2}  r_1^{\beta + \gamma - 1} r_2^{\beta + \gamma - 1}   \left(r_2^{-\ell \beta /\alpha} + r_2^{-\beta}  r_1^{(\alpha - \ell)\beta /\alpha} \right), & \ell \neq \alpha, \\
                   \int_0^{t - s} \ud r_2 \int_0^{r_2} \ud r_1 \: (r_2 - r_1)^{2H - 2}  r_1^{\beta + \gamma - 1} r_2^{\beta + \gamma - 1}   r_2^{-\beta}  \left( 1 + \log(r_2/r_1) \right),                                             & \ell = \alpha.
                 \end{dcases}
  \end{align*}
  If $\ell \neq \alpha$, it is not hard to further deduce that $I_1 \lesssim (t
    - s)^{2\rho_1}$. On the other hand, when $\ell = \alpha$, careful
  consideration of the logarithmic correction term is required, though it does
  not impact the final result significantly. This can be seen as follows. For
  any $a > -1$, $b > -1$ and $t > 0$, with a change of variables $s/t \mapsto r$,
  \begin{align}\label{E_:int-log}
    \int_0^t \ud s \: |t-s|^{a}  s^{b} \log(t/s) = - t^{a + b + 1} \int_0^1 \ud r \: |1-r|^{a}  r^{b} \log(r) \lesssim t^{a + b + 1} \approx  \int_0^t \ud s \: |t-s|^{a}  s^{b}.
  \end{align}
  On the other hand, when $\beta = 2$ and $\gamma \in (0,1)$ under the condition  $\ell < \alpha (\gamma + 1/2)$, part (ii) of Lemma~\ref{L:intml} implies
  % \begin{align*}
  %   I_1 \lesssim \begin{dcases}
  %      \int_{0}^{t-s} \ud r_2 \int_0^{ r_2} \ud r_1 \: (r_2 - r_1)^{2H-2} r_1^{ 2 \gamma - 2\ell/\alpha + 2} , & \ell/\alpha - 1/2 < \gamma < 2 \ell/\alpha - 1 \\
  %      \int_{0}^{t-s} \ud r_2 \int_0^{ r_2} \ud r_1 \: (r_2 - r_1)^{2H-2} r_1^{\gamma + 1} \big(1 + \log (r_2/r_1 )\big), & \gamma = 2 \ell/\alpha - 1, \\ 
  %      \int_{0}^{t-s} \ud r_2 \int_0^{ r_2} \ud r_1 \:  (r_2 - r_1)^{2H-2} r_1^{\gamma + 1} r_2^{\gamma - 2 \ell/\alpha + 1}, & \gamma > 2 \ell/\alpha - 1
  %   \end{dcases}
  % \end{align*}

  \begin{align*}
    I_1  \lesssim
    \begin{dcases}
      % \int_{0}^{t-s} \ud r_2 \int_0^{ r_2} \ud r_1 \: (r_2 - r_1)^{2H-2} r_1^{2\gamma - 2\ell/\alpha + 2}  + \left(r_2 - r_1\right)^{2H + 2\gamma - 2\ell/\alpha}, &  \gamma + 1 < \ell/\alpha , \\
      % \int_{0}^{t-s} \ud r_2 \int_0^{ r_2} \ud r_1  \: (r_2 - r_1)^{2H-2} \left(1 + |\log \left((r_2 - r_1)/r_2\right)| \right) , & \gamma + 1 = \ell/\alpha , \\
      \int_{0}^{t-s} \ud r_2 \int_0^{ r_2} \ud r_1  \: (r_2 - r_1)^{2H-2} r_1^{2 \gamma  - 2\ell/\alpha + 2},                                             & \ell/\alpha < \gamma + 1 < 2 \ell/\alpha , \\
      \int_{0}^{t-s} \ud r_2 \int_0^{ r_2} \ud r_1  \: (r_2 - r_1)^{2H-2} r_1^{\gamma + 1} r_2^{\gamma - 2\ell/\alpha + 1} \big(1 + \log (r_2/r_1 )\big), & \gamma + 1 =  2\ell/\alpha,                \\
      \int_{0}^{t-s} \ud r_2 \int_0^{ r_2} \ud r_1  \: (r_2 - r_1)^{2H-2} r_1^{\gamma + 1} r_2^{ \gamma - 2\ell/\alpha + 1},                              & \gamma + 1 > 2 \ell/\alpha.
    \end{dcases}
  \end{align*}

  Similar argument as in~\eqref{E_:int-log}, in each case, with constants $a, b, c$ such that $a + b + c = 2H + 2\gamma - 2\ell/\alpha > -1$ thanks to Dalang's condition~\eqref{E:main}, $\min \{a,b\} > -1$ and $a + b + c > -2$, we have
  \begin{align*}
    I_1 \lesssim  (t - s)^{2H + 2\gamma - 2\ell/\alpha + 2} \int_{0}^{1} \ud r_2 & \int_0^{r_2} \ud r_1 \:  (r_2 - r_1)^{a} r_1^{b} r_2^{c} \left(1 + \log (r_2/r_1) \right)  \lesssim (t - s)^{2 \rho_1}.
  \end{align*}
  As a consequence, we conclude that
  \begin{align*}
    I_1 \lesssim (t - s)^{2\rho_1}.
  \end{align*}
  for either $\beta \in (0,2)$; or $\beta = 2$ and $\gamma \ge 1$.
  \medskip

  If $H = 1/2$, by It\^o's isometry and a change of variables $r \mapsto t-r$,
  \begin{align*}
    I_1
    \approx \int_{0}^{t-s} \ud r\; r^{2(\beta + \gamma - 1)} \int_{0}^{\infty} \ud z\; z^{\ell/\alpha - 1}\left|E_{\beta, \beta + \gamma}\left(- r^\beta z\right)\right|^2.
  \end{align*}
  When $\beta \in (0,2)$, or $\beta = 2$ and $\gamma \ge 1$, the $z$-integral is
  finite by Lemma~\ref{L:intml} (with $s_1=s_2$) and gives the factor
  $\lesssim r^{-\ell\beta/\alpha}$. When $\beta=2$ and $\gamma\in(0,1)$, the
  $z$-integral is finite by the uniform bound~\eqref{E:ML-UB-1} in
  Lemma~\ref{L:ML} (iii) together with the standing condition
  $\ell<\alpha(\gamma+1/2)$. In all cases, we obtain $I_1 \lesssim (t-s)^{2\rho_1}$.

  \bigskip\noindent\textbf{Estimate for $I_2$.~} We estimate $I_2$ in two cases
  $H = 1/2$ and $H > 1/2$ separately. \bigskip

  \textbf{Case~1 ($H = 1/2$):~} Thanks to part (iv) of Lemma~\ref{L:ML}, we can
  write
  \begin{gather*}
    \begin{aligned}
      I_2 =   & 2\pi a_H \int_{\R^d} \ud \xi\; |\xi|^{\ell - d} \int_0^s \ud r  \bigg((t - r)^{\beta + \gamma - 1} E_{\beta, \beta + \gamma}\left(-2^{-1}\nu |\xi|^\alpha (t - r)^\beta\right)                   \\
              & \hspace{44mm}  - (s - r)^{\beta + \gamma -1} E_{\beta, \beta + \gamma}\left(-2^{-1}\nu |\xi|^\alpha (s - r)^\beta\right)\bigg)^2                                                                 \\
      \approx & \int_0^{\infty} \ud z\; z^{\ell/\alpha - 1} \int_0^s \ud r \left(\int_{s-r}^{t-r} \ud \theta \; \theta^{\beta + \gamma - 2} E_{\beta,\beta + \gamma - 1} \left(- \theta^\beta z\right) \right)^2 \\
      \approx & \int_s^t \ud \theta_2 \int_s^{\theta_2} \ud \theta_1 \int_0^s \ud r \: (\theta_1 - r)^{\beta + \gamma - 2} (\theta_2 - r)^{\beta + \gamma - 2} J (\theta_1, \theta_2, r),
    \end{aligned}
    \shortintertext{where}
    J(\theta_1, \theta_2, r) \coloneqq \int_0^{\infty} \ud z\; z^{\ell/\alpha - 1} E_{\beta,\beta + \gamma - 1} \left(-  (\theta_1 - r)^\beta z\right) E_{\beta,\beta + \gamma - 1} \left(-  (\theta_2 - r)^\beta z\right).
  \end{gather*}
  It follows from Lemma~\ref{L:intml} that if $\beta \in (0,2)$; or $\beta = 2$
  and $\gamma \geq 2$, with $\epsilon > 0$ arbitrarily small,
  \begin{align*}
    I_2 \lesssim \begin{dcases}
                   \int_s^t \ud \theta_2 \int_s^{\theta_2} \ud \theta_1 \int_0^s \ud r \; (\theta_1 - r)^{\beta + \gamma - 2} (\theta_2 - r)^{\beta + \gamma -\beta \ell /\alpha - 2}, & \ell < \alpha, \\
                   \int_s^t \ud \theta_2 \int_s^{\theta_2} \ud \theta_1 \int_0^s \ud r \; (\theta_1 - r)^{\beta + \gamma - 2 - \epsilon} (\theta_2 - r)^{\gamma - 2 + \epsilon} ,      & \ell = \alpha, \\
                   \int_s^t \ud \theta_2 \int_s^{\theta_2} \ud \theta_1 \int_0^s \ud r \; (\theta_1 - r)^{2 \beta + \gamma  - \ell \beta/ \alpha - 2} (\theta_2 - r)^{ \gamma - 2} ,   & \ell > \alpha;
                 \end{dcases}
  \end{align*}
  and if $\beta = 2$ and $\gamma \in (0,2)$, under the assumption that $\ell < \alpha (\gamma + 1/2)$,
  %  \begin{align*}
  %   I_2 \lesssim \begin{dcases}
  %      %  \begin{aligned}
  %      %   \int_s^t \ud \theta_2 & \int_s^{\theta_2} \ud \theta_1 \int_0^s \ud r \:\left( (\theta_1 - r)^{ 2 \gamma - 2\ell/\alpha}  +\left(\theta_2 - \theta_1\right)^{2\gamma - 2 \ell/\alpha}\right),
  %      %  \end{aligned} &  \gamma < \ell/\alpha , \\
  %      % \int_s^t \ud \theta_2  \int_s^{\theta_2} \ud \theta_1 \int_0^s \ud r \: \left(1 + |\log \left((\theta_2 - \theta_1)/(\theta_2 - r)\right)| \right) ,
  %      % & \gamma = \ell/\alpha , \\
  %      \int_s^t \ud \theta_2 \int_s^{\theta_2} \ud \theta_1 \int_0^s \ud r \: (\theta_1 - r)^{2\gamma - 2\ell/\alpha}  , & \ell/\alpha < \gamma < 2 \ell/\alpha \\
  %      \int_s^t \ud \theta_2 \int_s^{\theta_2} \ud \theta_1 \int_0^s \ud r \: (\theta_1 - r)^{ \gamma} \big(1 + \log ((\theta_2 - r)/(\theta_1 - r) )\big), & \gamma = 2 \ell/\alpha , \\ 
  %       \int_s^t \ud \theta_2 \int_s^{\theta_2} \ud \theta_1 \int_0^s \ud r \: (\theta_1 - r)^{ \gamma} (\theta_2 - r)^{ \gamma - 2 \ell/\alpha}, & \gamma > 2 \ell/\alpha 
  %   \end{dcases}
  % \end{align*}
  \begin{align*}
    I_2  \lesssim
    \begin{dcases}
      \int_s^t \ud \theta_2 \int_s^{\theta_2} \ud \theta_1 \int_0^s \ud r \: \left((\theta_1 - r)^{2 \gamma - 2\ell/\alpha } + \left(\theta_2 - \theta_1\right)^{2\gamma - 2\ell/\alpha}\right),
                                                                                                                                                                                                              & \gamma < \ell/\alpha ,                 \\
      \int_s^t \ud \theta_2 \int_s^{\theta_2} \ud \theta_1 \int_0^s \ud r \: \left(1 + \left|\log \left(\frac{\theta_2 - \theta_1}{\theta_2 - r}\right)\right| \right) ,                                      & \gamma = \ell/\alpha ,                 \\
      \int_s^t \ud \theta_2 \int_s^{\theta_2} \ud \theta_1 \int_0^s \ud r \: (\theta_1 - r)^{ 2\gamma - 2\ell/\alpha } ,                                                                                      & \ell/\alpha < \gamma < 2 \ell/\alpha , \\
      \int_s^t \ud \theta_2 \int_s^{\theta_2} \ud \theta_1 \int_0^s \ud r \: (\theta_1 - r)^{\gamma} (\theta_2 - r)^{\gamma - 2\ell/\alpha} \bigg(1 + \log \Big(\frac{\theta_2 - r}{\theta_1 - r}\Big)\bigg), & \gamma =  2\ell/\alpha,                \\
      \int_s^t \ud \theta_2 \int_s^{\theta_2} \ud \theta_1 \int_0^s \ud r \: (\theta_1 - r)^{ \gamma} (\theta_2 - r)^{\gamma - 2\ell/\alpha },                                                                & \gamma > 2 \ell/\alpha.
    \end{dcases}
  \end{align*}
  Therefore, the desired bound for $I_2$ follows as a consequence of
  Lemma~\ref{L:lmm_tu1}.

  \bigskip\textbf{Case~2 ($H > 1/2$):~} In this case, we have
  \begin{align*}
    I_2 \approx
     & \int_0^s \ud r_2 \int_0^s \ud r_1 \: (r_2 - r_1)^{2H-2} \int_{s-r_1}^{t-r_1} \ud \theta_1 \; \theta_1^{\beta + \gamma - 2}  \int_{s-r_2}^{t-r_2} \ud \theta_2 \; \theta_2^{\beta + \gamma - 2} \\
     & \times \int_0^{\infty} \ud z\; z^{\ell/\alpha - 1} E_{\beta,\beta + \gamma - 1} \left(- \theta_1^\beta z\right) E_{\beta,\beta + \gamma - 1} \left(- \theta_2^\beta z\right).
  \end{align*}

  Suppose $\beta \in (0,2)$; or $\beta = 2$ and $\gamma \geq 2$. It follows from
  part (i) of Lemma~\ref{L:intml} that if 1) $\ell > \alpha$,
  \begin{gather*}\begin{aligned}
      I_2 \lesssim
        & \int_0^s \ud r_2 \int_0^{r_2} \ud r_1 \: (r_2-r_1)^{2H-2}   \int_{s-r_1}^{t-r_1} \ud \theta_1 \; \theta_1^{\beta + \gamma - 2} \int_{s-r_2}^{t-r_2} \ud \theta_2 \; \theta_2^{\beta + \gamma - 2} \left(\theta_1 \wedge \theta_2\right)^{\beta(1 - \ell/\alpha)}  \left(\theta_1 \vee \theta_2\right)^{-\beta} \\
      = & \int_{s}^{t} \ud \theta_2 \int_{s}^{t} \ud \theta_1  \int_0^s \ud r_2 \int_0^{r_2} \ud r_1 \: (r_2-r_1)^{2H-2}  (\theta_1 - r_1)^{\beta + \gamma - 2}  \; (\theta_2 - r_2)^{\beta + \gamma - 2}                                                                                                                \\
        & \qquad \times \big((\theta_1 - r_1) \wedge (\theta_2 - r_2)\big)^{\beta(1 - \ell/\alpha)}  \big((\theta_1 - r_1) \vee (\theta_2 - r_2)\big)^{-\beta}                                                                                                                                                           \\
      = & \int_{s}^{t} \ud \theta_2 \int_{s}^{\theta_2} \ud \theta_1  \int_0^s \ud r_2 \int_{0 \vee (r_2 - \theta_2 + \theta_1)}^{r_2} \ud r_1 \: (r_2-r_1)^{2H-2}  (\theta_1 - r_1)^{2\beta + \gamma - \beta \ell/\alpha - 2}  (\theta_2 - r_2)^{\gamma - 2}                                                            \\
        & + \int_{s}^{t} \ud \theta_2 \int_{s}^{t} \ud \theta_1  \int_0^s \ud r_2 \int_{0}^{(r_2 - \theta_2 + \theta_1) \wedge r_2} \ud r_1 \: (r_2-r_1)^{2H-2} (\theta_1 - r_1)^{\gamma - 2}  (\theta_2 - r_2)^{2\beta + \gamma - \beta \ell/\alpha - 2};
    \end{aligned}\end{gather*}
  if 2) $\ell < \alpha$, then
  \begin{gather*}\begin{aligned}
      I_2 \lesssim
        & \int_0^s \ud r_2 \int_0^{r_2} \ud r_1 \: (r_2-r_1)^{2H-2}   \int_{s-r_1}^{t-r_1} \ud \theta_1 \; \theta_1^{\beta + \gamma - 2} \int_{s-r_2}^{t-r_2} \ud \theta_2 \; \theta_2^{\beta + \gamma - 2} \left(\theta_1 \vee \theta_2\right)^{- \beta\ell/\alpha} \\
      = & \int_{s}^{t} \ud \theta_2 \int_{s}^{\theta_2} \ud \theta_1  \int_0^s \ud r_2 \int_{0 \vee (r_2 - \theta_2 + \theta_1)}^{r_2} \ud r_1 \: (r_2-r_1)^{2H-2}  (\theta_1 - r_1)^{\beta + \gamma - 2}  (\theta_2 - r_2)^{\beta + \gamma - 2 - \beta\ell/\alpha}  \\
        & + \int_{s}^{t} \ud \theta_2 \int_{s}^{t} \ud \theta_1  \int_0^s \ud r_2 \int_{0}^{(r_2 - \theta_2 + \theta_1) \wedge r_2} \ud r_1 \: (r_2-r_1)^{2H-2} (\theta_1 - r_1)^{\beta + \gamma - \beta \ell/\alpha - 2}  (\theta_2 - r_2)^{\beta + \gamma - 2};
    \end{aligned}\end{gather*}
  and if 3) $\ell = \alpha$, then
  \begin{gather*}\begin{aligned}
      I_2 \lesssim & \int_{s}^{t} \ud \theta_2 \int_{s}^{\theta_2} \ud \theta_1  \int_0^s \ud r_2 \int_{0 \vee (r_2 - \theta_2 + \theta_1)}^{r_2} \ud r_1 \: (r_2-r_1)^{2H-2}  (\theta_1 - r_1)^{\beta + \gamma - 2}  (\theta_2 - r_2)^{ \gamma - 2}                       \\
                   & \hspace{67mm} \times \log \big((\theta_2 - r_2)/(\theta_1 - r_1)\big)                                                                                                                                                                                 \\
                   & + \int_{s}^{t} \ud \theta_2 \int_{s}^{t} \ud \theta_1  \int_0^s \ud r_2 \int_{0}^{(r_2 - \theta_2 + \theta_1) \wedge r_2} \ud r_1 \: (r_2-r_1)^{2H-2} (\theta_1 - r_1)^{\gamma  - 2}  (\theta_2 - r_2)^{\beta + \gamma - 2}                           \\
                   & \hspace{67mm} \times  \log \big((\theta_1 - r_1)/(\theta_2 - r_2)\big)                                                                                                                                                                                \\
      \lesssim     & \int_{s}^{t} \ud \theta_2 \int_{s}^{\theta_2} \ud \theta_1  \int_0^s \ud r_2 \int_{0 \vee (r_2 - \theta_2 + \theta_1)}^{r_2} \ud r_1 \: (r_2-r_1)^{2H-2}   (\theta_1 - r_1)^{\beta + \gamma - \epsilon - 2}  (\theta_2 - r_2)^{\gamma + \epsilon - 2} \\
                   & +\int_{s}^{t} \ud \theta_2 \int_{s}^{t} \ud \theta_1   \int_0^s \ud r_2 \int_{0}^{(r_2 - \theta_2 + \theta_1) \wedge r_2} \ud r_1 \: (r_2-r_1)^{2H-2} (\theta_1 - r_1)^{\gamma + \epsilon - 2}  (\theta_2 - r_2)^{\beta + \gamma - \epsilon - 2},
    \end{aligned}\end{gather*}
  with $\epsilon \in (0, \beta/2)$, where the last inequality is due to the fact
  that $\sup_{x\geq 1}\log (x)/x^{\epsilon} < \infty$ for all $\epsilon > 0$.
  Therefore, using Lemma~\ref{L:lmm_tu}, one can easily deduce that in all the
  cases, $I_2$ is bounded by the desired order of $t-s$. \bigskip

  On the other hand, the case when $\beta = 2$, $\gamma \in (0,1)\cup(1,2)$ such
  that $\ell < \alpha (\gamma + 1/2)$ can be proved in a similar
  way. In fact, by comparing~\eqref{E:intml} and~\eqref{E:intml-a=2}, it
  suffices to estimate the following quantities
  \begin{align*}
    I_{2,1}^* \coloneqq \int_0^s \ud r_2 \int_0^s \ud r_1 \: (r_2 - r_1)^{2H-2} \int_{s-r_1}^{t-r_1} \ud \theta_1   \int_{s-r_2}^{t-r_2} \ud \theta_2 \; \theta_1^{2\gamma - 2 \ell/\alpha} ,
  \end{align*}
  \begin{align*}
    I_{2,2}^* \coloneqq \int_0^s \ud r_2 \int_0^s \ud r_1 \: (r_2 - r_1)^{2H-2} \int_{s-r_1}^{t-r_1} \ud \theta_1   \int_{s-r_2}^{t-r_2} \ud \theta_2 \;  \left|\theta_2 - \theta_1\right|^{2\gamma - 2\ell/\alpha},
  \end{align*}
  when $\ell > \alpha \gamma$; and
  \begin{align*}
    I_{2,3}^{*} \coloneqq \int_0^s \ud r_2 \int_0^s \ud r_1 \: (r_2 - r_1)^{2H-2} \int_{s-r_1}^{t-r_1} \ud \theta_1   \int_{s-r_2}^{t-r_2} \ud \theta_2 \;   \left|\log \left(\left(\theta_2 - \theta_1\right)/\theta_2\right)\right|,
  \end{align*}
  if $\ell = \alpha \gamma$. With a change of variables
  $(\theta_1, \theta_2) \mapsto ((t - s)\theta_1' + s - r_1,\; (t - s)\theta_2' + s - r_2)$,
  it is not hard to deduce that
  \begin{align*}
    I_{2,1}^* + I_{2,2}^* + I_{2,3}^* \lesssim (t - s)^2.
  \end{align*}
  Finally, the case $\gamma = 1$ can be also estimated similarly using \eqref{E:int_sin}, the details are omitted for the sake of conciseness.  The proof
  of Proposition~\ref{T:Upper-t} is thus complete under the assumption that
  either $\beta \in (0,2)$; or $\beta = 2$ and $\gamma \geq 2$; or $\beta = 2$ and
  $\gamma \in (0,2)$ with $\ell < \alpha(\gamma + 1/2)$.
\end{proof}

\subsubsection{Proof of Proposition~\ref{T:Upper-t}: part (ii)}

The estimate for $I_1$ is more or less the same in Case (i). Thus for the sake
of simplification, we only provide the estimate for $I_2$. The next lemma is
used in case $H > 1/2$, whose proof is postponed to Section~\ref{SS:lm_tu}.
\begin{lemma}\label{L:recF}
  Let $t > s > 0$, $- 1 < a < 0$, $-3 < b < a$, and $F_b$ be defined as
  in~\eqref{E:F_tht}. Denote by,
  \begin{align*}
    I \coloneqq \int_0^s \ud r_2 \int_0^{r_2} \ud r_1 \; (r_2 - r_1)^{a} \big( & F_{b} (t-r_1,t-r_2) - F_{b} (t-r_1,s-r_2)              \\
                                                                               & - F_{b}(s - r_1, t - r_2) + F_{b} (s-r_1,s-r_2) \big).
  \end{align*}
  Furthermore, assume $t - s \ll s$. Then,
  \begin{align*}
    I \lesssim \begin{dcases}
                 (t - s)^{a - b},               & a - b < 2, \\
                 (t - s)^2 (1 + |\log(t - s)|), & a - b = 2, \\
                 (t - s)^2,                     & a - b >2.
               \end{dcases}
  \end{align*}
\end{lemma}
\begin{proof}[Proof of part (ii) of Proposition~\ref{T:Upper-t}]
  In case $H = 1/2$, recalling~\eqref{E:FG_b2g0}, we can write
  \begin{align*}
    I_2
    \approx & \int_0^s \ud r \int_{0}^{\infty} \ud y\; y^{2 \ell/\alpha - 3} \left( \sin \left((t - r) y\right) - \sin \left( (s - r) y\right) \right)^2 \\
    \approx & \int_0^s \ud r \; \left( F_{2\ell/\alpha - 3} (t-r,t-r) - 2 F_{2 \ell/\alpha - 3} (t-r,s-r) + F_{2 \ell/\alpha - 3} (s-r,s-r) \right).
  \end{align*}
  As a consequence of Lemma~\ref{L:F_tht},
  \begin{gather*}
    I_2 \approx  \int_0^s \ud r \; J (r),
    \shortintertext{with}
    J (r) \coloneqq \left| (2t - 2r)^{2 - 2\ell/\alpha} + (2s - 2r)^{2 - 2\ell/\alpha} - 2 (t + s - 2r)^{2 - 2\ell/\alpha} + 2 (t - s)^{2 - 2\ell/\alpha}\right|.
  \end{gather*}
  Notice that in case $H = 1/2$, Dalang's condition presented in
  Theorem~\ref{T:Main} requires that $0 < \ell < \alpha$. It follows that $2 -
    2\ell /\alpha \in (0,2)$. As a result, we need to work on the following two
  cases.

  If $2 - 2 \ell /\alpha \in (0,1]$, using sub-additivity, we have
  \begin{align*}
    J (r) \lesssim & \left| (2 (t - r))^{2 - 2\ell /\alpha} - (t + s - 2r)^{2 - 2\ell /\alpha} \right| + \left| (t + s - 2r)^{2 - 2\ell /\alpha} - (2 (s - r))^{2 - 2\ell /\alpha} \right| \\
                   & + (t - s)^{2 - 2\ell/\alpha}                                                                                                                                          \\
    \lesssim       & (t - s)^{2 - 2\ell /\alpha}.
  \end{align*}

  Otherwise, if $2 - 2\ell/\alpha \in (1,2)$, it follows from the fundamental
  theorem of calculus that
  \begin{align*}
    J(r) =   & \left|\int_{t + s - 2r}^{2(t - r)} \ud \tau \; \tau^{1 - 2\ell/\alpha} -  \int_{2(s - r)}^{t + s - 2r} \ud \tau\; \tau^{1 - 2\ell/\alpha}  \right| + 2 (t - s)^{2 - 2\ell/\alpha} \\
    =        & \left|\int_{t + s - 2r}^{2(t - r)} \ud \tau \; \left(\tau^{1 - 2\ell/\alpha} - (\tau - t + s)^{1 - 2\ell/\alpha} \right)  \right| + 2 (t - s)^{2 - 2\ell/\alpha}                  \\
    \lesssim & \int_{t + s - 2r}^{2(t - r)} \ud \tau \; (t - s)^{1 - 2\ell/\alpha}  + (t-s)^{2 - 2\ell/\alpha} \lesssim (t - s)^{2 - 2\ell/\alpha}.
  \end{align*}
  This proves part (ii) of Proposition~\ref{T:Upper-t} when $H = 1/2$. \bigskip

  Next, assume $ H > 1/2$. Then,
  \begin{align*}
    I_2 \approx & \int_0^s \ud r_2 \int_0^{r_2} \ud r_1 \; (r_2 - r_1)^{2H - 2} \int_{0}^{\infty} \ud y\; y^{\ell/\alpha - 3} \left( \sin \left((t - r_1) y\right) - \sin \left( (s - r_1) y\right) \right) \\
                & \times \left( \sin \left((t - r_2) y\right) - \sin \left( (s - r_2) y\right) \right)                                                                                                      \\
    \approx     & \int_0^s \ud r_2 \int_0^{r_2} \ud r_1 \; (r_2 - r_1)^{2H - 2} \left(F_{2\ell/\alpha - 3} (t-r_1,t-r_2) - F_{2 \ell/\alpha - 3} (t-r_1,s-r_2) \right.                                      \\
                & \hspace{50mm} \left. - F_{2\ell/\alpha - 3}(s - r_1, t - r_2) + F_{2 \ell/\alpha - 3} (s-r_1,s-r_2) \right).
  \end{align*}
  The proof of part (ii) of Proposition~\ref{T:Upper-t} is thus complete by
  referring to Lemma~\ref{L:recF}.
\end{proof}

\subsubsection{Proofs of Lemmas~\ref{L:lmm_tu1},~\ref{L:lmm_tu}, and~\ref{L:recF}}\label{SS:lm_tu}

In this section, we provide the proofs for
Lemmas~\ref{L:lmm_tu1},~\ref{L:lmm_tu}, and~\ref{L:recF}.

\begin{proof}[Proof of Lemma~\ref{L:lmm_tu1}]
  Denote by $I$ the left hand side of inequality~\eqref{E:lmm_tu1}. \bigskip

  If $\min \{a, b\} \geq 0$, then
  \begin{align*}
    I \lesssim \int_s^t \ud \theta_2 \int_s^{\theta_2} \ud \theta_1 \lesssim (t - s)^2.
  \end{align*}

  \bigskip
  Additionally, assume $\min \{a, b\} < 0$ and $a + b > -2$. Notice that
  \begin{align}\label{E_:tha_b<0}
    (\theta_2 - r)^{c} \lesssim (\theta_1 - r)^c \quad \text{whenever } c < 0 \text{ and } r < s < \theta_1 < \theta_2.
  \end{align}
  Thus,
  \begin{align*}
    I \lesssim \int_{s}^{t} \ud \theta_2 \int_{s}^{\theta_2} \ud \theta_1  \int_0^s \ud r\;  (\theta_1 - r)^{a + b} \lesssim
    \begin{dcases}
      \int_{s}^{t} \ud \theta_2 \int_{s}^{\theta_2} \ud \theta_1\; (\theta_1 - s)^{a + b + 1},              & a + b < -1, \\
      \int_{s}^{t} \ud \theta_2 \int_{s}^{\theta_2} \ud \theta_1\; \log \big(\theta_1/(\theta_1 - s) \big), & a + b = -1, \\
      \int_{s}^{t} \ud \theta_2 \int_{s}^{\theta_2} \ud \theta_1\; \theta_1^{a + b + 1},                    & a + b > -1.
    \end{dcases}
  \end{align*}
  The desired bounds for case $a + b \neq -1$ follow from a direct calculation.
  Additionally, in case $a + b = -1$, with a change of variables
  $(\theta_1, \theta_2) \mapsto (\theta_1 (t - s) + s, \theta_2(t - s) + s)$, we have
  \begin{align*}
    I \lesssim (t - s)^2 \int_{0}^{1} \ud \theta_2 \int_{0}^{\theta_2} \ud \theta_1\; \log \left(1 + \frac{s}{\theta_1 (t - s)} \right).
  \end{align*}
  Recall the assumption that $t - s \ll s$. It follows that
  \begin{align*}
    I \lesssim (t - s)^2 \int_{0}^{1} \ud \theta_2 \int_{0}^{\theta_2} \ud \theta_1\; \log \left( \left(1 + \theta_1^{-1}\right) \frac{s}{t - s} \right) \lesssim  (t - s)^2 \left(1 + \left|\log (t - s)\right| \right).
  \end{align*}

  Finally, assume $\min \{a, b\}  < 0$ and $a + b \leq -2$. As $\min\{a,b\} \geq -2$, it follows that $\max\{a,b\} \leq 0$. If $a  < -1$, then
  \begin{align*}
    \int_{s}^{t} \ud \theta_2 \int_{s}^{\theta_2} \ud \theta_1  \int_0^s \ud r\;
    (\theta_1 - r)^{a}  (\theta_2 - r)^{b}
    \lesssim & \int_{s}^{t} \ud \theta_2 \; (\theta_2 - s)^{b}\int_{s}^{\theta_2} \ud \theta_1  (\theta_1 - s)^{a + 1} \\
    \lesssim & (t - s)^{a + b + 3}.
  \end{align*}
  Otherwise suppose $a \geq -1$, positive number $\epsilon$ such that $a + 1 <
    \epsilon < \min \{- b, a + 2\}$ and denote $(a',b') = (a - \epsilon, b +
    \epsilon)$, then we have $-2 < a' < - 1$, $b' < 0$, $a' + b' = a + b$, and
  using~\eqref{E_:tha_b<0} one more time,
  \begin{align*}
    (\theta_1 - r)^a (\theta_2 - r)^b \leq (\theta_1 - r)^{a'} (\theta_2 - r)^{b'}.
  \end{align*}
  Hence,
  \begin{align*}
    \int_{s}^{t} \ud \theta_2 \int_{s}^{\theta_2} \ud \theta_1  \int_0^s \ud r\;
    (\theta_1 - r)^{a}  (\theta_2 - r)^{b}
    \lesssim & \int_{s}^{t} \ud \theta_2 \int_{s}^{\theta_2} \ud \theta_1  \int_0^s \ud r\;
    (\theta_1 - r)^{a'}  (\theta_2 - r)^{b'}                                                \\
    \lesssim & (t - s)^{a + b + 3}.
  \end{align*}
  The proof of this lemma is complete.
\end{proof}

We present the estimates for $I$ and $I'$ in Lemma~\ref{L:lmm_tu} separately.

\begin{proof}[Proof of Lemma~\ref{L:lmm_tu} (Estimate for $I$)]
  Without loss of generality, we can assume that neither $b$ nor $b + c$ is an
  integer. Otherwise, noting $r_2 - r_1 < \theta_1 - r_1$, we can choose
  $\epsilon \in (0, 1 + a)$, so that $I$ is bounded by the same integral with
  $(a,b,c)$ replacing by $(a',b',c') = (a - \epsilon, b + \epsilon, c)$. Next,
  we write
  \begin{gather*}
    I = I_1 + I_2,
    \shortintertext{where}
    I_1 \coloneqq \int_{s}^{t} \ud \theta_2 \int_{s}^{\theta_2} \ud \theta_1  \int_0^{s\wedge (\theta_2 - \theta_1)} \ud r_2 \int_{0}^{r_2} \ud r_1 \: (r_2-r_1)^{a}  (\theta_1 - r_1)^{b}  (\theta_2 - r_2)^{c}
    \shortintertext{and}
    I_2 \coloneqq \int_{s}^{t} \ud \theta_2 \int_{s}^{\theta_2} \ud \theta_1  \int_{s\wedge (\theta_2 - \theta_1)}^s \ud r_2 \int_{r_2 - \theta_2 + \theta_1}^{r_2} \ud r_1 \: (r_2-r_1)^{a}  (\theta_1 - r_1)^{b}  (\theta_2 - r_2)^{c}.
  \end{gather*}
  Recall that $\theta_1, \theta_2 \in (s,t)$, and thus $\theta_2 - \theta_1 <
    t-s$. Using the assumption that $t - s \ll s$, we have
  \begin{align*}
    I_1 \lesssim & \int_{s}^{t} \ud \theta_2 \int_{s}^{\theta_2} \ud \theta_1  \int_0^{t-s} \ud r_2 \int_{0}^{r_2} \ud r_1 \: (r_2-r_1)^{a}  (\theta_1 - r_1)^{b}  (\theta_2 - r_2)^{c}                                         \\
    <            & \max \left\{t^{b}, (2s - t)^{b} \right\} \max \left\{t^{c}, (2s - t)^{c}\right\}  \int_{s}^{t} \ud \theta_2 \int_{s}^{\theta_2} \ud \theta_1  \int_0^{t - s} \ud r_2 \int_{0}^{r_2} \ud r_1 \: (r_2-r_1)^{a} \\
    \lesssim     & (t - s)^{4 + a} \lesssim (t - s)^2,
  \end{align*}
  because $a > -1$. \bigskip

  For $I_2$, we consider the following cases.

  \bigskip\noindent\textbf{Case~1 ($c \ge 0$).~} In this case, $b > c \geq 0$,
  and similar as in the estimate for $I_1$, we have
  \begin{align*}
    I_2 \leq t^{b + c} \int_{s}^{t} \ud \theta_2 \int_{s}^{\theta_2} \ud \theta_1  \int_{0}^s \ud r_2 \int_{0}^{r_2} \ud r_1 \: (r_2-r_1)^{a}  \lesssim (t-s)^2.
  \end{align*}

  \bigskip\noindent\textbf{Case~2 ($c < 0$ and $b < -1$).~} In this cases, $-2
    \le c < b < -1$, and thus
  \begin{gather}\begin{aligned}\label{E:lmm_ut-gm>0}
      I_2 \leq & \int_{s}^{t} \ud \theta_2 \int_{s}^{\theta_2} \ud \theta_1  \int_{\theta_2 - \theta_1}^s \ud r_2 \; (\theta_1 - r_2)^{b}  (\theta_2 - r_2)^{c}  \int_{r_2 - \theta_2 + \theta_1}^{r_2} \ud r_1 \: (r_2-r_1)^{a} \\
      \lesssim & \int_{s}^{t} \ud \theta_2 \int_{s}^{\theta_2} \ud \theta_1 (\theta_2 - \theta_1)^{a + 1}  \int_{\theta_2 - \theta_1}^s \ud r_2 \; (\theta_1 - r_2)^{b}  (\theta_2 - r_2)^{c}                                    \\
      \lesssim & \int_{s}^{t} \ud \theta_2 \;  (\theta_2 - s)^{c}\int_{s}^{\theta_2} \ud \theta_1 (\theta_2 - \theta_1)^{a + 1} \int_{\theta_2 - \theta_1}^s \ud r_2 \; (\theta_1 - r_2)^{b}                                     \\
      \lesssim & \int_{s}^{t} \ud \theta_2 \; (\theta_2 - s)^{c}\int_{s}^{\theta_2} \ud \theta_1 \; (\theta_2 - \theta_1)^{a + 1} (\theta_1 - s)^{b + 1} \lesssim (t-s)^{a + b + c + 4}.
    \end{aligned}\end{gather}

  \bigskip\noindent\textbf{Case~3 ($c < 0$, $b \geq -1$, and $b + c < -2$).~} In
  this case, we have $b < 0$. Thus, similarly as in~\eqref{E:lmm_ut-gm>0}, we
  can show that
  \begin{gather}\label{E:lmm_ut-gm>0-1}
    I_2 \lesssim \int_{s}^{t} \ud \theta_2 \int_{s}^{\theta_2} \ud \theta_1 \; (\theta_2 - \theta_1)^{a + 1}  \int_{\theta_2 - \theta_1}^s \ud r_2 \; (\theta_1 - r_2)^{b}  (\theta_2 - r_2)^{c}.
  \end{gather}
  Since $b + c < -2$, and $\theta_1 \leq \theta_2$, we have
  \begin{align*}
    (\theta_1 - r_2)^{b}  (\theta_2 - r_2)^{c} \leq (\theta_1 - r_2)^{b - \epsilon}  (\theta_2 - r_2)^{c + \epsilon}
  \end{align*}
  with any $\epsilon \in (1 + b, (b - c)/2)$ so that $-2 \le c + \epsilon < b -
    \epsilon < -1$. As a result, by replacing $(b,c)$ by $(b - \epsilon, c +
    \epsilon)$ in~\eqref{E:lmm_ut-gm>0-1}, we can apply~\eqref{E:lmm_ut-gm>0} and
  get the $I_2 \lesssim (t-s)^{a + b + c + 4}$.

  \bigskip\noindent\textbf{Case~4 ($c < 0$, $-1 \le b \le 0$, and $b + c \ge
      -2$).~} Recall the assumption that $b + c$ is not an integer. It follows that
  $b + c > -2$. As a result,
  \begin{gather*}\begin{aligned}
      I_2\lesssim & \int_{s}^{t} \ud \theta_2  \int_{s}^{\theta_2} \ud \theta_1  \; (\theta_2 - \theta_1)^{a + 1} \int_{\theta_2 - \theta_1}^s \ud r_2 \; (\theta_1 - r_2)^{b + c}                              \\
      \lesssim    & (t-s)^{a + 1} \int_{s}^{t} \ud \theta_2 \int_{s}^{\theta_2} \ud \theta_1 \; \left( (\theta_1 - s)^{b + c + 1} \one_{\{- 2 < b + c < - 1\} } + t^{b + c + 1} \one_{\{b + c > - 1\} } \right) \\
      \lesssim    & (t-s)^{a + 1} \int_{s}^{t} \ud \theta_2 \int_{s}^{\theta_2} \ud \theta_1 \; \left( (\theta_1 - s)^{b + c + 1} \one_{\{- 2 < b + c < - 1\} } + t^{b + c + 1} \one_{\{b + c > - 1\} } \right) \\
      \lesssim    & (t-s)^{a + b + c + 4} + (t-s)^{2}.
    \end{aligned}\end{gather*}

  \bigskip\noindent\textbf{Case~5 ($c < 0 < b$).~} Finally, in this case, we
  deduce that
  \begin{gather*}\begin{aligned}
      I_2\lesssim & \int_{s}^{t} \ud \theta_2  \int_{s}^{\theta_2} \ud \theta_1  \; (\theta_2 - \theta_1)^{a + 1} \int_{\theta_2 - \theta_1}^s \ud r_2 \;  (\theta_2 - r_2)^{b}(\theta_1 - r_2)^{c}                                        \\
      \lesssim    & \int_{s}^{t} \ud \theta_2  \int_{s}^{\theta_2} \ud \theta_1  \; (\theta_2 - \theta_1)^{a + 1} \int_{\theta_2 - \theta_1}^s \ud r_2 \;  \left((\theta_2 - \theta_1)^b + (\theta_1 - r_2)^{b}\right)(\theta_1 - r_2)^{c} \\
      \lesssim    & (t-s)^{a + b + 1} \int_{s}^{t} \ud \theta_2 \int_{s}^{\theta_2} \ud \theta_1 \int_{\theta_2 - \theta_1}^s \ud r_2 \; (\theta_1 - r_2)^{c}                                                                              \\
                  & + (t-s)^{a + 1} \int_{s}^{t} \ud \theta_2  \int_{s}^{\theta_2} \ud \theta_1  \; (\theta_2 - \theta_1)^{a + 1} \int_{\theta_2 - \theta_1}^s \ud r_2 \;   (\theta_1 - r_2)^{b + c} .
    \end{aligned}\end{gather*}
  Then, following the same argument as in Case 4, we can show that
  \begin{align*}
    I_2 \lesssim & (t-s)^{a + b + c + 4} + (t-s)^{2}
  \end{align*}
  Hence, $I = I_1 + I_2 \lesssim (t-s)^{(a + b + c + 4) \wedge 2}$ for all
  cases.
\end{proof}

\begin{proof}[Proof of Lemma~\ref{L:lmm_tu} (Estimate for $I'$)]
  Similarly as in the estimate for $I$, we decompose
  \begin{gather*}
    I' = I'_1 + I'_2, \shortintertext{where}
    I'_1 \coloneqq \int_{s}^{t} \ud \theta_2 \int_{s}^{\theta_2} \ud \theta_1  \int_{\theta_2 - \theta_1}^s \ud r_2 \int_{0}^{r_2 - \theta_2 + \theta_1} \ud r_1 \: (r_2-r_1)^{a}  (\theta_1 - r_1)^{c} (\theta_2 - r_2)^{b} \shortintertext{and}
    I'_2 \coloneqq \int_{s}^{t} \ud \theta_1 \int_{s}^{\theta_1} \ud \theta_2  \int_0^s \ud r_2 \int_{0}^{r_2} \ud r_1 \: (r_2-r_1)^{a}  (\theta_1 - r_1)^{c} (\theta_2 - r_2)^{b}.
  \end{gather*}

  First, we estimate $I_1'$ as follows. When $c \geq 0$, it is clear that
  \begin{align*}
    I'_1 \lesssim t^{c+b} \int_{s}^{t} \ud \theta_2 \int_{s}^{\theta_2} \ud \theta_1  \int_0^s \ud r_2 \int_{0}^{r_2 - \theta_2 + \theta_1} \ud r_1 \: (r_2-r_1)^{a}  \lesssim (t - s)^2.
  \end{align*}

  On the other hand, if $c < 0$, we need to study the following three cases.

  \bigskip\noindent\textbf{Case~1 ($a + b + c < -2$).~} Noticing that $a \in
    (-1, 0)$ and $b > c > -2$. We see that $a + c < -1$. As a result,
  \begin{align*}
    I_1' \leq & \int_{s}^{t} \ud \theta_2 \int_{s}^{\theta_2} \ud \theta_1  \int_{\theta_2 - \theta_1}^s \ud r_2 \; (\theta_2 - r_2)^{b} \int_{0}^{r_2 - \theta_2 + \theta_1} \ud r_1 \: (r_2-r_1)^{a + c} \\
    \lesssim  & \int_{s}^{t} \ud \theta_2 \int_{s}^{\theta_2} \ud \theta_1  \int_{\theta_2 - \theta_1}^s \ud r_2 \; (\theta_2 - r_2)^{a + b + c + 1} \lesssim (t - s)^{a + b + c + 4}.
  \end{align*}

  \bigskip\noindent\textbf{Case~2: ($a + b + c = -2$).~} Just like the case $a +
    b + c < -2$, we have
  \begin{align*}
    I_1' \lesssim  \int_{s}^{t} \ud \theta_2 \int_{s}^{\theta_2} \ud \theta_1  \int_{\theta_2 - \theta_1}^s \ud r_2 \; (\theta_2 - r_2)^{-1}
    \lesssim  \int_{s}^{t} \ud \theta_2 \int_{s}^{\theta_2} \ud \theta_1 \; |\log (\theta_2 - s)|  \lesssim (t-s)^2 |\log (t-s)|.
  \end{align*}

  \bigskip\noindent \textbf{Case~3 ($a + b + c > -2$).~} In this case, because
  $\theta_1 \geq s \geq r_2$ and $r_1 \leq r_2 - \theta_2 + \theta_1$, we find
  that $\theta_1 - r_1 \geq r_2 - r_1$ and $\theta_1 - r_1 \geq \theta_2 - r_2$.
  Therefore, there exist $a' \in (-1, a]$ and $b' \in (-1, b]$ such that $a' +
      b' = a + b + c$. It follows that
      \begin{align*}
          (r_2 - r_1)^a (\theta_1 - r_1)^c (\theta_2 - r_2)^b \leq (r_2 - r_1)^{a'} (\theta_2 - r_2)^{b'},
        \end{align*}
      and thus
      \begin{align*}
          I_1' \leq & \int_{s}^{t} \ud \theta_2 \int_{s}^{\theta_2} \ud \theta_1  \int_{\theta_2 - \theta_1}^s \ud r_2 \; (\theta_2 - r_2)^{b'}  \int_{0}^{r_2 - \theta_2 + \theta_1} \ud r_1 \: (r_2 - r_1)^{a'} \\
          \lesssim  & t^{b' + 1} (t-s)^{a' + 1} \int_{s}^{t} \ud \theta_2  \int_{s}^{\theta_2} \ud \theta_1  \lesssim (t-s)^2.
        \end{align*}
      This completes the estimates for $I_1'$. As a consequence, it suffices to
      analyze $I_2'$.

      Note that for the cases i) $c \geq 0$ and ii) $c < 0$ and $a + b + c > 2$, we
    can show that $I_2' \lesssim (t-s)^2$ following the same arguments as that for
  $I_1'$. Therefore, in the following we assume that $c < 0$ and $a + b + c \leq
  -2$. Decompose $I_2' = I'_{2,1} + I'_{2,2}$, where
    \begin{gather*}
      I'_{2,1} \coloneqq \int_{s}^{t} \ud \theta_1 \int_{s}^{\theta_1} \ud \theta_2  \int_0^s \ud r_2 \int_{0}^{r_2 - \theta_1 + \theta_2} \ud r_1 \: (r_2-r_1)^{a}  (\theta_1 - r_1)^{c} (\theta_2 - r_2)^{b},
      \shortintertext{and}
      I'_{2,2} \coloneqq \int_{s}^{t} \ud \theta_1 \int_{s}^{\theta_1} \ud \theta_2  \int_0^s \ud r_2 \int_{r_2 - \theta_1 + \theta_2}^{r_2} \ud r_1 \: (r_2-r_1)^{a}  (\theta_1 - r_1)^{c} (\theta_2 - r_2)^{b}.
    \end{gather*}
    The estimate for $I_{2,1}'$ can be done using the same steps appearing in that
    for $I_{1}'$. For $I_{2,2}'$, we find that
    \begin{align*}
      I'_{2,2}
      \leq    & \int_{s}^{t} \ud \theta_1 \int_{s}^{\theta_1} \ud \theta_2  \int_0^s \ud r_2 (\theta_1 - r_2)^{c} (\theta_2 - r_2)^b  \int_{r_2 - \theta_1 + \theta_2}^{r_2} \ud r_1 \: (r_2-r_1)^{a} \\
      \approx & \int_{s}^{t} \ud \theta_1 \int_{s}^{\theta_1} \ud \theta_2 (\theta_1 - \theta_2)^{a + 1}  \int_0^s \ud r_2 \; (\theta_1 - r_2)^{c} (\theta_2 - r_2)^b.
    \end{align*}
    Recall that $a + b + c \leq - 2$ and $a \in (-1,0)$, we have $b + c < -1$. As
    a consequence, we can find $b' \in (-2, - 1)$ and $c' < 0$ such that $b' \leq
  b$, $c' \geq c$ and $b' + c' = b + c$. It follows that
    \begin{align*}
      (\theta_1 - r_2)^{c} (\theta_2 - r_2)^b \leq (\theta_1 - r_2)^{c'} (\theta_2 - r_2)^{b'},
    \end{align*}
    since $\theta_2 \leq \theta_1$. This implies that
    \begin{align*}
      I'_{2,2} \lesssim & \int_{s}^{t} \ud \theta_1 \; (\theta_1 - s)^{c'} \int_{s}^{\theta_1} \ud \theta_2 \; (\theta_1 - \theta_2)^{a + 1} (\theta_2 - s)^{b' + 1} \approx (t - s)^{a + b + c + 4}.
    \end{align*}
    This completes the estimate for $I_2'$. The proof of Lemma~\ref{L:lmm_tu} is
  complete.
\end{proof}

\begin{proof}[Proof of Lemma~\ref{L:recF}]
  The proof of this lemma is divided into the following two cases.

  \bigskip\noindent\textbf{Case~1 ($b \neq -1$):~} Using Lemma~\ref{L:F_tht}, we
  can deduce that
  \begin{gather}\label{E_:I_b-ne-1}
    I \approx I_1 + I_2 \coloneqq \int_0^s \ud r_2 \int_0^{r_2} \ud r_1 \; (r_2 - r_1)^{a} \big(\left|J_1 (r_1, r_2)\right| + \left|J_2(r_1, r_2)\right| \big),
    \shortintertext{where}
    \begin{aligned}
      J_1 (r_1, r_2) \coloneqq & |2t - r_1 - r_2|^{-1-b}
      - 2|t + s - r_1 - r_2|^{-1-b} + |2s - r_1 - r_2|^{-1-b}.\nonumber
    \end{aligned}
    \shortintertext{and}
    \begin{aligned}
      J_2 (r_1, r_2) \coloneqq & - 2|r_2 - r_1|^{-1-b} + |t - s - r_1 + r_2|^{-1-b} + |t - s + r_1 - r_2|^{-1-b}.\nonumber
    \end{aligned}
  \end{gather}
  In the following, we only provide the estimate for $I_2$, that for $I_1$ can
  be done using similar arguments. Decompose the integration region as follows,
  \begin{align*}
    I_2 \approx & I_{2,1} + I_{2,2} + I_{2,3}                                                                                                                                                                                     \\
    \coloneqq   & \left(\int_{t - s}^s \ud r_2 \int_0^{r_2 - t + s} \ud r_1 + \int_0^{t - s} \ud r_2 \int_{0}^{r_2} \ud r_1 + \int_{t - s}^{s} \ud r_2 \int_{r_2 - t + s}^{r_2} \ud r_1 \right) (r_2 - r_1)^{a} |J_2 (r_1, r_2)|.
  \end{align*}
  On the set $\{(r_1, r_2) \in \R_+^2 \colon r_1 \leq r_2 - t + s, t - s \leq
    r_2 \leq s\}$, thanks to the fundamental theorem of calculus, we have
  \begin{align*}
    |J_2 (r_1, r_2)|
    \approx & \left|\int_{r_2 - r_1}^{r_2 - r_1 + t - s} \ud \theta \; \theta^{-2-b} - \int_{r_2 - r_1 - (t - s)}^{r_2 - r_1} \ud \theta \; \theta^{-2-b} \right| \\
    =       & \int_{0}^{t - s} \ud \theta \; \left|(\theta + r_2 - r_1)^{-2 - b} - \big(\theta + r_2 - r_1 - (t - s) \big)^{-2-b}\right|                          \\
    \approx & \int_{0}^{t - s} \ud \theta_1 \int_{0}^{t - s} \ud \theta_1 \; \big(\theta_1 + \theta_2 + r_2 - r_1 - (t - s) \big)^{-3-b}.
  \end{align*}
  It follows from a change of variables $(v_1,v_2) = (r_2 - r_1, r_2)$ that
  \begin{align*}
    I_{2,1} \approx \int_{0}^{t - s} \ud \theta_1 \int_{0}^{t - s} \ud \theta_1 \; \int_{t-s}^s \ud v_2 \int_{t - s}^{v_2} \ud v_1 \; v_1^a \big(\theta_1 + \theta_2 + v_1 - (t - s) \big)^{-3-b}.
  \end{align*}
  Recalling that $a < 0$ and $b > -3$, we can write
  \begin{align*}
    v_1^a \big(\theta_1 + \theta_2 + v_1 - (t - s) \big)^{-3-b}
    \leq & \left[ v_1 \wedge \big(\theta_1 + \theta_2 + v_1 - (t - s) \big) \right]^{-3-b} \\
    \leq & v_1^{a - b - 3} + \big(\theta_1 + \theta_2 + v_1 - (t - s) \big)^{a - b - 3}.
  \end{align*}
  Therefore, one can easily deduce that
  \begin{align*}
    I_{2,1} \lesssim \begin{dcases}
                       (t-s)^{2\wedge (a-b)},                     & a-b \neq 2, \\
                       (t - s)^2 \left(1 + |\log(t - s)| \right), & a-b = 2.
                     \end{dcases}
  \end{align*}

  Additionally, for $I_{2,2}$, notice the integration region in $r_1$ and $r_2$
  is a subset of $(t-s)^2$. Thus, we can simply bound $J_2(r_1, r_2)$ by the
  summation of the absolute value of each summands, where the integrability of
  each term is ensured by the conditions $-1 < a < 0$, and $-3 < b < a$. Then,
  it is straightforward to see that each term is bounded by $(t-s)^{a - b + 1}$
  up to a constant. For example,
  \begin{align*}
    \int_0^{t - s} \ud r_2  \int_{0}^{r_2} \ud r_1 \; (r_2 - r_1)^a |t - s + r_1 - r_2|^{-1 - b}
    =        & \int_0^{t - s} \ud u_2 \int_{0}^{u_2} \ud u_1 \; u_1^a (t - s - u_1)^{-1 - b} \\
    \lesssim & (t - s)^{a - b + 1}.
  \end{align*}
  The estimates for other terms can be done in a similar way, and are thus
  skipped.

  The estimate for $I_{2,3}$ is quite similar, using the fact that $r_2 - r_1 <
    t - s$, one can also deduce that $I_{2,3} \lesssim (t - s)^{a - b}$ by simply
  bound $J_2(r_1, r_2)$ by the summation of the absolute value of each summands.

  \bigskip\noindent\textbf{Case~2 ($b = -1$):~} The proof for the case when $b =
    -1$ closely parallels the one for the case where $b \neq -1.$ Consequently, we
  will provide a brief outline of the proof without delving into the specific
  details. First, we decompose $I$ into an analogous formulation
  to~\eqref{E_:I_b-ne-1}. Instead of the power functions, in this case $J_1$ and
  $J_2$ are summation/subtraction of logarithmic functions. Note that in the
  estimate for the counterpart of $I_{2,1}$, we apply the fundamental theorem of
  calculus, which is still applicable in this case. Then, for the parallel term
  of $I_{2,2}$, one can use the inequality $\lim_{0< x < \infty} x^{\epsilon}
    \log (x) < \infty$ for any fixed $\epsilon > 0$. So one can replace the
  logarithmic functions by $x^{\epsilon}$. Note that in the previous case, we
  find $I_{2,2} \lesssim (t - s)^{a - b + 1}$. Thus, such replacement will not
  affect the final result. Now it suffices to estimate the correspondent term of
  $I_{2,3}$, which can be approximated as follows
  \begin{align*}
    I_{2,3}' \approx \int_{t - s}^{s} \ud r_2 \int_{r_2 - t + s}^{r_2} \ud r_1 \; (r_2 - r_1)^a \left|\log\left(\frac{t - s}{r_2 - r_1} + 1\right) + \log \left(\frac{t - s}{r_2 - r_1} - 1\right) \right|.
  \end{align*}
  Performing a change of variables $(v_1, v_2) = ((t-s)/(r_2 - r_1), r_2)$, we
  can write
  \begin{align*}
    I_{2,3}' \approx (t - s)^{a - b} (2s - t) \int_1^{\infty} \ud v_1 \; v_1^{- a - 2} \left( \log(v_1 + 1) + \log(v_1 - 1) \right).
  \end{align*}
  Recall that $a \in (-1, 0)$, the integral in the above expression is finite.
  The proof of Lemma~\ref{L:recF} is complete.
\end{proof}

\subsection{Proof of Proposition~\ref{T:Upper-s}}\label{S:Upper-s}

Without loss of generality, assume that $|x - y| \ll 1$.

\bigskip\noindent\textit{(i)} To prove~\eqref{E:ups}, just like the proof of
part (i) of Theorem~\ref{T:Dalang}, it suffices to show this inequality assuming
$H > 1/2$. Applying the Hardy--Littlewood--Sobolev inequality, we have
\begin{align*}
  \E\left[\left(u(t,x) - u(t,y)\right)^2\right]
  \lesssim & \int_{\R^d} \ud \xi\; |\xi|^{\ell - d} \left|e^{-i x \cdot \xi} - e^{- i y \cdot \xi}\right|^2                                                          \\
           & \times \left(\int_0^t \ud s\, s^{(\beta+\gamma-1)/H}\left|E_{\beta, \beta+\gamma} \left(-2^{-1}\nu |\xi|^\alpha s^\beta\right)\right|^{1/H}\right)^{2H} \\
  \lesssim &
  \int_{\R^d} \ud \xi\; \frac{|\xi|^{\ell - d}\left(1 - \cos ((x - y) \cdot \xi)\right)}{1+ |\xi|^{2\min\left(\beta, \beta+\gamma + H -1\right) \times \alpha/\beta}}.
\end{align*}
% In the following, we only prove \eqref{E:ups}, namely, the case when $\beta < 2$, or $\beta = 2$ and $\gamma \geq 1$, because the other case can be deduced similarly. 
Using the fact that $1 - \cos(x) \leq 2 \wedge (x^2/ 2)$ for all $x \in \R$, we
deduce that
\begin{align*}
  \E\left[\left(u(t,x) - u(t,y)\right)^2\right]
  \lesssim & \int_{0}^{|x - y|^{-1}} \ud z\; \frac{z^{\ell + 1} |x - y|^2 }{(1 + z)^{2\min\left(\beta, \beta+\gamma + H -1\right) \times \alpha/\beta}} \\
           & + \int_{|x - y|^{-1}}^{\infty} \ud z\; z^{\ell - 1 - 2\min\left(\beta, \beta+\gamma + H -1\right) \times \alpha/\beta}.
\end{align*}
Following the same idea as in the proof of~\cite[Proposition
  5.4]{chen.hu.ea:19:nonlinear}, we have that
\begin{align*}
  \E\left[\left(u(t,x) - u(t,y)\right)^2\right]
  \lesssim & |x - y|^{\Theta - \ell} + |x - y|^{-\ell} \lMr{2}{F}{1} (2 + \ell, \Theta; 2 + \ell; -|x - y|^{-1})                                                   \\
  =        & |x - y|^{\Theta - \ell} + C_{\ell,\Theta} |x - y|^{-\ell}  \FoxH{1,2}{2,2}{|x - y|^{-1}}{(- \ell - 1, 1),\:( 1 - \Theta ,1)}{(0,1),\:(- \ell - 2,1)},
\end{align*}
where the last equality is due to~\eqref{E:F=H} and $\Theta \coloneqq
  2\min\left(\beta, \beta+\gamma + H -1\right) \times \alpha/\beta$. Now when
\begin{align*}
  -\ell -1 \ne  1- \Theta \quad \Longleftrightarrow \quad \Theta -\ell \ne 2,
\end{align*}
we can apply the series expansion of the Fox $H$-function as given
in~\cite[Theorem~1.7]{kilbas.saigo:04:h-transforms} to obtain the first and the
third cases in~\eqref{E:ups}. Otherwise, when $\Theta-\ell=2$, one can apply the
series expansion of the hypergeometric function as given
in~\cite[Eq.~15.8.8]{olver.lozier.ea:10:nist} to obtain the second case
in~\eqref{E:ups}. % This completes the proof of~\eqref{E:ups}. The proof for~\eqref{E:ups'} is similar and thus omitted.

\bigskip\noindent\textit{(ii)} Next, suppose $\beta = 2$ and $\gamma \in [0,1)$. By
using the space Fourier transformation, and noting the fact that
\begin{align*}
  \left|1 - e^{i (x - y) \cdot \xi}\right| \lesssim |(x - y)\cdot \xi| \wedge 1 \leq \big(|x - y| \times |\xi| \big) \wedge 1,
\end{align*}
we can write
\begin{align*}
  \E & \left[\left(u(t,x) - u(t,y)\right)^2 \right] \approx
  \int_{\R^d} \ud \xi\; |\xi|^{\ell - d}  \left|1 - e^{i (x - y) \cdot \xi}\right|^2 J(t, |\xi|) \\
     & \quad \lesssim \int_0^1 \ud z\; z^{\ell - 1} (|x-y|z)^2 J(t,z)
  + \int_1^{|x-y|^{-1}} \ud z\; z^{\ell - 1} (|x-y|z)^2 J(t,z)
  + \int_{|x-y|^{-1}}^{\infty} \ud z\; z^{\ell - 1} J(t,z)                                       \\
     & \quad \eqqcolon I_1 + I_2 + I_3.
\end{align*}
where
\begin{align*}
  J(t, x) \coloneqq \int_{0}^t \ud s_2 \int_0^{s_2} \ud s_1 \; |s_2 - s_1|^{2H - 2}
  s_1^{\gamma + 1} \left|E_{2, 2 + \gamma} \left(-\nu s_1^2 x^{\alpha}/2\right)\right|
  s_2^{\gamma + 1} \left|E_{2, 2 + \gamma} \left(-\nu s_2^2 x^{\alpha}/2\right)\right|,
\end{align*}
for all $(t,x) \in \R_+^2$.
Applying Lemma~\ref{L:ML} (iii), in particular~\eqref{E:ML-UB-1}, we conclude that
\begin{align}\label{E_:Upper-s-beta=2-1}
  I_1 \lesssim |x - y|^2 \int_0^1 \ud z \; z^{\ell + 1}  \int_{0}^t \ud s_2 \int_0^{s_2} \ud s_1 \; |s_2 - s_1|^{2H - 2}   s_1^{\gamma + 1} s_2^{\gamma + 1}  \lesssim |x - y|^2.
\end{align}
Additionally, due to \eqref{E:asym_MT2-1}, it holds that
\begin{align*}
  I_2 \lesssim & I_{2,1} + I_{2,2}
\end{align*}
where, using the time Fourier representation, with notation $N$ defined as
in~\eqref{E:def_nt},
\begin{align*}
  I_{2,1}  \coloneqq & |x - y|^2 \int_1^{|x - y|^{-1}} \ud z \; z^{\ell - \alpha \gamma + 1}  N_{\alpha, -\frac{\gamma}{2}\pi} \left(t, (\nu/2)^{1/\alpha} z\right),
\end{align*}
and recalling \eqref{E:main},
\begin{align*}
  I_{2,2} \coloneqq & |x - y|^2 \int_1^{|x - y|^{-1}} \ud z \; z^{\ell - 2\alpha + 1} \int_{0}^t \ud s_2 \int_0^{s_2} \ud s_1 \; |s_2 - s_1|^{2H - 2}   s_1^{\gamma - 1} s_2^{\gamma - 1} \\
  \lesssim          &
  \begin{dcases}
    |x - y|^{2\alpha - \ell},                     & \ell > 2\alpha - 2,  \\
    |x - y|^2 \left(1 + |\log(|x - y|) | \right), & \ell = 2 \alpha - 2, \\
    |x - y|^2,                                    & \ell < 2 \alpha - 2.
  \end{dcases}.
\end{align*}
Moreover, Balan's lemma (Lemma~\ref{L:balan})
yields that for all $z \geq 1$,
\begin{align*}
  N_{\alpha, -\frac{\gamma}{2}\pi} \left(t, (\nu/2)^{1/\alpha} z\right)
  \lesssim & \frac{t^2}{\sqrt{ z^{\alpha} + 1}} \int_{\R} \ud \tau \; \frac{|\tau/t|^{1 - 2H} }{\tau^2 + t^2 z^{\alpha} + t^2}
  \lesssim  \frac{1}{z^{\alpha/2} + 1} \int_{\R} \ud \tau \; \frac{|\tau|^{1 - 2H} }{\tau^2 + z^{\alpha} + 1}                                           \\
  =        & (z^{\alpha/2} + 1)^{-1} (z^{\alpha} + 1)^{-H} \int_0^{\infty} \ud u \; \frac{u^{2H - 1}}{u^2 + 1} \lesssim z^{-\alpha (H + 1/2)} \wedge 1.
\end{align*}
Thus, we can further deduce that
\begin{align*}
  I_{2,1}
  \lesssim & |x - y|^2 \int_{1}^{|x - y|^{- 1}} \ud z \; z^{\ell - \alpha (\gamma + H + 1/2) + 1}             \\
  \lesssim & \begin{dcases}
               |x - y|^{\alpha (H + \gamma + 1/2) - \ell },  & \ell  > \alpha (H + \gamma + 1/2) - 2,  \\
               |x - y|^2 \left(1 + |\log (|x - y|) |\right), & \ell  = \alpha (H + \gamma + 1/2)  - 2, \\
               |x - y|^2,                                    & \ell  < \alpha (H + \gamma + 1/2)  - 2.
             \end{dcases}
\end{align*}
As a result,
\begin{align}
  I_2 \lesssim  \begin{dcases}
                  |x-y|^{2 \widetilde{\rho}_2 },              & \widetilde{\rho}_2 < 1, \\
                  |x-y|^2 \left(1 + |\log (|x - y|) |\right), & \widetilde{\rho}_2 = 1, \\
                  |x-y|^2 ,                                   & \widetilde{\rho}_2 > 1.
                \end{dcases}
\end{align}
The estimate for $I_3$ is almost identical to that for $I_2$, with very minor adjustments. Therefore, we omit the details and conclude that
\begin{align}\label{E_:Upper-s-beta=2-3}
  I_3 \lesssim  \int_{|x - y|^{-1}}^{\infty} \ud z \; z^{\ell - \alpha (\gamma + H + 1/2) - 1}
  + \int_{|x - y|^{-1}}^{\infty} \ud z \; z^{\ell - 2\alpha - 1}
  \lesssim  |x - y|^{2 \widetilde{\rho}_2}.
\end{align}
A combination of \eqref{E_:Upper-s-beta=2-1}--\eqref{E_:Upper-s-beta=2-3} confirms~\eqref{E:ups1}, thereby completing the proof of Proposition~\ref{T:Upper-s}. \qed

\section{Strong local nondeterminism in $t$ and moment lower bounds}\label{S:SLN_t}
\index{strong local nondeterminism (SLND)}\index{strong local nondeterminism (SLND)!temporal}

In this section, we establish strong local nondeterminism property for the temporal process
$\{u(t, x): t \ge 0\}$, where $x \in \R^d$ is fixed, and lower bounds for
the variances of the temporal increments $u(t, x) - u(s, x)$.
Recall the definition of $\rho$ in \eqref{E:Dalang}.

\subsection*{Main results}
\addcontentsline{toc}{subsection}{Main results}

For ease of reference, we first state the main result of this section as
Theorem~\ref{T:LND_t}. Its proof is divided into
Subsections~\ref{SS:LNDt-i}--\ref{SS:LNDt-iii}, corresponding to parts (i)--(iii),
respectively.

\begin{theorem}\label{T:LND_t}%
  \index{conditional variance}%
  Assume $\alpha > 0$, $\beta \in (0, 2]$, $\gamma \ge 0$,
  $H \in (0, 1)$, and $\ell \in (0, 2d)$. 
    Assume the Dalang condition~\eqref{E:main}, i.e., $\rho>0$ and $\ell < 2\alpha$. 
    Then, for all $0 < S < T < \infty$,
  there exist positive finite constants $C$ and $C'$ depending on $S, T, \alpha,
    \beta, \gamma, H, d, \ell$ such that the following statements hold.
  \begin{enumerate}[\rm (i)]

    \item If $\beta \in (0,2)$,\index{strong local nondeterminism (SLND)!one-sided} then for all $x \in \R^d$, for all integers $n
            \ge 1$, and for all $t, t_1, \dots, t_n \in [S, T]$ with $\max\{ t_1,
            \dots, t_n\} \le t$,
          \begin{align}\label{LNDt:CondVar}
            \Var\left( u(t, x)\big|u(t_1, x), \dots, u(t_n, x) \right) \ge C \min_{1 \le j\le n} (t-t_j)^{2\rho}.
          \end{align}
          Consequently, for all $t, t' \in [S, T]$, for all $x \in \R^d$,
          \begin{align}\label{LB:Tinc}
            \E \left[ \left(u(t, x) - u(t', x)\right)^2 \right] \ge C |t'-t|^{2\rho}.
          \end{align}
          In addition, there exists $\delta > 0$ such that for all $t, t' \in [S, T]$ with
          $|t'-t| \le \delta$, for all $x \in \R^d$,
          \begin{align}\label{LB:Tinc_log}
            \E \left[ \left(u(t', x) - u(t, x)\right)^2 \right] \ge \begin{cases}
                                                                      C' |t'-t|^2 (1+|\log(t'-t)|) & \text{if } \rho = 1, \\
                                                                      C' |t'-t|^2                  & \text{if } \rho > 1.
                                                                    \end{cases}
          \end{align}

    \item If $\beta = 1$, then \eqref{LNDt:CondVar} can be strengthened\index{strong local nondeterminism (SLND)!two-sided} as
          follows: for all $x \in \R^d$, for all integers $n \ge 1$, and for all $t,
            t_1, \dots, t_n \in [S, T]$,
          \begin{align}\label{LNDt:beta=1}
            \Var\left( u(t, x)\big|u(t_1, x), \dots, u(t_n, x) \right) \ge C \min_{1 \le j\le n} |t-t_j|^{2\rho}.
          \end{align}

    \item If $\beta = 2$, $\gamma \in [0, 1]$, and $H \in [1/2, 1)$, then
          for all $x \in \R^d$, for all integers $n \ge 1$, and for all $t, t_1,
            \dots, t_n \in [S, T]$ with $\max_j |t-t_j| \le S$,
          \begin{align}\label{LNDt:CondVar2}
            \Var\left( u(t, x) \big| u(t_1, x), \dots, u(t_n, x) \right) \ge C \min_{1\le j \le n}|t-t_j|^{2(\gamma+H+\frac{1}{2}-\frac{\ell}{\alpha})}.
          \end{align}
          and for all $t, t' \in [S, T]$, for all $x \in \R^d$,
          \begin{align}\label{LB:Tinc2}
            \E\left[ \left( u(t', x) - u(t, x) \right)^2 \right] \ge \begin{cases}
                                                                       C' |t'-t|^{2(\gamma+H+\frac{1}{2}-\frac{\ell}{\alpha})} & \text{if } \gamma + H + \frac12 - \frac\ell\alpha < 1, \\
                                                                       C'|t'-t|^2 (1+|\log(t'-t)|)                             & \text{if } \gamma + H + \frac12 - \frac\ell\alpha = 1, \\
                                                                       C' |t'-t|^2                                             & \text{if } \gamma + H + \frac12 -\frac\ell\alpha > 1.
                                                                     \end{cases}
          \end{align}
  \end{enumerate}
\end{theorem}

%\begin{remark}
%  In case when $\beta=2$, $\alpha=2$, and $\gamma=0$, Lee~\cite{lee:22:local}
%  showed that $\rho = H + \frac{1}{2}(1-\ell)$. In contrast, when formally
%  setting $\beta=2$, $\alpha=2$, and $\gamma=0$ in~\eqref{E:Dalang}, we have
%  $\rho = H + 1 -\ell/2$. There is a property jump at $\beta=2$, which is also
%  observed in~\cite{chen.hu:22:holder} \textcolor{magenta}{(more to explain)}.
%\end{remark}

\subsection{Proof of part (i) of Theorem \ref{T:LND_t}}\label{SS:LNDt-i}

\begin{proof}[Proof of part (i) of Theorem \ref{T:LND_t}]
  Suppose $\beta \in (0,2)$.
  Since $u$ is Gaussian, the conditional variance in~\eqref{LNDt:CondVar} is
  equal to the squared distance from $u(t, x)$ to the linear subspace spanned by
  $u(t_1, x), \dots, u(t_n, x)$ in $L^2(\mathbb{P})$ (see, e.g.,
  \cite[Lemma~A.4]{kalbasi.mountford:20:on}), namely,
  \begin{align*}
    \Var \left(u(t, x)\big|u(t_1, x), \dots, u(t_n, x) \right)
    = \inf_{a_1, \dots, a_n \in \R} \E\left[ \left( u(t, x) - \sum_{j=1}^n a_j u(t_j, x)\right)^2 \right].
  \end{align*}
  Hence, in order to prove \eqref{LNDt:CondVar}, it suffices to prove that there
  exists a positive finite constant $C$ such that for all $x \in \R^d$, for all
  $n \ge 1$, for all $t, t_1, \dots, t_n \in [S, T]$ with $\max\{t_1, \dots, t_n
    \} \le t$, and for all $a_1, \dots, a_n \in \R$,
  \begin{align}\label{LNDt:LB}
    \Theta \coloneqq \E\left[ \left( u(t, x) - \sum_{j=1}^n a_j u(t_j, x)\right)^2 \right] \ge C r^{2\rho},
    \quad \text{where } r \coloneqq \min_{1 \le j \le n}(t-t_j).
  \end{align}
  If $r = 0$, there is nothing to prove. So, we assume that $r > 0$. Then, by \eqref{E:Inner_1}, \eqref{E:u_corr}, and \eqref{E:FG},
  \begin{align}\label{E:LNDt_Theta}
    \begin{split}
      \Theta & = C_{H,d}\int_{\R} \ud \tau\; |\tau|^{1-2H} \int_{\R^d} \ud \xi\; |\xi|^{\ell - d}                                                                                                                                                            \\
             & \quad \times \left|\int_0^\infty \ud s \left[ e^{-ix\cdot\xi-it\tau}e^{i\tau s} s^{\beta+\gamma-1} E_{\beta, \beta+\gamma}\left(-2^{-1}\nu |\xi|^\alpha s^\beta\right) \one_{[0, t]}(s)\right.\right.                                         \\
             & \qquad \left.\left. - \sum_{j=1}^n a_j e^{-ix\cdot\xi-it_j\tau}e^{i\tau s}      s^{\beta+\gamma-1} E_{\beta, \beta+\gamma}\left(-2^{-1}\nu |\xi|^\alpha s^\beta\right) \one_{[0, t_{j}]}(s)\right]\right|^2                                   \\
             & = C_{H, d}\int_{\R} \ud \tau\; |\tau|^{1-2H} \int_{\R^d} \ud \xi\; |\xi|^{\ell - d}                                                                                                                                                           \\
             & \quad \times\left| \int_0^\infty \ud s \, e^{i\tau s} s^{\beta+\gamma-1}E_{\beta, \beta+\gamma}\left(-2^{-1}\nu|\xi|^\alpha s^\beta\right) \left[ \one_{[0, t]}(s) - \sum_{j=1}^n a_j e^{-i(t_j-t)\tau} \one_{[0, t_j]}(s) \right] \right|^2,
    \end{split}
  \end{align}
  where
  \begin{align*}
    C_{H,d} \coloneqq a_H 2^{2(H-1)} \frac{\Gamma(H-1/2)}{\Gamma(1-H)\sqrt{\pi}} (2\pi)^{-d} \quad \text{and} \quad
    a_H     \coloneqq H(2H-1).
  \end{align*}
  Let $f:\R \to \R$ and $g:\R^d \to \R$ be two test functions to be determined.
  For any $r > 0$, we define $f_r(s) = r^{-1}f(r^{-1}s)$ and $g_r(y) =
    r^{-d\beta/\alpha}g(r^{-\beta/\alpha}y)$. Consider the integral
  \begin{align}\begin{aligned}\label{E:LNDtI}
      I & := \int_{\R} \ud \tau \int_{\R^d} \ud \xi\, \widehat{f_r}(\tau) \widehat{g_r}(\xi)                                                                                                                                           \\
        & \quad \times \int_0^\infty \ud s\, e^{i\tau s} s^{\beta+\gamma-1}E_{\beta, \beta+\gamma}\left(-2^{-1}\nu|\xi|^\alpha s^\beta\right) \left[ \one_{[0, t]}(s) - \sum_{j=1}^n a_j e^{-i(t_j-t)\tau} \one_{[0, t_j]}(s) \right].
    \end{aligned}\end{align}
  Then, by inverse Fourier transform in $\tau$,
  \begin{align*}
    I & = 2\pi \int_{\R^d} \ud \xi \, \widehat{g_r}(\xi)                                                                         \\
      & \quad \times \int_0^\infty \ud s\, s^{\beta+\gamma-1} E_{\beta, \beta+\gamma}\left(-2^{-1}\nu|\xi|^\alpha s^\beta\right)
    \left[f_r(s)\one_{[0, t]}(s) - \sum_{j=1}^n a_j f_r(s-t_j+t)\one_{[0, t_j]}(s)\right].
  \end{align*}
  Thanks to~\eqref{E:Gscaling} and
  Theorem~\ref{T:Zero} (see also Remark \ref{rmk:G:+ve}), we can find $0<\delta_0 \le \min\{1, T^{-1}S\}$ and some neighborhood $U \subset \R^d$ of the
  origin such that
  \begin{align}\label{E:Gpositive}
    G \text{ is strictly positive or strictly negative on } (0, \delta_0) \times U.
  \end{align}
  Now choose and fix $f$ to be a smooth nonnegative function that is supported
  on $[0, \delta_0]$ and strictly positive on $(0, \delta_0)$, and $g$ to be a
  smooth nonnegative function that is supported on $\overline{U}$ and strictly
  positive on $U$. Since $r^{-1}(s-t_j+t) \ge r^{-1}(t-t_j) \ge 1 \ge \delta_0$ for all $s \in [0, t_j]$,
  we have $f_r(s-t_j+t) = 0$. Hence,
  \begin{align*}
    I = 2\pi \int_{\R^d} \ud \xi \, \widehat{g_r}(\xi) \int_0^t \ud s\, s^{\beta+\gamma-1} E_{\beta, \beta+\gamma}\left(-2^{-1}\nu|\xi|^\alpha s^\beta\right) f_r(s).
  \end{align*}
  Note that $\widehat{g_r}(\xi) = \widehat{g}(r^{\beta/\alpha}\xi)$ and the change of variable $s \mapsto rs$
  together imply
  \begin{align*}
    I = 2\pi r^{\beta+\gamma-1} \int_{\R^d} \ud \xi \, \widehat{g}(r^{\beta/\alpha}\xi) \int_0^{r^{-1}t} \ud s\, s^{\beta+\gamma-1} E_{\beta, \beta+\gamma}\left(-2^{-1}\nu|\xi|^\alpha r^\beta s^\beta\right) f(s).
  \end{align*}
  Then, we can use the change of variable $\xi \mapsto r^{-\beta/\alpha}\xi$ to
  get that
  \begin{align*}
    I = 2\pi r^{\beta+\gamma-1-d\beta/\alpha} \int_{\R^d} \ud\xi\, \widehat{g}(\xi) \int_0^{r^{-1}t} \ud s\, s^{\beta+\gamma-1} E_{\beta, \beta+\gamma} \left( -2^{-1}\nu |\xi|^\alpha s^\beta \right) f(s).
  \end{align*}
  Applying Plancherel's theorem to the $\ud \xi$-integral yields
  \begin{align*}
    I = (2\pi)^{1+d}r^{\beta+\gamma-1-d\beta/\alpha} \int_0^{r^{-1}t} \ud s \, f(s) \int_{\R^d} \ud y\, g(y) G(s, y).
  \end{align*}
  Thanks to \eqref{E:Gpositive} and the choice of the test functions, the integrand is either nonnegative or nonpositive, and since $r^{-1}t \ge T^{-1}S \ge \delta_0$, we have
  \begin{align}\label{LNDt:I}
    |I| \ge C_1 r^{\beta+\gamma-1-d\beta/\alpha},
  \end{align}
  where
  \begin{align*}
    C_1 \coloneqq (2\pi)^{1+d} \left|\int_0^{\delta_0} \ud s \, f(s) \int_{U} \ud y\, g(y) G(s, y)\right|
  \end{align*}
  is a positive constant. \bigskip

  On the other hand, applying the Cauchy--Schwarz inequality to the integral
  \eqref{E:LNDtI} and then using \eqref{E:LNDt_Theta}, we get that
  \begin{align*}
    |I|^2
    \le & \left(\int_{\R} \ud \tau \int_{\R^d} \ud \xi \; \left|\widehat{f_r}(\tau)\right|^2 \left|\widehat{g_r}(\xi)\right|^2 |\tau|^{2H-1} |\xi|^{d-\ell} \right)                                                                             \\
        & \times \int_\R \ud \tau \, |\tau|^{1-2H} \int_{\R^d} \ud \xi \, |\xi|^{\ell-d}                                                                                                                                                        \\
        & \times \left|\int_0^\infty \ud s \, e^{i\tau s} s^{\beta+\gamma-1} E_{\beta, \beta+\gamma}\left(-2^{-1}\nu|\xi|^\alpha s^\beta\right) \left[\one_{[0, t]}(s) - \sum_{j=1}^n a_j e^{-i(t_j-t)\tau} \one_{[0, t_j]}(s) \right]\right|^2 \\
    =   & C_{H, d}^{-1}\; \Theta \times \int_{\R} \ud \tau \int_{\R^d} \ud \xi \; \left|\widehat{f}(r\tau)\right|^2 \left|\widehat{g}(r^{\beta/\alpha}\xi)\right|^2 |\tau|^{2H-1} |\xi|^{d-\ell}.
  \end{align*}
  Then, by the change of variables $\tau \mapsto r^{-1} \tau$ and $\xi \mapsto
    r^{-\beta/\alpha} \xi$, we have
  \begin{align}\label{LNDt:I2}
    |I|^2 \le C_{H, d}^{-1}\; \Theta \int_{\R} \ud \tau \int_{\R^d} \ud \xi \, \left|\widehat{f}(r\tau)\right|^2 \left|\widehat{g}(r^{\beta/\alpha}\xi)\right|^2 |\tau|^{2H-1} |\xi|^{d-\ell}
    = C_2 r^{-2H-(2d-\ell)\beta/\alpha} \Theta,
  \end{align}
  where
  \begin{align*}
    C_2 \coloneqq C_{H, d}^{-1} \int_{\R} \ud \tau \int_{\R^d} \ud \xi \, \left|\widehat{f}(\tau)\right|^2 \left|\widehat{g}(\xi)\right|^2 |\tau|^{2H-1} |\xi|^{d-\ell}.
  \end{align*}
  Note that $C_2$ is finite since $2H-1 >-1$, $d-\ell > -d$, and
  $\widehat{f}$ and $\widehat{g}$ are rapidly decreasing functions. Finally, we
  can combine~\eqref{LNDt:I} and~\eqref{LNDt:I2} to conclude that
  \begin{align*}
    \Theta \ge C_1^2 C_2^{-1} r^{2(\beta+\gamma-1+H)-\ell\beta/\alpha}
    =  C_1^2 C_2^{-1} r^{2\rho},
  \end{align*}
  where $C_1^2C_2^{-1}$ is a positive constant. This proves
  \eqref{LNDt:CondVar}. Consequently, \eqref{LB:Tinc} follows
  from~\eqref{LNDt:LB} by taking $n = 1$ and $a_1 = 1$.\bigskip

  Now, we turn to the proof of \eqref{LB:Tinc_log}. Suppose $\rho \ge 1$. For
  any $t,t'>0$ and $x\in \R^d$, we can write
  \begin{align}\begin{split}\label{E:LB_t_log}
      \E\left[ \left( u(t',x)-u(t,x)\right)^2\right]
       & = C_{H,d} \int_\R \ud \tau \, |\tau|^{1-2H} \int_{\R^d} \ud \xi \, |\xi|^{\ell-d} |A_{t'}(\tau,\xi) - A_t(\tau,\xi)|^2       \\
       & = 2C_{H,d} \int_0^\infty \ud \tau \, \tau^{1-2H} \int_{\R^d} \ud \xi \, |\xi|^{\ell-d} |A_{t'}(\tau,\xi) - A_t(\tau,\xi)|^2,
    \end{split}\end{align}
  where
  \begin{align*}
    A_t(\tau,\xi) = e^{-i\tau t} \int_0^{t} e^{i\tau s} s^{\beta+\gamma-1} E_{\beta,\beta+\gamma}\left( - \frac{\nu}{2}|\xi|^\alpha s^\beta \right) \ud s.
  \end{align*}
  Let $\phi: \R^d\to \R$ be a compactly supported smooth test function to be
  determined, and let
  \begin{align*}
    \phi_\tau(y) = \tau^{d\beta/\alpha}\phi(\tau^{\beta/\alpha}y).
  \end{align*}
  Consider
  \begin{align*}
    K(\tau) := \int_{\R^d} \ud \xi \, \widehat{\phi_\tau}(\xi) \left[ A_{t'}(\tau,\xi) - A_t(\tau,\xi) \right].
  \end{align*}
  By Plancherel's theorem and \eqref{E:FG}, followed by the change of variables $s \mapsto \tau^{-1}
    s$, $y \mapsto \tau^{-\beta/\alpha}y$ and the scaling property
  \eqref{E:Gscaling} of $G$, we see that
  \begin{align}\label{E:LB:K}
     & K(\tau) = (2\pi)^d \left[ e^{-i\tau t'} \int_0^{t'} \ud s \, e^{i\tau s} \int_{\R^d} \ud y \, \phi_\tau(y) G(s,y)
    - e^{-i\tau t} \int_0^{t} \ud s \, e^{i\tau s} \int_{\R^d} \ud y \, \phi_\tau(y) G(s,y)\right]                                              \\
     & = (2\pi)^d \tau^{-\beta-\gamma+d\beta/\alpha} \left[ e^{-i\tau t'} \int_0^{\tau t'} \ud s \, e^{i s} \int_{\R^d} \ud y \, \phi(y) G(s,y)
      - e^{-i\tau t} \int_0^{\tau t} \ud s \, e^{i s} \int_{\R^d} \ud y \, \phi(y) G(s,y)\right]. \nonumber
  \end{align}
  On the other hand, by Cauchy--Schwarz inequality and scaling,
  \begin{align}\begin{split}\label{E:LB:K2}
      |K(\tau)|^2 & \le \int_{\R^d} |\xi|^{\ell-d} |A_{t'}(\tau,\xi)-A_t(\tau,\xi)|^2 \ud \xi \times \int_{\R^d} |\xi|^{-\ell+d} \left| \widehat{\phi}(\tau^{-\beta/\alpha}\xi) \right|^2 \ud \xi \\
                  & \lesssim \int_{\R^d} |\xi|^{\ell-d} |A_{t'}(\tau,\xi)-A_t(\tau,\xi)|^2 \ud \xi \times \tau^{2d\beta/\alpha - \ell\beta/\alpha}.
    \end{split}\end{align}
  Combining \eqref{E:LB:K} and \eqref{E:LB:K2} yields
  \begin{align*}
     & \int_{\R^d} |\xi|^{\ell-d} |A_{t'}(\tau,\xi)-A_t(\tau,\xi)|^2 \ud \xi                                                                                          \\
     & \gtrsim \tau^{-2\left(\beta+\gamma-\frac{\ell\beta}{2\alpha}\right)} \left|e^{-i\tau t'} \int_0^{\tau t'} \ud s \, e^{i s} \int_{\R^d} \ud y \, \phi(y) G(s,y)
    - e^{-i\tau t} \int_0^{\tau t} \ud s \, e^{i s} \int_{\R^d} \ud y \, \phi(y) G(s,y)\right|^2.
  \end{align*}
  Put this back into \eqref{E:LB_t_log}, recall that $\rho = \beta+\gamma+H-\ell\beta/(2\alpha)-1$, and use the change of variable $\tau
    \mapsto \tau/(t'-t)$ to find that
  \begin{align}\begin{split}\label{E:LB_t_log2}
      \E\left[ \left( u(t',x)-u(t,x)\right)^2\right]
       & \gtrsim \int_0^\infty \tau^{-1-2\rho} \left|F(\tau t') - F(\tau t)\right|^2 \ud \tau                                                           \\
       & = (t'-t)^{2\rho} \int_0^\infty \tau^{-1-2\rho} \left|F\left(\frac{\tau t'}{t'-t}\right) - F\left(\frac{\tau t}{t'-t}\right)\right|^2 \ud \tau,
    \end{split}\end{align}
  where
  \begin{align*}
    F(r) := e^{-ir} \int_0^r \ud s \, e^{is} \int_{\R^d} \ud y\, \phi(y) G(s,y).
  \end{align*}
  The derivative of $F$ is
  \begin{align*}
    F'(r) = -i F(r) + \int_{\R^d} \phi(y) G(r,y) \ud y.
  \end{align*}
  We claim that the compactly supported smooth function $\phi$ can be chosen
  such that
  \begin{align}\label{E:liminf:F'}
    \liminf_{r\to\infty, r \in A} |F'(r)|>0,
  \end{align}
  where $A = \bigcup_{k=0}^\infty [(2k+1/3)\pi-\theta, (2k+4/3)\pi-\theta]$ for
  some $\theta \in (-\pi/2, \pi/2]$.

    Indeed, if $\beta+\gamma-1-\beta(\frac{d}{\alpha}\wedge 1) \in [-1, 0)$, then
  by \eqref{E_:gInf} below,
  \begin{align*}
    \int_{\R^d} \phi(y) G(r,y) \ud y \to 0 \quad \text{as $r\to\infty$,}
  \end{align*}
  which implies that
  \begin{align*}
    \liminf_{r\to\infty} |F'(r)| = \liminf_{r\to\infty} |F(r)|,
  \end{align*}
  and hence according to Lemma \ref{L:LB_t_phi} below, $\phi$ can be chosen such
  that
  \begin{align*}
    \liminf_{r\to\infty} |F(r)| > 0,
  \end{align*}
  which proves \eqref{E:liminf:F'}. If
  $\beta+\gamma-1-\beta(\frac{d}{\alpha}\wedge 1)\ge 0$, then we can use
  Plancherel's theorem and integration by parts and \eqref{E:ML-deriv} to find that
  \begin{align*}
    F'(r) & = -i e^{-ir} \left[ (2\pi)^{-d}\int_{\R^d} \ud \xi\, \widehat{\phi}(\xi) \int_0^r \ud s \, e^{is} s^{\beta+\gamma-1}E_{\beta,\beta+\gamma}\left(-\frac{\nu}{2}|\xi|^\alpha s^\beta\right) \right] + \int_{\R^d} \phi(y) G(r,y) \ud y \\
          & = (2\pi)^{-d} e^{-ir} \int_{\R^d} \ud \xi \, \widehat{\phi}(\xi) \int_0^r \ud s \, e^{is} s^{\beta+\gamma-2} E_{\beta, \beta+\gamma-1}\left(-\frac{\nu}{2}|\xi|^\alpha s^\beta\right)                                                \\
          & = e^{-ir} \int_0^r \ud s \, e^{is} \int_{\R^d} \ud y \, \phi(y) Y_{\alpha, \beta, \tilde\gamma, d}(s,y),
  \end{align*}
  where $\tilde\gamma=\gamma-1$ and $Y_{\alpha, \beta, \gamma, d}$ was defined
  in Theorem \ref{T:PDE}. Hence, \eqref{E:liminf:F'} again follows from Lemma
  \ref{L:LB_t_phi} below. This verifies the claim. In particular, there exists
  $\delta>0$ small such that
  \begin{align}\label{E:liminf:F':2}
    |F'(r)|\ge L>0 \quad \text{for all $r\in [S/\delta, \infty) \cap A$.}
  \end{align}
  With this in hand, we now proceed to finish the proof of \eqref{LB:Tinc_log}.

  Suppose $\rho=1$. Let $t, t'\in [S, T]$ be such that $0 \le t'-t \le
    \delta/2$, and set
  \[
    A_k = \left[\frac{t'-t}{t}((2k+1/3)\pi-\theta), \frac{t'-t}{t}((2k+2/3)\pi-\theta)\right].
  \]
  Then, by \eqref{E:LB_t_log2}, the mean value theorem, and \eqref{E:liminf:F':2},
  \begin{align*}
    \E\left[ \left(u(t',x)-u(t,x)\right)^2\right]
     & \gtrsim (t'-t)^2 \sum_{k=0}^\infty \int_{[(t'-t)/\delta,\, 1]\cap A_k} \tau^{-3} \left|F\left( \frac{\tau t'}{t'-t} \right) - F\left(\frac{\tau t}{t'-t}\right) \right|^2 \ud \tau \\
     & \gtrsim L^2 (t'-t)^2 \sum_{k=0}^\infty \int_{[(t'-t)/\delta,\, 1]\cap A_k} \tau^{-1}  \ud \tau                                                                                     \\
     & \gtrsim (t'-t)^2 (1+|\log(t'-t)|).
  \end{align*}
  Suppose $\rho>1$. Again, by \eqref{E:LB_t_log2}, the mean value theorem, and \eqref{E:liminf:F':2},
  \begin{align*}
    \E\left[ \left(u(t',x)-u(t,x)\right)^2\right]
     & \gtrsim (t'-t)^{2\rho} \sum_{k=0}^\infty \int_{[(t'-t)/\delta,\, 2(t'-t)/\delta]\cap A_k} \tau^{-1-2\rho} \left| F\left( \frac{\tau t'}{t'-t} \right) - F\left( \frac{\tau t}{t'-t} \right)\right|^2 \ud \tau \\
     & \gtrsim L^2(t'-t)^{2\rho} \sum_{k=0}^\infty \int_{[(t'-t)/\delta,\, 2(t'-t)/\delta]\cap A_k} \tau^{1-2\rho} \ud \tau                                                                                          \\
     & \gtrsim (t'-t)^2.
  \end{align*}
  This completes the proof of part (i) of Theorem~\ref{T:LND_t}.
\end{proof}

\begin{lemma}\label{L:LB_t_phi}
  Suppose that $\gamma \ge 0$ and $\beta\in (0,2)$.
  \begin{enumerate}[\rm (i)]

    \item If $\alpha>0$ satisfies~\eqref{E:Rationals}, then there exists a
          compactly supported smooth function $\phi: \R^d \to \R$ such that
          \begin{align}\label{E:LB_t_phi}
            L_\phi \coloneqq \liminf_{\tau \to \infty} \left\lvert\int_0^\tau \ud s\, e^{is} \int_{\R^d} \ud y\, \phi(y) G(s, y)\right\rvert^2 > 0.
          \end{align}
          In particular, the above $\liminf$ is taken over all $\tau > 0$ when
          $\beta + \gamma - 1 - \beta \left(\frac{d}{\alpha} \wedge 1\right) \ne
            0$, and along the set $\bigcup_{k=0}^\infty \left[(2k+1/3)\pi - \theta,
              \left(2k+4/3\right) \pi - \theta\right]$ when $\beta + \gamma - 1 - \beta
            \left(\frac{d}{\alpha} \wedge 1\right) = 0$ for some $\theta \in (-\pi/2,
            \pi/2]$.

    \item Otherwise if $\alpha>0$ satisfies~\eqref{E:alpha-Q},
          then~\eqref{E:LB_t_phi} still holds.

  \end{enumerate}
\end{lemma}

\begin{remark}\label{R:LB-t-phi-liminf}
  In the statement of part (ii) of the above lemma, we do not explicitly specify
  how the \(\liminf\) is taken over \(\tau > 0\). This detail is less critical,
  and when necessary, it can be understood through the proof of part (ii)
  provided below.
\end{remark}

\begin{proof}[Proof of part (i) of Lemma~\ref{L:LB_t_phi}]
  Let $\psi:\R\to\R$ be a test function supported in
  $[1,2]$, and set $\phi(x) \coloneqq \psi\left(|x|\right)$. Let $f(x)$ be
  defined as in~\eqref{E:f(x)}. Then, by the scaling property of $G(s,y)$
  in~\eqref{E:Gscaling} and the relation between $G(1,x)$ and $f(x)$
  in~\eqref{E:G-f}, we see that
  \begin{align*}
    g(s)
    \coloneqq \int_{\R^d} \phi(y) G(s,y) \ud y
    = C s^{\beta+\gamma-1} \int_1^2 \ud r\: \psi(r) r^{-1} f\left(C_2 s^{-\beta} r^{\alpha}\right).
  \end{align*}
  Since $f(x)$ is a smooth function, bounded at $x=0$ (see
  Theorem~\ref{T:Zero}), and $\lim_{x\to\infty}f(x) =0$ (see
  Theorem~\ref{T:Infinity}), $f(x)$ is globally bounded. Therefore,
  \begin{align}\label{E_:gZero}
    |g(s)| \le C \left(\max_{r\ge 0}|f(r)|\int_1^2\ud r\: r^{-1}\psi (r) \right) s^{\beta+\gamma-1}, \quad \text{for all $s > 0$,}
  \end{align}
  which in particular implies that $g(s)$ is locally integrable near $s=0$.

  The precise behavior of $g(s)$ for large $s$ is more delicate. However, the
  case where $\alpha$ satisfies~\eqref{E:Rationals} is clear. Since $\alpha$
  satisfies~\eqref{E:Rationals}, by part (i) of Theorem~\ref{T:Zero}, and using
  the dominated convergence theorem, we see that
  \begin{align}\label{E_:gInf}
    g(s) \asymp s^{\beta+\gamma - 1 - \beta \left(\frac{d}{\alpha}\wedge 1\right)} , \quad \text{as $s\to\infty$}.
  \end{align}

  \noindent\textbf{Case I:~} We first consider the case when $\beta+\gamma - 1 -
    \beta \left(\frac{d}{\alpha}\wedge 1\right) \in (-1, 0)$. Notice that for any $\tau > 1$,
  \begin{align*}
    \int_0^{\tau} \ud s \; e^{i s} g(s) = \int_0^1 \ud s \; e^{i s} g(s) + \int_1^{\tau} \ud s \; \cos(s) g(s) + i \int_1^{\tau} \ud s \; \sin(s) g(s) \eqqcolon I_1 + I_2(\tau) + I_3 (\tau).
  \end{align*}
  By \eqref{E_:gZero}, we know that
  \begin{align*}
    |I_1| \leq \int_0^1 \ud s \; |g(s)| < \infty.
  \end{align*}
  Using \eqref{E_:gInf}, we can deduce that
  \begin{align*} % \label{E_:i2(tau)}
    I_2 (\tau) = \int_1^{\tau} \ud s \; \cos(s) s^{\beta+\gamma - 1 - \beta \left(\frac{d}{\alpha}\wedge 1\right)} + R_2 (\tau),
  \end{align*}
  where $R_2$ corresponds to the integral of higher order terms in $g$. Thanks to Dirichlet's test for integrals,
  \begin{align*}
    \lim_{\tau \to \infty} \int_1^{\tau} \ud s \; \cos(s) s^{\beta+\gamma - 1 - \beta \left(\frac{d}{\alpha}\wedge 1\right)}
  \end{align*}
  exists. On the other hand, by part (i) of Theorem~\ref{T:Zero}, the integrals in $R_2$ can be expressed into two parts. One part involves integrals of the form
  \begin{align*}
    \int_1^{\tau} \ud s \; \cos(s) s^{q}, \quad -1 \leq q < \beta+\gamma - 1 -
    \beta \left(\frac{d}{\alpha}\wedge 1\right) .
  \end{align*}
  Specifically, the number of such terms is finite. The other part can be bounded by an integral
  \begin{align*}
    \int_1^{\tau} \ud s \; \cos(s) s^{q} , \quad q < -1.
  \end{align*}
  The first part can be treated using Dirichlet's test, and the second part is an integral of absolute integrable function. As a result, as $\tau \to \infty$, $I_2 (\tau)$, and similarly, $I_3 (\tau)$ are convergent. Therefore,
  \begin{align}\label{E_:int-eig}
    \lim_{\tau \to \infty} \int_0^{\tau} \ud s \; e^{i s} g(s)
  \end{align}
  exists. Therefore, the validity of Lemma \ref{L:LB_t_phi} reduces to the existence of a test function $\psi$ such that the limit in \eqref{E_:int-eig} is not zero. This can be verified by contradiction.
  Suppose that for all test function $\psi$ supported on $[1,2]$,
  \begin{align*}
    \int_0^{\infty} e^{i s} g (s) = 0.
  \end{align*}
  Notice that for any $\tau > 0$, Fubini's theorem suggests that
  \begin{align*}
    \int_0^{\tau} e^{i s} g (s)
    = \int_1^2 \psi (r) \ud r \; r^{-1} \int_0^{\tau} \ud s \;  s^{\beta + \gamma - 1}  f\left(C_2 s^{-\beta} r^{\alpha}\right) e^{i s} .
  \end{align*}
  By the same argument as before, one can show that
  \begin{align*}
    \int_0^{\tau} \ud s \;  s^{\beta + \gamma - 1}  f\left(C_2 s^{-\beta} r^{\alpha}\right) e^{i s}
  \end{align*}
  is uniformly bounded in $r \in [1,2]$ and $\tau \in (0,\infty)$, and is convergent for every $r$ as $\tau \to \infty$. As a result of the bounded convergence theorem,
  \begin{align*}
    0 = \lim_{\tau \to \infty} \int_0^{\tau} e^{i s} g (s)
    = \int_1^2 \psi (r) \ud r \; r^{-1} \lim_{\tau \to \infty}\int_0^{\tau} \ud s \;  s^{\beta + \gamma - 1}  f\left(C_2 s^{-\beta} r^{\alpha}\right) e^{i s},
  \end{align*}
  for all test function $\psi$. This implies that as a function of $r \in [1,2]$,
  \begin{align*}
    \lim_{\tau \to \infty} \int_0^{\tau} \ud s \;  s^{\beta + \gamma - 1}  f\left(C_2 s^{-\beta} r^{\alpha}\right) e^{i s} = \int_0^{\infty} \ud s \;  s^{\beta + \gamma - 1}  f\left(C_2 s^{-\beta} r^{\alpha}\right) e^{i s} = 0
  \end{align*}
  With a change of variable $s \mapsto r^{\alpha /\beta} t$, we have
  \begin{align*}
    0 = r^{\alpha (\beta + \gamma ) /\beta} \int_0^{\infty} \ud t \;   t^{\beta + \gamma - 1}  f\left(C_2 t^{-\beta} \right) e^{i r^{\alpha /\beta} t}
  \end{align*}
  for all $r \in [1,2]$. Notice that the above argument work for all test functions $\psi$ supported on a compact subset of $\R \setminus \{0\}$. In other words, if Lemma \ref{L:LB_t_phi} fails, it holds that
  \begin{align*}
    0 = \int_0^{\infty} \ud t \;   t^{\beta + \gamma - 1}  f\left(C_2 t^{-\beta} \right) e^{i |r|^{\alpha /\beta} t}
  \end{align*}
  for all $r \in \R \setminus \{0\}$,
  which is not true. Therefore, there exists a test function $\psi$ such that the limit in \eqref{E_:int-eig} is not zero, which leads to \eqref{E:LB_t_phi} under the condition that $\beta+\gamma - 1 -
    \beta \left(\frac{d}{\alpha}\wedge 1\right) \in (-1, 0)$.

  \bigskip

  \noindent\textbf{Case II:~} Assume $\beta+\gamma - 1 - \beta
    \left(\frac{d}{\alpha}\wedge 1\right) \in (0,1)$. In this case, $g(0) = 0$
  (see~\eqref{E_:gZero}).  Denote
  \begin{align*}
    S(\tau) \coloneqq \int_0^\tau g'(s) \sin(s) \ud s \quad \text{and} \quad
    C(\tau) \coloneqq \int_0^\tau g'(s) \cos(s) \ud s.
  \end{align*}
  Then, by integration by parts formula, it follows that
  \begin{align}\label{E_:int_fs-ii}
    \left\lvert\int_0^\tau e^{is} g(s) \ud s\right\rvert^2
    % & = \left(\sin(\tau) g(\tau) - \int_0^\tau g'(s) \sin(s) \ud s\right)^2
    %    +\left(\cos(\tau) g(\tau) + \int_0^\tau g'(s) \cos(s) \ud s\right)^2 \nonumber \\
     & = g^2(\tau) - 2 \sin(\tau) g(\tau) S(\tau) + 2 \cos(\tau) g(\tau) C(\tau) + S^2(\tau) + C^2(\tau).
  \end{align}
  By Case I, we see that for both $S(\tau)$ and $C(\tau)$ are uniformly bounded
  functions. Thus, for $\tau$ large enough, with universal positive constants
  $C_1$ and $C_2$,
  \begin{align*}
    \left\lvert\int_0^\tau e^{is} g(s) \ud s\right\rvert^2
    \ge g^2(\tau) - C_1 |g(\tau) | - C_2.
  \end{align*}
  Then, we see that the lower bound is dominated by the term $g^2(\tau)$ due to
  the asymptotic behavior of $g$ as $\tau\to\infty$ (see~\eqref{E_:gInf}).
  Hence, the limit inferior, along $\tau >0$, of the right-hand side of the
  above inequality is bounded from below by a positive constant, which
  proves~\eqref{E:LB_t_phi}.\bigskip

  \noindent\textbf{Case III:~} If $\beta+\gamma - 1 - \beta
    \left(\frac{d}{\alpha}\wedge 1\right) = 0$, denote by $g_{\infty}$,
  $S_{\infty}$, and $C_{\infty}$ the limits of $g(\tau)$, $S(\tau)$, and
  $C(\tau)$, respectively, as $\tau \to \infty$. From~\eqref{E_:gInf}, we know
  that $g_{\infty} \neq 0$. Hence, if $S_{\infty} = C_{\infty} = 0$,
  then~\eqref{E_:int_fs-ii} approaches to $g_{\infty}^2 > 0$, as $\tau \to
    \infty$, which proves~\eqref{E:LB_t_phi} along $\tau>0$. Now if $S_{\infty}$
  and $C_{\infty}$ are not identically zero, choose
  \begin{align*}
    \theta \coloneqq \arcsin\left(\frac{S_{\infty}}{\sqrt{S_{\infty}^2 + C_{\infty}^2}}\right) \quad \text{so that }
    \sin (\theta) = \frac{S_{\infty}}{\sqrt{S_{\infty}^2 + C_{\infty}^2}} \text{ and }
    \cos (\theta) = \frac{C_{\infty}}{\sqrt{S_{\infty}^2 + C_{\infty}^2}}.
  \end{align*}
  Then, for some remainder term $R(\tau) \to 0$ as $\tau\to\infty$, we have that
  \begin{align*}
    \left\lvert\int_0^\tau e^{is} g(s) \ud s\right\rvert^2
     & = g_\infty^2 - 2 g_\infty \sqrt{S_{\infty}^2 + C_{\infty}^2} \cos(\tau + \theta) + S_{\infty}^2 + C_{\infty}^2 + R(\tau).
  \end{align*}
  Hence, for $\tau \in \bigcup_{k = 0}^{\infty} \left[ \left(2 k+1/3\right)\pi -
      \theta, \left(2k + 4/3\right) \pi - \theta\right ]$, we have that
  $|\cos(\tau+\theta)|\le 1/2$ and the above quantity is lower bounded by
  \begin{align*}
    \left\lvert\int_0^\tau e^{is} g(s) \ud s\right\rvert^2
     & \ge  |g_\infty|^2 - |g_\infty| \sqrt{S_{\infty}^2 + C_{\infty}^2} + S_{\infty}^2 + C_{\infty}^2 + R(\tau)                                        \\
     & \ge  \left(|g_\infty| - \frac{1}{2} \sqrt{S_{\infty}^2 + C_{\infty}^2}\right)^2 + \frac{3}{4} \left(S_{\infty}^2 + C_{\infty}^2\right) + R(\tau) \\
     & \ge  \frac{3}{4} \left(S_{\infty}^2 + C_{\infty}^2\right) + R(\tau) > 0
  \end{align*}
  when $\tau$ is large enough. This proves~\eqref{E:LB_t_phi} along $\tau \in
    \bigcup_{k = 0}^{\infty} \left[ \left(2 k+1/3\right)\pi - \theta, \left(2k +
      4/3\right) \pi - \theta\right ]$. \bigskip

  \noindent\textbf{Case IV:~} Finally, when $\beta+\gamma - 1 - \beta
    \left(\frac{d}{\alpha}\wedge 1\right) \ge 1$, we can apply integration by
  parts multiple times and then use the similar arguments as in Case II to
  prove~\eqref{E:LB_t_phi}. We will not go into the details here, and instead
  leave it to the interested reader. This completes the proof when $\alpha$
  satisfies the condition~\eqref{E:Rationals}.
\end{proof}

\begin{proof}[Proof of part (ii) of Lemma~\ref{L:LB_t_phi}]
  When $\alpha$ does not satisfy~\eqref{E:Rationals}, or equivalently, when it
  is a rational number of the form given in the condition~\eqref{E:alpha-Q}, the
  dominant order of $f(x)$ as $x\downarrow 0$ is determined by the largest term
  among the two power series and the term with logarithm correction
  in~\eqref{E:PowerLogSeries}. If one of the two power series contributes the
  dominant order, then the proof is similar to part (i). Otherwise, if the term
  with logarithm correction dominates, then for some $\ell \in
    \mathbb{N}$ and $C\ne 0$, $f(x) = C x^{\ell} \log x + o\left(x^{\ell}\log
    x\right)$ as $x\downarrow 0$. As a consequence,
  \begin{align*}
    g(s) \asymp s^{\beta+\gamma - 1 - \beta \ell} \log s \quad \text{as $s\to\infty$}.
  \end{align*}
  When $\beta+\gamma-1-\beta\ell<0$, the proof remains the same as the Case I of
  the proof of part (i). When $\beta+\gamma-1-\beta\ell\in [0,1)$, with $0$
  being included, the arguments from Case II in the proof of part (i) apply
  here. Case III can be combined with Case II due to the growth contributed by
  the logarithmic term.  Finally, when $\beta+\gamma-1-\beta\ell\ge 1$,
  integration by parts can be applied multiple times to reduce the problem to
  the previous cases. This completes the proof of part (ii) of
  Lemma~\ref{L:LB_t_phi}.
\end{proof}

\subsection{Proof of part (ii) of Theorem \ref{T:LND_t}}\label{SS:LNDt-ii}

% \begin{proof}[Proof of part (ii) of Theorem \ref{T:LND_t}]
The $\beta = 1$, $\gamma = 0$ case has been proved in \cite{lee.xiao:23:chung-type}, so we focus on proving the $\beta = 1$, $\gamma > 0$ case.
It is enough to show that there exists a positive finite constant $C$ such that for all $x \in \R^d$, for all $n \ge 1$, for all $t, t_1, \dots, t_n \in [S, T]$, and for all $a_1, \dots, a_n \in \R$,
\begin{align*}
  \Theta \coloneqq \E\left[ \left( u(t, x) - \sum_{j=1}^n a_j u(t_j, x) \right)^2\right]
  \ge C r^{2\gamma + 2H - \ell/\alpha}, \quad
  \text{where } r\coloneqq \min_{1 \le j \le n} |t-t_j|.
\end{align*}
Without loss of generality, assume $r > 0$.
Recall the formula \cite[(1.100)]{podlubny:99:fractional}
\begin{align}\label{E:ML:FracInt}
  s^{b+\nu-1} E_{a, b+\nu} ( \lambda s^b) = \frac{1}{\Gamma(\nu)} \int_0^s (s-v)^{\nu-1} v^{b-1} E_{a, b}( \lambda v^b ) \ud v
\end{align}
and the fact that $E_{1, 1}(z) = \exp(z)$.
This and \eqref{E:LNDt_Theta} together imply that
\begin{align}\label{E:LNDt:b1:Theta}
  \Theta = C_{H, d} \int_\R \ud \tau\, |\tau|^{1-2H} \int_{\R^d} \ud \xi \, |\xi|^{\ell-d} \left| e^{-i\tau t} A_t(\tau, \xi) - \sum_{j=1}^n a_j e^{-i\tau t_j} A_{t_j}(\tau, \xi) \right|^2,
\end{align}
where
\begin{align*} % \label{E:LNDt:A}
  A_{t}(\tau, \xi)
   & = \frac{1}{\Gamma(\gamma)} \int_0^t \ud s \int_0^s \ud v \, e^{i\tau s} (s-v)^{\gamma-1} e^{-\frac{\nu}{2}|\xi|^\alpha v}.
\end{align*}
By interchanging the order of integration, a change of variable, and integration by parts, we have
\begin{align}
  \begin{split}\label{E:LNDt:b1:Atilde}
    A_{t}(\tau, \xi)
     & = \frac{1}{\Gamma(\gamma)}\int_0^t \ud v\, e^{-\frac{\nu}{2}|\xi|^\alpha v} \int_v^t \ud s\, e^{i\tau s} (s-v)^{\gamma-1}                                                                                                                                  \\
     & = \frac{1}{\Gamma(\gamma)}\int_0^t \ud v \, e^{-\frac{\nu}{2}|\xi|^\alpha v} e^{i\tau v}\int_0^{t-v} \ud s \, e^{i\tau s} s^{\gamma-1}                                                                                                                     \\
     & = \frac{1}{i\tau-\frac{\nu}{2}|\xi|^\alpha} \underbrace{\frac{1}{\Gamma(\gamma)}\left[ -\int_0^t e^{i\tau s} s^{\gamma-1} \ud s + e^{i\tau t} \int_0^t  e^{-\frac{\nu}{2}|\xi|^\alpha v} (t-v)^{\gamma-1} \ud v\right]}_{=: \widetilde{A}_{t}(\tau, \xi)}.
  \end{split}
\end{align}
Let $f_r: \R \to \R$ and $g_r : \R^d \to \R$ be two test functions to be determined.
Consider the integral
\begin{align}\label{E:LNDt:b1:int}
  I & \coloneqq - \int_\R \ud \tau \int_{\R^d} \ud \xi \, \widehat{f_r}(\tau) \widehat{g_r}(\xi) \overline{\left( e^{-i\tau t} \widetilde{A}_{t}(\tau, \xi) - \sum_{j=1}^n a_j e^{-i\tau t_j} \widetilde{A}_{t_j}(\tau, \xi) \right)}.
\end{align}
By applying inverse Fourier transform in $\tau$, we have
\begin{align*}
  I & = \frac{2\pi}{\Gamma(\gamma)} \int_{\R^d} \ud \xi \, \widehat{g_r}(\xi) \left[ \int_0^t f_r(t-s) s^{\gamma-1} \ud s - f_r(0) \int_0^t e^{-\frac{\nu}{2}|\xi|^\alpha v} (t-v)^{\gamma-1} \ud v \right. \\
    & \quad \left. - \sum_{j=1}^n a_j \left( \int_0^{t_j} f_r(t_j-s) s^{\gamma-1} \ud s - f_r(0) \int_0^{t_j} e^{-\frac{\nu}{2}|\xi|^\alpha v} (t-v)^{\gamma-1} \ud v \right)\right].
\end{align*}
Note that
\begin{align*}
  \frac{1}{\Gamma(\gamma)}\int_0^t f_r(t-s) s^{\gamma-1} \ud s = \frac{1}{\Gamma(\gamma)}\int_0^t (t-s)^{\gamma-1} f_r(s) \ud s
  = \prescript{}{0}D_t^{-\gamma} f_r (t)
\end{align*}
is a fractional integral of $f_r$; see, e.g., \cite[Chapter 2]{podlubny:99:fractional}.
Take $F:\R \to \R$ to be a nonnegative smooth function supported in $[-\delta, \delta]$ with $\delta = \min\{1, S/T\}$ such that $F(0) = 1$, and let
\begin{align*}
  F_r(s) = F\left( \frac{s-t}{r}\right).
\end{align*}
Define $f_r: \R \to \R$ by the fractional derivative
\begin{align}\label{E:LNDt:b1:fr}
  \qquad f_r(s) = \prescript{}{0}D_s^{\gamma} F_r(s).
\end{align}
Note that $F_r$ is identically 0 in a neighborhood of 0 because of the choice of the support of $F$ and $-t/r \le -S/(T-S) < -S/T$.
Then by \cite[\S 2.2.6]{podlubny:99:fractional},
\begin{align*}
  \prescript{}{0}D_s^{-\gamma} f_r (s) = \prescript{}{0}D_s^{-\gamma} \left(\prescript{}{0}D_s^{\gamma} F_r(s)\right) = F_r(s).
\end{align*}
Note that $f_r(0) = 0$.
Also, $F_r(t_j) = 0$ for $j = 1, \dots, n$ since $|t_j-t|/r \ge 1$.
Take $g: \R^d \to \R$ to be a compactly supported nonnegative smooth function with $g(0) = 1$ and $g_r(x) = g(r^{-1/\alpha}x)$.
By such a choice of $f_r$ and $g_r$, it follows that
\begin{align}\label{E:LNDt:b1:I}
  I = 2\pi \int_{\R^d} \ud \xi \, \widehat{g_r}(\xi) \left[ F_r(t) - \sum_{j=1}^n a_j F_r(t_j)\right] = (2\pi)^{1+d}.
\end{align}

On the other hand, applying Cauchy--Schwarz inequality to the integral \eqref{E:LNDt:b1:int}, and using \eqref{E:LNDt:b1:Theta}--\eqref{E:LNDt:b1:Atilde}, we get that
\begin{align}\label{E:LNDt:b1:CS}
  |I|^2 & \le C_{H, d}^{-1} \Theta \times \int_\R \ud \tau \int_{\R^d} \ud \xi \, | \widehat{f_r}(\tau)|^2 \left| \widehat{g_r}(\xi) \right|^2 |\tau|^{2H-1} |\xi|^{d-\ell} \left( |\tau|^2 + \frac{\nu^2}{4} |\xi|^{2\alpha} \right).
\end{align}
By \eqref{E:LNDt:b1:fr} and re-scaling, we see that
\begin{align*}
  f_r(s)
   & = \frac{\ud^{\lceil \gamma \rceil}}{\ud s^{\lceil \gamma \rceil}} \left( \prescript{}{0}D_s^{\gamma - \lceil \gamma \rceil} F_r(s) \right)                                                                                                                                          \\
   & = \frac{\ud^{\lceil \gamma \rceil}}{\ud s^{\lceil \gamma \rceil}} \left( \frac{1}{\Gamma(\lceil \gamma\rceil - \gamma)} \int_0^s (s-v)^{\lceil \gamma \rceil - \gamma - 1} F\left( \frac{v-t}{r} \right) \ud v \right)                                                              \\
   & = \frac{\ud^{\lceil \gamma \rceil}}{\ud s^{\lceil \gamma \rceil}} \left( \frac{r^{\lceil \gamma \rceil - \gamma}}{\Gamma(\lceil \gamma \rceil - \gamma)} \int_{-\frac{t}{r}}^{\frac{s-t}{r}} \left( \frac{s-t}{r} - v\right)^{\lceil \gamma \rceil - \gamma - 1} F(v) \ud v \right) \\
   & = \frac{\ud^{\lceil \gamma \rceil}}{\ud s^{\lceil \gamma \rceil}} \left( \frac{r^{\lceil \gamma \rceil - \gamma}}{\Gamma(\lceil \gamma \rceil - \gamma)} \int_{-\infty}^{\frac{s-t}{r}} \left( \frac{s-t}{r} - v\right)^{\lceil \gamma \rceil - \gamma - 1} F(v) \ud v \right)      \\
   & = \frac{1}{r^{\gamma}} \phi\left( \frac{s-t}{r} \right),
\end{align*}
where
\begin{align*}
  \phi(s) = \prescript{}{-\infty}D_s^{\gamma}F(s).
\end{align*}
Then we have $|\widehat{f_r}(\tau)| = r^{1-\gamma}|\widehat{\phi}(r\tau)|$ and $|\widehat{g_r}(\xi)| = r^{d/\alpha}|\widehat{g}(r^{1/\alpha} \xi)|$.
By Fourier transform of fractional derivative \cite[\S 2.9.3]{podlubny:99:fractional}, we have
\begin{align*}
  \widehat{\phi}(\tau) = (-i\tau)^{\gamma} \widehat{F}(\tau),
\end{align*}
Putting these back into \eqref{E:LNDt:b1:CS} and by scaling, we have
\begin{align}\label{E:LNDt:b1:I2}
  |I|^2 \le C_{H, d}^{-1} \Theta \times C_0 r^{-2\gamma-2H+\ell/\alpha},
\end{align}
where
\begin{align*}
  C_0 = \int_\R \ud \tau \int_{\R^d} \ud \xi \, | \widehat{F}(\tau) |^2 \left| \widehat{g}(\xi) \right|^2 |\tau|^{2H+2\gamma-1} |\xi|^{d-\ell} \left( |\tau|^2 + \frac{\nu^2}{4}|\xi|^{2\alpha} \right),
\end{align*}
which is a finite constant since $\widehat{F}$ and $\widehat{g}$ are rapidly decreasing functions.
Combining \eqref{E:LNDt:b1:I} and \eqref{E:LNDt:b1:I2}, we conclude that
\begin{align*}
  \Theta \ge (2\pi)^{2+2d} C_{H, d} C_0^{-1} r^{2\gamma + 2H - \ell/\alpha}.
\end{align*}
This completes the proof of \eqref{LNDt:beta=1}. \qed
% \end{proof}

\subsection{Proof of part (iii) of Theorem \ref{T:LND_t}}\label{SS:LNDt-iii}

\subsubsection{The $\beta = 2$, $\gamma = 0$ case}

\begin{proof}
  We first consider the $\beta = 2$, $\gamma = 0$ case.
  Assume Dalang's condition \eqref{E:main}.
  To prove \eqref{LNDt:CondVar2}, it suffices to show that there exists a
  positive finite constant $C$ such that for all $x \in \R^d$, for all $n \ge
    1$, for all $t, t_1, \dots, t_n \in [S, T]$, and for all $a_1, \dots, a_n \in
    \R$,
  \begin{align}\label{E:LNDt2_LB}
    \Theta \coloneqq \E\left[ \left( u(t, x) - \sum_{j=1}^n a_j u(t_j, x)\right)^2\right]
    \ge C r^{2H+1-\frac{2\ell}{\alpha}}, \quad \text{where } r\coloneqq \min_{1\le j \le n}|t-t_j|.
  \end{align}
  We prove \eqref{E:LNDt2_LB} by adapting the method in \cite{lee:22:local}.
  Without loss of generality, we assume that $r > 0$. By \eqref{E:Inner_1}, \eqref{E:u_corr}, and \eqref{E:FG_b2g0}, we
  have
  \begin{align*}
    \Theta & = C\int_\R \ud \tau \, |\tau|^{1-2H} \int_{\R^d} \ud \xi \, |\xi|^{\ell-d} \left| \mathcal{F}[G(t-\cdot, \ast)\one_{[0, t]}(\cdot)](\tau, \xi) - \sum_{j=1}^n a_j \mathcal{F}[G(t_j-\cdot, \ast) \one_{[0, t_j]}(\cdot)](\tau, \xi)\right|^2 \\
           & = C\int_\R \ud \tau \, |\tau|^{1-2H} \int_{\R^d} \ud \xi \, |\xi|^{\ell-d-\alpha} \left| F\left(t, \tau, \sqrt{\nu/2}|\xi|^{\alpha/2}\right) - \sum_{j=1}^n a_j F\left(t_j, \tau, \sqrt{\nu/2}|\xi|^{\alpha/2}\right) \right|^2,
  \end{align*}
  where
  \begin{align}\label{E:FTsin}
    F(t, \tau, z) = \int_0^t e^{-i\tau s} \sin\left((t-s)z\right) \ud s.
  \end{align}
  Then, using spherical coordinates with $y = |\xi|$, followed by the change of
  variable $z = \sqrt{\nu/2}\,y^{\alpha/2}$, we get that
  \begin{align*}
    \Theta
     & = C\int_\R \ud \tau \, |\tau|^{1-2H} \int_0^\infty \ud y \, y^{\ell-\alpha-1} \left| F\left(t, \tau, \sqrt{\nu/2} \, y^{\alpha/2} \right) - \sum_{j=1}^n a_j F\left(t_j, \tau, \sqrt{\nu/2} \,  y^{\alpha/2}\right)\right|^2 \\
     & = C'\int_\R \ud \tau \, |\tau|^{1-2H} \int_0^\infty \ud z \, z^{-3+\frac{2\ell}{\alpha}} \left| F(t, \tau, z) - \sum_{j=1}^n a_j F(t_j, \tau, z)\right|^2.
  \end{align*}
  Thanks to the property $F(t, \tau, -z) = -F(t, \tau, z)$, we can write the
  last integral as
  \begin{align}\label{E:LNDt_b2g0_Var}
    \Theta = C_1\int_\R \ud \tau \, |\tau|^{1-2H} \int_\R \ud z \, |z|^{-3+\frac{2\ell}{\alpha}} \left| F(t, \tau, z) - \sum_{j=1}^n a_j F(t_j, \tau, z)\right|^2,
  \end{align}
  where $C_1=C'/2$.

  Let $\phi, \psi: \R \to \R$ be two nonnegative smooth test functions
  satisfying the following properties: $\phi$ is supported on $[0, S/2]$ and
  $\int \phi = 1$; $\psi$ is supported on $[-1, 1]$ and $\psi(0) = 1$. Let
  $\psi_r(x) = r^{-1}\psi(r^{-1}x)$. Consider the integral
  \begin{align}\label{D:LNDt_b2g0_I}
    I \coloneqq \int_\R \ud z \int_\R \ud \tau \overline{\left( F(t, \tau, z) - \sum_{j=1}^n a_j F(t_j, \tau, z) \right)} e^{-itz}\widehat{\phi}(\tau-z)\widehat{\psi_r}(z).
  \end{align}
  Note that for fixed $z$, $\tau \mapsto \widehat{\phi}(\tau-z)$ is the Fourier
  transform of $s \mapsto e^{isz} \phi(s)$, and $\tau \mapsto F(t, \tau, z)$ is
  the Fourier transform of $s\mapsto \sin((t-s)z)\one_{[0, t]}(s)$ by
  \eqref{E:FTsin}. Then, applying the Plancherel theorem with respect to to the $\tau$ variable, we get that
  \begin{align*}
    I = 2\pi \int_\R \ud z \int_\R \ud s {\left(\sin((t-s)z)\one_{[0, t]}(s) - \sum_{j=1}^n a_j \sin((t_j - s)z)\one_{[0, t_j]}(s)\right)}
    e^{-i(t-s)z}\phi(s)\widehat{\psi_r}(z).
  \end{align*}
  Since $\phi$ is supported on $[0, S/2]$ and $\sin z = \frac{1}{2i}(e^{iz} -
    e^{-iz})$, we can write
  \begin{align*}
    I = -\pi i \int_0^{S/2} \ud s \, \phi(s) \int_\R \ud z \left[ \left(e^{i(t-s)z} - e^{-i(t-s)z}\right) - \sum_{j=1}^n a_j\left(e^{i(t_j-s)z} - e^{-i(t_j-s)z}\right) \right] e^{-i(t-s)z} \widehat{\psi_r}(z).
  \end{align*}
  Then, by inverse Fourier transform, we have
  \begin{align*}
    I = -2\pi^2 i \int_0^{S/2} \ud s\,  \phi(s) \left[ \big(\psi_r(0) - \psi_r(2(t-s))\big) - \sum_{j=1}^n \big(\psi_r(t-t_j) - \psi_r(t_j-s+t-s)\big) \right].
  \end{align*}
  Recall that $\psi$ is supported on $[-1, 1]$. For all $s \in [0, S/2]$, we
  have $2(t-s)/r \ge 2(S-S/2)/S = 1$ and $(t_j-s + t-s)/r \ge (S/2+S/2)/S =
    1$, so
  \begin{align*}
    \psi_r(2(t-s))        = 0 \quad \text{and} \quad
    \psi_r(t_j-s+t-s) = 0.
  \end{align*}
  Also, by definition, $r \le |t-t_j|$, which implies that
  \begin{align*}
    \psi_r(t-t_j) = 0.
  \end{align*}
  Moreover, $\psi_r(0) = r^{-1}$ and $\int \phi = 1$. Therefore, we have
  \begin{align}\label{E:LNDt_b2g0_I}
    |I| = 2\pi^2 r^{-1}.
  \end{align}

  On the other hand, we can apply the Cauchy--Schwarz inequality to the
  integral~\eqref{D:LNDt_b2g0_I} and use~\eqref{E:LNDt_b2g0_Var} to get that
  \begin{align*}
    |I|^2 \le C_1^{-1} \Theta \times \int_\R \ud \tau \int_\R \ud z \, \left|\widehat{\phi}(\tau-z)\right|^2 \left| \widehat{\psi_r}(z) \right|^2|\tau|^{2H-1}\,|z|^{3-\frac{2\ell}{\alpha}}.
  \end{align*}
  Note that $\widehat{\psi_r}(z) = \widehat{\psi}(rz)$. We split the double
  integral into two parts where $|\tau| \le |z|$ and $|\tau| > |z|$. For the
  first part, we use $|\tau|^{2H-1} \le |z|^{2H-1}$ (since $H \ge 1/2$) and
  scaling to get that
  \begin{align}\label{E:LNDt_b2g0_int1}
    \begin{split}
       & \int_\R \ud z \int_{|\tau|\le |z|} \ud \tau \left|\widehat{\phi}(\tau-z)\right|^2 \left| \widehat{\psi}(rz) \right|^2|\tau|^{2H-1}\,|z|^{3-\frac{2\ell}{\alpha}} \\
       & \le \int_\R \ud z \, \left| \widehat{\psi}(rz)\right|^2 |z|^{2H+2-\frac{2\ell}{\alpha}} \int_\R \ud \tau\, \left| \widehat{\phi}(\tau)\right|^2                  \\
       & = C_2 r^{-2H-3+\frac{2\ell}{\alpha}},
    \end{split}
  \end{align}
  where
  \begin{align*}
    C_2 \coloneqq \int_\R \ud z \, \left| \widehat{\psi}(z)\right|^2 |z|^{2H+2-\frac{2\ell}{\alpha}} \int_\R \ud \tau\, \left| \widehat{\phi}(\tau)\right|^2
  \end{align*}
  is a positive finite constant due to Dalang's condition \eqref{E:main} (which
  implies that $2H+2-2\ell/\alpha > -1$) and the property that $\widehat{\psi}$
  and $\widehat{\phi}$ are rapidly decreasing functions.
  For the second part,
  we can let $w = \tau-z$ to get that
  \begin{align}\label{E:LNDt_b2g0_int2}
    \begin{split}
       & \int_\R \ud z \int_{|\tau| > |z|} \ud \tau \, \left|\widehat{\phi}(\tau-z)\right|^2 \left| \widehat{\psi}(rz) \right|^2|\tau|^{2H-1}\,|z|^{3-\frac{2\ell}{\alpha}}                                                                                                                     \\
       & \le \int_\R \ud z \int_\R \ud w \, \left|\widehat{\phi}(w)\right|^2 \left| \widehat{\psi}(rz)\right|^2 |z+w|^{2H-1} |z|^{3-\frac{2\ell}{\alpha}}                                                                                                                                       \\
       & \lesssim \int_\R \ud z \int_\R \ud w \, \left|\widehat{\phi}(w)\right|^2 \left| \widehat{\psi}(rz)\right|^2 |z|^{2H+2-\frac{2\ell}{\alpha}}+\int_\R \ud z \int_\R \ud w \, \left|\widehat{\phi}(w)\right|^2 \left| \widehat{\psi}(rz)\right|^2 |w|^{2H-1} |z|^{3-\frac{2\ell}{\alpha}} \\
       & = C_3 r^{-2H-3+\frac{2\ell}{\alpha}} + C_4 r^{-4+\frac{2\ell}{\alpha}},
    \end{split}
  \end{align}
  where $C_3$ and $C_4$ are positive finite constants (note that $3-2\ell/\alpha > -1$ by condition \eqref{E:main}). Also, $r^{-4+2\ell/\alpha} \le T^{2H-1} r^{-2H-3+2\ell/\alpha}$ because $r \le T$ and $H\ge 1/2$.
  Hence, combining \eqref{E:LNDt_b2g0_int1} and \eqref{E:LNDt_b2g0_int2},
  we obtain
  \begin{align}\label{E:LNDt_b2g0_I2}
    |I|^2 \le C_1^{-1} \Theta \times C_0 r^{-2H-3+\frac{2\ell}{\alpha}}
  \end{align}
  for some constant $0<C_0<\infty$. Finally, we can conclude
  from~\eqref{E:LNDt_b2g0_I} and~\eqref{E:LNDt_b2g0_I2} that
  \begin{align*}
    \Theta \ge 4\pi^4 C_1 C_0^{-1} r^{2H+1-\frac{2\ell}{\alpha}}.
  \end{align*}
  This proves~\eqref{E:LNDt2_LB} and hence~\eqref{LNDt:CondVar2}.

  We proceed to prove~\eqref{LB:Tinc2}. Denote $\theta \coloneqq
    H+1/2-\ell/\alpha$. The case $\theta < 1$
  follows from~\eqref{E:LNDt2_LB} by taking $n = 1$ and $a_1 = 1$. Suppose
  $\theta \ge 1$, and let $S \le t < t' \le T$. It suffices to consider $0 < t'-
    t \le \delta$ for some fixed constant $\delta \in (0, 1)$.
  As in~\eqref{E:LNDt_b2g0_Var}, write
  \begin{align*}
    \E\left[ \left(u(t', x) - u(t, x) \right)^2\right]
     & = C_1 \int_\R \ud \tau \, |\tau|^{1-2H} \int_\R \ud z \, |z|^{-3+\frac{2\ell}{\alpha}}
    |F(t',\tau,z)-F(t,\tau,z)|^2,
  \end{align*}
  where
  \begin{align*}
    F(t,\tau,z) = \int_0^t e^{-i\tau s} \sin((t-s)z) \ud s
    = -\frac{1}{2} \left( \frac{e^{izt}-e^{-i\tau t}}{\tau+z} - \frac{e^{-izt}-e^{-i\tau t}}{\tau-z} \right).
  \end{align*}
  Use the change of variables $v=(\tau+z)/2$ and $w=(\tau-z)/2$ to get that
  \begin{align*}
    \E\left[ \left(u(t', x) - u(t, x) \right)^2\right]
     & \gtrsim \iint_A \ud v\, \ud w \, |v+w|^{1-2H} |v-w|^{-3+\frac{2\ell}{\alpha}}
    \left| f(t',v,w) - f(t,v,w) \right|^2,
  \end{align*}
  where $A$ is a subset of $\R^2$ to be determined and
  \begin{align*}
    f(t,v,w) = \frac{e^{i(v-w)t}-e^{-i(v+w)t}}{v} - \frac{e^{-i(v-w)t}-e^{-i(v+w)t}}{w}.
  \end{align*}
  % By Taylor expansion, for any $0<t<t'$ and $v,w\in \R$, there exists $t^* = t^*(t,t',v,w) \in [t,t']$ such that
  % \begin{align*}
  %     f(t',v,w) - f(t,v,w) = \partial_t f(t,v,w) (t'-t) + \frac12 \partial_t^2 f(t^*,v,w) (t'-t)^2.
  % \end{align*}
  By Taylor expansion, for any $0<t<t'$ and $v,w\in \R$,
  \begin{align*}
    f(t',v,w) - f(t,v,w) = \partial_t f(t,v,w) (t'-t) + \int_t^{t'}   (t'-s) \partial_t^2 f(s,v,w) \ud s.
  \end{align*}
  This, together with the elementary inequalities $|p+q| \ge |p|-|q|$ for $p,q\in \mathbb{C}$ and $(a-b)^2 \ge \frac12 a^2 - b^2$ for $a, b \in \R_+$, implies that
  \begin{align}\label{E:u-u:t:J1J2}
    \E\left[ \left(u(t', x) - u(t, x) \right)^2\right]
    \gtrsim \frac12 J_1 - J_2,
  \end{align}
  where
  \begin{align*}
     & J_1 = (t'-t)^2 \iint_A \ud v\, \ud w \, |v+w|^{1-2H} |v-w|^{-3+\frac{2\ell}{\alpha}} |\partial_t f(t,v,w)|^2,                             \\
     & J_2 = \iint_A \ud v\, \ud w \, |v+w|^{1-2H} |v-w|^{-3+\frac{2\ell}{\alpha}}  \left|\int_t^{t'} (t'-s) \partial_t^2 f(s,v,w) ds \right|^2.
  \end{align*}
  Next, we derive a lower bound for $J_1$ and an upper bound for $J_2$.
  By direct computation,
  \begin{align*}
    \partial_t f(t,v,w)
    %&= \frac{i(v-w) e^{i(v-w)t} + i(v+w) e^{-i(v+w)t}}{v} - \frac{-i(v-w) e^{-i(v-w)t} + i(v+w) e^{-i(v+w)t}}{w}\\
    %&= i \frac{w[(v-w) e^{i(v-w)t} + (v+w) e^{-i(v+w)t}]+v[(v-w) e^{-i(v-w)t} - (v+w) e^{-i(v+w)t}]}{vw}\\
    %&= i \frac{w(v-w)e^{i(v-w)t} + v(v-w)e^{-i(v-w)t}+(w^2-v^2)e^{-i(v+w)t}}{vw}\\
     & = i(v-w) \frac{we^{i(v-w)t} + v e^{-i(v-w)t}-(v+w)e^{-i(v+w)t}}{vw}
  \end{align*}
  and hence
  \begin{align*}
    |\partial_t f(t,v,w)|^2
     & \ge \frac{|v-w|^2 |\mathrm{Re}(we^{i(v-w)t} + v e^{-i(v-w)t}-(v+w)e^{-i(v+w)t})|^2}{|v|^2 |w|^2} \\
     & = \frac{|v-w|^2 |v+w|^2 \left(\cos((v-w)t)-\cos((v+w)t)\right)^2}{|v|^2 |w|^2}                   \\
     & = \frac{4 |v-w|^2|v+w|^2}{|v|^2|w|^2} \sin^2(vt) \sin^2(wt).
  \end{align*}
  It follows that
  \begin{align*}
    J_1 \gtrsim (t'-t)^2 \iint_A \ud v\, \ud w \, \frac{|v+w|^{3-2H} |v-w|^{-1+\frac{2\ell}{\alpha}}}{|v|^2|w|^2} \sin^2(vt) \sin^2(wt).
  \end{align*}
  Next, we have
  \begin{align*}
    \partial_t^2 f(s,v,w) = \frac{-(v-w)}{vw}\left( w(v-w) e^{i(v-w)s} - v(v-w)e^{-i(v-w)s} + (v+w)^2 e^{-i(v+w)s} \right)
  \end{align*}
  and
  \begin{align*}
    |\partial_t^2 f(s, v, w)|^2
     & = \frac{|v-w|^2}{|v|^2 |w|^2} \left\{ \left[ (v-w)^2 \cos((v-w)s) - (v+w)^2 \cos((v+w)s) \right]^2 \right.                                \\
     & \qquad \left.+ \left[ (v+w)(v-w) \sin((v-w)s) -(v+w)^2 \sin((v+w)s)  \right]^2 \right\}                                                   \\
     & \lesssim \frac{|v-w|^2}{|v|^2|w|^2} \left\{ \left[ (v-w)^2-(v+w)^2 \right]^2 + (v+w)^4 \left[ \cos((v-w)s)-\cos((v+w)s) \right]^2 \right. \\
     & \qquad \left. + (v+w)^2 \left[ 2w \sin((v-w)s) \right]^2 + (v+w)^4\left[ \sin((v-w)s) - \sin((v+w)s)\right]^2 \right\}                    \\
     & \lesssim \frac{|v-w|^2}{|v|^2|w|^2} \left\{ (|v|^2+|w|^2) |w|^2 + (|v|^4+|w|^4) \sin^2(ws)  \right\}                                      \\
     & \lesssim \frac{|v-w|^2}{|v|^2|w|^2} \left\{ (|v|^2+|w|^2) |w|^2 + (|v|^4+|w|^4) |w|^2  \right\},
  \end{align*}
  valid uniformly for all $s \in [t, t'] \subset [S,T]$ and $v,w \in \R$.
  It follows that
  \begin{align*}
    J_2 \lesssim
    (t'-t)^4 \iint_A \ud v\, \ud w \, |v+w|^{1-2H} |v-w|^{-3+\frac{2\ell}{\alpha}}
    \frac{|v-w|^2\left[(|v|^2+|w|^2)|w|^2 + (|v|^4+|w|^4)|w|^2 \right]}{|v|^2|w|^2}
  \end{align*}
  uniformly for all $t<t'$ in $[S,T]$.

  Suppose $\theta = 1$, i.e., $H-\ell/\alpha = 1/2$. Choose
  \begin{align*}
    A = \left\{ (v, w) \in \R^2 : 0 \le w \le \frac{1}{\delta t} \text{ and } \frac{2}{\delta t} \le v \le \frac{2}{(t'-t)t} \right\}.
  \end{align*}
  Then, for all $(v, w) \in A$,
  \begin{align}\label{E:v+w:v-w}
    v \le v+w \le 3v/2, \quad v/2 \le v-w \le v.
  \end{align}
  Hence, for all $t, t' \in [S,T]$ with $0<t'-t\le \delta$,
  \begin{align*}
    J_1 & \gtrsim (t'-t)^2 \int_{2/(\delta t)}^{2/((t'-t)t)} \ud v \, v^{-2H+\frac{2\ell}{\alpha}} \sin^2(vt) \int_0^{1/(\delta t)} \ud w \, \frac{\sin^2(wt)}{w^2} \\
        & \gtrsim (t'-t)^2 \int_{2/\delta}^{2/(t'-t)} \ud v \, \frac{\sin^2(v)}{v}\int_0^{1/\delta} \ud w \,\frac{\sin^2(w)}{w^2}                                   \\
        & \gtrsim (t'-t)^2 \sum_{k=0}^\infty \int_{[2/\delta, 2/(t'-t)]\cap[(k+1/4)\pi,(k+3/4)\pi]} \frac{\sin^2(v)}{v} \ud v                                       \\
        & \gtrsim (t'-t)^2 (1+|\log(t'-t)|).
  \end{align*}
  and
  \begin{align*}
    J_2 & \lesssim
    (t'-t)^4 \int_{2/(\delta t)}^{2/((t'-t)t)} \ud v \int_0^{1/(\delta t)} \ud w \, v^{-2-2H+\frac{2\ell}{\alpha}} \frac{v^2 ( v^4 w^2 )}{v^2 w^2} \\
        & \lesssim
    (t'-t)^4 \int_0^{2/((t'-t)t)} \ud v \, v \int_0^{1/(\delta t)} \ud w                                                                           \\
        & \lesssim (t'-t)^2.
  \end{align*}
  Putting these back into \eqref{E:u-u:t:J1J2} implies that if $\delta$ is sufficiently small, then
  \begin{align*}
    \E\left[ \left( u(t',x)-u(t,x)\right)^2 \right] \gtrsim (t'-t)^2 (1+|\log(t'-t)|)
  \end{align*}
  uniformly for all $t,t'\in [S, T]$ such that $0 < t'-t \le \delta$.

  Suppose $\theta > 1$.
  Choose
  \[
    A = \left\{ (v, w) \in \R^2 : \frac{1}{2t} \le v \le \frac{1}{t} \text{ and } \frac{v}{3} \le w \le \frac{v}{2}\right\}.
  \]
  Then \eqref{E:v+w:v-w} still holds for all $(v, w) \in A$.
  Hence, by using \eqref{E:v+w:v-w} and $\inf_{z \in (0, 1]}|\sin(z)/z|>0$, and recalling that $\theta = H+1/2-\ell/\alpha$, we deduce that for all $t, t' \in [S,T]$ with $0 < t'-t \le \delta$,
  \begin{align*}
    J_1 & \gtrsim (t'-t)^2 \int_{1/(2t)}^{1/t} \ud v \, v^{2-2H+\frac{2\ell}{\alpha}} \frac{\sin^2(vt)}{v^2} \int_{v/3}^{v/2} \ud w \, \frac{\sin^2(wt)}{w^2} \\
        & \gtrsim (t'-t)^2 \int_{1/(2t)}^{1/t} \ud v \, v^{3-2\theta} \int_{v/3}^{v/2} \ud w                                                                  \\
        & \gtrsim (t'-t)^2
  \end{align*}
  and
  \begin{align*}
    J_2 & \lesssim
    (t'-t)^4 \int_{1/(2t)}^{1/t} \ud v \int_{v/3}^{v/2} \ud w \, v^{-2-2H+\frac{2\ell}{\alpha}} \frac{v^2(v^4w^2)}{v^2w^2} \\
        & \lesssim (t'-t)^4 \int_{1/(2t)}^{1/t} \ud v \, v^{3-2\theta} \int_{v/3}^{v/2} \ud w                              \\
        & \lesssim (t'-t)^4.
  \end{align*}
  It follows that for $\delta>0$ small enough,
  \begin{align*}
    \E\left[ \left( u(t',x)-u(t,x) \right)^2 \right] \gtrsim (t'-t)^2
  \end{align*}
  uniformly for all $t,t'\in [S, T]$ such that $0<t'-t\le\delta$.
  This completes the proof for the case of $\beta=2$, $\gamma=0$ for Theorem \ref{T:LND_t} (iii).
\end{proof}

\subsubsection{The $\beta = 2$, $\gamma \in (0, 1]$ case}

\begin{proof}
  We now treat the $\beta = 2$, $\gamma \in (0, 1]$ case. In this case, by \eqref{E:Inner_1}, \eqref{E:u_corr}, \eqref{E:FG}, and \eqref{E:ML:FracInt}, we have \eqref{E:LNDt_b2g0_Var}, where $F$ is replaced by
  \begin{align*}
    F(t, \tau, z) = \frac{1}{\Gamma(\gamma)} \int_0^t \ud s \, e^{-i\tau s} \int_0^{t-s} \ud v \, (t-s-v)^{\gamma-1} \sin(zv).
  \end{align*}
  Let $\phi:\R \to \R$ be a nonnegative smooth test function supported in $[0,
        S/2]$ with $\int \phi = 1$. Let $\psi_r :\R \to \R$ be another test function
  to be determined. Consider the integral
  \begin{align}\label{E:b=2:g>0:int}
    I \coloneqq \frac{1}{\Gamma(\gamma)}\int_\R \ud z \int_\R \ud \tau \overline{\left( F(t, \tau, z) - \sum_{i=1}^n a_j F(t_j, \tau, z) \right)} \widehat{\phi}(\tau-z) \widehat{\psi_r}(z).
  \end{align}
  By Plancherel's theorem in $\tau$ variable,
  \begin{align*}
    I = \frac{2\pi}{\Gamma(\gamma)} \int_\R \ud z \int_\R \ud s
     & \Bigg[ \one_{[0, t](s)} \int_0^{t-s} (t-s-v)^{\gamma-1} \sin(zv)\ud v                                                               \\
     & - \sum_{j=1}^n a_j \one_{[0, t_j]}(s) \int_0^{t_j-s} (t_j-s-v)^{\gamma-1} \sin(zv)\ud v \Bigg] e^{isz} \phi(s) \widehat{\psi_r}(z).
  \end{align*}
  Since $\sin{z} = \frac{1}{2i}(e^{iz} - e^{-iz})$, we can apply inverse Fourier
  transform in $z$ variable and then a change of variable to get that
  \begin{align*}
    I & = -\frac{2\pi^2 i}{\Gamma(\gamma)} \int_0^{S/2} \ud s \, \phi(s) \Bigg[ \int_0^{t-s} (t-s-v)^{\gamma-1} \left( \psi_r(s+v) - \psi_r(s-v) \right) \ud v \\
      & \hspace{100pt}- \sum_{j=1}^n a_j \int_0^{t_j-s} (t_j-s-v)^{\gamma-1} \left( \psi_r(s+v) - \psi_r(s-v) \right) \ud v \Bigg]                             \\
      & = -\frac{2\pi^2 i}{\Gamma(\gamma)} \int_0^{S/2} \ud s \, \phi(s) \Bigg[ \int_s^t (t-v)^{\gamma-1} \left( \psi_r(v) - \psi_r(2s-v) \right) \ud v        \\
      & \hspace{100pt}- \sum_{j=1}^n a_j \int_s^{t_j} (t_j-v)^{\gamma-1} \left( \psi_r(v) - \psi_r(2s-v) \right) \ud v \Bigg].
  \end{align*}
  Take $\Psi:\R \to \R$ to be a nonnegative smooth function supported in
  $[-\delta, \delta]$ with $\delta = \min\{1, S/(2T)\}$ such that $\Psi(0) = 1$,
  and let
  \begin{align*}
    \Psi_r(x) = \Psi\left( \frac{x-t}{r}\right).
  \end{align*}
  Define $\psi_r: \R \to \R$ by the fractional derivative
  \begin{align*}
    \psi_r(x) = \prescript{}{-\infty}D_x^\gamma \Psi_r(x)
    = \frac{\ud^{\lceil \gamma \rceil}}{\ud x^{\lceil \gamma \rceil}} \left( \frac{1}{\Gamma(\gamma)} \int_{-\infty}^x (x-v)^{\lceil \gamma \rceil - \gamma - 1} \Psi_r(v) \ud v \right).
  \end{align*}
  Note that $\psi_r(x) = 0$ for $x \le S/2$ since $\Psi$ is supported in
  $[-\delta, \delta]$, and
  $v \le S/2$ implies $(v-t)/r \le (S/2-S)/(T-S) < -\delta$.
  This implies that $\psi_r(2s-v) = 0$ for $s \in [0, S/2]$ and $v \ge s$, and
  \begin{align*}
    I & = -\frac{2\pi^2 i}{\Gamma(\gamma)} \int_0^{S/2} \ud s \, \phi(s) \Bigg[ \int_{-\infty}^t (t-v)^{\gamma-1}  \psi_r(v) \ud v - \sum_{j=1}^n a_j \int_{-\infty}^{t_j} (t_j-v)^{\gamma-1}  \psi_r(v) \ud v \Bigg] \\
      & = -2\pi^2 i \int_0^{S/2} \ud s \, \phi(s) \left[ \Psi_r(t) - \sum_{j=1}^n a_j \Psi_r(t_j) \right].
  \end{align*}
  Note that $\Psi_r(t_j) = 0$ for $j = 1 ,\dots, n$ since $|t_j-t|/r \ge 1 \ge
    \delta$. Hence
  \begin{align}\label{E:b=2:g>0:I}
    I = -2\pi^2 i \int_0^{S/2} \phi(s) \Psi(0) \ud s = -2\pi^2 i.
  \end{align}
  On the other hand, applying the Cauchy--Schwarz inequality to the integral
  \eqref{E:b=2:g>0:int}, we have
  \begin{align*}
    |I|^2 & \le C_1^{-1} \Theta \times \int_\R \ud \tau \int_\R \ud z\, \left|\widehat{\phi}(\tau-z)\right|^2 \left|\widehat{\psi_r}(z) \right|^2 |\tau|^{2H-1} |z|^{3-\frac{2\ell}{\alpha}}
    % \\& = C_1^{-1} \Theta \times (J_1 + J_2),
  \end{align*}
  % where $J_1$ denotes the integral over $|\tau| \le |z|$ and $J_2$ denotes the integral over $|\tau| > |z|$. 
  Note that
  \begin{align*}
    \psi_r(x) & = \frac{\ud^{\lceil \gamma \rceil}}{\ud x^{\lceil \gamma \rceil}} \left( \frac{1}{\Gamma(\gamma)} \int_{-\infty}^x (x-v)^{\lceil \gamma \rceil - \gamma - 1} \Psi\left(\frac{v-t}{r}\right) \ud v \right)                                               \\
              & = \frac{\ud^{\lceil \gamma \rceil}}{\ud x^{\lceil \gamma \rceil}} \left( \frac{r^{\lceil \gamma \rceil - \gamma}}{\Gamma(\gamma)} \int_{-\infty}^{\frac{x-t}{r}} \left(\frac{x-t}{r}-v\right)^{\lceil \gamma \rceil - \gamma - 1} \Psi(v) \ud v \right) \\
              & = \frac{1}{r^\gamma} f\left( \frac{x-t}{r} \right),
  \end{align*}
  where $f(s) = \prescript{}{-\infty} D_s^\gamma \Psi(s)$. Then
  $|\widehat{\psi_r}(z)| = r^{1-\gamma} |\widehat{f}(rz)|$ and since $H \ge
    1/2$, we have
  % \begin{align*}
  %   J_1 & = r^{2-2\gamma} \int_\R \ud z \int_{|\tau| \le |z|} \ud \tau \, \left| \widehat{\phi}(\tau-z)\right|^2 \left|\widehat{f}(rz) \right|^2 |\tau|^{2H-1} |z|^{3-\frac{2\ell}{\alpha}} \\
  %       & \le r^{2-2\gamma} \int_\R \ud z \left| \widehat{f}(rz)\right|^2 |z|^{2H+2-\frac{2\ell}{\alpha}} \int_\R \ud \tau \left|\widehat{\phi}(\tau) \right|^2                             \\
  %       & = C_2 r^{-2\gamma-2H+\frac{2\ell}{\alpha}-1},
  % \end{align*}
  % where $C_2$ is a finite constant. Moreover,
  \begin{align*}
    % J_2 & = 
     & \int_\R \ud \tau \int_\R \ud z\, \left|\widehat{\phi}(\tau-z)\right|^2 \left|\widehat{\psi_r}(z) \right|^2 |\tau|^{2H-1} |z|^{3-\frac{2\ell}{\alpha}}     \\
     & =r^{2-2\gamma} \int_\R \ud z \int_\R dw \left| \widehat{\phi}(w) \right|^2 \left| \widehat{f}(rz) \right|^2 |z+w|^{2H-1} |z|^{3-\frac{2\ell}{\alpha}}     \\
     & \lesssim r^{2-2\gamma} \left[\int_\R \ud z \int_\R dw \left| \widehat{\phi}(w) \right|^2 \left| \widehat{f}(rz) \right|^2 |z|^{2H+2-\frac{2\ell}{\alpha}}
    + \int_\R \ud z \int_\R dw \left| \widehat{\phi}(w) \right|^2 \left| \widehat{f}(rz) \right|^2 |w|^{2H-1} |z|^{3-\frac{2\ell}{\alpha}}\right]                \\
     & \lesssim r^{2-2\gamma} \left[ r^{-2H-3+\frac{2\ell}{\alpha}} +  r^{-4+\frac{2\ell}{\alpha}} \right]                                                       \\
     & \lesssim r^{-2\gamma-2H + \frac{2\ell}{\alpha}-1}.
  \end{align*}
  It follows that
  \begin{align}\label{E:b=2:g>0:I2}
    |I|^2 \le C_1^{-1} \Theta \times C_0 r^{-2\gamma-2H + \frac{2\ell}{\alpha}-1}
  \end{align}
  for some positive finite constant $C_0$. Combining \eqref{E:b=2:g>0:I} and
  \eqref{E:b=2:g>0:I2} yields
  \begin{align*}
    \Theta \ge 4\pi^4 C_1 C_0^{-1} r^{2\gamma+2H-\frac{2\ell}{\alpha}+1} = C r^{2\rho}
  \end{align*}
  when $\beta = 2$ and $\gamma \in (0, 1]$. This proves \eqref{LNDt:CondVar2}.

  We turn to the proof of \eqref{LB:Tinc2}. The case of  $\theta \coloneqq \gamma + H +
    \frac{1}{2} - \frac{\ell}{\alpha} < 1$ follows from
  \eqref{LNDt:CondVar2}. It remains to prove \eqref{LB:Tinc2} for $\theta \ge
    1$. Let $S \le t < t' \le T$.
  As in~\eqref{E:LNDt_b2g0_Var}, we write
  \begin{align*}
    \E\left[\left(u(t',x)-u(t,x)\right)^2\right]
    = C_1 \int_\R \ud \tau \, |\tau|^{1-2H} \int_\R \ud z \, |z|^{-3+\frac{2\ell}{\alpha}} \left|F(t',\tau,z)-F(t,\tau,z) \right|^2,
  \end{align*}
  where
  \begin{align*}
    F(t,\tau,z) = \frac{e^{-i\tau t}}{\Gamma(\gamma)} \int_0^t \ud s \, e^{i\tau s} \int_0^s \ud r \, (s-r)^{\gamma-1} \sin(zr).
  \end{align*}
  Then, by the change of variables $v=(\tau+z)/2$, $w=(\tau-z)/2$, Taylor expansion, and the elementary inequality $|p+q|^2 \ge \frac12 |p|^2 - |q|^2$ for $p,q \in \mathbb{C}$,
  \begin{align}\begin{split}\label{E:LNDt:b2g>0:J1J2}
       & \E\left[ \left( u(t',x)-u(t,x)\right)^2\right]                                                                     \\
       & \gtrsim \iint_A \ud v \, \ud w \, |v-w|^{1-2H} |v+w|^{-3+\frac{2\ell}{\alpha}} \left| f(t',v,w)-f(t,v,w) \right|^2
      \gtrsim \frac12 J_1 - J_2,
    \end{split}\end{align}
  where $A$ is a subset of $\R^2$ to be determined,
  \begin{align*}
     & f(t,v,w) = e^{-i(v+w)t} \int_0^t \ud s \, e^{i(v+w)s} \int_0^s \ud r \, (s-r)^{\gamma-1} \sin((v-w)r),                                      \\
     & J_1 = (t'-t)^2 \iint_A \ud v\, \ud w \, |v+w|^{1-2H} |v-w|^{-3+\frac{2\ell}{\alpha}} |\partial_t f(t,v,w)|^2,                               \\
     & J_2 = \iint_A \ud v\, \ud w \, |v+w|^{1-2H} |v-w|^{-3+\frac{2\ell}{\alpha}}  \left|\int_t^{t'} (t'-h) \partial_t^2 f(h,v,w) \ud h\right|^2.
  \end{align*}
  It is easy to see that
  \begin{align*}
    \partial_t f(t,v,w)
    = -i(v+w)f(t,v,w) + \frac{1}{\Gamma(\gamma)} \int_0^t (t-s)^{\gamma-1} \sin((v-w)s) \ud s.
  \end{align*}
  In particular, we can use the identity $\sin(z)=\frac{1}{2i}(e^{iz}-e^{-iz})$ to deduce that
  \begin{align*}
     & f(t,v,w)                                                                                                                                \\
     & = \frac{e^{-i(v+w)t}}{2i}\int_0^t \ud s \, e^{i(v+w)s} \int_0^s \ud r \, r^{\gamma-1} (e^{i(v-w)(s-r)}-e^{-i(v-w)(s-r)})                \\
     & =\frac{-1}{4} \left[ \frac{1}{v} \int_0^t s^{\gamma-1} e^{i(v-w)(t-s)} \ud s - \frac{1}{w} \int_0^t s^{\gamma-1} e^{-i(v-w)(t-s)} \ud s
      - \left( \frac{1}{v}-\frac{1}{w} \right) \int_0^t s^{\gamma-1} e^{-i(v+w)(t-s)} \ud s \right]
  \end{align*}
  and hence
  \begin{align*}
    |\partial_t f(t,v,w)|^2
     & \gtrsim \frac{|v+w|^2|v-w|^2}{|v|^2|w|^2} \left( \int_0^t s^{\gamma-1} \left( \cos((v-w)(t-s)) -\cos((v+w)(t-s)) \right) \ud s \right)^2 \\
     & \gtrsim \frac{|v+w|^2|v-w|^2}{|v|^2|w|^2} \left( \int_0^t s^{\gamma-1} \sin(v(t-s))\sin(w(t-s)) \ud s \right)^2.
  \end{align*}
  This together with a simple scaling implies that
  \begin{align}\label{E:LNDt:b2g>0:J1}
    J_1 \gtrsim (t'-t)^2 \iint_{\tilde A} \ud v \, \ud w \, |v+w|^{1-2H} |v-w|^{-3+\frac{2\ell}{\alpha}}
    \frac{|v+w|^2|v-w|^2}{|v|^2|w|^2} |I(v,w)|^2
  \end{align}
  uniformly for all $t<t'$ in $[S,T]$, where
  \[
    \tilde{A} = \{ (v,w) \in \R^2 : (v/t, w/t) \in A \}
  \]
  and
  \begin{align}\label{E:I(v,w)}
    I(v,w) := \int_0^1 (1-s)^{\gamma-1} \sin(vs)\sin(ws) \ud s
    = v^{-\gamma} \int_0^v (v-s)^{\gamma-1} \sin(s) \sin(ws/v)\ud s.
  \end{align}
  Next, for any $h \in [t, t']$ and $v, w \in \R$,
  \begin{align*}
    \partial_t^2 f(h, v,w)
     & = \frac{-i}{4}\left[ \frac{v-w}{v} \int_0^h s^{\gamma-1} e^{i(v-w)(h-s)} \ud s + \frac{v-w}{w} \int_0^h s^{\gamma-1} e^{-i(v-w)(h-s)} \ud s \right. \\
     & \qquad \left. + \left( \frac{1}{v}-\frac{1}{w} \right) (v+w) \int_0^h s^{\gamma-1} e^{-i(v+w)(h-s)} \ud s \right]
  \end{align*}
  and hence
  \begin{align*}
    |\partial_t^2 f(h, v,w)|^2 & \lesssim \frac{|v+w|^2|v-w|}{|v|^2|w|^2} \left( \int_0^h s^{\gamma-1} \left( \cos((v+w)(h-s))-\cos((v-w)(h-s)) \right) \ud s \right)^2 \\
                               & \quad + |v-w|^2 \left(\int_0^h s^{\gamma-1} \cos((v-w)(h-s)) \ud s \right)^2                                                           \\
                               & \quad + \frac{|v-w|^2|v+w|}{|v|^2|w|^2} \left(\int_0^h s^{\gamma-1} \left( \sin((v-w)(h-s)) - \sin((v+w)(h-s)) \right) \ud s \right)^2 \\
                               & \lesssim \frac{|v+w|^2|v-w|}{|v|^2|w|^2} \left( \int_0^h s^{\gamma-1} \sin(v(h-s))\sin(w(h-s)) \ud s \right)^2                         \\
                               & \quad + |v-w|^2 + \frac{|v-w|^2|v+w|}{|v|^2|w|^2} \left(\int_0^h s^{\gamma-1} \sin(w(h-s)) \cos(v(h-s))\ud s \right)^2                 \\
                               & \lesssim \frac{|v-w|}{|v|^2}\left(|v+w|^2 + |v-w| |v|^2 + |v-w||v+w|\right).
  \end{align*}
  It follows that
  \begin{align}\begin{split}\label{E:LNDt:b2g>0:J2}
      J_2 & \lesssim (t'-t)^4                                                                                        \\
          & \quad \times\iint_{\tilde A} \ud v\, \ud w \, \frac{|v+w|^{1-2H} |v-w|^{-2+\frac{2\ell}{\alpha}}}{|v|^2}
      \left( |v+w|^2 + |v-w| |v|^2 + |v-w||v+w| \right).
    \end{split}\end{align}

  Recall \eqref{E:I(v,w)}.
  We claim that there exist a small number $\delta>0$ and an interval $L \subset [0,\infty)$ such that for all $t,t'\in[S,T]$ with $0<t'-t<\delta$, for all $v\in [1/\delta,\infty) \cap \bigcup_{n=1}^\infty [2n\pi - 3\pi/8, 2n\pi-\pi/8]$ and $w \in L$,
        \begin{align}\label{E:LNDt:b2g>0:claim}
          |I(v, w)| \ge C \left( (v-w)^{-\gamma} - (v+w)^{-\gamma} \right) (v+w),
        \end{align}
        where $C>0$ is a constant. Indeed, if $\gamma \in (0, 1]$, by the change of variable $s \to v-s$, product-to-sum formula,
  and angle-sum formula,
  \begin{align*}
    I(v, w)
     & = \frac{v^{-\gamma}}{2} \int_0^v s^{\gamma-1} \cos((1-w/v)(v-s)) \ud s - \frac{v^{-\gamma}}{2} \int_0^v s^{\gamma-1} \cos((1+w/v)(v-s)) \ud s            \\
     & = \frac{v^{-\gamma}}{2}  \left[ \cos(v-w)\int_0^v s^{\gamma-1} \cos((1-w/v)s) \ud s - \cos(v+w)\int_0^v s^{\gamma-1} \cos((1+w/v)s) \ud s \right]        \\
     & \quad + \frac{v^{-\gamma}}{2}  \left[\sin(v-w) \int_0^v s^{\gamma-1} \sin((1-w/v)s) \ud s - \sin(v+w)\int_0^v s^{\gamma-1} \sin((1+w/v)s) \ud s \right].
  \end{align*}
  Then, by changing variables,
  \begin{align*}
    I(v, w)
     & = \frac{1}{2}\left[ \cos(v-w) (v-w)^{-\gamma} \int_0^{v-w} s^{\gamma-1} \cos(s) \ud s - \cos(v+w) (v+w)^{-\gamma} \int_0^{v+w} s^{\gamma-1} \cos(s) \ud s \right]       \\
     & \quad + \frac{1}{2} \left[\sin(v-w) (v-w)^{-\gamma} \int_0^{v-w} s^{\gamma-1} \sin(s) \ud s - \sin(v+w)  (v+w)^{-\gamma} \int_0^{v+w} s^{\gamma-1} \sin(s) \ud s\right] \\
     & = \frac{1}{2} \left( (v-w)^{-\gamma} - (v+w)^{-\gamma}\right) (v+w)                                                                                                     \\
     & \quad \times \left[ A_1 + A_2 + \left(\cos(v-w) - \cos(v+w) \right) A_3+ B_1+B_2+\left(\sin(v-w) - \sin(v+w) \right)B_3 \right],
  \end{align*}
  where
  \begin{align*}
    A_1 & = \cos(v-w) (v+w)^{-1} \int_0^{v-w} s^{\gamma-1} \cos(s) \ud s,                                                    \\
    A_2 & = - \cos(v-w) \frac{(v+w)^{-1}\int_{v-w}^{v+w} s^{\gamma-1}\cos(s)\ud s}{\left(\frac{v+w}{v-w}\right)^{\gamma}-1}, \\
    A_3 & = \frac{(v+w)^{-1} \int_0^{v+w} s^{\gamma-1} \cos(s) \ud s}{\left( \frac{v+w}{v-w} \right)^{\gamma}-1 },
  \end{align*}
  and $B_1, B_2, B_3$ are defined in the same way but with $\cos$ replaced by $\sin$.
  We may take
  \[
    L = [\pi/8, \pi/4]
  \]
  so that $\cos(v-w) - \cos(v+w) \le -c_0 < 0$ and $\sin(v-w) -
    \sin(v+w) \le -c_0 < 0$ uniformly for all $v$ with $v \in [1/\delta, \infty) \cap  \bigcup_{n=1}^\infty [2n\pi - 3\pi/8, 2n\pi - \pi/8]$ and all $w\in L$. From the proof of
        Lemma \ref{L:F_tht}, we see that $\int_0^x s^{\gamma-1} \cos(s) \ud s$ and
      $\int_0^x s^{\gamma-1} \sin(s) \ud s$ both converge to a positive limit as $x
      \to +\infty$, so $A_1, A_2, B_1, B_2 \to 0$.
        % By l'Hospital's rule, $A_3, B_3 \to +\infty$ as $v \to \infty$ uniformly for $w \in L$. 
        Also, we have
        \begin{align*}
          A_3 & = \frac{(v^{-1} + O(v^{-2})) \int_0^{v+w} s^{\gamma-1}\cos(s) \ud s}{2\gamma w v^{-1} + O(v^{-2})} \to \frac{\int_0^\infty s^{\gamma-1} \cos(s) \ud s}{2\gamma w} > 0, \\
          B_3 & = \frac{(v^{-1} + O(v^{-2})) \int_0^{v+w} s^{\gamma-1}\sin(s) \ud s}{2\gamma w v^{-1} + O(v^{-2})} \to \frac{\int_0^\infty s^{\gamma-1}\sin(s) \ud s}{2\gamma w} > 0
        \end{align*}
        as $v \to \infty$, uniformly for all $w \in L$.
        It follows that
        \[
          A_1 + A_2 + \left(\cos(v-w) - \cos(v+w) \right) A_3+ B_1+B_2+\left(\sin(v-w) - \sin(v+w) \right)B_3
        \]
        remains bounded above by a strictly negative number $-C<0$ uniformly for all large-enough $v \in \bigcup_{n=1}^\infty [2n\pi - 3\pi/8, 2n\pi - \pi/8]$ and all $w \in L$.
        Hence, the
        claim \eqref{E:LNDt:b2g>0:claim} holds for some sufficiently small number $\delta>0$.
        %If $\gamma = 1$, the proof of \eqref{E:LNDt:b2g>0:claim} is much simpler and thus omitted.

        Suppose $\theta=1$.
        Choose
        \begin{align*}
          \tilde A = \left\{ (v,w)\in \R^2 : \frac{1}{\delta} \le v \le \frac{1}{t'-t} \text{ and } \frac{\pi}{8} \le w \le \frac{\pi}{4}\right\}.
        \end{align*}
        Note that if $\delta>0$ is small enough, then $v/2 \le v\pm w \le 3v/2$ for all $(v,w)\in \tilde A$.
        Hence, by \eqref{E:LNDt:b2g>0:J1}, \eqref{E:LNDt:b2g>0:claim}, and $\theta = \gamma+H+1/2-\ell/\alpha = 1$,
        \begin{align*}
          J_1 & \gtrsim (t'-t)^2 \sum_{n=1}^\infty \int_{[1/\delta, 1/(t'-t)]\cap [2n\pi-3\pi/8, 2n\pi - \pi/8]} \ud v \, v^{-2\gamma-2H+\frac{2\ell}{\alpha}} \int_{\pi/8}^{\pi/4} \ud w \\
              & \gtrsim (t'-t)^2 \sum_{n=1}^\infty \int_{[1/\delta, 1/(t'-t)]\cap [2n\pi-3\pi/8, 2n\pi - \pi/8]} \ud v \, v^{-1}                                                          \\
              & \gtrsim (t'-t)^2 (1+|\log(t'-t)|).
        \end{align*}
        On the other hand, \eqref{E:LNDt:b2g>0:J2} and $\gamma \in (0,1]$ imply that
  \begin{align*}
    J_2 & \lesssim (t'-t)^4 \int_{1/\delta}^{1/(t'-t)} \ud v \, v^{-2H+\frac{2\ell}{\alpha}} \int_{\pi/8}^{\pi/4}\ud w \\
        & \lesssim (t'-t)^4 \int_{1/\delta}^{1/(t'-t)} \ud v\, v^{2\gamma-1}                                           \\
        & \lesssim (t'-t)^{4-2\gamma}.
  \end{align*}
  Putting together the preceding two estimates and \eqref{E:LNDt:b2g>0:J1J2} shows that if $\delta>0$ is sufficiently small, then
  \begin{align*}
    \E\left[ \left( u(t',x)-u(t,x)\right)^2 \right]
    \gtrsim (t'-t)^2 (1+|\log(t'-t)|)
  \end{align*}
  uniformly for all $t,t'\in[S,T]$ such that $0<t'-t\le\delta$.

  Suppose $\theta>1$. Choose
  \[
    \tilde A = \left\{ (v,w)\in\R^2 : \frac{1}{\delta} \le v \le \frac{2}{\delta} \text{ and } \frac{\pi}{8} \le w \le \frac{\pi}{4} \right\}.
  \]
  If $\delta>0$ is small enough, then again we have $v/2 \le v\pm w \le 3v/2$ for all $(v,w)\in \tilde A$, and by \eqref{E:LNDt:b2g>0:J1} and \eqref{E:LNDt:b2g>0:claim},
  \begin{align*}
    J_1 & \gtrsim (t'-t)^2 \sum_{n=1}^\infty \int_{[1/\delta,2/\delta] \cap [2n\pi-3\pi/8, 2n\pi-\pi/8]} \ud v \, v^{-2\gamma-2H+ \frac{2\ell}{\alpha}} \int_{\pi/8}^{\pi/4} \ud w \\
        & \gtrsim (t'-t)^2 \sum_{n=1}^\infty \int_{[1/\delta,2/\delta] \cap [2n\pi-3\pi/8, 2n\pi-\pi/8]} \ud v\, v^{1-2\theta}                                                     \\
        & \gtrsim (t'-t)^2.
  \end{align*}
  On the other hand, by \eqref{E:LNDt:b2g>0:J2},
  \begin{align*}
    J_2 \lesssim (t'-t)^4 \int_{1/\delta}^{2/\delta} \ud v \, v^{2\gamma-1}\int_{\pi/8}^{\pi/4} \ud w
    \lesssim (t'-t)^4.
  \end{align*}
  Putting these back into \eqref{E:LNDt:b2g>0:J1J2}, we conclude that if $\delta>0$ is sufficiently small, then
  \begin{align*}
    \E\left[\left( u(t',x)-u(t,x)\right)^2\right]
    \gtrsim (t'-t)^2
  \end{align*}
  uniformly for all $t,t'\in[S,T]$ with $0<t'-t\le\delta$.
  The proof of Theorem \ref{T:LND_t} is complete.
\end{proof}

\section{Strong local nondeterminism in $x$ and moment lower bounds}\label{S:SLN_x}
\index{strong local nondeterminism (SLND)}\index{strong local nondeterminism (SLND)!spatial}

This section is devoted to proving the following theorem, which establishes
strong local nondeterminism in the variable $x$, with $t$ fixed, and matching
moment lower bounds for spatial increments.

\subsection*{Main results}
\addcontentsline{toc}{subsection}{Main results}

For ease of reference, we state the main result of this section as
Theorem~\ref{T:LND_x}. Its proof is split into
Subsections~\ref{SS:LNDx-i} and~\ref{SS:LNDx-ii}.
Recall $\rho_1, \tilde\rho_1, \rho_2, \tilde\rho_2$ defined in \eqref{E:rho1-rho2} and \eqref{E:trho1-trho2}, i.e.,
\begin{align*}
  \rho_1 \coloneqq \beta + \gamma + H -\frac{\ell \beta}{2\alpha} - 1, \quad \rho_2 \coloneqq \gamma + H - \ell/\alpha + 1/2,
\end{align*}
\begin{align*}
  \widetilde{\rho}_1 \coloneqq \min
  \{\alpha\rho_1/\beta, \alpha - \ell/2\}, \quad \text{and} \quad \widetilde{\rho}_2 \coloneqq \min
  \{\alpha\rho_2/2, \alpha - \ell/2\}.
\end{align*}

\begin{theorem}\label{T:LND_x}%
  \index{conditional variance}%
  Assume $\alpha > 0$, $\beta \in (0, 2]$, $\gamma \ge 0$, $H \in (0, 1)$, and
  $\ell \in (0, 2d)$. Assume the Dalang condition \eqref{E:main}, namely, $\rho >
    0$ and $\ell < 2\alpha$. Then, for all $0 < S < T < \infty$ and $1 \le M <
    \infty$, there exist positive finite constants $C = C(S,T,M, \alpha, \beta,
    \gamma, H, d, \ell)$ and $C' = C'(S,T,M, \alpha, \beta, \gamma, H, d, \ell)$
  such that the following statements hold.

  \begin{enumerate}[\rm (i)]

    \item If either $\beta \in (0, 2)$ or $(\beta, \gamma) \in \{2\} \times [1,
            \infty)$, then for all $t \in [S, T]$, for all integers $n \ge 1$, and for
          all $x, x_1, \dots, x_n \in B(0, M) \subset \R^d$, we have
          \begin{align}\label{E:LNDx_CV_LB}
            \Var\left(u(t, x)\big|u(t, x_1), \dots, u(t, x_n)\right)
            \ge C \min_{1 \le j \le n}|x-x_j|^{2\tilde\rho_1},
          \end{align}
          and for all $t \in [S, T]$, for all $x, y \in B(0, M)$, we have
          \begin{align}\label{E:LNDx_V_LB}
            \E\left[\left(u(t, x) - u(t, y)\right)^2\right] \ge
            \begin{cases}
              C'|x-y|^{2\tilde\rho_1} & \text{if } \tilde\rho_1 \in (0, 1), \\
              C'|x-y|^2 (1+|\log|x-y||)                            & \text{if } \tilde\rho_1 = 1,        \\
              C'|x-y|^2                                            & \text{if } \tilde\rho_1 > 1.
            \end{cases}
          \end{align}

    \item If $\beta = 2$ and $\gamma \in [0,1)$, then for all $t \in [S, T]$,
          for all integers $n \ge 1$, for all $x, x_1, \dots, x_n \in B(0, M)
            \subset \R^d$, we have
          \begin{align}\label{E:LNDx2_CV_LB}
            \Var\left( u(t, x) \big| u(t, x_1), \dots, u(t, x_n) \right) \ge C \min_{1\le j \le n}|x-x_j|^{2\tilde\rho_2},
          \end{align}
          and for all $t \in [S, T]$, for all $x, y \in B(0, M)$, we have
          \begin{align}\label{E:LNDx2_V_LB}
            \E\left[ \left( u(t, x) - u(t, y) \right)^2\right] \ge
            \begin{cases}
              C' |x-y|^{2\tilde\rho_2} & \text{if } \tilde\rho_2 \in (0, 1), \\
              C' |x-y|^2 (1+|\log|x-y||)           & \text{if } \tilde\rho_2 = 1,        \\
              C' |x-y|^2                           & \text{if } \tilde\rho_2 > 1.
            \end{cases}
          \end{align}

  \end{enumerate}
\end{theorem}

% \begin{remark}
%     Recall $\rho_2$ and $\tilde\rho_2$ defined
%     in~\eqref{E:rho1-rho2}--\eqref{E:trho1-trho2}. When $H \in [1/2,1)$, $\beta
%     = 2$ and $\gamma\in[0,1)$, we have $\alpha\rho_2 = \alpha(\gamma+H+1/2)-\ell
%     < 2\alpha-\ell$, and thus
%     \[
%       \alpha(\gamma+H+1/2)-\ell
%       = \min\{ \alpha\rho_2/2, \alpha-\ell/2\}
%       = \tilde\rho_2.
%     \]
%     Hence, the lower bounds in~\eqref{E:LNDx_V_LB} and~\eqref{E:LNDx2_V_LB}
%     match with the upper bounds in~\eqref{E:ups} and~\eqref{E:ups1}.
% \end{remark}

\subsection{Proof of part (i) of Theorem \ref{T:LND_x}}\label{SS:LNDx-i}

We need two technical lemmas. Their proofs are given in Section~\ref{S7:lemma}.

\begin{lemma}\label{L:TestFun}
  For any $m\ge 1$, there exists $g\in C_c^\infty(\R^d)$, i.e., a smooth
  function with compact support, such that $g(0)>0$ and
  $|\widehat{g}(\xi)|\lesssim |\xi|^{2m}$ for $\xi$ near zero.
\end{lemma}

\begin{lemma}\label{L:Delta_t}
  For any $t > 0$ and $\lambda > 0$, let
  \begin{align}\label{E:Delta_t}
    \Delta_t(\lambda) \coloneqq \int_\R \ud \tau \, |\tau|^{1-2H} \left| \int_0^t \ud s \, e^{i\tau s} s^{\beta+\gamma-1} E_{\beta, \beta+\gamma}\left(-2^{-1}\nu\lambda s^\beta\right) \right|^2.
  \end{align}
  For all $0 < S < T < \infty$, the following statements hold:
  \begin{enumerate}[\rm (i)]

    \item If $(\beta, \gamma) \in (0, 2] \times [0, \infty)$, then there exists
          a positive finite constant $C_{S, T}$ such that
          \begin{align}\label{E:Delta_t-i}
            \Delta_t(\lambda) \ge C_{S, T} \frac{\lambda^{2(1-H)/\beta}}{1+\lambda^{2(\beta+\gamma)/\beta}}
            \quad \text{for all } t \in [S, T] \text{ and } \lambda > 0.
          \end{align}

    \item If $(\beta, \gamma) \in (0, 2) \times (0, \infty)$ or $(\beta, \gamma)
            \in \{2\} \times [1, \infty)$, then there exists a positive finite
          constant $\widetilde{C}_{S, T}$ such that
          \begin{align}\label{E:Delta_t-ii}
            \Delta_t(\lambda) \ge \widetilde{C}_{S, T} \max\left\{ \frac{\lambda^{2(1-H)/\beta}}{1+\lambda^{2(\beta+\gamma)/\beta}}, \frac{1}{1+\lambda^2} \right\}
            \quad \text{for all } t \in [S, T] \text{ and } \lambda > 0.
          \end{align}

  \end{enumerate}
\end{lemma}

Assuming these two lemmas, we can now prove part (i) of Theorem~\ref{T:LND_x}.

\begin{proof}[Proof of part (i) of Theorem~\ref{T:LND_x}]
  Similarly to the proof of Theorem \ref{T:LND_t}, in order to prove~\eqref{E:LNDx_CV_LB}, it suffices to prove that there exists a
  positive finite constant $C$ such that for all $t \in [S, T]$, for all
  integers $n \ge 1$, for all $x, x_1, \dots, x_n \in \R^d$, for all $a_1,
    \dots, a_n \in \R$,
  \begin{align}\label{LNDx:LB}
    \Theta
    \coloneqq \E\left[ \left( u(t, x) - \sum_{j=1}^n a_j u(t, x_j) \right)^2 \right]
    \ge C r^{\min(2\alpha \rho /\beta, \, 2\alpha-\ell)}, \quad \text{where } r = \min_{1 \le j \le n}\left|x-x_j\right|.
  \end{align}
  To this end, we first use \eqref{E:Inner_1}, \eqref{E:u_corr}, and \eqref{E:FG} to write
  \begin{align}\label{LNDx:Theta}
    \begin{split}
      \Theta & = C_{H, d} \int_{\R} \ud \tau \, |\tau|^{1-2H} \int_{\R^d} \ud \xi \, |\xi|^{\ell-d} \left| 1 - \sum_{j=1}^n a_j e^{i\left(x-x_j\right)\cdot \xi} \right|^2 \\
             & \quad \times \left| \int_0^t e^{i\tau s} s^{\beta+\gamma-1} E_{\beta, \beta+\gamma}\left(-2^{-1}\nu |\xi|^\alpha s^\beta\right)\, \ud s \right|^2.
    \end{split}
  \end{align}
  Applying the lower bound~\eqref{E:Delta_t-i} with $\lambda = |\xi|^\alpha$
  to~\eqref{LNDx:Theta}, we have that
  \begin{align}\label{E:LNDx_Th_LB1}
    \Theta \ge C_{H, d}\; C_{S, T} \int_{\R^d} \left|1 - \sum_{j=1}^n a_j e^{i(x-x_j)\cdot \xi} \right|^2 |\xi|^{\ell-d+\frac{2\alpha}{\beta}(1-H)}\left(1 \wedge |\xi|^{-\frac{2\alpha}{\beta}(\beta+\gamma)}\right)\ud \xi.
  \end{align}
  Let $m\ge 2$ be an even number to be determined later. Let
  $g: \R^d \to \R$ be a smooth test function obtained from
  Lemma~\ref{L:TestFun}. By a suitable scaling, we may assume that $g$ is
  supported on the unit ball with $g(0) = 1$. Define $g_r(y) \coloneqq
    r^{-d}g\left(r^{-1}y\right)$. Consider the integral
  \begin{align}\label{E:LNDxI}
    I \coloneqq \int_{\R^d} \widehat{g_r}(\xi) \left( 1 -\sum_{j=1}^n a_j e^{i(x-x_j)\cdot \xi} \right) \ud \xi.
  \end{align}
  By the inverse Fourier transform, we see that
  \begin{align*}
    I = (2\pi)^d \left( g_r(0) - \sum_{j=1}^n a_j g_r(x-x_j)\right).
  \end{align*}
  By the definition of $r$, for each $j \in \{1, \dots, n\}$, we have $|x-x_j|
    \ge r$ and hence $g_r(x-x_j) = 0$. It follows that
  \begin{align}\label{E:LNDx-I}
    I = (2\pi)^d g_r(0) = (2\pi)^d r^{-d}.
  \end{align}
  On the other hand, applying the Cauchy--Schwarz inequality and
  $\widehat{g_r}(\xi) = \widehat{g}(r\xi)$, and using \eqref{E:LNDx_Th_LB1}, we get that
  \begin{align*}
    I^2 & \le C_{H, d}^{-1}  C_{S, T}^{-1} \Theta \times \int_{\R^d} |\widehat{g}(r\xi)|^2 \frac{|\xi|^{d-\ell-\frac{2\alpha}{\beta}(1-H)}}{1 \wedge |\xi|^{-\frac{2\alpha}{\beta}(\beta+\gamma)}} \ud \xi                                                               \\
        & =   C_{H, d}^{-1}  C_{S, T}^{-1} \Theta \times \left[\int_{|\xi| \le 1} |\widehat{g}(r\xi)|^2 |\xi|^{d-\ell-\frac{2\alpha}{\beta}(1-H)}\ud \xi + \int_{|\xi| > 1} |\widehat{g}(r\xi)|^2 |\xi|^{d-\ell+\frac{2\alpha}{\beta}(\beta+\gamma-1+H)} \ud \xi\right].
  \end{align*}
  Then, by the change of variable $\xi \mapsto r^{-1}\xi$,
  \begin{align*}
    I^2 \le C_{H, d}^{-1}  C_{S, T}^{-1} \Theta\times
     & \left[ r^{-2d +\ell + \frac{2\alpha}{\beta}(1-H)} \int_{|\xi| \le r} \left|\widehat{g}(\xi)\right|^2 |\xi|^{d-\ell-\frac{2\alpha}{\beta}(1-H)} \ud \xi \right.                   \\
     & \quad + \left. r^{-2d+\ell-\frac{2\alpha}{\beta}(\beta+\gamma-1+H)} \int_{|\xi| > r} |\widehat{g}(\xi)|^2 |\xi|^{d-\ell+\frac{2\alpha}{\beta}(\beta+\gamma-1+H)} \ud \xi\right].
  \end{align*}
  Since $x, x_1, \dots, x_n \in B(0, M)$, we have $r \le 2M$ and thus
  \begin{align}\label{E:LNDx-Isq}
    I^2 \le \Theta \times r^{-2d+\ell-\frac{2\alpha}{\beta}(\beta+\gamma-1+H)}\left(C_1 + C_2\right),
  \end{align}
  where
  \begin{align*}
    C_1 & \coloneqq C_{H, d}^{-1}  C_{S, T}^{-1} (2M)^{\frac{2\alpha}{\beta}(\beta+\gamma)} \int_{|\xi| \le 2M} |\widehat{g}(\xi)|^2 |\xi|^{d-\ell-\frac{2\alpha}{\beta}(1-H)} \ud \xi, \\
    C_2 & \coloneqq C_{H, d}^{-1}  C_{S, T}^{-1} \int_{\R^d} |\widehat{g}(\xi)|^2 |\xi|^{d-\ell+\frac{2\alpha}{\beta}(\beta+\gamma-1+H)} \ud \xi.
  \end{align*}
  By Lemma~\ref{L:TestFun},
  \begin{align*} % \label{E:LNDx_g}
    |\widehat{g}(\xi)|\lesssim |\xi|^{2m} \quad \text{for } |\xi| \text{ small}.
  \end{align*}
  Hence, by choosing an even $m\ge 2$ such that
  \begin{align*}
    m >  -\frac{d}{2}+\frac{\alpha -\alpha  H}{2 \beta }+\frac{\ell}{4}\;,
  \end{align*}
  we have $d-\ell-\frac{2\alpha}{\beta}(1-H) + 4m > -d$ and therefore ensure
  that $C_1$ is finite. As for $C_2$, the Dalang condition~\eqref{E:main} implies that
  $d-\ell+\frac{2\alpha}{\beta}(\beta+\gamma-1+H) \ge d > -d$. Hence,
  considering that $\widehat{g}$ is a rapidly decreasing function, the constant
  $C_2$ is also finite. Finally,~\eqref{E:LNDx-I} and~\eqref{E:LNDx-Isq}
  together imply that
  \begin{align*}
    \Theta \ge (2\pi)^{2d} \left(C_1+C_2\right)^{-1} r^{\frac{2\alpha}{\beta}(\beta+\gamma-1+H)-\ell}
    = (2\pi)^{2d} \left(C_1+C_2\right)^{-1} r^{2\alpha\rho/\beta}.
  \end{align*}
  This proves~\eqref{LNDx:LB} when $\min(2\alpha\rho/\beta, 2\alpha-\ell) =
    2\alpha\rho/\beta$.

  It remains to show~\eqref{LNDx:LB} for the case $\min(2\alpha\rho/\beta,
    2\alpha-\ell) = 2\alpha - \ell$. In this case, it can be readily seen that
  $\gamma > 0$, so one of the following conditions holds:
  \begin{itemize}
    \item $(\beta, \gamma) \in (0, 2) \times (0, \infty)$;
    \item $(\beta, \gamma) \in \{2\}  \times [1, \infty)$.
  \end{itemize}
  Hence, we can apply~\eqref{E:Delta_t-ii} with $\lambda = |\xi|^\alpha$
  to~\eqref{LNDx:Theta} to get that
  \begin{align}\label{E:LNDx_Th_LB2}
    \Theta \ge C_{H, d} \widetilde{C}_{S, T} \int_{\R^d} \left| 1-\sum_{j=1}^n a_j e^{i(x-x_j)\cdot \xi} \right|^2 \frac{|\xi|^{\ell-d}}{1+|\xi|^{2\alpha}} \ud \xi.
  \end{align}
  Let $h: \R^d \to \R$ be a smooth test function supported on the unit ball with
  $h(0) = 1$. Define $h_r(y) = r^{-d} h(r^{-1}y)$ and consider the integral
  \begin{align*} % \label{E:LNDxJ}
    J \coloneqq \int_{\R^d} \widehat{h_r}(\xi) \left( 1 - \sum_{j=1}^n a_j e^{i(x-x_j)\cdot \xi} \right) \ud \xi.
  \end{align*}
  By a similar argument leading~\eqref{E:LNDxI} to~\eqref{E:LNDx-I}, we get
  that
  \begin{align}\label{E:LNDx-J1}
    J = (2\pi)^d r^{-d}.
  \end{align}
  On the other hand, by the Cauchy--Schwarz inequality, we have
  \begin{align*}
    J^2 & \le C_{H, d}^{-1} \widetilde{C}_{S, T}^{-1} \Theta \int_{\R^d} |\widehat{h}(r\xi)|^2 |\xi|^{d-\ell}(1+|\xi|^{2\alpha})   \\
        & \le 2C_{H, d}^{-1} \widetilde{C}_{S, T}^{-1} \Theta \left[ \int_{|\xi|\le 1} |\widehat{h}(r\xi)|^2 |\xi|^{d-\ell}\ud \xi
    + \int_{|\xi| > 1} |\widehat{h}(r\xi)|^2 |\xi|^{d-\ell+2\alpha} \ud \xi\right].
  \end{align*}
  Then, by the change of variable $\xi \mapsto r^{-1}\xi$ and the fact that $r
    \le 2M$,
  \begin{align}\label{E:LNDx-J2}
    J^2 \le \Theta \times r^{-2d+\ell-2\alpha} (C_3 + C_4),
  \end{align}
  where
  \begin{gather*}
    C_3 \coloneqq 2C_{H, d}^{-1} \widetilde{C}_{S, T}^{-1} (2M)^{2\alpha} \int_{|\xi| \le 2M} |\widehat{h}(\xi)|^2 |\xi|^{d-\ell} \ud \xi \shortintertext{and}
    C_4 \coloneqq 2C_{H, d}^{-1} \widetilde{C}_{S, T}^{-1} \int_{\R^d} |\widehat{h}(\xi)|^2 |\xi|^{d-\ell+2\alpha} \ud \xi.
  \end{gather*}
  Note that $C_3$ is finite since $d-\ell > -d$, and $C_4$ is finite since
  $d-\ell+2\alpha > d > -d$ and $\widehat{h}$ is rapidly decreasing.
  Combining~\eqref{E:LNDx-J1} and~\eqref{E:LNDx-J2}, we
  conclude~\eqref{LNDx:LB}. This completes the proof of~\eqref{E:LNDx_CV_LB}.
  \bigskip

  Next, we proceed to prove~\eqref{E:LNDx_V_LB}. If $\min( 2\alpha \rho
    /\beta,\, 2\alpha-\ell) \in (0, 2)$, then follows from~\eqref{LNDx:LB} by
  taking $n = 1$, $x_1 = y$ and $a_1 = 1$. It remains to
  show~\eqref{E:LNDx_V_LB} when $\min( 2\alpha \rho /\beta,\, 2\alpha-\ell) \ge
    2$. Note that at least one of the following two cases hold: (1) $2\alpha-\ell
    \ge 2$, and (2) $2\alpha \rho /\beta \ge 2$. We will prove
  that~\eqref{E:LNDx_V_LB} holds in each of the two cases.

  \bigskip\noindent\textbf{Case (1):} In this case, we assume that $2\alpha -
    \ell \ge 2$. By~\eqref{E:LNDx_Th_LB2} and changing variable $\xi \mapsto
    (2M|x-y|)^{-1} \xi$, we see that
  \begin{align}\label{E:LND_x_IIa}
    \E\left[\left(u(t, x)-u(t, y)\right)^2\right]
     & \gtrsim \int_{\R^d} \left|1-e^{i(x-y)\cdot\xi}\right|^2 \frac{|\xi|^{\ell-d}}{1+|\xi|^{2\alpha}} \ud \xi \nonumber                     \\
     & = (2M|x-y|)^{-\ell} \int_{\R^d} 2(1-\cos(v\cdot \xi)) \frac{|\xi|^{\ell-d}}{1 + (2M|x-y|)^{-2\alpha}|\xi|^{2\alpha}} \ud \xi \nonumber \\
     & \gtrsim |x-y|^{2\alpha-\ell} \int_{\R^d} 2(1-\cos(v\cdot \xi)) \frac{|\xi|^{\ell-d}}{|x-y|^{2\alpha} + |\xi|^{2\alpha}} \ud \xi,
  \end{align}
  where $v = (2M|x-y|)^{-1}(x-y)$ and in the last inequality, we have used
  $(|x-y|^{2\alpha}+(2M)^{-2\alpha} |\xi|^{2\alpha})^{-1} \ge
    (|x-y|^{2\alpha}+|\xi|^{2\alpha})^{-1}$ since $M \ge 1$. Note that $|v \cdot
    \xi| = (2M)^{-1}|\xi| \cos \theta_{v, \xi}$, where $\theta_{v, \xi} \in [0,
      \pi]$ is the angle between $v$ and $\xi$. Restrict the domain of the last
  integral to
  \begin{align}\label{E:LND_x_D}
    D = \{ \xi \in \R^d: |\xi| \le 2M,\, \theta_{v, \xi} \in [0, \pi/4] \}.
  \end{align}
  Then, using spherical coordinates and the fact that
  \begin{align}\label{E:1-cos_LB}
    1 - \cos z \ge z^2/4 \quad \text{for } |z| \le 1,
  \end{align}
  we obtain
  \begin{align}\label{E:LND_x_IIb}
    \begin{split}
      \E\left[\left(u(t, x)-u(t, y)\right)^2\right]
       & \gtrsim |x-y|^{2\alpha-\ell} \int_{D} \frac{|\xi|^{\ell-d+2}}{|x-y|^{2\alpha} + |\xi|^{2\alpha}} \ud \xi \\
       & = |x-y|^{2\alpha-\ell} \int_0^{2M} \frac{z^{\ell+1}}{|x-y|^{2\alpha} + z^{2\alpha}} \ud z.
    \end{split}
  \end{align}
  By splitting the integral over the intervals $0 \le z \le |x-y|$ and $|x-y| <
    z \le 2M$, we see that
  \begin{align*}
    \int_0^{2M} \frac{z^{\ell+1}}{|x-y|^{2\alpha} + z^{2\alpha}} \ud z
    \gtrsim \begin{cases}
              |x-y|^{2-2\alpha+\ell} & \text{if } 2\alpha-\ell > 2, \\
              1+|\log|x-y||          & \text{if } 2\alpha-\ell = 2.
            \end{cases}
  \end{align*}
  This proves \eqref{E:LNDx_V_LB} in Case (1).

  \bigskip\noindent\textbf{Case (2):} In this case, we assume that $2\alpha \rho
    / \beta \ge 2$. Similarly, by~\eqref{E:LNDx_Th_LB1} and scaling, we have
  \begin{align*}
    \E & \left[\left( u(t, x) - u(t, y) \right)^2\right]                                                                                                                                                                                             \\
       & \gtrsim \int_{\R^d} 2(1-\cos((x-y)\cdot \xi)) |\xi|^{\ell-d-\frac{2\alpha}{\beta}(H-1)}\left( 1 \wedge |\xi|^{-\frac{2\alpha}{\beta}(\beta+\gamma)}\right) \ud \xi                                                                          \\
       & \gtrsim |x-y|^{\frac{2\alpha\rho}{\beta}} \int_{\R^d} 2(1-\cos(v\cdot\xi)) |\xi|^{\ell-d+\frac{2\alpha}{\beta}(1-H)}\left( |x-y|^{-\frac{2\alpha}{\beta}(\beta+\gamma)} \wedge |\xi|^{-\frac{2\alpha}{\beta}(\beta+\gamma)}\right) \ud \xi,
  \end{align*}
  where $v = (2M|x-y|)^{-1}(x-y)$. Then, by restricting the domain of
  integration to $D$ defined in~\eqref{E:LND_x_D} and using spherical
  coordinates and~\eqref{E:1-cos_LB}, we get that
  \begin{align*}
    \E\left[\left( u(t, x) - u(t, y) \right)^2\right]
    \gtrsim |x-y|^{\frac{2\alpha\rho}{\beta}} \int_0^{2M} z^{1+\ell-\frac{2\alpha}{\beta}(H-1)} \left( |x-y|^{-\frac{2\alpha}{\beta}(\beta+\gamma)} \wedge z^{-\frac{2\alpha}{\beta}(\beta+\gamma)} \right) \ud z.
  \end{align*}
  By considering $0 \le z \le |x-y|$ and $|x-y| < z \le 2M$ separately, we see
  that
  \begin{align*}
    \int_0^{2M} z^{1+\ell-\frac{2\alpha}{\beta}(H-1)} \left( |x-y|^{-\frac{2\alpha}{\beta}(\beta+\gamma)} \wedge z^{-\frac{2\alpha}{\beta}(\beta+\gamma)} \right) \ud z
    \gtrsim \begin{cases}
              |x-y|^{2-\frac{2\alpha\rho}{\beta}} & \text{if } 2\alpha\rho/\beta > 2, \\
              1+ |\log|x-y||                      & \text{if } 2\alpha\rho/\beta = 2.
            \end{cases}
  \end{align*}
  The proof of~\eqref{E:LNDx_V_LB} and hence part (i) of Theorem~\ref{T:LND_x}
  is complete.
\end{proof}

\subsection{Proof of part (ii) of Theorem \ref{T:LND_x}}\label{SS:LNDx-ii}

\subsubsection{The $\beta=2$, $\gamma=0$ case}

\begin{proof}
  To prove~\eqref{E:LNDx2_CV_LB}, it suffices to prove the existence of a
  positive finite constant $C$ such that for all $t \in [S, T]$, for all $n \ge
    1$, for all $x, x_1, \dots, x_n \in B(0, M)$, for all $a_1, \dots, a_n \in
    \R$,
  \begin{align}\label{E:LNDx_b2g0_Theta}
    \Theta \coloneqq \E\left[ \left( u(t, x) - \sum_{j=1}^n a_j u(t, x_j) \right)^2 \right]
    \ge C r^{\alpha H - \ell + \alpha/2}, \quad
    \text{where } r \coloneqq \min_{1 \le j \le n}|x-x_j|.
  \end{align}
  Without loss of generality, assume $r > 0$. By \eqref{E:Inner_1}, \eqref{E:u_corr}, and \eqref{E:FG_b2g0},
  \begin{align*} % \label{E:LNDx_b2g0_V}
    \Theta = C_1 \int_{\R^d} \left| 1 - \sum_{j=1}^n a_j e^{i(x-x_j)\cdot \xi} \right|^2 N_t(\xi) |\xi|^{\ell-d} \ud \xi,
  \end{align*}
  where
  \[
    N_t(\xi) = \frac{\nu}{2} |\xi|^{-\alpha} \int_\R \ud \tau \, |\tau|^{1-2H} \left| \int_0^t \ud s \, e^{i\tau s} \sin\left( \sqrt{\nu/2} |\xi|^{\alpha/2} s \right) \right|^2.
  \]
  Thanks to Balan's lemma~\cite[Lemma 6.2]{balan:12:linear}, there exists a
  constant $C > 0$ such that for all $t \in [S, T]$ and $\xi \in \R^d$,
  \begin{align}\label{E:LNDx_b2g0_N}
    N_t(\xi) \ge \frac{C}{\sqrt{K |\xi|^\alpha + 1}} \int_\R \frac{|\tau|^{1-2H}}{\tau^2 + K |\xi|^\alpha + 1} \ud \tau,
  \end{align}
  where $K = 2^{-1}\nu T^2$.

  We claim that there exists a constant $C_2 > 0$ such that for all $\xi \in
    \R^d$,
  \begin{align}\label{E:LNDx_b2g0_claim}
    \frac{C}{\sqrt{K |\xi|^\alpha + 1}} \int_\R \frac{|\tau|^{1-2H}}{\tau^2 + K |\xi|^\alpha + 1} \ud \tau \ge \frac{C_2(1 \wedge |\xi|^{\alpha H})}{|\xi|^{\alpha H} (1+|\xi|^{\alpha/2})}.
  \end{align}
  Indeed, by a change of variable $\tau \mapsto |\xi|^{\alpha/2}\tau$, we have
  \begin{align*}
    \int_\R \frac{|\tau|^{1-2H}}{\tau^2 + K |\xi|^\alpha + 1} \ud \tau = |\xi|^{-\alpha H} \int_\R \frac{|\tau|^{1-2H}}{\tau^2 + |\xi|^{-\alpha} + K} \ud \tau.
  \end{align*}
  If $|\xi| \le K^{-1/\alpha}$, then
  \begin{align*}
    \int_\R \frac{|\tau|^{1-2H}}{\tau^2 + |\xi|^{-\alpha} + K} \ud \tau
    \gtrsim \int_{|\tau| \le |\xi|^{-\alpha/2}} \frac{|\tau|^{1-2H}}{|\xi|^{-\alpha}} \ud \tau + \int_{|\tau| > |\xi|^{-\alpha/2}} |\tau|^{-1-2H} \ud \tau
    \gtrsim |\xi|^{\alpha H}.
  \end{align*}
  If $|\xi| > K^{-1/\alpha}$, then
  \begin{align*}
    \int_\R \frac{|\tau|^{1-2H}}{\tau^2 + |\xi|^{-\alpha} + K} \ud \tau
    \gtrsim \int_{\R} \frac{|\tau|^{1-2H}}{\tau^2 + 2K} \ud \tau = C_{H, K}.
  \end{align*}
  Hence, the claim~\eqref{E:LNDx_b2g0_claim} follows. Therefore,
  \eqref{E:LNDx_b2g0_Theta}, \eqref{E:LNDx_b2g0_N} and~\eqref{E:LNDx_b2g0_claim}
  together imply that
  \begin{align}\label{E:LNDx_b2g0_T_LB}
    \Theta \ge C_1 C_2 \int_{\R^d} \left| 1 - \sum_{j=1}^n a_j e^{i(x-x_j)\cdot\xi} \right|^2 \frac{1 \wedge |\xi|^{\alpha H}}{1+|\xi|^{\alpha/2}}|\xi|^{-d-\alpha H +\ell} \ud \xi.
  \end{align}

  Choose and fix a smooth test function $\phi: \R^d \to \R$ supported on the
  unit ball with $\phi(0) = 1$. Define $\phi_r(x) = r^{-d} \phi(r^{-1}x)$ and
  consider the integral
  \begin{align}\label{E:LNDx_b2g0_I}
    I \coloneqq \int_{\R^d} \left( 1 - \sum_{j=1}^n a_j e^{i(x-x_j)\cdot \xi} \right) \widehat{\phi_r}(\xi) \ud \xi.
  \end{align}
  Then, by inverse Fourier transform, and the fact that $r \le |x-x_j|$, which
  implies $\phi_r(x-x_j) = 0$, we have
  \begin{align}\label{E:LNDx_b2g0_I2}
    I = (2\pi)^d \phi_r(0) = (2\pi)^{d} r^{-d}.
  \end{align}

  On the other hand, we can apply Cauchy--Schwarz inequality to
  \eqref{E:LNDx_b2g0_I}. By~\eqref{E:LNDx_b2g0_T_LB}, we have
  \begin{align*}
    I^2 \le C_1^{-1} C_2^{-1} \Theta \times \int_{\R^d} \left| \widehat{\phi}(r\xi) \right|^2 \frac{1+|\xi|^{\alpha/2}}{1 \wedge |\xi|^{\alpha H}} |\xi|^{d+\alpha H - \ell} \ud \xi.
  \end{align*}
  By splitting the integral into two parts, changing variable, and noting that
  $r \le 2M$, we obtain
  \begin{align*}
     & \int_{\R^d} \left| \widehat{\phi}(r\xi) \right|^2 \frac{1+|\xi|^{\alpha/2}}{1 \wedge |\xi|^{\alpha H}} |\xi|^{d+\alpha H - \ell} \ud \xi \\
     & \le \int_{|\xi| \le 1} \left| \widehat{\phi}(r\xi) \right|^2 \left(1+|\xi|^{\alpha/2}\right) |\xi|^{d-\ell} \ud \xi
    + \int_{|\xi| > 1} \left| \widehat{\phi}(r\xi) \right|^2 \left(1+|\xi|^{\alpha/2}\right)|\xi|^{d+\alpha H - \ell} \ud \xi                   \\
     & \le C_3 r^{-2d-\alpha H - \frac{\alpha}{2} + \ell},
  \end{align*}
  where
  \begin{align*}
    C_3 & = (2M)^{\alpha H} \int_{|\xi| \le 2M} \left| \widehat{\phi}(\xi) \right|^2 \left((2M)^{\alpha/2} + |\xi|^{\alpha/2}\right) |\xi|^{d-\ell} \ud \xi \\
        & \quad + \int_{\R^d} \left| \widehat{\phi}(\xi) \right|^2 \left((2M)^{\alpha/2} + |\xi|^{\alpha/2}\right) |\xi|^{d+\alpha H -\ell} \ud \xi.
  \end{align*}
  Note that $C_3$ is a finite constant since $\ell \in (0, 2d)$ and
  $\widehat{\phi}$ is a rapidly decreasing function. Hence,
  \begin{align}\label{E:LNDx_b2g0_I3}
    I^2 \le C_1^{-1} C_2^{-1} \Theta \times C_3 r^{-2d-\alpha H - \frac{\alpha}{2} + \ell}.
  \end{align}
  Combining~\eqref{E:LNDx_b2g0_I2} and~\eqref{E:LNDx_b2g0_I3}
  yields~\eqref{E:LNDx_b2g0_Theta} with $C = (2\pi)^{2d} C_1 C_2 C_3^{-1}$, and
  this completes the proof of~\eqref{E:LNDx2_CV_LB}.

  Now, we turn to the proof of~\eqref{E:LNDx2_V_LB}. If $\alpha H - \ell +
    \alpha/2 \in (0, 2)$, then~\eqref{E:LNDx2_V_LB} follows
  from~\eqref{E:LNDx_b2g0_Theta} by taking $n = 1$, $x_1 = y$ and $a_1 = 1$. It
  remains to show \eqref{E:LNDx2_V_LB} when $\alpha H - \ell + \alpha/2 \ge 2$.
  In this case, thanks to \eqref{E:LNDx_b2g0_T_LB}, we have
  \begin{align*}
    \E\left[ \left( u(t, x) - u(t, y) \right)^2 \right] \gtrsim \int_{\R^d} 2(1-\cos((x-y)\cdot \xi)) \frac{1 \wedge |\xi|^{\alpha H}}{1+|\xi|^{\alpha/2}} |\xi|^{-d-\alpha H +\ell} \ud \xi.
  \end{align*}
  Then, as in \eqref{E:LND_x_IIa}--\eqref{E:LND_x_IIb}, the same changes of
  variables and similar calculations lead to
  \begin{align*}
    \E\left[ \left( u(t, x) - u(t, y) \right)^2 \right]
    \gtrsim |x-y|^{\alpha H -\ell + \alpha/2} \int_0^{2M} \frac{1 \wedge (|x-y|^{-\alpha H} z^{\alpha H})}{ |x-y|^{\alpha/2} + z^{\alpha/2}} z^{1-\alpha H + \ell} \ud z.
  \end{align*}
  Then, \eqref{E:LNDx2_V_LB} follows from the estimate
  \begin{align*}
    \int_0^{2M} \frac{1 \wedge (|x-y|^{-\alpha H} z^{\alpha H})}{ |x-y|^{\alpha/2} + z^{\alpha/2}} z^{1-\alpha H + \ell} \ud z
    \gtrsim \begin{cases}
              |x-y|^{2-(\alpha H - \ell + \alpha/2)} & \text{if } \alpha H - \ell + \alpha/2 > 2, \\
              1+|\log|x-y||                          & \text{if } \alpha H - \ell + \alpha/2 = 2,
            \end{cases}
  \end{align*}
  which can be easily obtained by splitting the integral over the intervals $0
    \le z \le |x-y|$ and $|x-y| < z \le 2M$. This completes the proof of the
  $\beta=2$, $\gamma=0$ case for Theorem \ref{T:LND_x} (ii).
  % The proof of part (ii) of Theorem \ref{T:LND_x} is hence complete.
\end{proof}

\subsubsection{The $\beta=2$, $\gamma\in(0,1)$ case}

\begin{lemma}\label{lem:B}
  Let
  \begin{align}\label{def:B}
    B(t,a) = a^{-2} \int_\R \ud \tau\, |\tau|^{1-2H} \left|\int_0^t \ud s \, e^{i\tau s} \int_0^s \ud r \, (s-r)^{\gamma-1}\sin(ar) \right|^2.
  \end{align}
  Then for any $0<S<T<\infty$, there exists $0<C<\infty$ such that for all $a >
    0$,
  \begin{align*}
    B(t,a) \ge \frac{C}{1+a^{2\gamma+2H+1}}.
  \end{align*}
\end{lemma}

\begin{proof}
  First, since
  \begin{align*}
    B(t,a) \ge \int_0^1 \ud \tau\, |\tau|^{1-2H} \left|\int_0^t \ud s \, e^{i\tau s} \int_0^s \ud r \, (s-r)^{\gamma-1}r\,\frac{\sin(ar)}{ar} \right|^2,
  \end{align*}
  we can use $\sin(x)/x \to 1$ as $x \to 0$ and the dominated convergence
  theorem to deduce that there exists $c>0$ and $\delta_0>0$ such that $B(t,a)
    \ge c$ for all $a\in(0,\delta_0)$.

  Next, consider $a \ge \delta_0$. By scaling, we can write
  \begin{align}\label{E:B:LB}
    B(t,a) = a^{-2\gamma-2H-2} \underbrace{\int_\R \ud \tau \, |\tau|^{1-2H} \left| \int_0^{ta} \ud s \, e^{i\tau s} \int_0^s \ud r \, (s-r)^{\gamma-1} \sin(r) \right|^2}_{=:I(ta)}.
  \end{align}
  Let $M>0$ be a constant to be determined. For any $z > 0$, we have
  \begin{align*}
     & I(z) \ge \int_0^M \ud \tau \, |\tau|^{1-2H} \left| \int_0^{z} \ud s \, e^{i\tau s} \int_0^s \ud r \, (s-r)^{\gamma-1} \sin(r) \right|^2                                                                                                                          \\
     & \ge M^{1-2H} \int_0^M \ud \tau \left| \int_0^{z} \ud s \, e^{i\tau s} \int_0^s \ud r \, (s-r)^{\gamma-1} \sin(r) \right|^2                                                                                                                                       \\
     & =  M^{1-2H} \left[\int_0^\infty \ud \tau \left| \int_0^{z} \ud s \, e^{i\tau s} \int_0^s \ud r \, (s-r)^{\gamma-1} \sin(r) \right|^2 - \int_M^\infty \ud \tau \left| \int_0^{z} \ud s \, e^{i\tau s} \int_0^s \ud r \, (s-r)^{\gamma-1} \sin(r) \right|^2\right] \\
     & =:  M^{1-2H} \left[J_1(z) - J_2(z)\right].
  \end{align*}
  By Plancherel's theorem and \eqref{E:ML-deriv},
  \begin{align*}
    J_1(z)
    = 2 \pi \int_0^z \ud s \left| \int_0^s \ud r \, (s-r)^{\gamma-1} \sin(r) \right|^2
    = C \int_0^z \ud s \left| s^{\gamma+1} E_{2,2+\gamma}(-s^2) \right|^2.
  \end{align*}
  Now recall \eqref{E:asym_MT1-4}, which implies that
  \[
    s^{\gamma+1}E_{2,2+\gamma}(-s^2) = \cos\left(s-(1+\gamma)\pi/2\right) + \frac{1}{\Gamma(\gamma)} s^{\gamma-1} + O(s^{\gamma-3}) \quad \text{as $s\to \infty$},
  \]
  and $E_{2,2+\gamma}(-s^2)$ can be 0 only at isolated points, hence
  \[
    J_1(z) \ge C_1 z \quad \text{for all $z \ge S \delta_0$.}
  \]
  As for $J_2(z)$, we can use integration by parts twice to find that
  \begin{align*}
     & \int_0^{z} \ud s \, e^{i\tau s} \int_0^s \ud r \, (s-r)^{\gamma-1} \sin(r)                                                                  \\
     & = \frac{1}{i\tau} \left[ e^{i\tau z} z^{\gamma+1} E_{2,2+\gamma}(-z^2) - \int_0^z e^{i\tau s} s^\gamma E_{2,1+\gamma}(-s^2)\, \ud s \right] \\
     & = \frac{1}{i\tau} e^{i\tau z} z^{\gamma+1} E_{2,2+\gamma}(-z^2)
    + \frac{1}{\tau^2} e^{i\tau z} z^\gamma E_{2,1+\gamma}(-z^2) - \frac{1}{\tau^2} \int_0^z e^{i\tau s} s^{\gamma-1} E_{2,\gamma}(-s^2)\, \ud s.
  \end{align*}
  By \eqref{E:asym_MT1-4} again, we see that
  \[
    \sup_{z\ge S\delta_0} [z^{\gamma+1} E_{2,2+\gamma}(-z^2)] < \infty, \qquad
    \sup_{z\ge S\delta_0} [z^\gamma     E_{2,1+\gamma}(-z^2)] < \infty,
  \]
  and
  \begin{align*}
    \left|\int_0^z e^{i\tau s} s^{\gamma-1} E_{2,\gamma}(-s^2)\, \ud s\right|
    \le \left| \int_0^z e^{i\tau s} \cos\left(s+(1-\gamma)\pi/2\right)\, \ud s \right| + C \int_0^z s^{\gamma-3} \,\ud s,
  \end{align*}
  which is also bounded for all $z\ge S\delta_0$. It follows that
  \[
    J_2(z) \le C_2 \int_M^\infty \tau^{-2} \ud \tau.
  \]
  By choosing $M>0$ large enough, we can ensure that for all $z \ge S \delta_0$,
  \[
    J_2(z) \le \frac{1}{2}C_1z.
  \]
  Hence, for all $t \in [S,T]$ and $a \ge \delta_0$,
  \[
    I(ta) \ge M^{1-2H}\left[ C_1 ta - \frac{1}{2}C_1ta\right] \ge \frac{M^{1-2H}C_1 S}{2} a.
  \]
  Combine this with~\eqref{E:B:LB} to finish the proof of Lemma~\ref{lem:B}.
\end{proof}

We now finish the proof of the $\beta=2$, $\gamma\in(0,1)$ case for
part (ii) of Theorem~\ref{T:LND_x}.

\begin{proof}
  By Theorem \ref{T:PDE}, we can write
  \begin{align}\label{Theta:B}
    \Theta \coloneqq \E\left[ \left( u(t,x)-\sum_{j=1}^n a_j u(t, x_j)\right)^2\right]
    = C_1 \int_{\R^d} \left| 1- \sum_{j=1}^n a_j e^{i(x-x_j)\cdot \xi} \right|^2 B(t,|\xi|^{\alpha/2})  |\xi|^{\ell-d} \,\ud \xi,
  \end{align}
  where $B(t,|\xi|^{\alpha/2})$ was defined by \eqref{def:B} in Lemma~\ref{lem:B}. 

  {\bf Case (1):} $\alpha(\gamma+H+1/2)\le 2\alpha$, so that $2\tilde\rho_2 = \alpha(\gamma+H+1/2)-\ell$.
  By \eqref{Theta:B} and Lemma \ref{lem:B},
  \[
    \Theta \gtrsim \int_{\R^d} \left| 1- \sum_{j=1}^n a_j e^{i(x-x_j)\cdot \xi} \right|^2  \frac{|\xi|^{\ell-d}}{1+|\xi|^{\alpha(\gamma+H+1/2)}} \,\ud \xi.
  \]
  Let $\phi:\R^d\to\R$ be nonnegative smooth function supported on the unit ball
  with $\phi(0)=1$, let $\phi_r(z) = r^{-d}\phi(r^{-1}z)$, where $r = \min_{1\le
      j \le n}|x-x_j|$, and consider
  \[
    I \coloneqq \int_{\R^d} \left(1- \sum_{j=1}^n a_j e^{i(x-x_j)\cdot \xi} \right) \widehat{\phi_r}(\xi).
  \]
  By Fourier inversion,
  \[
    I = (2\pi)^d \left( \phi_r(0) - \sum_{j = 1}^n a_j \phi_r(x-x_j) \right)
    = (2\pi)^d r^{-d}.
  \]
  By Cauchy--Schwarz inequality,
  \begin{align*}
    (2\pi)^{2d} r^{-2d} \le |I|^2 \lesssim \Theta \times \int_{\R^d} \left| \widehat{\phi_r}(\xi) \right|^2 \frac{1+|\xi|^{\alpha(\gamma+H+1/2)}}{|\xi|^{\ell-d}}\ud \xi.
  \end{align*}
  We estimate the last integral using a change of variable and the property that
  $\widehat{\phi}$ is rapidly decreasing:
  \begin{align*}
     & \int_{\R^d} \left| \widehat{\phi}(r\xi) \right|^2 |\xi|^{-\ell+d} \ud \xi +  \int_{\R^d} \left| \widehat{\phi}(r\xi) \right|^2 |\xi|^{\alpha(\gamma+H+1/2)-\ell+d} \ud \xi \\
     & \le r^{\ell-2d} \int_{\R^d} \left| \widehat{\phi}(\xi) \right|^2 |\xi|^{-\ell+d} \ud \xi
    + r^{-\alpha(\gamma+H+1/2)+\ell-2d} \int_{\R^d} \left| \widehat{\phi}(\xi) \right|^2 |\xi|^{\alpha(\gamma+H+1/2)-\ell+d} \ud \xi                                              \\
     & \lesssim r^{-\alpha(\gamma+H+1/2)+\ell-2d}.
  \end{align*}
  It follows that
  \[
    \Theta \gtrsim r^{\alpha(\gamma+H+1/2)-\ell}.
  \]
  This proves \eqref{E:LNDx2_CV_LB}, which also implies the first case
  $\tilde\rho_2 = \frac12 [\alpha(\gamma+H+1/2) - \ell] \in (0,1)$ in~\eqref{E:LNDx2_V_LB}.

  Next, we prove \eqref{E:LNDx2_V_LB} for $\tilde\rho_2 = \frac12 [ \alpha(\gamma+H+1/2)-\ell] \ge 1$.
  By Lemma \ref{lem:B},
  \[
    \E\left[ \left( u(t,x)-u(t,y)\right)^2\right]
    \gtrsim \int_{\R^d} 2(1-\cos((x-y)\cdot \xi)) \frac{|\xi|^{\ell-d}}{1+|\xi|^{\alpha(\gamma+H+1/2)}} \ud \xi.
  \]
  Then, similarly to \eqref{E:LND_x_IIa}--\eqref{E:LND_x_IIb}, we can deduce
  that
  \[
    \E\left[ \left( u(t,x)-u(t,y)\right)^2\right]
    \gtrsim |x-y|^{\alpha(\gamma+H+1/2)-\ell} \int_0^{2M} \frac{z^{\ell+1}}{|x-y|^{\alpha(\gamma+H+1/2)}+z^{\alpha(\gamma+H+1/2)}} \ud z.
  \]
  By splitting the last integral over $0 \le z \le |x-y|$ and $|x-y| < z \le
    2M$, it is easy to show that
  \[
    \int_0^{2M} \frac{z^{\ell+1}}{|x-y|^{\alpha(\gamma+H+1/2)}+z^{\alpha(\gamma+H+1/2)}} \ud z
    \gtrsim \begin{cases}
      |x-y|^{2-\alpha(\gamma+H+1/2)} & \text{if } \alpha(\gamma+H+1/2) > 2, \\
      1+|\log|x-y||                  & \text{if } \alpha(\gamma+H+1/2) = 2.
    \end{cases}
  \]
  This yields the last two cases in \eqref{E:LNDx2_V_LB}.

  {\bf Case (2):} $\alpha(\gamma+H+1/2)> 2\alpha$, so that $2\tilde\rho_2 = 2\alpha-\ell$.
  By \eqref{E:asym_MT1-4} and $\sup_{x \ge 0} |E_{2,2+\gamma}(-x)| < \infty$, there is $a_0\ge 1$ such that
  \begin{align}\label{E:cos:Phi}
      s^{1+\gamma} E_{2,2+\gamma}(-|\xi|^\alpha s^2) 
      = |\xi|^{-\frac{\alpha}{2}(1+\gamma)} \cos(|\xi|^{\alpha/2}s - (1+\gamma)\pi/2) + \Phi(s,\xi),
  \end{align}
  uniformly for all $|\xi| \ge a_0$ and $s \in [0,T]$, where
  \begin{align}\label{Phi:bound}
    \begin{cases}
        \Phi(s, \xi) \gtrsim s^{\gamma-1} |\xi|^{-\alpha} & \text{for $s\ge |\xi|^{-\alpha/2}$},\\
        |\Phi(s, \xi)| = O(|\xi|^{-\frac{\alpha}{2}(1+\gamma)}) & \text{for $s\le |\xi|^{-\alpha/2}$.}
    \end{cases}  
  \end{align}
  Recall $B(t,a)$ defined in \eqref{def:B}.
  Then by \eqref{E:cos:Phi},
  \begin{align*}
      B(t, |\xi|^{\alpha/2})
      =\int_\R \ud \tau \, |\tau|^{1-2H} \left| \int_0^t e^{i\tau s} s^{1+\gamma} E_{2,2+\gamma}(- |\xi|^\alpha s^2) \ud s \right|^2
      \ge I_1 - I_2,
  \end{align*}
  where
  \begin{align*}
      I_1 &= \frac12 \int_0^t \ud s_1  \int_0^t \ud s_2 \, |s_1-s_2|^{2H-2} \Phi(s_1,\xi) \Phi(s_2,\xi),\\
      I_2 &= |\xi|^{-\alpha(1+\gamma)} \int_\R \ud \tau \, |\tau|^{1-2H} \left| \int_0^t e^{i\tau s} \cos(|\xi|^{\alpha/2} s^2 - (1+\gamma)\pi/2) \ud s \right|^2.
  \end{align*}
  It follows from \eqref{Phi:bound} that there is $c_0>0$ such that for all $|\xi| \ge a_0$ and $t \in [S,T]$,
  \begin{align*}
      I_1 \ge c_0 |\xi|^{-2\alpha}.
  \end{align*}
  By a generalized Balan's lemma (Lemma \ref{L:balan}),
  \begin{align*}
      I_2 &= |\xi|^{-\alpha(1+\gamma)} |\xi|^\alpha N_{\alpha,-\gamma \pi/2}(t,|\xi|) = |\xi|^{-\alpha\gamma} t^4 A_{-\gamma\pi/2}(t, t|\xi|^{\alpha/2})\\
      & \lesssim \frac{|\xi|^{-\alpha\gamma}}{\sqrt{1+|\xi|^\alpha}} \int_\R \frac{|\tau|^{1-2H}}{\tau^2 + |\xi|^\alpha} \ud \tau \lesssim |\xi|^{-\alpha(\gamma+H+1/2)}.
  \end{align*}
  Since $\alpha(\gamma+H+1/2)>2\alpha$, we can find $a_1 \ge a_0$ such that $I_2 \le (c_0/2)|\xi|^{-2\alpha}$ for all $|\xi| \ge a_1$.
  Also, Lemma \ref{lem:B} implies that $B(t, |\xi|^\alpha)$ is bounded below for all $|\xi| \le a_1$ and $t \in [S,T]$.
  It follows that
  \begin{align}\label{B:lb:2}
    B(t,|\xi|^{\alpha/2}) \gtrsim \frac{1}{1+|\xi|^{2\alpha}} \quad \text{for all $\xi \in \R^d$, $t \in [S,T]$.}
  \end{align}
  This together with \eqref{Theta:B} yields
  \[
    \Theta \gtrsim \int_{\R^d} \left| 1- \sum_{j=1}^n a_j e^{i(x-x_j)\cdot \xi} \right|^2  \frac{|\xi|^{\ell-d}}{1+|\xi|^{2\alpha}} \,\ud \xi.
  \]
  Hence, as in {\bf Case (1)}, we can deduce that $\Theta \gtrsim r^{2\alpha-\ell}$. 
  This proves \eqref{E:LNDx2_CV_LB} and implies the first case
  $\tilde\rho_2 = \alpha-\ell/2 \in (0,1)$ in~\eqref{E:LNDx2_V_LB}.

  Finally, for $\tilde\rho_2 = \alpha-\ell/2 \ge 1$, we use \eqref{B:lb:2} again to get
  \[
    \E\left[(u(t,x)-u(t,y))^2\right]
    \gtrsim \int_{\R^d} 2(1-\cos((x-y)\cdot \xi)) \frac{|\xi|^{\ell-d}}{1+|\xi|^{2\alpha}} \ud \xi.
  \]
  As in {\bf Case (1)}, we can show that
  \[
    \E\left[(u(t,x)-u(t,y))^2\right]
    \gtrsim \begin{cases}
        |x-y|^2 (1+|\log|x-y||) & \text{if $2\alpha-\ell = 2$,}\\
        |x-y|^2 & \text{if $2\alpha - \ell > 2$.}
    \end{cases}
  \]
  The proof of part (ii) of Theorem \ref{T:LND_x} is hence complete.
\end{proof}

\subsection{Proofs of Lemmas~\ref{L:TestFun} and~\ref{L:Delta_t}}\label{S7:lemma}

\begin{proof}[Proof of Lemma \ref{L:TestFun}]
  Fix an arbitrary $m\ge 1$. Without loss of generality, we assume that $m$ is
  an even integer. We first consider the case when $d=1$. Let $\phi\in
    C_c^\infty(\R)$ be an arbitrary nonnegative and even test function. Set
  $h(x)\coloneqq \cos(x) \one_{[-\pi,\pi]}(x)$. Elementary calculation shows
  that $\widehat{h}(\xi) = \frac{2\xi\sin(\pi\xi)}{1-\xi^2}$ (see
  Figure~\ref{F:ghat}). For $m\ge 1$, $h_m(x) \coloneqq \left( h * \cdots *
    h\right)(x)$, where there are $m-1$ convolutions. Denote $g(x) \coloneqq
    \left(h_m * \phi * \phi \right)(x)$. Now, it is ready to check that $g\in
    C_c^\infty(\R)$. Notice that
  \begin{align}\label{E:ghat}
    \widehat{g}(\xi)
    = \widehat{h_m}(\xi) \widehat{\phi}(\xi)^2
    = \left( \widehat{h}(\xi)\right)^m \widehat{\phi}(\xi)^2
    = \left( \frac{2 \xi  \sin (\pi  \xi )}{1-\xi ^2}\right)^m \left|\widehat{\phi}(\xi)\right|^2.
  \end{align}
  Hence, we see that
  \begin{align*}
    g(0) = (2\pi)^{-1} \int_\R \left( \frac{2 \xi  \sin (\pi  \xi)}{1-\xi ^2}\right)^m \left|\widehat{\phi}(\xi)\right|^2 \ud \xi> 0,
  \end{align*}
  where we used the fact that $m$ is even. From~\eqref{E:ghat}, we see
  $|\widehat{g}(\xi)|\lesssim |\xi|^{2m}$ for $|\xi|$ small.

  As for the case when $d\ge 2$, for $x\in\R^d$, denote $x=(x_1,x_2)$, with
  $x_1\in\R$ and $x_2\in\R^{d-1}$. Let $g_1(x_1)$ be the function constructed in
  the case $d=1$. Let $g_2\in C_c^\infty(\R^{d-1})$ be an arbitrary test
  function with $g_2(0)>0$. Then, we define $g(x)\coloneqq g_1(x_1) g_2(x_2)$.
  It is ready to verify that $g\in C_c^\infty(\R^d)$ and $g(0) = g_1(0)
    g_2(0)>0$. Moreover, notice that
  \begin{align*}
    |\widehat{g}(\xi)| = |\widehat{g_1}(\xi_1)||\widehat{g_2}(\xi_2)|
    \lesssim |\xi_1|^{2m}|\widehat{g_2}(\xi_2)|
    \le |\xi|^{2m}|\widehat{g_2}(\xi_2)|, \quad \text{for $|\xi|$ small.}
  \end{align*}
  Hence, we see that $|\widehat{g}(\xi)|\lesssim |\xi|^{2m}$ for $|\xi|$ small.
  This completes the proof of Lemma~\ref{L:TestFun}.
\end{proof}

\begin{figure}[htpb]
  \centering
  \begin{tikzpicture}[scale=0.9, transform shape]
    \tikzset{>=latex}
    \begin{axis}[
        axis lines = center,
        unit vector ratio*={4 2},
        ytick={1,2,3.14159},
        yticklabels={1,2,$\pi$},
        xmin=-6.5, xmax= 6.5,
        ymin=-1.2, ymax= 3.5,
        xlabel={$\xi$},
        xlabel style={at=(current axis.right of origin), anchor=center, xshift = 4.5em, yshift = +0.5em},
        extra x ticks={-6,-5,-4,-3,-2,-1,1,2,3,4,5,6},
        extra x tick style={
            grid=major,
            grid style={dashed},
            tick label style={
                draw=none, % Hide tick labels
              },
          },
        legend style={at={(1,0.9)}, anchor=north east,draw=none},
      ]
      \addplot[domain=-5.6:5.6, blue, solid, thick] table {ghat.csv};
      \addplot[dotted, gray, thick] coordinates {(-1,3.1415) (1,3.1415)};
      \addplot[only marks, mark=*, color=blue, mark size = 1pt] coordinates {(-1,3.1415)};
      \addplot[only marks, mark=*, color=blue, mark size = 1pt] coordinates {(+1,3.1415)};
      % \node[] () at (1,20) {$\widehat{h}(\xi) = \frac{2\xi\sin(\pi\xi)}{1-\xi^2}$};
      \addlegendentry{$ \displaystyle \widehat{h}(\xi) = \frac{2\xi\sin(\pi\xi)}{1-\xi^2}$}; % Legend entry for the plot
    \end{axis}
  \end{tikzpicture}
  \caption{Plot of function $\widehat{h}(\xi)$.}
  \label{F:ghat}
\end{figure}
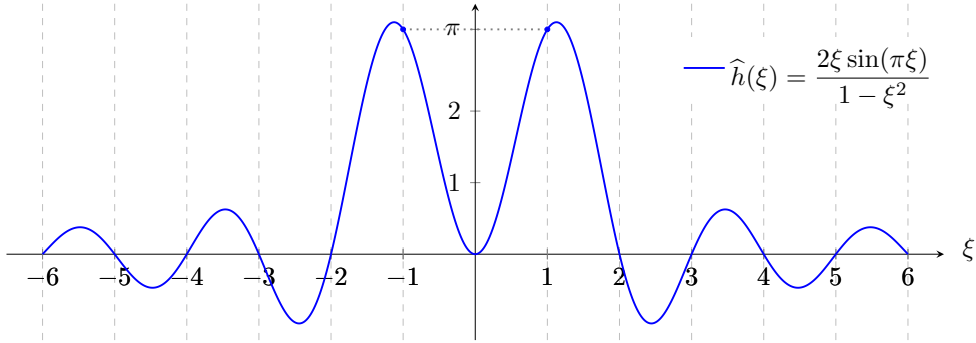

% Do not remove the following commented codes.
% \begin{codes} %math
%
% (* Plot Theta *)
% res = Table[{\[Xi], -((2 \[Xi] Sin[\[Pi] \[Xi]])/(-1 + \[Xi]^2))}, {\[Xi], -6, 6, 0.02}]
% Export["ghat.csv", res, "Table"]
%
% \end{codes}

\begin{proof}[Proof of Lemma \ref{L:Delta_t}]
  We first prove that in both cases (i) and (ii), there exists a positive finite
  constant $C$ such that
  \begin{align}\label{E:LNDx-claim1}
    \Delta_t(\lambda) \ge C  \frac{\lambda^{2(1-H)/\beta}}{1+\lambda^{2(\beta+\gamma)/\beta}} \quad \text{for all } t \in [S, T] \text{ and } \lambda > 0.
  \end{align}
  To prove this lower bound, let $f: \R \to \R$ be a test function to be
  determined, define $f_\lambda(s) = \lambda^{1/\beta}f(\lambda^{1/\beta}s)$,
  and consider the integral
  \begin{align*}
    I_1 \coloneqq \int_\R \ud \tau\, \widehat{f_\lambda}(\tau) \int_0^t \ud s \, e^{i\tau s} s^{\beta+\gamma-1} E_{\beta, \beta+\gamma}\left(-2^{-1}\nu\lambda s^\beta\right).
  \end{align*}
  By inverse Fourier transform in $\tau$,
  \begin{align*}
    I_1 = 2\pi \int_0^t f_\lambda(s)\, s^{\beta+\gamma-1}E_{\beta, \beta+\gamma}\left(-2^{-1}\nu\lambda s^\beta\right)\, \ud s.
  \end{align*}
  Then, we make the change of variable $s \mapsto \lambda^{-1/\beta}s$ to get
  that
  \begin{align*}
    I_1 & = 2\pi \lambda^{-1/\beta} \int_0^{t\lambda^{1/\beta}} f_\lambda\left(\lambda^{-1/\beta}s\right) \left(\lambda^{-1/\beta}s\right)^{\beta+\gamma-1} E_{\beta, \beta+\gamma}\left(-2^{-1}\nu s^\beta\right)\, \ud s \\
        & = 2\pi \lambda^{-(\beta+\gamma-1)/\beta} \int_0^{t\lambda^{1/\beta}} f(s) \, s^{\beta+\gamma-1} E_{\beta, \beta+\gamma}\left(-2^{-1}\nu s^\beta\right)\, \ud s.
  \end{align*}
  By the continuity of $E_{\beta, \beta+\gamma}(\cdot)$ and $E_{\beta,
        \beta+\gamma}(0) = 1/\Gamma(\beta+\gamma) > 0$, we can find $\delta_0 \in (0,
    1]$ such that for all $s \in [0, \delta_0]$, $E_{\beta,
        \beta+\gamma}\left(-2^{-1}\nu s^\beta\right) \ge 1/[2\Gamma(\beta+\gamma)]
    \eqqcolon c_0 > 0$. Now, we choose and fix $f$ to be a smooth nonnegative
  function supported on $[-\delta_0, \delta_0]$ with $f \ge 1$ on $[-\delta_0/2,
        \delta_0/2]$. Choose $\lambda_0 > 0$ small enough such that $T
    \lambda_0^{1/\beta} \le \delta_0/2$. If $\lambda \le \lambda_0$, then $t
    \lambda^{1/\beta} \le T\lambda_0^{1/\beta} \le \delta_0/2$, so we have
  \begin{align*}
    I_1 \ge 2\pi \lambda^{-(\beta+\gamma-1)/\beta} \int_0^{t\lambda^{1/\beta}} s^{\beta+\gamma-1} c_0 \, \ud s
    =  \frac{2\pi c_0 t^{\beta+\gamma}}{\beta+\gamma} \lambda^{1/\beta}
    \ge \frac{2\pi c_0 S^{\beta+\gamma}}{\beta+\gamma} \lambda^{1/\beta}.
  \end{align*}
  If $\lambda > \lambda_0$, then $t \lambda^{1/\beta} \ge S
    \lambda_0^{1/\beta}$, and since $f(s) s^{\beta+\gamma-1}E_{\beta,
    \beta+\gamma}\left(-2^{-1}\nu s^\beta\right) \ge 0$ on $[0, \infty) \cap
    \operatorname{supp}{f}$, it follows that
  \begin{align*}
    I_1 \ge 2\pi \lambda^{-(\beta+\gamma-1)/\beta}\int_0^{S\lambda_0^{1/\beta}} f(s) s^{\beta+\gamma-1} E_{\beta, \beta+\gamma}\left(-2^{-1}\nu s^\beta\right) \, \ud s
    =  C \lambda^{-(\beta+\gamma-1)/\beta}.
  \end{align*}
  This implies that, for some positive finite constant $C_3$, it holds that
  \begin{align}\label{E:LNDx-I1}
    I_1 \ge C_3 \lambda^{1/\beta} \times \left(1 \wedge \lambda^{-(\beta+\gamma)/\beta}\right).
  \end{align}
  On the other hand, by the Cauchy--Schwarz inequality and
  $\widehat{f_\lambda}(\tau) = \widehat{f}\left(\lambda^{-1/\beta}\tau\right)$,
  \begin{align*}
    |I_1|^2 \le \Delta_t(\lambda) \int_{\R} |\tau|^{2H-1} \left|\widehat{f}\left(\lambda^{-1/\beta}\tau\right)\right|^2 \ud \tau.
  \end{align*}
  Then, we make the change of variable $\tau \mapsto \lambda^{1/\beta}\tau$ to
  get that
  \begin{align}\label{E:LNDx-I2}
    |I_1|^2 \le C_4\lambda^{2H/\beta} \Delta_t(\lambda),
  \end{align}
  with the constant $C_4$ given by
  \begin{align*}
    C_4 \coloneqq \int_{\R} |\tau|^{2H-1} \left|\widehat{f}(\tau)\right|^2 \ud \tau.
  \end{align*}
  Note that $C_4$ is a finite constant since $2H-1 > -1$ and $\widehat{f}$ is a
  rapidly decreasing function. Therefore, combining~\eqref{E:LNDx-I1}
  and~\eqref{E:LNDx-I2} yields
  \begin{align*}
    \Delta_t(\lambda) \ge C_3^2C_4^{-1} \lambda^{(2-2H)/\beta}\left(1 \wedge \lambda^{-(\beta+\gamma)/\beta}\right)^2,
  \end{align*}
  which proves~\eqref{E:LNDx-claim1} with $C = C_3^2 C_4^{-1}$.

  It remains to prove that if $(\beta, \gamma) \in (0, 2) \times (0, \infty)$ or
  $(\beta, \gamma) \in \{2\} \times [1, \infty)$, then there exists a positive
  finite constant $C$ such that
  \begin{align}\label{E:LNDx-claim2}
    \Delta_t(\lambda) \ge \frac{C}{1+\lambda^2} \quad \text{for all } t \in [S, T] \text{ and } \lambda > 0.
  \end{align}
  Fix a constant $\lambda_0>0$. Notice that, if $(\beta,\gamma)\in (0,2)\times(0,\infty)$ or
  $(\beta,\gamma)\in\{2\}\times(1,\infty)$, then the asymptotics
  in~\eqref{E:asym_MT1} imply that as $z\to\infty$,
  \[
    E_{\beta,\beta+\gamma}(-z)=\frac{1}{\Gamma(\gamma)z}+o(z^{-1}).
  \]
  In this case, by enlarging $\lambda_0$ if necessary, we may assume that
  \begin{align}\label{E:ML_LB}
    E_{\beta, \beta+\gamma}\left(-z\right) \ge \frac{1}{2\Gamma(\gamma)z} \quad \text{for all } z \ge 2^{-1}\nu \lambda_0 (S/2)^\beta.
  \end{align}
  When $(\beta,\gamma)=(2,1)$, we instead use the identity (for $z>0$)
  \begin{align}\label{E:ML-23-exact}
    E_{2,3}(-z)=\frac{1-\cos(\sqrt{z})}{z},
  \end{align}
  which follows by comparing the power series of both sides.
  In order to prove~\eqref{E:LNDx-claim2}, we consider two cases: (1) $\lambda
    \ge \lambda_0$, and (2) $0 < \lambda < \lambda_0$.

  \bigskip\noindent\textbf{Case (1):~} $\lambda \ge \lambda_0$. Let $\phi: \R
    \to \R$ be a smooth nonnegative function supported on $[S/2, S]$ and consider
  the integral
  \begin{align*}
    I_2 \coloneqq \int_\R \ud \tau \widehat{\phi}(\tau) \int_0^t \ud s \, e^{i\tau s} s^{\beta+\gamma-1} E_{\beta, \beta+\gamma}\left(-2^{-1} \nu \lambda s^\beta \right).
  \end{align*}
  Applying inverse Fourier transform in $\tau$ and the property that $\phi$ is
  supported on $[S/2, S]$, we have
  \begin{align*}
    I_2 = 2\pi \int_{S/2}^S \phi(s) s^{\beta+\gamma-1} E_{\beta, \beta+\gamma}\left(-2^{-1}\nu \lambda s^\beta \right) \ud s.
  \end{align*}
  Then, in the case $(\beta,\gamma)\ne(2,1)$, by~\eqref{E:ML_LB}, we get that
  \begin{align}\label{E:I2_LB}
    I_2 \ge C_5 \lambda^{-1}, \qquad\text{where~~}
    C_5 \coloneqq \frac{2\pi}{\Gamma(\gamma) \nu}\int_{S/2}^S \phi(s) s^{\gamma-1} \ud s \in (0, \infty).
  \end{align}
  In the remaining case $(\beta,\gamma)=(2,1)$, by~\eqref{E:ML-23-exact}, we can
  write
  \begin{align*}
    I_2
    = \frac{4\pi}{\nu\lambda}\int_{S/2}^S \phi(s) \left[1-\cos\!\left(\sqrt{\nu\lambda/2}\:s\right)\right]\ud s.
  \end{align*}
  Since $\phi\in C_c^\infty(\R)$, integration by parts yields
  $|\int_{S/2}^S \phi(s)e^{ias}\ud s|\lesssim a^{-1}$ for all $a>0$. Hence, by
  enlarging $\lambda_0$ if necessary, we can ensure that for all
  $\lambda\ge\lambda_0$,
  \[
    \left|\int_{S/2}^S \phi(s)\cos\!\left(\sqrt{\nu\lambda/2}\:s\right)\ud s\right|
    \le \frac{1}{2}\int_{S/2}^S \phi(s)\ud s,
  \]
  which implies~\eqref{E:I2_LB} (with $\gamma=1$).
  On the other hand, by Cauchy--Schwarz inequality,
  \begin{align}\label{E:I2_UB}
    |I_2|^2 \le C_6 \Delta_t(\lambda), \qquad\text{where~~}
    C_6 \coloneqq \int_\R |\tau|^{2H-1} \left|\widehat{\phi}(\tau) \right|^2 \ud \tau < \infty.
  \end{align}
  Combining~\eqref{E:I2_LB} and~\eqref{E:I2_UB}, we obtain
  \begin{align}\label{E:Delta_LB1}
    \Delta_t(\lambda) \ge C_5^2 C_6^{-1} \lambda^{-2} \quad \text{for all }  t \in [S, T] \text{ and } \lambda \ge \lambda_0
  \end{align}

  \noindent\textbf{Case (2):~} $0 < \lambda < \lambda_0$. Choose a small
  constant $\delta_1 \in (0, 1)$ such that
  \begin{align}\label{E:ML_LB0}
    E_{\beta, \beta+\gamma}(-z) \ge \frac{1}{2\Gamma(\beta+\gamma)} \quad \text{for all } z \in [0, 2^{-1}\nu\lambda_0 \delta_1^\beta].
  \end{align}
  Let $\psi: \R \to \R$ be a smooth nonnegative function supported on $[0,
        \delta_1]$ and consider the integral
  \begin{align*}
    I_3 \coloneqq \int_\R \ud \tau \, \widehat{\psi}(\tau) \int_0^t \ud s \, e^{i\tau s} s^{\beta+\gamma-1} E_{\beta, \beta+\gamma}\left(-2^{-1}\nu\lambda s^\beta\right).
  \end{align*}
  Then, using inverse Fourier transform, the fact that $\psi$ is supported on
  $[0, \delta_1]$, and the lower bound~\eqref{E:ML_LB0}, we have
  \begin{align}\label{E:I3_LB}
    I_3 = 2\pi \int_0^{\delta_1} \psi(s) s^{\beta+\gamma-1} E_{\beta, \beta+\gamma}\left( -2^{-1} \nu \lambda s^\beta\right) \ud s
    \ge \frac{\pi}{\Gamma(\beta+\gamma)} \int_0^{\delta_1} \psi(s) s^{\beta+\gamma-1} \ud s
    \eqqcolon C_7,
  \end{align}
  where $C_7$ is a positive finite constant. On the other hand, by the
  Cauchy--Schwarz inequality,
  \begin{align}\label{E:I3_UB}
    |I_3|^2 \le C_8 \Delta_t(\lambda), \qquad\text{where~~}
    C_8 \coloneqq \int_\R |\tau|^{2H-1} \left| \widehat{\psi}(\tau) \right|^2 \ud \tau < \infty.
  \end{align}
  Hence, we can combine~\eqref{E:I3_LB} and~\eqref{E:I3_UB} to get
  \begin{align}\label{E:Delta_LB2}
    \Delta_t(\lambda) \ge C_7^2 C_8^{-1} > 0 \quad \text{for all } t \in [S, T] \text{ and } \lambda \in (0, \lambda_0).
  \end{align}
  Finally, we can put~\eqref{E:Delta_LB1} and~\eqref{E:Delta_LB2} together to
  conclude the lower bound~\eqref{E:LNDx-claim2}, which completes the proof of
  Lemma~\ref{L:Delta_t}.
\end{proof}

\section{Zero-one laws and proofs of corollaries}
\label{S:Pf:Cor}
\index{zero-one law}

In this section, we prove the sample path regularity results stated in the
Introduction. We first establish the zero-one laws in Theorem~\ref{T:01law}
(Subsection~8.1), which imply that the limits in the exact modulus
and Chung/LIL-type statements are constant. We then prove
Corollaries~\ref{C:MC},~\ref{C:LIL}, and~\ref{C:Chung} (Subsection~8.2) and the
small ball estimates in Corollary~\ref{C:smallball} (Subsection~8.3).

\subsection{Zero-one laws}

We first introduce the harmonizable representation\index{harmonizable representation} of the solution and establish
the zero-one laws in Theorem~\ref{T:01law} below. Let $M_1$ and $M_2$ be
independent space-time Gaussian white noise on $\R \times \R^d$, and set $M =
  M_1 + iM_2$. For any $t > 0$ and $x \in \R^d$, define
\begin{align}\label{E:har_rep}
  v(t, x) = \Re \iint_{\R \times \R^d} \mathcal{F}[G(t-\cdot, x-\ast)\one_{[0, t]}(\cdot)](\tau, \xi)\, |\tau|^{\frac{1-2H}{2}} |\xi|^{\frac{\ell-d}{2}}\, M(\ud\tau, \ud\xi).
\end{align}
Then $v = \{v(t, x): t > 0, x \in \R^d\}$ has the same law as the solution $u =
  \{u(t, x): t > 0, x \in \R^d\}$ to~\eqref{E:fde} since $v$ is a centered
Gaussian random field and it has the same covariance function as $u$
by~\eqref{E:Inner_1} and~\eqref{E:u_corr}. As in
\cites{dalang.mueller.ea:17:polarity, dalang.lee.ea:21:multiple,
  lee.xiao:23:chung-type}, we call~\eqref{E:har_rep} the harmonizable
representation of $u$. Let $\mathcal{B}(\R \times \R^d)$ denote the Borel
$\sigma$-algebra on $\R \times \R^d$. For any $t > 0$, $x \in \R^d$ and $A \in
  \mathcal{B}(\R \times \R^d)$, define $v(A,t,x)$ by truncating the harmonizable
representation
\begin{align}\label{v(A,t,x)}
  v(A, t, x) = \Re \iint_{(\tau, \xi) \in A} \mathcal{F}[G(t-\cdot, x-\ast)\one_{[0, t]}(\cdot)](\tau, \xi)\, |\tau|^{\frac{1-2H}{2}} |\xi|^{\frac{\ell-d}{2}}\, M(\ud\tau, \ud\xi).
\end{align}
Note that the Gaussian processes $v(A, \cdot)$ and $v(B, \cdot)$ are independent
whenever $A$ and $B$ are disjoint.

\begin{theorem}[Zero-one laws]\label{T:01law}%
  \index{zero-one law}%
  Assume condition \eqref{E:main}. Further assume $\ell < \alpha (\gamma+1/2)$ if $\beta = 2$ and $\gamma \in (0, 2]$ in part {\rm (i)}
  below. Fix $a \ge 0$.
  \begin{enumerate}[\rm (i)]

    \item If $\rho \in (0, 1)$, then for any compact interval $I$ in $(0, \infty)$
          and any $x_0 \in \R^d$, there exist constants $K_1, K_2, K_3, K_4 \in [0,
              \infty]$ such that for each $t_0 \in I$,
          \begin{align}\label{E:01law:mc:t}
            \lim_{\varepsilon \to 0^+} \sup_{t,s\in I: 0<|t-s|\le \varepsilon} \frac{|u(t,x_0)-u(s,x_0)|}{|t-s|^\rho\sqrt{\log(1+|t-s|^{-1})}} = K_1 \quad \text{a.s.},
          \end{align}
          \begin{align}\label{E:01law:LIL:t}
            \lim_{\varepsilon \to 0^+} \sup_{t: 0<|t-t_0|\le \varepsilon} \frac{|u(t,x_0)-u(t_0,x_0)|}{|t-t_0|^\rho\sqrt{\log\log(|t-t_0|^{-1})}} = K_2 \quad \text{a.s.},
          \end{align}
          \begin{align}\label{E:01law1}
            \liminf_{\varepsilon \to 0^+} \sup_{t : |t-t_0| \le \varepsilon} \frac{|u(t, x_0) - u(t_0, x_0)|}{\varepsilon^\rho (\log\log(1/\varepsilon))^{-a}} = K_3 \quad \text{a.s.},
          \end{align}
          \begin{align}\label{E:01law2}
            \liminf_{\varepsilon \to 0^+} \inf_{t_0 \in I} \sup_{t : |t-t_0| \le \varepsilon} \frac{|u(t, x_0) - u(t_0, x_0)|}{\varepsilon^\rho(\log(1/\varepsilon))^{-a}} = K_4 \quad \text{a.s.}
          \end{align}

    \item If $\tilde\rho \in (0, 1)$, then for any $t_0 > 0$, for any compact
          interval $J$ in $\R^d$, there exist constants $K_1', K_2', K_3', K_4' \in
            [0, \infty]$ such that for any $x_0 \in J$,
          \begin{align}\label{E:01law:mc:x}
            \lim_{\varepsilon\to0^+} \sup_{x,y\in J: 0<|x-y|\le \varepsilon} \frac{|u(t_0,x)-u(t_0,y)|}{|x-y|^{\tilde\rho}\sqrt{\log(1+|x-y|^{-1})}} = K_1' \quad \text{a.s.},
          \end{align}
          \begin{align}\label{E:01law:LIL:x}
            \lim_{\varepsilon\to0^+} \sup_{x: 0<|x-x_0|\le \varepsilon} \frac{|u(t_0,x)-u(t_0,x_0)|}{|x-x_0|^{\tilde\rho}\sqrt{\log\log(|x-x_0|^{-1})}} = K_2' \quad \text{a.s.},
          \end{align}
          \begin{align}\label{E:01law3}
            \liminf_{\varepsilon \to 0^+} \sup_{x: |x-x_0|\le \varepsilon} \frac{|u(t_0, x) - u(t_0, x_0)|}{\varepsilon^{\tilde{\rho}}(\log\log(1/\varepsilon))^{-a}} = K_3' \quad \text{a.s.},
          \end{align}
          \begin{align}\label{E:01law4}
            \liminf_{\varepsilon \to 0^+} \inf_{x_0 \in J} \sup_{x: |x-x_0| \le \varepsilon} \frac{|u(t_0, x) - u(t_0, x_0)|}{\varepsilon^{\tilde{\rho}}(\log(1/\varepsilon))^{-a}} = K_4' \quad \text{a.s.}
          \end{align}

  \end{enumerate}
\end{theorem}

\begin{proof}
  If $\rho\in(0,1)$ and $\tilde\rho\in(0,1)$, then by
  Theorem~\ref{T:main-holder},
  \[
    \lim_{\varepsilon\to0^+} \sup_{t,s\in I: 0<|t-s|\le\varepsilon} \frac{\E[(u(t,x_0)-u(s,x_0))^2]^{1/2}}{|t-s|^\rho\sqrt{\log(1+|t-s|^{-1})}}=0
  \]
  and
  \[
    \lim_{\varepsilon\to0^+}\sup_{x,y\in J: 0<|x-y|\le\varepsilon}\frac{\E[(u(t_0,x)-u(t_0,y))^2]^{1/2}}{|x-y|^{\tilde\rho}\sqrt{\log(1+|x-y|^{-1})}} = 0.
  \]
  Hence, an application of Lemma 7.1.1 of \cite{marcus.rosen:06:markov} yields
  \eqref{E:01law:mc:t}, \eqref{E:01law:LIL:t}, \eqref{E:01law:mc:x} and
  \eqref{E:01law:LIL:x} for some constants $K_1,K_2,K_1',K_2'\in[0,\infty]$.

  For the rest of the proof, we may identify $u=v$. For each $n \ge 1$, let $A_n
    = \{(\tau, \xi) \in \R \times \R^d : n-1 \le |\tau| \vee |\xi| < n\}$ and
  $v_n(t, x) = v(A_n, t, x)$. Then
  \begin{align*}
    u(t, x) = \sum_{n=1}^\infty v_n(t, x),
  \end{align*}
  where $\{v_n\}_{n \ge 1}$ is a sequence of independent Gaussian processes. Set
  \begin{align*}
    Y_n(t, x) = \sum_{m=1}^n v_m(t, x), \quad \text{and} \quad
    Z_n(t, x) = \sum_{m=n+1}^\infty v_m(t, x).
  \end{align*}
  Then $u = Y_n + Z_n$ for each $n \ge 1$. To show \eqref{E:01law1} and \eqref{E:01law2}, let us fix $0 < S < T$,
  and use \eqref{v(A,t,x)} and \eqref{E:FG} to deduce that for each fixed $n \ge 1$ and for $t < t'$ in $I \coloneqq [S,
      T]$ and $x_0 \in \R^d$,
  \begin{align*}
     & \E\left[ \left(Y_n(t', x_0) - Y_n(t, x_0)\right)^2 \right]                                                                                                                                                                                           \\
     & \lesssim \iint_{|\tau| \vee |\xi| \le n} \ud \tau\, \ud \xi \, |\tau|^{1-2H} |\xi|^{\ell-d} |e^{-i\tau t'} - e^{-i\tau t}|^2 \left| \int_0^t e^{i\tau s} s^{\beta+\gamma-1} E_{\beta, \beta+\gamma}(-2^{-1}\nu |\xi|^\alpha s^\beta) \ud s \right|^2 \\
     & \quad + \iint_{|\tau| \vee |\xi| \le n} \ud \tau \, \ud \xi \, |\tau|^{1-2H} |\xi|^{\ell-d} \left| \int_t^{t'} e^{i\tau s} s^{\beta+\gamma-1} E_{\beta, \beta+\gamma}(-2^{-1}\nu|\xi|^\alpha s^\beta) \ud s\right|^2.
  \end{align*}
  By $|e^{-i\tau t'} - e^{-i\tau t}| \le 2|\tau| |t'-t|$ and Lemma \ref{L:ML}, the above is
  \begin{align*}
     & \lesssim n^2 (t'-t)^2 \int_{\R} \int_{\R^d} \ud \tau \, \ud \xi \, |\tau|^{1-2H} |\xi|^{\ell-d} \left| \int_0^t e^{i\tau s} s^{\beta+\gamma-1} E_{\beta, \beta+\gamma}(-2^{-1}\nu|\xi|^\alpha s^\beta) \ud s \right|^2 \\
     & \quad + (t'-t)^2 \int_{|\tau| \le n} \ud \tau \, |\tau|^{1-2H} \int_{|\xi|\le n} \ud \xi \, |\xi|^{\ell-d}                                                                                                             \\
     & \le C_n (t'-t)^2,
  \end{align*}
  where the first integral is equal to $C \E[(u(t,0))^2]$ which is bounded over all $t \in I$ by \eqref{E:utx-rho} and Theorem \ref{T:Dalang} under condition \eqref{E:main}.
  Then, by Kolmogorov's continuity theorem, for any $\delta \in (0, 1)$, $Y_n(t, x_0)$
  is a.s.\ $(1-\delta)$-H\"older continuous with respect to $t$ on $I$. As $\rho
    \in (0, 1)$, it follows that
  \begin{align*}
    \limsup_{\varepsilon \to 0^+} \sup_{t: |t-t_0| \le \varepsilon} \frac{|Y_n(t, x_0) - Y_n(t_0, x_0)|}{\varepsilon^\rho(\log\log(1/\varepsilon))^{-a}} = 0 \quad \text{a.s.}
  \end{align*}
  and
  \begin{align*}
    \limsup_{\varepsilon \to 0^+} \sup_{t_0 \in I} \sup_{t: |t-t_0| \le \varepsilon} \frac{|Y_n(t, x_0) - Y_n(t_0, x_0)|}{\varepsilon^{\rho}(\log(1/\varepsilon))^{-a}} = 0 \quad \text{a.s.}
  \end{align*}
  Hence, this implies that the random variables appearing on left-hand
  sides of \eqref{E:01law1} and \eqref{E:01law2} are a.s. equal to
  \begin{align*}
    \liminf_{\varepsilon \to 0^+} \sup_{t: |t-t_0| \le \varepsilon} \frac{|Z_n(t, x_0) - Z_n(t_0, x_0)|}{\varepsilon^\rho(\log\log(1/\varepsilon))^{-a}}
    \quad \text{and} \quad
    \liminf_{\varepsilon \to 0^+} \inf_{t_0 \in I} \sup_{t: |t-t_0| \le \varepsilon} \frac{|Z_n(t, x_0) - Z_n(t_0, x_0)|}{\varepsilon^{\rho}(\log(1/\varepsilon))^{-a}}
  \end{align*}
  respectively, which is true for arbitrary $n \ge 1$, so they are both
  measurable with respect to the tail $\sigma$-algebra generated by $v_n$. By
  independence of $v_n$ and Kolmogorov's 0-1 law, those random variables in \eqref{E:01law1} and \eqref{E:01law2} are
  constant a.s.

  Finally, \eqref{E:01law3} and \eqref{E:01law4} can be proved in a similar way based on the following estimate:
  \begin{align*}
     & \E\left[ \left( Y_n(t_0, x') - Y_n(t_0, x)\right)^2 \right]                                                                                                                                                                                                      \\
     & \lesssim \iint_{|\tau| \vee |\xi| \le n} \ud \tau \, \ud \xi \, |\tau|^{1-2H} |\xi|^{\ell-d} |e^{-i\xi\cdot x'} - e^{-i\xi \cdot x}|^2 \left| \int_0^{t_0} e^{i\tau s} s^{\beta+\gamma-1} E_{\beta, \beta+\gamma}(-2^{-1}\nu|\xi|^\alpha s^\beta) \ud s\right|^2 \\
     & \lesssim n^2 |x'-x|^2 \int_{|\tau| \le n} \ud \tau \, |\tau|^{1-2H} \int_{|\xi| \le n} \ud \xi \, |\xi|^{\ell-d}                                                                                                                                                 \\
     & \lesssim C_n |x'-x|^2.
    \qedhere
  \end{align*}
\end{proof}

\subsection{Proof of Corollaries \ref{C:MC}, \ref{C:LIL} and \ref{C:Chung}}
\index{modulus of continuity!uniform}%
\index{modulus of continuity!local}%
\index{law of the iterated logarithm (LIL)}%
\index{Chung's law of the iterated logarithm}

\begin{proof}[Proof of Corollary \ref{C:MC}]
  Theorem 6.3.3 of \cite{marcus.rosen:06:markov} applied to the Gaussian process
  $\{ u(t, x_0): t \in I \}$ yields a random variable $Z$ with $\E[Z] < \infty$
  such that for all $t, s \in I$,
  \begin{align*}
    |u(t, x_0) - u(s, x_0)|
    \le Z \int_0^{d_1(t, s)} \left[ \left( \log\frac{1}{\lambda(B_{d_1}(t, r))} \right)^{1/2} + \left( \log\frac{1}{\lambda(B_{d_1}(s, r))} \right)^{1/2} \right] dr \quad \text{a.s.},
  \end{align*}
  where $d_1(t, s) = \E[(u(t, x_0)-u(s, x_0))^2]^{1/2}$, $\lambda$ is the
  (normalized) Lebesgue measure on $I$, and $B_{d_1}(t, r) = \{ t' \in I :
    d_1(t', t) \le r \}$. Using Theorem \ref{T:main-holder}, we can verify that if
  $\rho \in (0, 1]$, then
  \begin{align*}
    \int_0^{d_1(t, s)} \left( \log\frac{1}{\lambda(B_{d_1}(t, r))} \right)^{1/2} dr \le C w_1(|t-s|).
  \end{align*}
  This proves \eqref{E:MC1}. Similarly, an application of Theorem 6.3.3 of
  \cite{marcus.rosen:06:markov} to the spatial process $\{ u(t_0, x), x \in J\}$
  shows \eqref{E:MC2}. If $\rho \in (0, 1)$, then by Theorem \ref{T:01law},
  there exists a constant $C_1 \in [0, \infty]$ such that
  \begin{align*}
    \lim_{\varepsilon \to 0^+} \sup_{t, s \in I:\, 0<|t-s|\le \varepsilon} \frac{|u(t, x_0) - u(s, x_0)|}{{|t-s|}^\rho \sqrt{\log(1+ |t-s|^{-1})}} = C_1 \quad \text{a.s.}
  \end{align*}
  From \eqref{E:MC1}, we have $C_1 < \infty$. Using one-sided or two-sided
  strong local nondeterminism, we can prove $C_1 > 0$ as in
  \cites{meerschaert.wang.ea:13:fernique-type, lee.xiao:19:local}. This
  proves \eqref{E:exactMC1}. The proof of \eqref{E:exactMC2} is similar.
\end{proof}

Recall the harmonizable representation \eqref{E:har_rep} for the solution
$u(t,x)$. Let $\mathcal{B}(\R_+)$ denote the Borel $\sigma$-algebra on $\R_+$.
For any $t > 0$, $x \in \R^d$ and $A \in \mathcal{B}(\R_+)$, define the
following processes by truncating the harmonizable representation:
\begin{align*}
  v_1(A, t, x) = \Re \iint_{|\tau|^\rho \in A, \xi \in \R^d} \mathcal{F}[G(t-\cdot, x-\ast)\one_{[0, t]}(\cdot)](\tau, \xi)\, |\tau|^{\frac{1-2H}{2}} |\xi|^{\frac{\ell-d}{2}}\, M(\ud\tau, \ud\xi)
\end{align*}
and
\begin{align*}
  v_2(A, t, x) = \Re \iint_{\tau \in \R , |\xi|^\theta \in A} \mathcal{F}[G(t-\cdot, x-\ast)\one_{[0, t]}(\cdot)](\tau, \xi)\, |\tau|^{\frac{1-2H}{2}} |\xi|^{\frac{\ell-d}{2}}\, M(\ud\tau, \ud\xi),
\end{align*}
where 
\[
    \rho = \beta+\gamma+H-\frac{\ell\beta}{2\alpha}-1 \quad \text{and} \quad \theta = \min\left\{\frac{\alpha\rho}{\beta}, \alpha - \frac{\ell}{2}\right\}.
\]
In the next lemma, we
verify Assumption 2.1 in \cite{lee.xiao:23:chung-type} for the temporal process
$\{v(t, x_0): t > 0\}$ and the spatial process $\{v(t_0, x): x \in \R^d\}$,
showing that small increments of the truncated processes serve as good
approximations of the increments of the original processes.

\begin{lemma}\label{L:LX23cond}
  Let $\alpha > 0$, $\beta \in (0, 2]$, $\gamma \ge 0$, $H \in [1/2, 1)$ and
  $\ell \in (0, 2d)$. Assume Dalang's condition
  \eqref{E:main}, and suppose $\rho \in (0, 1)$ and $\theta \in (0, 1)$. Then
  the following estimates hold:
  \begin{enumerate}[\rm (i)]

    \item If $\beta \in (0,1]$ and $\gamma \in [0,1-\beta]$, for any compact interval $I = [S, T]$ in $(0,
            \infty)$ and $x_0 \in \R^d$, there exists a positive finite constant $C$
          such that for all $t, t' \in I$ and $a_0 := 1 \vee (1/S^\rho) \le a < b \le \infty$,
          \begin{align}\label{E:LXA-t}
            \| v_1([a, b), t', x_0) - v(t', x_0) - v_1([a, b), t, x_0) + v(t, x_0) \|_{L^2}
            \le C \left( a^{\frac{1}{\rho} - 1} |t'-t| + b^{-1} \right)
          \end{align}
          and
          \begin{align}\label{E:LXA-t2}
            \|v_1([0,a_0),t',x_0)-v_1([0,a_0),t,x_0)\|_{L^2} \le C |t'-t|.
          \end{align}

    \item If $\beta \in (0,2]$ and $\gamma \ge 0$, for any compact rectangle $J$ in $\R^d$ and $t_0 > 0$, there exists
          a positive finite constant $C'$ such that for all $x, x' \in J$ and $0
            \le a < b \le \infty$,
          \begin{align}\label{E:LXA-x}
            \| v_2([a, b), t_0, x) - v(t_0, x) - v_2([a, b), t_0, x') + v(t_0, x') \|_{L^2}
            \le C' \left( a^{\frac{1}{\theta} - 1} |x-x'| + b^{-1} \right).
          \end{align}

  \end{enumerate}
\end{lemma}

\begin{proof}
  (i). Let $t$, $t' \in I$ be such that $t < t'$, and let $a_0 \le a < b \le
    \infty$. Note that
  \begin{align*}
     & \|v_1([a, b), t', x_0) -  v(t', x_0) - v_1([a, b), t, x_0) + v(t, x_0)\|_{L^2}           \\
     & = \|v_1([0, a) \cup [b, \infty), t', x_0) - v_1([0, a) \cup [b, \infty), t, x_0)\|_{L^2} \\
     & \lesssim I_1 + I_2,
  \end{align*}
  where
  \begin{align*}
    I_1 & = \int_{|\tau|^\rho < a } \ud \tau\,|\tau|^{1-2H}  \int_{\R^d}
    \ud \xi\, |\xi|^{\ell-d} \left| F(t',\tau,\xi)-F(t,\tau,\xi)\right|^2, \\
    I_2 & = \int_{|\tau|^\rho \ge b} \ud \tau\, |\tau|^{1-2H} \int_{\R^d}
    \,\ud \xi\,  |\xi|^{\ell-d} \left| F(t',\tau,\xi)-F(t,\tau,\xi) \right|^2,
  \end{align*}
  and
  \begin{align*} % \label{E:F(t,tau,xi)}
    F(t,\tau,\xi) = e^{-i\tau t} \int_0^t
    e^{i\tau s}s^{\beta+\gamma-1} E_{\beta,
        \beta+\gamma}\left(-2^{-1}\nu|\xi|^\alpha s^\beta\right) \ud s.
  \end{align*}
  In order to estimate $I_1$, we split the integral as follows:
  \begin{align*}
    I_1 = J_1 + J_2,
  \end{align*}
  where
  \begin{align*}
     & J_1 = \int_0^{1/S} \ud \tau \, \tau^{1-2H} \int_{\R^d} \ud \xi \, |\xi|^{\ell-d} \left|F(t',\tau,\xi)-F(t,\tau,\xi) \right|^2,            \\
     & J_2 = \int_{1/S}^{a^{1/\rho}} \ud \tau \, \tau^{1-2H} \int_{\R^d} \ud \xi \, |\xi|^{\ell-d} \left|F(t',\tau,\xi)-F(t,\tau,\xi) \right|^2.
  \end{align*}
  Since
  \begin{align} \label{E:bound}
    \sup_{x\ge0}|(1 + x) E_{\beta,c}(-x)|<\infty, \quad \sup_{x\ge0}|(1 + x^2) E_{\beta,\beta}(-x)|<\infty \quad \text{for $\beta \in (0,2),c \in \R$}
  \end{align} (see Lemma \ref{L:ML}), we have
  \begin{align*}
    \begin{split} % \label{E:|F-F|:bd}
      |F(t',\tau,\xi) - F(t,\tau,\xi)|^2
       & \le 2\left|e^{-i\tau t'} - e^{-i\tau t}\right|^2 \left| \int_0^{t'} e^{i\tau s} s^{\beta+\gamma-1} E_{\beta, \beta+\gamma}(-2^{-1}\nu |\xi|^\alpha s^\beta)\,\ud s\right|^2 \\
       & \quad + 2  \left| \int_t^{t'} e^{i\tau s} s^{\beta+\gamma-1} E_{\beta, \beta+\gamma}(-2^{-1}\nu|\xi|^\alpha s^\beta)\, \ud s\right|^2                                       \\
       & \lesssim (\tau^2 +1) (t'-t)^2 (1 \wedge |\xi|^{-2\alpha}),
    \end{split}
  \end{align*}
  which holds uniformly for all $\tau \ge 0$, $\xi \in \R^d$ and $t<t'$ in $[S, T]$.
  This together with $\ell<2\alpha$ (from Dalang's condition \eqref{E:Dalang2}) implies that
  \begin{align*}
    J_1
     & \lesssim (t'-t)^2 \int_0^{1/S} \ud \tau \, \tau^{1-2H} (1+\tau^2) \int_{\R^d} \ud \xi \, |\xi|^{\ell-d} (1 \wedge |\xi|^{-2\alpha}) \\
     & \lesssim (t'-t)^2.
  \end{align*}
  To estimate $J_2$, we use the change of variables $s \mapsto \tau^{-1} s$ and $\xi \mapsto \tau^{\beta/\alpha} \xi$ to see that
  \begin{align*}
    J_2 & = \int_{1/S}^{a^{1/\rho}} \ud \tau \, \tau^{-1-2\rho} \int_{\R^d} \ud \xi \, |\xi|^{\ell-d} \left| f(\tau t', \xi) - f(\tau t, \xi) \right|^2,
  \end{align*}
  where
  \begin{align}\label{E:f(r,xi)}
    f(r, \xi) = e^{-i r} \int_0^{r} e^{is} s^{\beta+\gamma-1} E_{\beta,\beta+\gamma}(-2^{-1}\nu|\xi|^\alpha s^\beta)\, \ud s.
  \end{align}
  The derivative of $f$ with respect to $r$ is given by
  \begin{align*} %\label{partial_r:f}
    \partial_r f(r,\xi)
    = -i e^{-ir} \int_0^r e^{is} s^{\beta+\gamma-1}E_{\beta,\beta+\gamma}(-2^{-1}\nu|\xi|^\alpha s^\beta) \, \ud s
    + r^{\beta+\gamma-1} E_{\beta, \beta+\gamma}(-2^{-1}\nu |\xi|^\alpha r^\beta).
  \end{align*}
  If $0\le \gamma<1-\beta$, we can split the integral, integrate by parts and use \eqref{E:ML-deriv} to deduce that for $r\ge 1$,
  \begin{align}\begin{split}\label{f_r}
      \partial_r f(r,\xi)
       & = -i e^{-ir} \int_0^1 e^{is} s^{\beta+\gamma-1}E_{\beta, \beta+\gamma}(-2^{-1}\nu|\xi|^\alpha s^\beta)\,\ud s                                                                     \\
       & \quad + e^{-ir} e^i E_{\beta,\beta+\gamma}(-2^{-1}\nu |\xi|^\alpha) + e^{ir} \int_1^r e^{is} s^{\beta+\gamma-2} E_{\beta,\beta+\gamma-1}(-2^{-1}\nu|\xi|^\alpha s^\beta)\, \ud s.
    \end{split}\end{align}
  Then, we can use \eqref{E:bound} to deduce that for $r \ge 1$,
  \begin{align*}
    |\partial_r f(r,\xi)| \lesssim (1 + r^{\beta+\gamma-1})(1 \wedge |\xi|^{-\alpha}) \lesssim 1 \wedge |\xi|^{-\alpha}.
  \end{align*}
  By Lemma \ref{L:ML}, we can deduce that
  % for $\beta\in (0,2)$, $b \in \R$ and $\varepsilon \in [0, \beta]$,
  % \begin{align}\label{E:sE:bd}
  %     \sup_{s\ge0} |s^{\beta-\varepsilon}E_{\beta,b}(-2^{-1}\nu s^\beta)| < \infty,
  % \end{align}
  % and that 
  if $\beta\in(0,2)$, $\gamma\ge 0$, and $c\ge 1$, then for all $s > 0$ and $\xi \in \R^d$,
  \begin{align}\label{E:sE:bd2}
    |s^{\beta+\gamma-c}E_{\beta,\beta+\gamma+(c-1)}(-2^{-1}\nu |\xi|^\alpha s^\beta)|\lesssim
    \begin{cases}
      s^{\beta+\gamma-c}           & \text{if }|\xi| \le 1,  \\
      s^{\gamma-c} |\xi|^{-\alpha} & \text{if } |\xi| \ge 1.
    \end{cases}
  \end{align}
  If $\gamma=1-\beta$, we can apply integration by parts to the last integral in \eqref{f_r} and apply \eqref{E:bound} and  \eqref{E:sE:bd2} to deduce that for all $r \in [\tau t, \tau t']$ with $\tau \ge 1/S$ and $t<t'$ in $[S, T]$,
  \begin{align*}
    |\partial_r f(r,\xi)| & \lesssim 1 \wedge |\xi|^{-\alpha} + |r^{\beta+\gamma-2}E_{\beta,\beta+\gamma-1}(-2^{-1}\nu|\xi|^\alpha r^\beta)|+ \int_1^r |s^{\beta+\gamma-3}E_{\beta,\beta+\gamma-2}(-2^{-1}\nu |\xi|^\alpha s^\beta)| \ud s\\
        & \lesssim 1 \wedge |\xi|^{-\alpha}.
  \end{align*}
  % and if $\gamma = 2-\beta$, we can integrate by parts again and then apply \eqref{E:bound} and  \eqref{E:sE:bd2} to deduce that
  % \begin{align*}
  %   |\partial_r f(r,\xi)|
  %    & \lesssim 1 \wedge |\xi|^{-\alpha} + |r^{\beta+\gamma-3}E_{\beta,\beta+\gamma-2}(-2^{-1}\nu|\xi|^\alpha r^\beta)|+ \int_1^r |s^{\beta+\gamma-4}E_{\beta,\beta+\gamma-3}(-2^{-1}\nu |\xi|^\alpha s^\beta)| \ud s \\
 %         & \lesssim 1 \wedge |\xi|^{-\alpha}.
  % \end{align*}
  Hence, in both cases, by Taylor's theorem, we have
  \begin{align*}
    J_2
     & = \int_{1/S}^{a^{1/\rho}} \ud \tau \, \tau^{-1-2\rho} \int_{\R^d} \ud \xi \, |\xi|^{\ell-d} \left| \int_{\tau t}^{\tau t'} \partial_r f(r, \xi) \ud r \right|^2 \\
     & \lesssim (t'-t)^2 \int_{1/S}^{a^{1/\rho}} \ud \tau \, \tau^{1-2\rho} \int_{\R^d} \ud \xi \, |\xi|^{\ell-d}(1 \wedge |\xi|^{-2\alpha})                           \\
     & \lesssim (t'-t)^2 a^{2/\rho-2}.
  \end{align*}
  Combining the two cases, we have
  \[
    I_1 \lesssim (t'-t)^2 a^{2/\rho-2}.
  \]

  In order to estimate $I_2$, we use the change of variables $s\mapsto \tau^{-1} s$ and $\xi \mapsto \tau^{\beta/\alpha}\xi$ to see that
  \begin{align*}
    I_2 = \int_{b^{1/\rho}}^\infty \ud \tau \, \tau^{-1-2\rho} \int_{\R^d} \ud \xi \, |\xi|^{\ell-d} \left|f(\tau t',\xi) - f(\tau t, \xi)\right|^2,
  \end{align*}
  where $f$ is defined by \eqref{E:f(r,xi)}.
  For $r\ge 1$, split the integral and integrate by parts to get that
  \begin{align*}
    f(r, \xi)
     & = e^{-ir} \int_0^1 e^{is} s^{\beta+\gamma-1}E_{\beta,\beta+\gamma}(-2^{-1}\nu|\xi|^\alpha s^\beta) \, \ud s                                        \\
     & \quad -i r^{\beta+\gamma-1} E_{\beta,\beta+\gamma}(-2^{-1}\nu|\xi|^\alpha r^\beta) + i e^{-ir} e^i E_{\beta,\beta+\gamma}(-2^{-1}\nu |\xi|^\alpha) \\
     & \quad + i \int_1^r e^{is} s^{\beta+\gamma-2}E_{\beta,\beta+\gamma-1}(-2^{-1}\nu|\xi|^\alpha s^\beta) \, \ud s.
  \end{align*}
  If $0\le \gamma<1-\beta$, then
  by \eqref{E:bound} and \eqref{E:sE:bd2}, for all $r \ge 1$ and $\xi \in \R^d$,
  \begin{align*}
    |f(r,\xi)| &\lesssim  1 \wedge |\xi|^{-\alpha} + |r^{\beta+\gamma-1}E_{\beta,\beta+\gamma}(-2^{-1}\nu|\xi|^\alpha r^\beta)| + \int_1^r |s^{\beta+\gamma-2}E_{\beta,\beta+\gamma-1}(-2^{-1}\nu|\xi|^\alpha s^\beta)| \ud s\\
        & \lesssim 1 \wedge |\xi|^{-\alpha};
  \end{align*}
  % if $\beta+\gamma-1 > 0$ and $\gamma \in [0, 1)$, 
  and if $\gamma=1-\beta$,
  % then by choosing $\varepsilon$ such that $0 < \varepsilon < 1-\gamma$, 
  then integrate by parts again and use \eqref{E:bound} and \eqref{E:sE:bd2} to see that for all $r\ge 1$ and $\xi \in \R^d$,
  \begin{align*}
    |f(r,\xi)| & \lesssim 1\wedge |\xi|^{-\alpha} + |r^{\beta+\gamma-2}E_{\beta,\beta+\gamma-1}(-2^{-1}\nu|\xi|^\alpha r^\beta)| + \int_1^r  |s^{\beta+\gamma-3} E_{\beta,\beta+\gamma-2}(-2^{-1}\nu|\xi|^\alpha s^\beta)|\, \ud s \\
        &\lesssim 1 \wedge |\xi|^{-\alpha}.
  \end{align*}
  % if $1\le \beta+\gamma-1< 2$, then integrate by parts another time and use \eqref{E:bound} and \eqref{E:sE:bd2} again to find that for all $r \ge 1$ and $\xi \in \R^d$,
  % \begin{align*}
  %   |f(r,\xi)| \lesssim 1\wedge |\xi|^{-\alpha} + \int_1^r |s^{\beta+\gamma-4} E_{\beta,\beta+\gamma-3}(-2^{-1}\nu|\xi|^\alpha s^\beta)| \ud s \lesssim 1\wedge |\xi|^{-\alpha}.
  % \end{align*}
  Hence,
  \begin{align*}
    I_2 & \lesssim
    \int_{b^{1/\rho}}^\infty \ud \tau \, \tau^{-1-2\rho} \int_{\R^d} \ud \xi \, |\xi|^{\ell-d} \left(|f(\tau t',\xi)|^2 + |f(\tau t, \xi)|\right)^2 \\
        & \lesssim \int_{b^{1/\rho}}^\infty \ud \tau \, \tau^{-1-2\rho} \int_{\R^d} \ud \xi \, |\xi|^{\ell-d}(1\wedge |\xi|^{-2\alpha})
    \lesssim b^{-2}.
  \end{align*}
  Combining the above estimates for $I_1$ and $I_2$ yields \eqref{E:LXA-t}.
  In particular, \eqref{E:LXA-t2} follows from \eqref{E:LXA-t} by taking $a = a_0$ and sending $b\to\infty$.

  \bigskip

  (ii). Note that
  \begin{align*}
     & \|v_2([a, b), t_0, x) - v(t_0, x) - v_2([a, b), t_0, x') + v(t_0, x')\|_{L^2}                                             \\
     & = \|v_2([0, a) \cup [b, \infty), t_0, x) - v_2([0, a) \cup [b, \infty), t_0, x')\|_{L^2}                                  \\
     & \le \|v_2([0, a), t_0, x) - v_2([0, a), t_0, x')\|_{L^2} + \|v_2([b, \infty), t_0, x) - v_2([b, \infty), t_0, x')\|_{L^2}
  \end{align*}
  and
  \begin{align*}
    \|v_2(A, t_0, x) - v_2(A, t_0, x')\|_{L^2}^2
    = \int_{|\xi|^\theta \in A}  |\xi|^{\ell-d} |1-e^{i(x-x')\cdot \xi}|^2 \Delta_{t_0}(|\xi|^\alpha)\, \ud \xi,
  \end{align*}
  where $\Delta_{t_0}(\cdot)$ is defined by~\eqref{E:Delta_t} with $t=t_0$. Hence, for $H = 1/2$,
  \begin{align*}
    \Delta_{t_0}(|\xi|^\alpha)
    = 2\pi\int_0^{t_0} s^{2(\beta+\gamma-1)} E_{\beta, \beta+\gamma}^2\left( -2^{-1} \nu |\xi|^\alpha s^\beta \right)\, \ud s,
  \end{align*}
  and for $H > 1/2$,
  \begin{align*}
    \Delta_{t_0}(|\xi|^\alpha)
     & = 2\pi a_H \iint_{[0, t_0]^2} \ud s_1\ud s_2\, |s_1-s_2|^{2H-2}                                                                                                                                         \\
     & \quad \times s_1^{\beta+\gamma-1}E_{\beta, \beta+\gamma}\left(-2^{-1}\nu|\xi|^\alpha s_1^\beta\right) \times s_2^{\beta+\gamma-1}E_{\beta, \beta+\gamma}\left( -2^{-1}\nu|\xi|^\alpha s_2^\beta\right).
  \end{align*}
  Then, similarly to what we did in~\eqref{E:hls} and \eqref{E:SecMomUB}, in both
  cases we get that
  \begin{align}\label{E:DeltaUB}
    \Delta_{t_0}(|\xi|^\alpha) \lesssim  \frac{t_0^{2(\beta+\gamma+H-1)}}{1+\left[ t_0(2/\nu)^{1/\beta} |\xi|^{\alpha/\beta} \right]^{2\min(\beta, \beta+\gamma+H-1)}}.
  \end{align}
  Let $\tilde{a} = a^{1/\theta}$. Then using~\eqref{E:DeltaUB} and the
  inequality $|1-e^{iz}| \le |z|$, we have
  \begin{align*}
    \|v_2([0, a), t_0, x) - v_2([0, a), t_0, x')\|_{L^2}^2
     & \lesssim |x-x'|^2 \int_{|\xi| < \tilde{a}} \frac{|\xi|^{2+\ell-d}}{|\xi|^{\frac{2\alpha}{\beta}\min(\beta, \beta+\gamma+H-1)}} \ud \xi
  \end{align*}
  It is easy to see that
  \begin{align}\label{E:2theta}
    \frac{2\alpha}{\beta}\min(\beta, \beta+\gamma+H-1)-\ell = 2\theta.
  \end{align}
  Since $\theta < 1$, the last integral is finite and we obtain
  \begin{align}\label{E:LXA-x1}
    \begin{split}
      \|v_2([0, a), t_0, x) - v_2([0, a), t_0, x')\|_{L^2}^2
       & \lesssim |x-x'|^2 \tilde{a}^{2-\frac{2\alpha}{\beta}\min(\beta, \beta+\gamma+H-1) + \ell} \\
       & \le |x-x'|^2 a^{\frac{2}{\theta}-2}.
    \end{split}
  \end{align}
  Similarly, if we let $\tilde{b} = b^{1/\theta}$, then by~\eqref{E:DeltaUB}, the
  inequality $|1-e^{iz}| \le 2$ and the relation~\eqref{E:2theta}, we get that
  \begin{align*}
    \|v_2([b, \infty), t_0, x) - v_2([b, \infty), t_0, x')\|_{L^2}^2
    \lesssim \int_{|\xi| \ge \tilde{b}} \frac{|\xi|^{\ell-d}}{|\xi|^{\frac{2\alpha}{\beta}\min(\beta, \beta+\gamma+H-1)}} \ud \xi.
  \end{align*}
  Using the condition that $\theta > 0$, we see that the last integral is finite
  and
  \begin{align}\label{E:LXA-x2}
    \begin{split}
      \|v_2([b, \infty), t_0, x) - v_2([b, \infty), t_0, x')\|_{L^2}^2
      \lesssim \tilde{b}^{-\frac{2\alpha}{\beta}\min(\beta, \beta+\gamma+H-1) + \ell} = b^{-2}.
    \end{split}
  \end{align}
  Combining~\eqref{E:LXA-x1} and~\eqref{E:LXA-x2} yields~\eqref{E:LXA-x}.
\end{proof}

\begin{proof}[Proof of Corollary \ref{C:LIL}]
  For $\rho \in (0, 1)$, Theorem \ref{T:01law} implies that \eqref{E:LIL1} and
  \eqref{E:LIL2} hold for some constants $C_3 \in [0, \infty]$ and $C_4 \in [0,
      \infty]$. As in the proof of Theorem 5.2 of \cite{lee.xiao:23:chung-type}, we
  can prove that $C_3 < \infty$ and $C_4 < \infty$ under \eqref{E:main-holder-t}
  and \eqref{E:main-holder-s}, respectively. Moreover, Lemma \ref{L:LX23cond}
  above shows that Assumption 2.1 in \cite{lee.xiao:23:chung-type} is satisfied
  by the temporal process $\{v(t,x_0): t\in I\}$ when 
  $\gamma\in[0,1-\beta]$ and by the spatial process $\{v(t_0,x): x\in J\}$. Assumption
  2.3 in \cite{lee.xiao:23:chung-type} is also satisfied thanks to Theorem
  \ref{T:main-holder}. Therefore, we may apply Theorem 5.2 of
  \cite{lee.xiao:23:chung-type} to find that $C_3>0$ and \eqref{E:LIL1b} holds
  when $\gamma\in[0,1-\beta]$, and find that $C_4>0$ and
  \eqref{E:LIL2b} holds.
\end{proof}

\begin{proof}[Proof of Corollary \ref{C:Chung}]
  First, Theorem \ref{T:01law} above shows that each liminf in
  \eqref{E:Chung1}--\eqref{E:mod-non2} takes a constant value in $[0, \infty]$
  a.s. Theorem \ref{T:SLND} shows that strong local nondeterminism holds for the
  temporal process when $\beta=1$, $\gamma\ge0$ or $\beta=2$, $\gamma\in[0,1]$
  and for the spatial process for the whole range of parameters. Together with Theorem \ref{T:main-holder},
  this shows that the two conditions in \cite{lee.xiao:23:local} hold.
  Therefore, Theorem 3.3 of~\cite{lee.xiao:23:local} shows that the constants
  $C_5, C_5', C_6, C_6'$ are all strictly positive.

  Moreover,Assumptions 2.1 and 2.2 in
  \cite{lee.xiao:23:chung-type} hold for the spatial process $\{v(t_0,x): x\in
    J\}$ thanks to Lemma \ref{L:LX23cond}(ii) and strong local nondeterminism
  \eqref{E:SLND:x}. Hence, we may appeal to Theorem 4.4 of
  \cite{lee.xiao:23:chung-type} to see that $0<C_6<\infty$. When $\gamma\in[0,1-\beta]$, Lemma \ref{L:LX23cond}(i) shows that Assumption 2.1 in
  \cite{lee.xiao:23:chung-type} holds for the temporal process $\{v(t,x_0):t\in
    I\}$. By Theorem~\ref{T:SLND}(i)(b), one-sided strong local nondeterminism
  holds for this process. A careful examination reveals that Proposition 4.2 of
  \cite{lee.xiao:23:chung-type} continues to hold when $N=1$ and when its
  Assumption 2.2 is replaced by one-sided local nondeterminism
  \eqref{E:1SLND:t}. Therefore, Theorem 4.4 of \cite{lee.xiao:23:chung-type}
  continues to hold. This implies $0<C_5<\infty$.
\end{proof}

\subsection{Small ball probabilities and proof of Corollary \ref{C:smallball}}
\index{small ball probability}

The following lemma establishes small ball probability estimates for Gaussian
processes under one-sided strong local nondeterminism.

\begin{lemma}\label{L:X:smallball}
  Let $I \subset \R$ be a compact interval and $X=\{X(t): t\in I\}$ be a centered
  Gaussian process. Suppose there exist constants $\rho\in(0,1)$ and $c_1, c_2,
    c_3\in (0,\infty)$ such that
  \begin{align}\label{E:X:incr}
    c_1 |t-s|^\rho \le \E[(X(t)-X(s))^2]^{1/2} \le c_2 |t-s|^\rho \quad \text{for all $t,s\in I$}
  \end{align}
  and $X$ satisfies one-sided strong local nondeterminism:
  \begin{align}\label{E:X:SLND}
    \Var(X(t)|X(t_1),\dots,X(t_n)) \ge c_3 \min_{1\le i \le n}(t-t_i)^{2\rho}
  \end{align}
  uniformly for all integers $n \ge 1$ and for all $t,t_1,\dots, t_n\in I$ with
  $\max\{t_1,\dots, t_n\} \le t$. Then there exist constants $0<C_1\le C_2
    <\infty$ such that for all $\varepsilon \in (0,1)$,
  \begin{align}\label{E:X:smallball}
    \exp\left(- \frac{C_2}{\varepsilon^{1/\rho}} \right) \le \mathbb{P}\left\{ \sup_{t\in I}|X(t)| \le \varepsilon \right\} \le \exp\left( - \frac{C_1}{\varepsilon^{1/\rho}} \right).
  \end{align}
\end{lemma}

\begin{proof}
  The lower bound in \eqref{E:X:smallball} follows from \eqref{E:X:incr} and
  Lemma 5.2 in \cite{xiao:09:sample}. The proof of the upper bound in
  \eqref{E:X:smallball} is similar to that of Theorem 5.1 in
  \cite{xiao:09:sample}. Without loss of generality, we assume that $I=[0,1]$.
  For any $\varepsilon\in(0,1)$, take $n$ such that $n \le
    (\sqrt{c_3}\varepsilon)^{-1} < 2n$. Let $t_i = i n^{-1/\rho}$ for $1 \le i \le
    n^{1/\rho}$. Then
  \begin{align*}
    \mathbb{P}\left\{ \sup_{t\in I}|X(t)| \le \varepsilon \right\} \le \mathbb{P}\left\{ \max_{1 \le i \le n^{1/\rho}} |X(t_i)| \le \varepsilon\right\}.
  \end{align*}
  For any $1 \le k \le n^{1/\rho}$, the conditional distribution of $X(t_k)$ given
  $X(t_1),\dots, X(t_{k-1})$ is a Gaussian distribution. By \eqref{E:X:SLND},
  the conditional variance of this distribution has the lower bound:
  \begin{align*}
    \mathrm{Var}(X(t_k)|X(t_1),\dots, X(t_{k-1})) \ge c_3 n^{-2}.
  \end{align*}
  This together with Anderson's inequality~\cite{anderson:55:integral} implies
  that
  \begin{align*}
    \mathbb{P}\left\{ |X(t_k)|\le\varepsilon | X(t_1),\dots, X(t_{k-1}) \right\}
    \le \mathbb{P}\left\{ |Z| \le \frac{n\varepsilon}{\sqrt{c_3}} \right\}
    \le \mathbb{P}\left\{ |Z| \le 1 \right\}
    \eqqcolon p,
  \end{align*}
  where $Z$ denotes a standard normal random variable. Thanks to the preceding
  bound, we can deduce by successive conditioning and induction that
  \begin{align*}
    \mathbb{P}\left\{ \max_{1 \le i \le n^{1/\rho}} |X(t_i)|
    \le \varepsilon\right\}
    \le p^{n^{1/\rho}}
    =  \exp\left( - n^{1/\rho} \log(1/p) \right)
    \le \exp\left( - \frac{\log(1/p)}{(2\sqrt{c_3})^{1/\rho}\varepsilon^{1/\rho}} \right).
  \end{align*}
  This completes the proof of Lemma~\ref{L:X:smallball}.
\end{proof}

\begin{proof}[Proof of Corollary \ref{C:smallball}]
  According to Theorem~\ref{T:SLND}(i)(b), when $\beta\in(0,2)$ the temporal
  process satisfies one-sided strong local nondeterminism. Moreover, when
  $\beta=2$ and $\gamma\in [0,1]$, Theorem~\ref{T:SLND}(i)(a) yields (two-sided) strong
  local nondeterminism in time, which implies the one-sided estimate
  \eqref{E:X:SLND} needed in Lemma~\ref{L:X:smallball}. This and the variance
  bounds \eqref{E:main-holder-t} allow us to apply Lemma~\ref{L:X:smallball}
  above to obtain the small ball estimate \eqref{E:smallball1}.

  The small ball estimate \eqref{E:smallball2} for the spatial process is a
  direct consequence of Theorem 5.1 of \cite{xiao:09:sample} under the variance
  bounds \eqref{E:main-holder-s} and strong local nondeterminism
  \eqref{E:SLND:x} for the whole range of parameters.
\end{proof}

\appendix
\setcounter{section}{1} % start appendices at A
\setcounter{subsection}{0}
\section*{Appendix}
\addcontentsline{toc}{section}{Appendix}
\subsection{A trigonometric integral}\label{S:Trigonometry}\index{trigonometric integral}

For wave-type equations (i.e., $\beta>1$), the Fourier transforms of their
fundamental solution $\widehat{G}(t,\xi)$ exhibit intricate oscillatory
behavior, crossing zero infinitely many times when $\beta = 2$. This distinct
oscillation imparts unique characteristics to the solutions, setting them apart
from equations with $\beta \le 1$. It also introduces greater complexity when
attempting to establish precise estimates for these solutions. In this study,
we are only able to derive precise results under the assumption that $\beta=2$
and $\gamma = 0$, where the explicit form of $\widehat{G}(t,\xi)$ is given by
equation~\eqref{E:FG_b2g0}. This section is dedicated to providing rigorous
estimates for trigonometric integrals, which play a pivotal role in the
analysis of the wave equation case ($\beta = 2$ and $\gamma = 0$).

For any $\theta \in \R$ and $s_1, s_2 \in \R_+$ with $s_1 \leq s_2$, denote by
  \begin{align}\label{E:F_tht}
      F_{\theta}(s_1,s_2) \coloneqq & \int_0^{\infty} \ud x \; x^{\theta} \sin (s_1 x) \sin(s_2 x).
  \end{align}

\bigskip\noindent\textit{Usage in the main text.}
Lemma~\ref{L:F_tht} below analyzes the oscillatory integral $F_\theta(s_1,s_2)$,
which is the basic building block for the wave equation regime $\beta=2$ and
$\gamma=0$, where
$\widehat{G}(t,\xi)=\sin(\sqrt{\nu/2}\,t|\xi|^{\alpha/2})/|\xi|^{\alpha/2}$ (see
\eqref{E:FG_b2g0}). In Section~\ref{S:Dalang}, Subsection~\ref{SS:ii} (proof of
Theorem~\ref{T:Dalang}(ii)), the variance $\E[u(t,x)^2]$ is reduced to a
time-integral of $F_{2\ell/\alpha-3}(s_1,s_2)$, and Lemma~\ref{L:F_tht} provides
the sharp finiteness range in $\theta$ and the explicit evaluations in the
diagonal ($s_1=s_2$) and off-diagonal ($s_1<s_2$) cases; see
\eqref{E:int_sin0}--\eqref{E:int_sin}. The same formulas are used in
Section~\ref{S:Upper-t} to estimate temporal increments in the wave case via
reductions to combinations of $F_{2\ell/\alpha-3}$; see, e.g.,
Proposition~\ref{T:Upper-t}(ii) and the proof of Lemma~\ref{L:recF}. Moreover,
in Section~\ref{S:SLN_t} (the wave case in the proof of Theorem~\ref{T:LND_t}),
we need the convergence of $\int_0^x s^{\gamma-1}\cos(s)\,\ud s$ and
$\int_0^x s^{\gamma-1}\sin(s)\,\ud s$ as $x\to\infty$; this is shown in the
proof of Lemma~\ref{L:F_tht} and is used there to control oscillatory
expressions.

Then, we have the following lemma.
\begin{lemma}\label{L:F_tht}
    \begin{enumerate}[\rm (i)]
    \item If $s_1 = s_2 = s$, then $F_{\theta}(s,s)$ is convergent if and only
      if $-3 < \theta < -1$, and in that case it holds that
      \begin{gather}\label{E:int_sin0}
        F_{\theta}(s,s) =
        \begin{dcases}
          \frac{\Gamma( 1+ \theta) \sin (\theta \pi/2)}{2^{2 + \theta}} s^{-1-\theta}, & \theta \in (-3,-1)\setminus \{-2\},\\
          \frac{\pi}{2}s, & \theta = -2.
        \end{dcases}
      \end{gather}

    \item If $s_1 < s_2$, then $F_{\theta}(s_1,s_2)$ is nonnegative and finite
      if and only if $-3 < \theta < 0$, and in that case it holds that
      \begin{gather}\label{E:int_sin}
        F_{\theta}(s_1,s_2) =
        \begin{dcases}
          \frac{1}{2} \Gamma(1+\theta)\sin\left(\frac{\pi\theta}{2}\right) \left( (s_2 + s_1)^{-1-\theta} -  (s_2 - s_1)^{-1-\theta} \right), & \theta \in (-3, 0)\setminus \{-2, -1\}, \\
          \frac{\pi}{2} s_1,                                                                                                                  & \theta = -2,                            \\
          \frac{1}{2} \log \left(\frac{s_1 + s_2}{s_2 - s_1} \right),                                                                         & \theta = -1.
        \end{dcases}
      \end{gather}

  \end{enumerate}
\end{lemma}

\begin{proof}
  Denote by
  \begin{align*}
    f_{\theta,s_1,s_2} (x) \coloneqq  x^{\theta} \sin (s_1 x) \sin(s_2 x)
    = \frac{x^{\theta}}{2} \Big( \cos \big( (s_2 - s_1 )x \big) - \cos \big( (s_1+ s_2 )x \big)\Big) \quad \text{for all $x \in \R$}.
  \end{align*}
  In case $s_1 = s_2 = s$, $f_{\theta, s,s}$ is integrable on $[0,\infty)$ if
  and only if $\theta \in (-3, -1)$ because $f_{\theta, s,s}(\cdot)\ge 0$ and
  $x^{\theta} \sin^2(s x) \asymp O\left(x^{\theta + 2}\right)$ as $x\downarrow
  0$. Hence, for part (i), it suffices to prove~\eqref{E:int_sin0}, which can be
  combined in the proof of part (ii); see Case~2 below. As for part (ii), we
  will prove it in the following cases:

  \bigskip\noindent\textbf{Case~1.~} If $-1 < \theta < 0$ and $s_1 < s_2$, we can write
  \begin{align*}
    F_{\theta}(s_1,s_2) = & \frac{1}{2} \int_0^{\infty} \ud x \; x^{\theta} \Big( \cos \big( (s_2 - s_1 ) x \big) - \cos \big( (s_1+ s_2 )x \big)\Big).
  \end{align*}
  Because $-1 < \theta < 0$,  we can apply the Mellin transform (see, e.g.,
  \cite[eq.~5.2 on p.~42]{oberhettinger:74:tables}) for the cosine function to
  obtain that
  \begin{align*}
    \int_0^{\infty} \ud x \; x^{\theta} \cos(x)
    = \Gamma(1+\theta) \cos\left((1+\theta)\pi /2\right)
    = -\Gamma(1+\theta) \sin(\theta \pi /2).
  \end{align*}
  One can thus deduce~\eqref{E:int_sin} after change of variables. Note that in
  this case $\theta+1>0$ and hence one needs to avoid the case $s_1=s_2$.
  Moreover, since $\Gamma(1+\theta)\sin(\pi\theta/2)<0$ (see
  Fig.~\ref{F:GammaSin}), together with $\theta+1>0$, we see that
  $F_\theta(s_1,s_2) \ge 0$.

  \bigskip\noindent\textbf{Case~2.~} If $-3 < \theta < -1$ and $s_1 \leq s_2$,
  it follows from integration by parts formula
  \begin{align*}
    F_{\theta} (s_1, s_2) =
    & \frac{1}{2} \left( \frac{x^{\theta + 1}}{\theta + 1} \Big( \cos \big( (s_2 - s_1 ) x \big) - \cos \big( (s_1+ s_2 )x \big)\Big) \bigg|_0^{\infty} \right. \\
    & \qquad + \left. \int_0^{\infty} \ud x \; \frac{x^{\theta + 1}}{\theta + 1} \Big( (s_2 - s_1) \sin \big( (s_2 - s_1 ) x \big) - (s_1 + s_2) \sin \big( (s_1+ s_2 )x \big)\Big) \right).
  \end{align*}
  Since $\cos\left((s_2 - s_1 )x\right) - \cos\left((s_1+ s_2 )x\right) =
  2s_1s_2 x^2 + O(x^3)$ as $x\downarrow 0$, we see that
  \begin{align*}
      x^{1 + \theta}\Big( \cos \big( (s_2 - s_1 ) x \big) - \cos \big( (s_1+ s_2 )x \big)\Big) \to  0,\quad \text{as $x\downarrow 0$ and $x \uparrow \infty$}.
  \end{align*}
  By the Mellin transform for the sine function (see, e.g., \cite[eq.~5.1 on
  p.~42]{oberhettinger:74:tables}), when $\theta \in (-2,-1)$,
  \begin{align}\label{E:Mellin_2}
    \int_0^{\infty} \ud x\, x^{1 + \theta} \sin (x)
    = \Gamma (2 + \theta) \sin \left( (2+\theta) \pi /2 \right)
    = - (1 + \theta) \Gamma (1 + \theta) \sin (\theta \pi /2).
  \end{align}
  Since the integral on the left-hand side is absolutely convergent for
  $\theta\in(-3,-2)$ as well, the identity~\eqref{E:Mellin_2} extends to
  $\theta\in(-3,-1)\setminus\{-2\}$ by analytic continuation.
  When $\theta = - 2$, we have $\displaystyle\int_0^{\infty} \frac{\sin
  (x)}{x}\ud x = \pi/2$. The proof of~\eqref{E:int_sin0} and~\eqref{E:int_sin}
  assuming $-3 < \theta < -1$ is complete after change of variables. Note that,
  in this case, $\theta+1<0$ and hence, the arguments above allow the case
  $s_1=s_2$. Moreover, since $\Gamma(1+\theta)\sin(\pi\theta/2)>0$ (see
  Fig.~\ref{F:GammaSin}) and $1+\theta <0$, we see that $F_\theta(s_1,s_2)\ge
  0$. \bigskip

  \begin{figure}[htpb]
    \centering
    \begin{tikzpicture}
      \begin{axis}[
          axis lines = center,
          domain=-3.2:0.7,
          ymin=-2, ymax=3.2,
          xlabel={$\theta$},
          ylabel={$\Gamma(1+\theta) \sin\left(\frac{\theta\pi}{2}\right)$},
          grid=major,
          ytick={-1,1,1.5708,2,3},
          yticklabels={-1,1,$\displaystyle \sfrac{\pi}{2}$,2,3},
          % ytick align=outside,
          % yticklabel pos=right,
          every axis y label/.style={
              at={(ticklabel cs:0.75)},
              anchor=north,
              shift={(0,-0pt)}
          },
          xtick={-3,-2,-1,-0.5},
          xticklabels={ $-3$, $-2$, $-1$, $-\sfrac{1}{2}$},
          xlabel style={at={(ticklabel* cs:1.05)}, anchor=west},
          unit vector ratio*={11 2},
        ]
        \addplot[domain=-3.2:-0.5, mark=none, very thick, opacity=0.5, dotted] {pi/2};
        \addplot[domain=-3.2:-2.5, blue, solid, thick] table {GammaSin1.csv};
        \addplot[domain=-1.8:+0.5, blue, solid, thick] table {GammaSin2.csv};
        \addplot[only marks, mark=*, color=blue, mark size = 2pt] coordinates {(-2,1.5708)};
        \addplot[dashed, thick] coordinates {(-3,-2) (-3,3.2)};
        \addplot[dashed, thick] coordinates {(-1,-2) (-1,3.2)};
      \end{axis}
    \end{tikzpicture}
    \caption{Plot of the factor $\Gamma(1+\theta)
    \sin\left(\theta\pi/2\right)$ in Lemma~\ref{L:F_tht}.}
    \label{F:GammaSin}
  \end{figure}
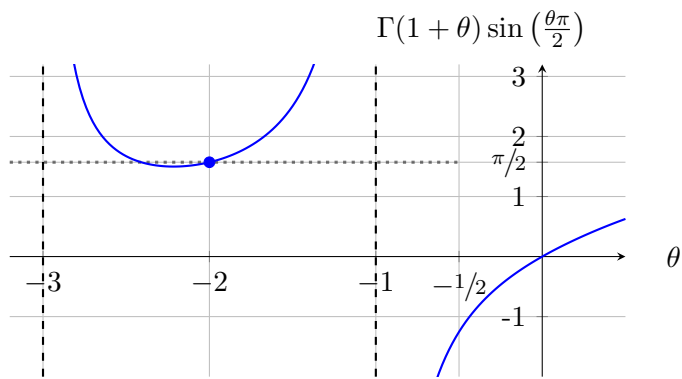
  % Do not remove the following commented codes.
  % \begin{codes} %math
  %
  % (* Plot Theta *)
  % res1 = Table[{x, Gamma[1 + x] Sin[(x \[Pi])/2]}, {x, -3,  -1.2, 0.02}]
  % res2 = Table[{x, Gamma[1 + x] Sin[(x \[Pi])/2]}, {x, -0.8, 0.5, 0.02}]
  % Export["GammaSin1.csv", res1, "Table"]
  % Export["GammaSin2.csv", res2, "Table"]
  %
  % \end{codes}

    \noindent\textbf{Case~3.~} If $\theta = -1$, by the fundamental theorem of
    calculus (in the parameter $r$), we have, for any $N>0$,
    \begin{align*}
      \int_0^N \ud x \; x^{-1}\sin(s_1x)\sin(s_2x)
      = \frac{1}{2}\int_0^N \ud x \; \frac{\cos\big((s_2-s_1)x\big)-\cos\big((s_1+s_2)x\big)}{x}.
    \end{align*}
    Since $s_1<s_2$, for any $x\ge 0$,
    \begin{align*}
      \cos\big((s_2-s_1)x\big)-\cos\big((s_1+s_2)x\big)
      = \int_{s_2-s_1}^{s_1+s_2} x\sin(rx)\,\ud r.
    \end{align*}
    Hence, by Fubini's theorem,
    \begin{align*}
      \int_0^N \ud x \; \frac{\cos\big((s_2-s_1)x\big)-\cos\big((s_1+s_2)x\big)}{x}
      = \int_{s_2-s_1}^{s_1+s_2} \ud r \int_0^N \sin(rx)\,\ud x
      = \int_{s_2-s_1}^{s_1+s_2} \frac{1-\cos(rN)}{r}\,\ud r.
    \end{align*}
    Moreover, by integration by parts,
    \begin{align*}
      \int_{s_2-s_1}^{s_1+s_2} \frac{\cos(rN)}{r}\,\ud r
      = \frac{\sin(rN)}{Nr}\bigg|_{s_2-s_1}^{s_1+s_2} + \frac{1}{N}\int_{s_2-s_1}^{s_1+s_2} \frac{\sin(rN)}{r^2}\,\ud r \to 0,
      \quad \text{as } N\uparrow \infty.
    \end{align*}
    Consequently,
    \begin{align*}
      F_{-1}(s_1,s_2)
      = \lim_{N\to\infty} \int_0^N \ud x \; x^{-1}\sin(s_1x)\sin(s_2x)
      = \frac{1}{2}\int_{s_2-s_1}^{s_1+s_2} \frac{\ud r}{r}
      = \frac{1}{2}\log\left(\frac{s_1+s_2}{s_2-s_1}\right),
    \end{align*}
    which proves~\eqref{E:int_sin} in the case $\theta=-1$.

  \bigskip\noindent\textbf{Case~4.~} If $\theta \leq -3$, since $f_{\theta,
  s_1,s_2}(x)\asymp s_1 s_2 x^{\theta +2}$ as $x\to 0_+$, we see that
  $f_{\theta, s_1,s_2}(x)$ is not integrable at zero. On the other hand, it is
  easy to see that when $\theta \ge 0$, $f_{\theta, s_1,s_2}(x)$ is not
  integrable at $\infty$. This completes the proof of Lemma~\ref{L:F_tht}.
\end{proof}

If we restrict the Mellin-transform integrals for sine/cosine functions to
$[t,\infty)$, then we have the following lemma.

\bigskip\noindent\textit{Usage in the main text.}
Lemma~\ref{L:tri_st} below provides two tail bounds, \eqref{E:ibp} and
\eqref{E:mvt},
for oscillatory Mellin-type integrals over $[t,\infty)$. While we do not cite
Lemma~\ref{L:tri_st} explicitly in Sections~\ref{S:Dalang}--\ref{S:Pf:Cor}, the
inequalities \eqref{E:ibp} and \eqref{E:mvt} are invoked in Appendix~\ref{A:Tech}
(see, e.g., the proof of Lemma~\ref{L:intml}) to control oscillatory remainder
terms arising from the large-argument asymptotics of Mittag--Leffler functions.
Those bounds then feed back into the main text through applications of
Lemma~\ref{L:intml} in Section~\ref{S:Dalang} (proof of
Theorem~\ref{T:Dalang}(iii)) and Section~\ref{S:Upper-t} (increment estimates
that depend on Lemma~\ref{L:intml}).

\begin{lemma}\label{L:tri_st}
Let $t > 0$, let $\theta < 0$, and let $\zeta \in \R$. Then, the following two inequalities hold
\begin{gather}
  \left|\int_t^{\infty} \ud x \; x^{\theta} \cos (x + \zeta) \right| \lesssim t^{\theta}. \label{E:ibp}
  \shortintertext{and}
   \left|\int_t^{\infty} \ud x \; x^{\theta} \cos (x  + \zeta)\right| \lesssim
  \begin{dcases}
  1,  & \theta > -1,\\
  1 + |\log(t)|, & \theta = -1,\\
   t^{\theta + 1}, & \theta < - 1;
  \end{dcases} \label{E:mvt}
\end{gather}
Moreover,~\eqref{E:ibp} remains valid when substituting the cosine with the sine function.
\end{lemma}
\begin{proof}
  Using integration by parts formula, one can write
  \begin{align*}
  \int_t^{\infty} \ud x\; x^{\theta} \cos(x  + \zeta) = - t^{\theta} \sin(t  + \zeta) + \int_t^{\infty} \ud x \; \theta x^{\theta - 1} \sin(x  + \zeta),
  \end{align*}
  for all $\theta < 0$. Then, the proof of~\eqref{E:ibp} is complete as a result of the boundedness of sine function. \bigskip

  Next, we will prove inequality~\eqref{E:mvt}. If $\theta > -1$, then~\eqref{E:mvt} follows from the finiteness of the Mellin transform for cosine function. Suppose $\theta = -1$, then~\eqref{E:mvt} is a result of~\eqref{E:ibp} if $t \geq 1$. Otherwise, we can write
  \begin{align*}
       \left|\int_t^{\infty} \ud x \; x^{-1} \cos (x + \zeta) \right| \le &  \left|\int_t^1 \ud x \; x^{-1} \cos (x + \zeta) \right| +  \left|\int_1^{\infty} \ud x \; x^{-1} \cos (x + \zeta) \right| \\ \lesssim & 1 + \left|\int_t^1 \ud x \; x^{-1} \right| \lesssim |\log (t)|.
  \end{align*}
  Finally, when $\theta < -1$, it holds that
  \begin{align*}
      \left|\int_t^{\infty} \ud x \; x^{\theta} \cos (x + \zeta) \right| \le \left|\int_t^{\infty} \ud x \; x^{\theta} \right| \lesssim t^{\theta + 1}.
  \end{align*}
  This completes the proof of~\eqref{E:mvt}.
\end{proof}

\subsection{Some technical lemmas}\label{A:Tech}
\index{Mittag-Leffler function}

The following lemma is a generalization of Lemma~5.1
of~\cite{chen.hu.ea:19:nonlinear} from $H=1/2$ to $H\in(0,1)$.

\bigskip\noindent\textit{Usage in the main text.}
Lemma~\ref{L:ML-Est} below is used in Section~\ref{S:Dalang} to prove upper bounds for
the second moment $\E[u(t,x)^2]$ via the Hardy--Littlewood--Sobolev inequality.
More precisely, it is the main analytic input in Subsection~\ref{SS:i} (proof of
the sufficiency in Theorem~\ref{T:Dalang}(i)), where it yields the decay in
$|\xi|$ needed to identify the Dalang-type integrability condition. It is also
invoked later in the same section in Subsection~\ref{SS:iii} (proof of
Theorem~\ref{T:Dalang}(iii)) to treat the cases $\beta=2$ and $\gamma\in(0,1]$
when applying the same Sobolev strategy.

\begin{lemma}\label{L:ML-Est}
  Suppose that $\theta+ H>1$, $H\in (0,1)$, and $\beta\in (0,2]$. Then we have:
    \begin{enumerate}[\rm (i)]
      \item Suppose one of the following conditions: i) $\beta\in(0,2)$, and ii) $\beta = 2$, $\theta \geq 3$. Then, with  $C_{\beta,\theta} > 0$, for
          all $t>0$ and $\lambda>0$, it holds that
          \begin{align}\label{E:ML-Est1}
            \int_0^t w^{(\theta-1) / H} \left| E_{\beta,\theta}\left(-\lambda w^\beta \right)\right|^{1/H}\ud w
            \le C_{\beta,\theta}\:\frac{t^{(\theta + H-1) /H}}{1+(t\lambda^{1/\beta})^{\min\left(\beta /H,(\theta + H-1) /H\right)}};
          \end{align}

    \item If $\beta = 2$ and $\theta \in [2,3)$, for some positive constant $C_{\theta}$, for all
          $t>0$ and $\lambda>0$, it holds that
          \begin{align}\label{E:ML-Est2}
            \int_0^t w^{(\theta-1) / H} \left| E_{2,\theta}\left(-\lambda w^2 \right)\right|^{1/H}\ud w
            \le C_{\theta}\:\frac{t^{(\theta + H-1) /H}}{1+(t\lambda^{1/2})^{(\theta - 1)/H}}.
          \end{align}

  \end{enumerate}
\end{lemma}
\begin{proof}
  Let $C$ denote a generic positive constant that may change its value at each
  appearance.

  \bigskip\noindent\textbf{(i)~} Under this setting, by~\eqref{E:ML-UB-2}, we
  see that
  \begin{align*} % \label{E_:WEC}
    \int_0^t w^{(\theta-1) /H} \left| E_{\beta,\theta}\left(-\lambda w^\beta \right)\right|^{1/H}\ud w
     & \le C \int_0^t  \frac{w^{(\theta-1)/H}}{(1+\lambda^{1/\beta} w)^{\beta/H}}\ud w                                                              \\ \notag
     & = C \:t^{(\theta-1) /H + 1}\int_0^1  \frac{u^{(\theta-1)/H}}{(1+\lambda^{1/\beta} t  u)^{\beta /H}}\ud u                                     \\ \notag
     & = C \:t^{(\theta-1) /H + 1}\lMr{2}{F}{1}\left(\frac{\beta}{H}, \frac{H + \theta -1}{H}; \frac{2H + \theta -1}{H} ;-t\lambda^{1/\beta}\right) \\ \notag
     & = C \:t^{(\theta-1) /H + 1}\FoxH{1,2}{2,2}{t\lambda^{1/\beta}}{(1-\beta/H,1),\:((1-\theta)/H,1)}{(0,1),\:((1-\theta-H) /H,1)},
  \end{align*}
  where the second equality is due to~(15.6.1)
  of~\cite{olver.lozier.ea:10:nist}, for which we need the condition that
  $\theta + H>1$. The last equality is due to the identity~\eqref{E:F=H}. Notice
  that $\Delta=0$ for the above \textit{Fox H-function}, which allows us to
  apply Theorems~1.7 and~1.11 (\textit{ibid.}). In particular, by Theorem~1.11
  (\textit{ibid.}), we know that
  \begin{align*}
    \FoxH{1,2}{2,2}{x}{(1-\beta/H,1),\:((1-\theta) /H,1)}{(0,1),\:((1-\theta-H) /H,1)}\sim O(1),\quad\text{as $x\rightarrow 0$}.
  \end{align*}
  The following condition guarantees that the above Fox-H function cannot be
  reduced to the simpler form:
  \begin{align*}
    1- \frac{\beta}{H} \ne \frac{1-\theta-H}{H} \quad \Longleftrightarrow \quad
    \theta \ne 1 + \beta - 2H.
  \end{align*}
  Hence, when $\theta \ne 1 + \beta - 2H$, by Theorem~1.7 (\textit{ibid.}),
  \begin{align*}
    \FoxH{1,2}{2,2}{x}{(1-\beta/H,1),\:((1-\theta) /H,1)}{(0,1),\:((1-\theta-H) /H,1)}
    \sim O\left(x^{-\min(\beta/H, (\theta+H-1)/H)}\right),\quad\text{as $x\rightarrow \infty$}.
  \end{align*}
  On the other hand, when $\theta = 1 + \beta - 2H$, by Property~2.2 and~(2.9.5)
  (\textit{ibid.}),
  \begin{align*}
    \FoxH{1,2}{2,2}{x}{(1-\beta/H,1),\:((1-\theta) /H,1)}{(0,1),\:((1-\theta-H) /H,1)}
    = \FoxH{1,1}{1,1}{x}{((1-\theta) /H,1)}{(0,1)}
    = \Gamma\left(\frac{\theta+H-1}{H}\right)(1+x)^{(1-\theta-H)/H}.
  \end{align*}
  Combining the above asymptotics at $0$ and $\infty$ proves~\eqref{E:ML-Est1}.

  \bigskip\noindent\textbf{(ii)~} The proof of part (b) is similar to that of
  part (a). One only needs to use a different asymptotic behavior of
  $E_{2,\theta}$ at infinity given in~\eqref{E:ML-UB-1}. Indeed, as above,
  \begin{align*}
    \int_0^t w^{(\theta-1) /H} \left| E_{\beta,\theta}\left(-\lambda w^\beta \right)\right|^{1/H}\ud w
     & \le C \int_0^t  \frac{w^{(\theta-1)/H}}{(1+\lambda^{1/\beta} w)^{\beta (\theta - 1)/(2H)}}\ud w                                                                    \\ \notag
     & = C \:t^{(\theta-1) /H + 1}\int_0^1  \frac{u^{(\theta-1)/H}}{(1+\lambda^{1/\beta} t  u)^{\beta (\theta - 1)/(2H)}}\ud u                                            \\ \notag
     & = C \:t^{(\theta-1) /H + 1}\lMr{2}{F}{1}\left(\frac{\beta  (\theta - 1)}{2H}, \frac{H + \theta -1}{H}; \frac{2H + \theta -1}{H} ;-t\lambda^{1/\beta}\right)        \\ \notag
     & = C \:t^{(\theta-1) /H + 1}\FoxH{1,2}{2,2}{t\lambda^{1/\beta}}{(1-\frac{\beta  (\theta - 1)}{2H},1),\:(\frac{1-\theta}{H} ,1)}{(0,1),\:(\frac{1-\theta-H}{H} ,1)}.
  \end{align*}
  Again, since $\Delta=0$ we can apply Theorem~1.11
  of~\cite{kilbas.saigo:04:h-transforms} to see that
  \begin{align*}
    \FoxH{1,2}{2,2}{x}{(1-\frac{\beta  (\theta - 1)}{2H},1),\:(\frac{1-\theta}{H} ,1)}{(0,1),\:((1-\theta-H) /H,1)}\sim O(1),\quad\text{as $x\rightarrow 0$}.
  \end{align*}
  The above Fox-H function cannot be reduced to the simpler form because, with
  $\beta = 2$,
  \begin{align*}
    1-\frac{\beta  (\theta - 1)}{2H} \ne \frac{1-\theta -H}{H} \quad \Longleftrightarrow \quad
    H \ne 1.
  \end{align*}
  Hence, by Theorem~1.7 (\textit{ibid.}), with $\beta$ replaced by $2$,
  \begin{align*}
    \FoxH{1,2}{2,2}{x}{(1-\frac{\beta  (\theta - 1)}{2H},1),\:((1-\theta) /H,1)}{(0,1),\:((1-\theta-H) /H,1)}
    \sim O\left(x^{-\min(\frac{\beta  (\theta - 1)}{2H}, (\theta+H-1)/H)}\right)
    =   O\left(x^{-(\theta-1)/H}\right),
  \end{align*}
  as $x\rightarrow \infty$. This proves~\eqref{E:ML-Est2}.
\end{proof}

In case $\beta \notin \{1,2\}$, we do not have an explicit formulation for the
Mittag--Leffler function, and thus cannot obtain an exact value of related
integrals as in Lemma~\ref{L:F_tht} corresponding to $\beta = 2$. Despite this,
we still have the next lemma providing an upper bound for an integral of the
Mittag--Leffler function, which
extends~\cite[Lemma~A.2]{guo.song.ea:24:stochastic}.

\bigskip\noindent\textit{Usage in the main text.}
Lemma~\ref{L:intml} below provides sharp bounds for integrals of products of
Mittag--Leffler functions. These estimates are repeatedly used to control
spectral integrals after a Fourier transform in space:
in Section~\ref{S:Dalang}, Subsection~\ref{SS:iii} (proof of
Theorem~\ref{T:Dalang}(iii)), part~(ii) of Lemma~\ref{L:intml} is applied to
bound the key $y$-integral that appears when $\beta=2$ and $\gamma\in(0,1]$; see
the application near condition~\eqref{E_:iii-cond}. In Section~\ref{S:Upper-t},
Lemma~\ref{L:intml} is used to derive the moment upper bounds for temporal and
spatial increments, notably in Proposition~\ref{T:Upper-t}(i)(c) (where the
assumption $\ell<\alpha(\gamma+1/2)$ is imposed specifically to ensure that
Lemma~\ref{L:intml}(ii) applies) and at several points in the proofs of
Propositions~\ref{T:Upper-t} and~\ref{T:Upper-s}.

\begin{lemma}\label{L:intml}
    Let $T > 0$, and let $0 < s_1 \leq s_2 < T$. Denote by
    \begin{align*}
      I (s_1,s_2) \coloneqq \int_0^{\infty} \ud x \; x^{\theta} E_{a,b} (- s_1 x) E_{a,b}(- s_2 x),
  \end{align*}
  where the parameter $\theta \in (-1,1)$.
  \begin{enumerate}[\rm (i)]
    \item Assume either $(a,b) \in (0,2) \times [a,\infty)$; or $a = 2$ and $b \in [3,\infty)$. Then, it holds that
          \begin{align}\label{E:intml}
            I (s_1, s_2)  \lesssim
            \begin{dcases}
              s_1^{-\theta} s_2^{-1},                   & \theta \in (0,1),  \\
              s_2^{-1} (1 + \log \left(s_2/s_1\right)), & \theta = 0,        \\
              s_2^{- \theta - 1},                       & \theta \in (-1,0).
            \end{dcases}
          \end{align}
    \item In case $a = 2$, $b \in (1,2) \cup (2,3]$, such that $b - \theta > 3/2$, then the next inequality is true,\footnote{The condition $b - \theta > 3/2$ is necessary to ensure the integrability of $I_{3,1,1}$ and $I_{3,1,2}$ in~\eqref{E_:intml-a=2-3-1-1} and~\eqref{E_:intml-a=2-3-1-2}.}
          \begin{align}\label{E:intml-a=2}
            I (s_1, s_2)  \lesssim
            \begin{dcases}
              s_1^{b/2 - \theta - 3/2} s_2^{- \frac{b-1}{2}} + s_1^{-\frac{b-1}{2}} s_2^{-\frac{b-1}{2}} \left(\sqrt{s_2} - \sqrt{s_1}\right)^{2b - 2\theta -4}, & b < 2 + \theta ,               \\
              s_1^{-(b-1)/2} s_2^{-(b-1)/2} \left(1 + |\log \left((\sqrt{s_2} - \sqrt{s_1})/\sqrt{s_2}\right)| \right) ,                                         & b = 2 + \theta,                \\
              s_1^{b/2 - \theta - 3/2} s_2^{-(b-1)/2},                                                                                                           & 2 + \theta < b < 3 + 2\theta , \\
              s_2^{- \theta - 1} \big(1 + \log (s_2/s_1 )\big),                                                                                                  & b = 3 +  2\theta,              \\
              s_2^{- \theta - 1},                                                                                                                                & b > 3 + 2\theta.
            \end{dcases}
          \end{align}
  \end{enumerate}
\end{lemma}

\begin{proof}
  Due to parts (i) and (ii) of Lemma~\ref{L:ML}, the function $E_{a,b}(-x)$ is
  bounded around $0$ and with certain decay when $x\to\infty$. Thus, we
  decompose
  \begin{align*}
    I (s_1,s_2) = I_1 (s_1,s_2) + I_2 (s_1,s_2) + I_3 (s_1, s_2),
    \shortintertext{where}
    I_1 (s_1,s_2) \coloneqq \int_0^{s_2^{-1}} \ud x  \; x^{\theta} E_{a, b}\left(- s_1 x\right)  E_{a,b}\left(- s_2 x\right) , \\
    I_2 (s_1,s_2) \coloneqq \int_{s_2^{-1}}^{s_1^{-1}} \ud x  \; x^{\theta} E_{a, b}\left(- s_1 x\right)  E_{a,b}\left(- s_2 x\right),
    \shortintertext{and}
    I_3 (s_1, s_2) \coloneqq \int_{s_1^{-1}}^{\infty}\ud x \; x^{\theta} E_{a, b}\left(- s_1 x\right)  E_{a,b}\left(- s_2 x\right).
  \end{align*}
  Next, we present the proofs for~\eqref{E:intml} and~\eqref{E:intml-a=2}
  separately.

  \bigskip\noindent\textbf{Proof of~\eqref{E:intml}.~} The proof
  of~\eqref{E:intml} is relatively easy. A direct application of parts (i) and
  (ii) of Lemma~\ref{L:ML} yields that
  \begin{gather*}
    |I_1 (s_1,s_2)| \lesssim \int_0^{s_2^{-1}} \ud x \; x^{\theta} \lesssim s_2^{- \theta - 1}, \quad |I_2 (s_1,s_2)| \lesssim s_2^{-1}\int_{s_2^{-1}}^{s_1^{-1}} \ud x \; x^{\theta - 1} \lesssim s_2^{-1}\left(s_1^{-\theta} + s_2^{-\theta}\right)
    \shortintertext{and}
    |I_3 (s_1,s_2)| \lesssim  s_1^{-1} s_2^{-1} \int_{s_1^{-1}}^{\infty} \ud x \; x^{\theta - 2} \lesssim s_1^{-\theta} s_2^{-1},
  \end{gather*}
  provided that $\theta \neq 0$; and on the other hand, in case $\theta = 0$,
  the logarithmic correction appears in the upper bound, because
  \begin{align*}
    |I_2 (s_1,s_2)| \lesssim s_2^{-1} \log \left(s_2/s_1\right).
  \end{align*}
  Then,~\eqref{E:intml} for all $\theta \in (-1,1)$ follows immediately by
  referring to previous inequalities.

  \bigskip\noindent\textbf{Proof of~\eqref{E:intml-a=2}.~} It follows from the
  same argument as in the previous case that
  \begin{align*}% \label{E_:intml-a=2-1}
    |I_1 (s_1, s_2)|\lesssim s_2^{-\theta - 1}.
  \end{align*}
  Thus, it suffices to estimate the rest two terms. First,
  using~\eqref{E:asym_MT1-4}, and simply bounding the cosine function by $1$, we
  can write
  \begin{align*} % \label{E_:intml-a=2-2}
    |I_2 (s_1, s_2)| \lesssim  s_2^{- (b-1)/2}\int_{s_2^{-1}}^{s_1^{-1}} \ud x \; x^{\theta - (b-1)/2} \lesssim \begin{dcases}
                                                                                                                  s_1^{b/2 - \theta - 3/2} s_2^{-(b-1)/2},                     & b - 2\theta < 3  \\
                                                                                                                  s_2^{- \theta - 1} \left(1 + \log \big(s_2/s_1 \big)\right), & b - 2\theta = 3  \\
                                                                                                                  s_2^{- \theta - 1},                                          & b - 2\theta > 3.
                                                                                                                \end{dcases}
  \end{align*}

  Next, we estimate $I_3(s_1, s_2)$. It follows from~\eqref{E:asym_MT1-4} again
  that with $\zeta \coloneqq (b - 1) \pi/2$,
  \begin{gather*} %\label{E_:intml-a=2-3}
    |I_3 (s_1, s_2)| \lesssim  I_{3, 1}(s_1,s_2) + I_{3, 2}(s_1,s_2) + I_{3, 3}(s_1,s_2) + I_{3, 4}(s_1,s_2),
  \end{gather*}
  where
  \begin{gather*}
    I_{3, 1}(s_1,s_2)  \coloneqq   s_1^{-(b - 1)/2} s_2^{-(b - 1)/2}\left|\int_{s_1^{-1}}^{\infty}  \ud x\;  x^{\theta - (b - 1)} \cos \left(\sqrt{s_1 x} - \zeta \right) \cos \left(\sqrt{s_2 x} - \zeta\right) \right|, \\
    I_{3, 2}(s_1,s_2)  \coloneqq  s_1^{-(b - 1)/2} s_2^{-1} \left| \int_{s_1^{-1}}^{\infty} \ud x\;  x^{\theta - (b - 1)/2 - 1} \cos \left(\sqrt{s_1 x} - \zeta \right) \right|,                            \\
    I_{3, 3}(s_1,s_2)  \coloneqq  s_1^{-1} s_2^{-(b - 1)/2} \left|\int_{s_1^{-1}}^{\infty}  \ud x\; x^{\theta - (b - 1)/2  - 1} \cos \left(\sqrt{s_2 x} - \zeta \right) \right|,                            \shortintertext{and}
    I_{3, 4}(s_1,s_2)  \coloneqq  s_1^{ -1} s_2^{ -1} \left|\int_{s_1^{-1}}^{\infty}\ud x\: x^{\theta - 2} \right|.
  \end{gather*}
  \bigskip Recall that $\theta \in (-1, 1)$. It follows that
  \begin{align*} %\label{E_:intml-a=2-3-4}
    I_{3, 4} (s_1, s_2) \lesssim s_1^{-\theta} s_2^{-1}.
  \end{align*}
  With the change of variables $\sqrt{s_1 x} = y$,
  \begin{gather*}
    I_{3, 2} (s_1,s_2)  \approx s_1^{ - \theta} s_2^{-1} \left| \int_{1}^{\infty} \ud y\: y^{2\theta - b} \cos \left( y - \zeta\right) \right| \lesssim s_1^{ - \theta} s_2^{-1} , %\label{E_:intml-a=2-3-2}
  \end{gather*}
  where the last inequality is a consequence of~\eqref{E:ibp} and the fact
  $2\theta - b < 0$ resulting from the assumptions $\theta \in (-1,1)$ and $b -
    \theta > 1$. Similarly, we find that
  \begin{gather*}%\label{E_:intml-a=2-3-3}
    I_{3, 3} (s_1,s_2)  \approx  s_1^{-1} s_2^{-\theta} \left| \int_{ \sqrt{s_2/s_1}}^{\infty} \ud y\: y^{2\theta - b} \cos \left( y - \zeta\right) \right| \lesssim s_1^{-1} s_2^{-\theta} (s_2/s_1)^{2\theta - b} \lesssim s_1^{ - \theta} s_2^{-1} .
  \end{gather*}
  Finally, we estimate $I_{3,1}$ as follows. Write
  \begin{align*}
    I_{3, 1} (s_1,s_2) \lesssim & s_1^{-(b - 1)/2} s_2^{-(b - 1)/2} \left|\int_{s_1^{-1}}^{\infty}  \ud x\;  x^{\theta - (b - 1)} \cos \left(\left(\sqrt{s_1} + \sqrt{s_2}\right) \sqrt{x} - 2 \zeta \right)  \right| \\
                                & + s_1^{-(b - 1)/2} s_2^{-(b - 1)/2} \left|\int_{s_1^{-1}}^{\infty}  \ud x\;  x^{\theta - (b - 1)} \cos \left(\left(\sqrt{s_2} - \sqrt{s_1}\right)  \sqrt{x} \right) \right|         \\
    \eqqcolon                   & I_{3,1,1} (s_1,s_2) + I_{3,1,2} (s_1,s_2).
  \end{align*}
  Performing the change of variables $\left(\sqrt{s_2} + \sqrt{s_1}\right)  \sqrt{x}  = y$, using~\eqref{E:ibp}, and recalling the assumption $b - \theta > 3/2$, we get
  \begin{gather}
    \begin{aligned}\label{E_:intml-a=2-3-1-1}
      I_{3, 1, 1} (s_1,s_2) \approx & s_1^{-(b - 1)/2} s_2^{-(b - 1)/2}\left(\sqrt{s_1} + \sqrt{s_2}\right)^{2b - 2\theta - 4} \left|\int_{\frac{\sqrt{s_1} + \sqrt{s_2}}{\sqrt{s_1}}}^{\infty}  \ud y\; y^{2\theta - 2b + 3} \cos \left(y - 2 \zeta \right)  \right| \\
      \lesssim                      & s_1^{b/2 - \theta - 1} s_2^{- (b - 1)/2} \left(\sqrt{s_1} + \sqrt{s_2}\right)^{-1} \lesssim s_1^{b/2 - \theta - 1} s_2^{- b /2} .
    \end{aligned}
  \end{gather}
  For $I_{3, 1, 2} (s_1,s_2)$, we do not apply~\eqref{E:ibp}.
  Instead,~\eqref{E:mvt} yields that
  \begin{gather}
    \begin{aligned}\label{E_:intml-a=2-3-1-2}
      I_{3, 1, 2} (s_1,s_2) \lesssim & s_1^{-(b - 1)/2} s_2^{-(b - 1)/2}\left(\sqrt{s_2} - \sqrt{s_1}\right)^{2b - 2\theta - 4} \left|\int_{\frac{\sqrt{s_2} - \sqrt{s_1}}{\sqrt{s_1}}}^{\infty}  \ud y\; y^{2\theta - 2b + 3} \cos \left(y - 2 \zeta \right)  \right| \\
      \lesssim                       & \begin{dcases}
                                         s_1^{-(b - 1)/2} s_2^{-(b - 1)/2}\left(\sqrt{s_2} - \sqrt{s_1}\right)^{2b - 2\theta - 4} ,                              & b - \theta < 2  \\
                                         s_1^{-(b - 1)/2} s_2^{-(b - 1)/2} \left(1 + \left|\log \left((\sqrt{s_2} - \sqrt{s_1})/\sqrt{s_2}\right)\right|\right), & b - \theta = 2, \\
                                         s_1^{b/2 - \theta - 3/2} s_2^{- (b - 1)/2},                                                                             & b - \theta > 2.
                                       \end{dcases}
      % \lesssim & \begin{dcases}
      %     \left(\sqrt{s_2} - \sqrt{s_1}\right)^{2b - 2\theta - 4} ,                & b - \theta < 2  \\
      %     1 + \left|\log \left((\sqrt{s_2} - \sqrt{s_1})/\sqrt{s_2}\right)\right|, & b - \theta = 2, \\
      %     1,                                                                       & b - \theta > 2.
      % \end{dcases}
    \end{aligned}
  \end{gather}
  Recall that $\theta \in (-1,1)$, $b \in (1,2) \cup (2,3]$, and $b - \theta > 3/2$. It follows that when $b < 2 + \theta < 3 + 2 \theta$,
  \begin{align*}
    I(s_1, s_2) \lesssim & s_2^{- \theta - 1} + s_1^{b/2 - \theta - 3/2} s_2^{- (b - 1)/2} + s_1^{- \theta} s_2^{-1} + s_1^{-(b - 1)/2} s_2^{-(b - 1)/2}\left(\sqrt{s_2} - \sqrt{s_1}\right)^{2b - 2\theta - 4} \\
    \lesssim             & s_1^{b/2 - \theta - 3/2} s_2^{- (b - 1)/2} + s_1^{-(b - 1)/2} s_2^{-(b - 1)/2}\left(\sqrt{s_2} - \sqrt{s_1}\right)^{2b - 2\theta - 4};
  \end{align*}
  when $b = 2 + \theta$,
  \begin{align*}
    I(s_1, s_2) \lesssim & s_2^{- (b - 1)} + s_1^{- (b - 1)/2} s_2^{- (b - 1)/2} + s_1^{- \theta} s_2^{-1} + s_1^{-(b - 1)/2} s_2^{-(b - 1)/2} \left(1 + \left|\log \left((\sqrt{s_2} - \sqrt{s_1})/\sqrt{s_2}\right)\right|\right) \\
    \lesssim             & s_1^{-(b - 1)/2} s_2^{-(b - 1)/2} \left(1 + \left|\log \left((\sqrt{s_2} - \sqrt{s_1})/\sqrt{s_2}\right)\right|\right);
  \end{align*}
  when $2 + \theta < b < 3 + 2\theta$,
  \begin{align*}
    I(s_1, s_2) \lesssim & s_2^{- \theta - 1} + s_1^{b/2 - \theta - 3/2} s_2^{- (b - 1)/2} + s_1^{- \theta} s_2^{-1} + s_1^{b/2 - \theta - 3/2} s_2^{- (b - 1)/2}
    \lesssim  s_1^{b/2 - \theta - 3/2} s_2^{- (b - 1)/2};
  \end{align*}
  when $ b = 3 + 2\theta$,
  \begin{align*}
    I(s_1, s_2) \lesssim & s_2^{- (b - 1)/2} + s_2^{- (b - 1)/2} \left(1 + \log \big(s_2/s_1 \big)\right) + s_1^{- (b - 3)/2} s_2^{-1} + s_2^{- (b - 1)/2} \\
    \lesssim             & s_2^{- \theta - 1} \left(1 + \log \big(s_2/s_1 \big)\right);
  \end{align*}
  and when $b > 3 + 2\theta$,
  \begin{align*}
    I(s_1, s_2) \lesssim & s_2^{- \theta - 1} + s_2^{-\theta - 1} + s_1^{- \theta} s_2^{-1} + s_1^{b/2 - \theta - 3/2} s_2^{- (b - 1)/2}   \lesssim  s_2^{- \theta - 1}.
  \end{align*}
  This completes the
  proof of Lemma~\ref{L:intml}.
\end{proof}

\subsection{A generalized Balan's lemma}\label{A:Balan}\index{Balan's lemma}

In this section, we present a generalized version of Balan's lemma~\cite{balan:12:linear}*{Lemma
6.2} tailored to our framework. We begin with the following
lemma, which is a modified version of~\cite{balan.tudor:10:stochastic}*{Lemma
B.1}.

\bigskip\noindent\textit{Usage in the main text.}
Lemma~\ref{L:ipv_balan} below is a purely technical identity that converts the complex
oscillatory integral $F_\eta(\theta,\tau)=\int_0^\theta e^{is\tau}\sin(s+\eta)\,\ud s$
into an explicit rational-trigonometric expression in $\tau$; see
\eqref{E:ipv_balan}. Its main purpose is to justify the representation of the
quantity $N_{\alpha,\eta}(t,z)$ in \eqref{E:def_nt} in terms of the integral
$A_\eta(t,a)$, which is the object bounded in Lemma~\ref{L:balan}. While
Lemma~\ref{L:ipv_balan} is not cited directly elsewhere in the paper, it is an
essential intermediate step in obtaining the generalized Balan estimate that is
used later for spatial increment bounds in the oscillatory wave regime.

\begin{lemma}\label{L:ipv_balan}
  Let $\eta \in \R$. For any $\theta > 0$, denote by
  \begin{align*}
    F_{\eta} (\theta, \tau) \coloneqq \int_0^{\theta} \ud s \; e^{i s \tau} \sin (s  + \eta) .
  \end{align*}
  Then, for all $\tau\in\R\setminus\{\pm1\}$,
  \begin{align}\label{E:ipv_balan}
    |F_{\eta} (\theta, \tau)|^2 = \frac{1}{\left(1 - \tau^2\right)^2}                                                            
    & \left( \left(\cos(\eta) \sin(\tau \theta) + \tau \sin(\eta) \cos(\tau \theta) - \tau \sin(\theta + \eta) \right)^2 \right. \nonumber \\
    & \left.+ \left(\cos(\eta) \cos(\tau \theta) - \tau \sin(\eta) \sin(\tau \theta) - \cos(\theta + \eta) \right)^2\right).
  \end{align}
  Moreover, the identity extends to $\tau=\pm1$ by continuity.
\end{lemma}
\begin{proof}
  For $\tau\ne \pm 1$, writing $\sin(s+\eta)=\frac{1}{2i}(e^{i(s+\eta)}-e^{-i(s+\eta)})$,
  a direct computation yields the explicit representation
  \begin{align*}
    F_{\eta}(\theta,\tau)
    = \frac{1}{2}\left[e^{i\eta}\frac{e^{i(\tau+1)\theta}-1}{\tau+1}
      -e^{-i\eta}\frac{e^{i(\tau-1)\theta}-1}{\tau-1}\right].
  \end{align*}
  Expanding the right-hand side and simplifying yields~\eqref{E:ipv_balan}. Since
  the left-hand side is continuous in $\tau$ and the right-hand side has
  removable singularities at $\tau=\pm 1$, the identity extends to $\tau=\pm 1$
  by continuity.
\end{proof}

Fix $\alpha > 0$ and $\eta \in \R$. We define
\begin{align}\label{E:def_nt}
  N_{\alpha, \eta} (t, z) \coloneqq  z^{- \alpha} \int_{\R} \ud \tau \; |\tau|^{1 - 2H}  \left|\int_{0}^t \ud s \; e^{i \tau s} \sin \left( s z^{\alpha/2} + \eta\right) \right|^2
\end{align}
for all $(t,z) \in \R_+^2$. With the variable change $s z^{\alpha/2} \mapsto s$,
it is not hard to deduce that
\begin{align*}
  N_{\alpha, \eta} (t, z) = z^{-2 \alpha} \int_{\R} \ud \tau \; |\tau|^{1 - 2H} \left|F_{\eta} \left(t z^{\alpha/2}, \tau/z^{\alpha/2}  \right)\right|^2.
\end{align*}
Therefore, applying Lemma~\ref{L:ipv_balan} and performing the change of variables
$\tau \mapsto t\tau$, we obtain the next identity analogous to~\cite{balan:12:linear}*{Equation
(45)}:
\begin{align*} %\label{E:N=A}
   N_{\alpha, \eta} (t, z) = t^4 A_{\eta} (t, t z^{\alpha /2}),
\end{align*}
where
\begin{align*} % \label{E:def_at}
  A_{\eta} (t, a) \coloneqq  t^{2H - 2}  \int_{\R} \ud \tau \;   \frac{|\tau|^{1 - 2H}}{\left(\tau^2 - a^2\right)^2}
  &  \left[
\left(\cos(\eta) \sin(\tau) + \frac{\tau}{a} \sin(\eta) \cos(\tau) - \frac{\tau}{a} \sin(a + \eta) \right)^2 \right. \nonumber \\
& \quad \left.+ \left(\cos(\eta) \cos(\tau) - \frac{\tau}{a} \sin(\eta) \sin(\tau) - \cos(a + \eta) \right)^2\right],
\end{align*}
for all $(t, a) \in \R_+^2$. We are now ready to present the main result of this
section: a generalized Balan lemma. The proof of the following lemma closely
follows that of~\cite[Lemma 6.2]{balan:12:linear}. Therefore, we omit
the proof and instead provide relevant comments in Remark~\ref{R:balan-im}.

\bigskip\noindent\textit{Usage in the main text.}
Lemma~\ref{L:balan} below is used in Section~\ref{S:Upper-t} (moment upper bounds for
increments), in the proof of Proposition~\ref{T:Upper-s}(ii) (spatial increments
when $\beta=2$ and $\gamma\in[0,1)$). In that proof, the increment variance is
reduced (via a time Fourier representation) to integrals involving
$N_{\alpha,\eta}(t,z)$ from \eqref{E:def_nt}; the upper bound in
Lemma~\ref{L:balan} yields the decay estimate
$N_{\alpha,\eta}(t,z)\lesssim z^{-\alpha(H+1/2)}\wedge 1$ for $z\ge 1$, which in
turn determines the sharp spatial increment exponent $\widetilde{\rho}_2$ in
\eqref{E:ups1}. The restriction $a\ge b>0$ in Lemma~\ref{L:balan} (for general
$\eta$) aligns with the fact that, in the main text, the lemma is applied only
on frequency ranges bounded away from $0$ (e.g., $z\ge 1$).

\begin{lemma}\label{L:balan}
  For all $t > 0$, $a \in [b, \infty) \subseteq (0, \infty)$, and $H \in (0,1)$, there exist strictly
  positive constants $C^{(1)}$ and $C^{(2)}$ depending on $t$, $b$, and $\eta$, such that
  \begin{align}\label{E:balan}
    \frac{C^{(2)}}{t \sqrt{a^2 + t^2}} \int_{\R} \frac{|\tau/t|^{1 - 2H}}{\tau^2 + a^2 + t^2} \;\ud \tau
    \leq A_{\eta}(t,a) \leq
    \frac{C^{(1)}}{t \sqrt{a^2 + t^2}} \int_{\R} \frac{|\tau/t|^{1 - 2H}}{\tau^2 + a^2 + t^2} \; \ud \tau.
  \end{align}
  In particular, when $\eta = 0$, inequality~\eqref{E:balan} holds for all $a \in (0, \infty)$, with constants $C^{(1)}$ and $C^{(2)}$ depending only on $t$.
\end{lemma}

\begin{remark}\label{R:balan-im}
  The original Balan lemma~\cite[Lemma 6.2]{balan:12:linear} established the
  upper and lower bounds for $A_{\eta}(t,a)$ in a special case when $\eta = 0$. As
  stated in Lemma \ref{L:balan}, in this setting,~\eqref{E:balan} holds for all
  $a > 0$, and $C^{(1)}$ and $C^{(2)}$ depend only on $t$. However, to extend
  the result for all $\eta \in \R$, an additional condition is required: namely, that $a$ bounded below by a positive constant. This necessity can be illustrated as follows. When $\eta = 0$,
  \begin{align*}
  \left|\cos(\eta) \sin(\tau) + \frac{\tau}{a} \sin(\eta) \cos(\tau) - \frac{\tau}{a} \sin(a + \eta)\right| =  \left|\sin(\tau) - \frac{\tau}{a} \sin(a)\right| \to \left|\sin (\tau) - \tau\right| < \infty,
  \end{align*}
  as $a \downarrow 0$. In contrast, for general $\eta$ such that $\sin (\eta)
  \neq 0$, it holds that
    \begin{align*}
  \left|\cos(\eta) \sin(\tau) + \frac{\tau}{a} \sin(\eta) \cos(\tau) - \frac{\tau}{a} \sin(a + \eta)\right| \geq  \frac{\tau}{a}\left|  \sin(\eta) \cos(\tau) - \sin(a + \eta)\right| - 1,
  \end{align*}
  increases to $\infty$ of order $O(a^{-1})$, as $a \downarrow 0$. This
  observation suggests that the following inequality
  \begin{align*}
    & \left[ \left(\cos(\eta) \sin(\tau) + \frac{\tau}{a} \sin(\eta) \cos(\tau) - \frac{\tau}{a} \sin(a + \eta) \right)^2 \right. \nonumber \\
    & \quad \left.+ \left(\cos(\eta) \cos(\tau) - \frac{\tau}{a} \sin(\eta) \sin(\tau) - \cos(a + \eta) \right)^2\right] \lesssim \frac{1}{a} (\tau + a)^2
  \end{align*}
  is analogous to~\cite[Inequality (47)]{balan:12:linear} and is valid only for $a$
  bounded below by a positive constant.
\end{remark}

\subsection{Hypergeometric function}\label{A:2F1}
\index{hypergeometric function}

We use the Gauss hypergeometric function (see Chapter~15
of~\cite{olver.lozier.ea:10:nist}), defined by
\begin{align}\label{E:2F1}
  \lMr{2}{F}{1}(a,b;c;z)
  \coloneqq \sum_{n=0}^{\infty} \frac{(a)_n (b)_n}{(c)_n n!} z^n
\end{align}
where $(q)_0\coloneqq 1$ and $(q)_n\coloneqq q(q+1)\cdots (q+n-1)$ for $n\ge 1$.
The series is well defined on the disk $|z|<1$, and extends elsewhere by
analytic continuation. We will use the following Euler integral representation
(see~\cite[eq.~15.6.1]{olver.lozier.ea:10:nist}):
\begin{align}\label{E:HGF}
  \lMr{2}{F}{1}(a,b;c;z)
   & = \frac{\Gamma(c)}{\Gamma(b)\Gamma(c-b)}
  \int_0^1 t^{b-1} (1-t)^{c-b-1} (1-zt)^{-a} \ud t,                                     \\
   & \qquad \text{provided that $\Re(c) > \Re(b) > 0$ and $|\arg(1-z)|<\pi$.} \nonumber
\end{align}
In particular, we will need the following special cases:

\bigskip\noindent\textit{Usage in the main text.}
Lemma~\ref{L:2F1} below records special evaluations that are used in Section~\ref{S:Dalang}
to compare and match various explicit formulas for the variance constant
$K(\alpha,\beta,\gamma;H,\ell;\nu,d)$ in the wave case $\beta=2$ and $\gamma=0$.
More precisely, \eqref{E:2F1-1} and \eqref{E:2F1-2} are invoked in
Remark~\ref{R:K-limits} (following Lemma~\ref{L:K}) to justify the limits as
$H\downarrow 1/2$ and as $\ell\to \alpha/2$ when comparing the special-case
constants.
\begin{lemma}\label{L:2F1}
    We have that
    \begin{align}\label{E:2F1-1}
       & \lMr{2}{F}{1}(1,b;1;-1) = 2^{-b}, \qquad \text{for all $0<b<1$,} \shortintertext{and}
     & \lMr{2}{F}{1}(1,-1;2H;-1) = \frac{1+2H}{2H}, \qquad \text{for all $H\in [1/2,1)$.}
    \label{E:2F1-2}
  \end{align}
\end{lemma}
\begin{proof}
  For~\eqref{E:2F1-1}, since $(1)_n = n!$ for all $n\ge 0$, the definition
  \eqref{E:2F1} yields
  \begin{align*}
    \lMr{2}{F}{1}(1,b;1;-1)
    = \sum_{n=0}^{\infty}\frac{(b)_n}{n!}(-1)^n
    = 2^{-b},
  \end{align*}
  where the series converges by the alternating test since $(b)_n/n!\asymp
    n^{b-1}\to 0$ for $0<b<1$. Moreover, by Abel's theorem applied to the identity
  $(1-z)^{-b}=\sum_{n=0}^\infty (b)_n z^n/n!$ (valid for $|z|<1$), we obtain that
  this sum equals $\lim_{r\uparrow 1} (1+r)^{-b}=2^{-b}$.

  For~\eqref{E:2F1-2}, since $(-1)_0=1$, $(-1)_1=-1$, and $(-1)_n=0$ for all
  $n\ge 2$, the series \eqref{E:2F1} truncates and gives
  \begin{align*}
    \lMr{2}{F}{1}(1,-1;2H;-1)
    = 1 + \frac{(1)_1(-1)_1}{(2H)_1\,1!}(-1)
    = 1 + \frac{1}{2H}
    = \frac{1+2H}{2H}.
  \end{align*}
  The proof of this lemma is complete.
\end{proof}

\bigskip\noindent\textit{Usage in the main text.}
Lemma~\ref{L:2F1-Log} below is used in Section~\ref{S:Dalang}, Subsection~\ref{SS:ii}
(proof of Theorem~\ref{T:Dalang}(ii)), in the borderline sub-case
$\ell=\alpha$ (see the computation leading to \eqref{E:K-5}): the double
time-integral reduces to an integral of the form
$\int_0^t (t-s)^{2H-2}\log\big(\frac{t+s}{t-s}\big)\,\ud s$, and
Lemma~\ref{L:2F1-Log} provides an explicit evaluation in terms of
$\lMr{2}{F}{1}(1,1;1+2H;-1)$.
\begin{lemma}\label{L:2F1-Log}
    For all $H \in (1/2, 1)$ and $t > 0$, the following holds:
    \begin{align*}
      \int_0^t (t-s)^{2H-2} \log\left(\frac{t+s}{t-s}\right) \ud s =
    t^{2H-1} \left(\frac{\Gamma(2H-1)}{\Gamma(1+2H)} \lMr{2}{F}{1}(1,1;1+2H;-1) + \frac{1}{(2H-1)^2}\right).
  \end{align*}
\end{lemma}

\begin{proof}
  By integration by parts, we have
  \begin{align*}
    \int_0^t (t-s)^{2H-2} \log(t-s) \ud s =
    \frac{t^{2H-1}}{2H-1} \left(\frac{1}{1-2H} + \log(t)\right),
  \end{align*}
  and
  \begin{align*}
    \int_0^t (t-s)^{2H-2} \log(t+s) \ud s
    = & \frac{t^{2H-1} \log(t)}{2H-1} + \frac{1}{2H-1} \int_0^t \frac{(t-s)^{2H-1}}{t+s} \ud s  \\
    = & \frac{t^{2H-1} \log(t)}{2H-1} + \frac{\lMr{2}{F}{1}(1,1;1+2H;-1) t^{2H-1}}{2 (2H-1) H},
  \end{align*}
  where the last equality follows from~\eqref{E:HGF}. Combining the above two
  integrals proves the lemma.
\end{proof}

Note that the hypergeometric function can be expressed using the Fox H-function%
\index{Fox H-function}
(see, e.g., (2.9.15) of~\cite{kilbas.saigo:04:h-transforms}):
\begin{align}\label{E:F=H}
  \lMr{2}{F}{1}\left(a,b;c;-z\right)
  = \frac{\Gamma(c)}{\Gamma(a)\Gamma(b)} \FoxH{1,2}{2,2}{z}{(1-a, 1),\:( 1 - b,1)}{(0,1),\:(1-c,1)}.
\end{align}

% List of Notation
% This section provides a reference for the main symbols and notation used throughout the paper.

\subsection{List of Notation}\label{A:Notation}

\subsubsection*{General notation}

\begin{tabular}{ll}
  $\R$, $\R^d$, $\R_+$            & Real numbers, $d$-dimensional Euclidean space, positive reals \\[0.3em]
  $\mathbb{N}$, $\mathbb{Z}$      & Natural numbers, integers                                     \\[0.3em]
  $\E$                            & Expectation                                                   \\[0.3em]
  $\Pro$                          & Probability measure                                           \\[0.3em]
  $\Var$                          & Variance                                                      \\[0.3em]
  $\one$                          & Indicator function                                            \\[0.3em]
  $\Ceil{\cdot}$, $\Floor{\cdot}$ & Ceiling and floor functions                                   \\[0.3em]
  $\Norm{\cdot}$                  & Norm                                                          \\[0.3em]
  $\InPrd{\cdot}$                 & Inner product                                                 \\[0.3em]
  $\sgn$                          & Sign function                                                 \\[0.3em]
  $\spt{\cdot}$                   & Support of a function                                         \\[0.3em]
  $B(x,r)$                        & Ball centered at $x$ with radius $r$                          \\[0.3em]
  $a \wedge b$, $a \vee b$        & Minimum and maximum of $a$ and $b$                            \\[0.3em]
  $f \asymp g$                    & $f$ and $g$ are comparable (two-sided bounds)                 \\[0.3em]
  $f \lesssim g$                  & $f \le C g$ for some constant $C > 0$                         \\[0.3em]
\end{tabular}

\subsubsection*{Parameters and exponents}

\begin{tabular}{ll}
  $\alpha > 0$       & Order of the fractional Laplacian $(-\Delta)^{\alpha/2}$          \\[0.3em]
  $\beta \in (0,2]$  & Order of the Caputo time-fractional derivative $\partial_t^\beta$ \\[0.3em]
  $\gamma \ge 0$     & Order of the Riemann--Liouville integral $I_t^\gamma$             \\[0.3em]
  $H \in [1/2, 1)$   & Hurst parameter\index{Hurst parameter} (temporal regularity of the noise)                \\[0.3em]
  $\ell \in (0, 2d)$ & Spatial roughness parameter of the noise                          \\[0.3em]
  $\rho$             & Dalang-type exponent\index{Dalang-type exponent}; see~\eqref{E:Dalang}                        \\[0.3em]
  $\widetilde{\rho}$ & Spatial regularity exponent\index{spatial regularity exponent}; see~\eqref{E:Dalang}                 \\[0.3em]
\end{tabular}

\subsubsection*{Differential and integral operators}

\begin{tabular}{ll}
  $\Delta$               & Laplace operator                                         \\[0.3em]
  $(-\Delta)^{\alpha/2}$ & Fractional Laplacian of order $\alpha$                   \\[0.3em]
  $\partial_t^\beta$     & Caputo fractional derivative of order $\beta$            \\[0.3em]
  $I_t^\gamma$           & Riemann--Liouville fractional integral of order $\gamma$ \\[0.3em]
  $\ud$                  & Differential symbol                                      \\[0.3em]
\end{tabular}

\subsubsection*{Noise and solution}

\begin{tabular}{ll}
  $\W = \dot{W}$    & Centered Gaussian noise\index{Gaussian noise!centered}                      \\[0.3em]
  $W$               & Isonormal Gaussian process                   \\[0.3em]
  $\calH$           & Hilbert space associated with the noise      \\[0.3em]
  $\Lambda_H(\tau)$ & Temporal spectral density: $|\tau|^{1-2H}$   \\[0.3em]
  $K_\ell(\xi)$     & Spatial spectral density: $|\xi|^{\ell - d}$ \\[0.3em]
  $u(t,x)$          & Solution to the SPDE~\eqref{E:fde}           \\[0.3em]
  $G(t,x)$          & Fundamental solution\index{fundamental solution} (Green's function)      \\[0.3em]
\end{tabular}

\subsubsection*{Special functions}

\begin{tabular}{ll}
  $\Gamma(\cdot)$                            & Gamma function                                                         \\[0.3em]
  $E_{\alpha,\beta}(\cdot)$                  & Mittag--Leffler function                                               \\[0.3em]
  $W_{\lambda,\mu}(\cdot)$                   & Wright function                                                        \\[0.3em]
  $\FoxH{m,n}{p,q}{z}{(a_p,A_p)}{(b_q,B_q)}$ & Fox $H$-function                                                       \\[0.3em]
  $\Erf$, $\Erfc$, $\Erfi$                   & Error function, complementary error function, imaginary error function \\[0.3em]
  $\Ei$                                      & Exponential integral                                                   \\[0.3em]
  $\Shi$, $\Chi$                             & Hyperbolic sine and cosine integrals                                   \\[0.3em]
  $\Daw$                                     & Dawson function                                                        \\[0.3em]
  ${}_2F_1$                                  & Gauss hypergeometric function                                          \\[0.3em]
\end{tabular}

\subsubsection*{Transforms}

\begin{tabular}{ll}
  $\mathcal{F}(\phi)$, $\widehat{\phi}$ & Fourier transform (no $2\pi$ factor): $\int e^{-i(\tau t + \xi \cdot x)} \phi \, \ud t \, \ud x$ \\[0.3em]
  $\mathcal{L}$                         & Laplace transform                                                                                \\[0.3em]
\end{tabular}

\subsubsection*{Regularity functions}

\begin{tabular}{ll}
  $m_1(r)$, $m_2(r)$     & Variance scaling functions\index{variance scaling function}; see~\eqref{E:m1-m2}            \\[0.3em]
  $w_1(r)$, $w_2(r)$     & Modulus of continuity functions\index{modulus of continuity!functions}; see~\eqref{E:w1w2}        \\[0.3em]
  $d_1(t,s)$, $d_2(x,y)$ & Canonical metrics\index{canonical metric}: $\E[(u(t,x_0)-u(s,x_0))^2]^{1/2}$, etc. \\[0.3em]
\end{tabular}

% \include{Complement-proof-Int_Q}
% \include{Complement-proof-M_1}
% \include{Complement-proof-M_2}

% and Davar's paper~\cite{khoshnevisan.sanz-sole:22:optimal} for references.

% \begin{thebibliography}{999} \end{thebibliography}
% \printbibliography[title={References}]

\addcontentsline{toc}{section}{References}
\bibliographystyle{amsrefs}
\bibliography{All,dalang-sanz-sole-26-stochastic}

\cleardoublepage
\addcontentsline{toc}{section}{Index}
\printindex
\end{document}